\documentclass[11pt]{amsart}
\usepackage{amssymb, latexsym, graphicx, mathrsfs, enumerate, amsmath, amsthm, textcomp, url, tikz-cd,bbm}
\usepackage[T1]{fontenc}
\usepackage{mathdots}
\usepackage{comment}
\usepackage[toc,page]{appendix}
\usepackage{hyperref}
\hypersetup{
    colorlinks=true,
    linkcolor=blue,
    filecolor=magenta,      
    urlcolor=cyan,
}

\usepackage{color}

\input xy
\xyoption{all}

\allowdisplaybreaks
\allowdisplaybreaks

\numberwithin{equation}{section}
\newtheorem{theorem}{Theorem}[section]
\newtheorem{proposition}[theorem]{Proposition}
\newtheorem{corollary}[theorem]{Corollary}
\newtheorem{lemma}[theorem]{Lemma}
\newtheorem{conjecture}[theorem]{Conjecture}

\newtheorem*{remark*}{Remark}

\theoremstyle{definition}
\newtheorem{definition}[theorem]{Definition}

\newtheorem{remark}[theorem]{Remark}
\newcommand{\vpi}{\varphi}
\newcommand{\mfo}{\mathfrak{o}}
\newcommand{\mfu}{\mathfrak{g}}
\newcommand{\wn}{\widetilde{n}}
\newcommand{\mfof}{\mathfrak{o}_{\widetilde{F}}}
\newcommand{\kf}{\kappa_{\widetilde{F}}}
\newcommand{\T}{\mathcal{T}}
\newcommand{\JT}{\mathcal{JT}}

\newcommand{\sog}{\mathcal{SO}_{\gamma}}

\newcommand{\gl}{\mathfrak{gl}}
\newcommand{\glno}{\mathfrak{gl}_{n, \mathfrak{o}}}
\newcommand{\GL}{\mathfrak{gl}}
\newcommand{\End}{\mathrm{End}}
\newcommand{\cL}{\mathcal{L}}
\newcommand{\Z}{\mathbb{Z}}
\newcommand{\uGr}{\underline{G}_{\gamma}}
\newcommand{\ord}{\mathrm{ord}}
\newcommand*{\Rom}[1]{\expandafter\@slowromancap\romannumeral #1@}
\begin{document}

\title[Stable orbital integrals for classical  Lie algebras and smooth integral models]
{Stable orbital integrals for  classical  Lie algebras and smooth integral models}


\author[Sungmun Cho]{Sungmun Cho}
\author[Taeyeoup Kang]{Taeyeoup Kang}
\author[Yuchan Lee]{Yuchan Lee}
\thanks{The authors are supported by  Samsung Science and Technology Foundation under Project Number SSTF-BA2001-04.}

\address{Sungmun Cho \\  Department of Mathematics, POSTECH, 77, Cheongam-ro, Nam-gu, Pohang-si, Gyeongsangbuk-do, 37673, KOREA}

\email{sungmuncho12@gmail.com}

\address{Taeyeoup Kang \\  Department of Mathematics, POSTECH, 77, Cheongam-ro, Nam-gu, Pohang-si, Gyeongsangbuk-do, 37673, KOREA}

\email{taeyeoupkang@gmail.com}

\address{Yuchan Lee \\  Department of Mathematics, POSTECH, 77, Cheongam-ro, Nam-gu, Pohang-si, Gyeongsangbuk-do, 37673, KOREA}

\email{yuchanlee329@gmail.com}

\maketitle

\begin{abstract}
A main goal of this paper is to introduce a new description of the stable orbital integral for a regular semisimple element and  for the unit element of the Hecke algebra in the case of $\mathfrak{gl}_{n,F}$, $\mathfrak{u}_{n,F}$, and $\mathfrak{sp}_{2n,F}$, by assigning a certain stratification  and then smoothening each stratum, where $F$ is a non-Archimedean local field of any characteristic.

  As  applications,  we will provide a closed formula for the stable orbital integral for $\mathfrak{gl}_{2,F}$, $\mathfrak{gl}_{3,F}$, and $\mathfrak{u}_{2,F}$.
We will also  provide a lower bound for the stable orbital integral for $\mathfrak{gl}_{n,F}$, $\mathfrak{u}_{n,F}$, and $\mathfrak{sp}_{2n,F}$ with all $n$.
Finally we will propose conjectures that our lower bounds are optimal in a sense of the second leading term for $\mathfrak{gl}_{n,F}$ and the first leading term for $\mathfrak{u}_{n,F}$ and $\mathfrak{sp}_{2n,F}$. 
There is a restriction about the factorization of the characteristic polynomial arising from the parabolic descent when we work with $\mathfrak{u}_{n,F}$ and $\mathfrak{sp}_{2n,F}$, whereas this assumption does not appear  in  $\mathfrak{gl}_{n,F}$ case.
\end{abstract}

\tableofcontents

\section{Introduction}
An orbital integral over a non-Archimedean local field is the volume of a compact $p$-adic manifold given by a conjugacy class of a fixed element in a reductive group (or its Lie algebra) with respect to a certain Haar measure. A traditional way to study an orbital integral is through either an analytic method to analyze the integral or a geometric method using Bruhat-Tits building. 

In this paper, we propose another geometric method to investigate an orbital integral for a  regular and semisimple element, using smoothening of a certain scheme defined over an henselian ring $R$.

\subsection{Background on  smoothening}
In this subsection, we will explain a general geometric setting.
We fix the following notations:
\[
\left\{
\begin{array}{l}
\textit{Let $F$ be a  non-Archimedean local field  of any characteristic};\\
\textit{Let $\mathfrak{o}_F$  be the ring of integers in $F$ with $\pi$ a uniformizer};\\
\textit{Let $\kappa$  be the residue field of $\mfo_F$};\\
\textit{Let $q$ be the cardinality of $\kappa$}.
\end{array}\right.
\]
We sometimes use $\mfo$ for $\mfo_{F}$, if there is no confusion.
We consider the following general setting:
\[
\left\{
\begin{array}{l}
\textit{Let $X$ be a smooth scheme defined over $\mfo$ of finite type};\\
\textit{Let $\mathbb{A}_{\mathfrak{o}}$ be an affine space defined over $\mfo$ of finite dimensional};\\
\textit{Let $\varphi$ be a morphism $\varphi: X \longrightarrow \mathbb{A}_{\mathfrak{o}}$ defined over $\mfo$
whose generic fiber  is smooth over $F$.}
\end{array}\right.
\]
For a suitable $a\in \mathbb{A}_{\mathfrak{o}}(\mfo)$, our main question is 
\[
\textit{What is the volume of $\varphi^{-1}(a)(\mfo)$?}
\]
Here, the volume form is the quotient of a canonical volume form of $X$  by that of $\mathbb{A}_{\mathfrak{o}}$ (see Section \ref{measure} for more discussion).
The easiest situation   is when $\varphi$ is smooth over $\mfo$ so that  $\varphi^{-1}(a)$ is a smooth scheme over $\mfo$.
It is well-known that the desired volume is $\frac{\#\varphi^{-1}(a)(\kappa)}{q^{\dim_{\mfo}{X}-\dim_{\mfo}{\mathbb{A}_{\mfo}}}}$ due to Weil's formula given in  \cite[Theorem 2.2.5]{Weil}.

This formula is no longer true if $\varphi$ is not smooth. 
In many  but interesting cases, the morphism $\varphi$ is  far from smoothness unfortunately. 
This fact makes the volume difficult  to understand as well as to compute.

In the following  subsections, we will introduce three specific cases of the setting $\varphi: X \longrightarrow \mathbb{A}_{\mathfrak{o}}$
 which  have rich and important applications. 
 We will briefly explain how to ``overcome'' non-smoothness issue so as to use Weil's formula, as a motivation of the idea used in this paper.
\\

\subsubsection{\textbf{Local density of a bilinear form: Congruence conditions}}

Consider
\[
\varphi: \mathfrak{gl}_{n, \mfo} \longrightarrow \mathrm{Sym}_{n, \mfo}, ~~~~~~~~~~~~m \mapsto f\circ m,
 \]
where $f$ is a quadratic form over $\mfo$ with $n$ variables whose discriminant is nonzero and $\mathrm{Sym}_{n, \mfo}$ represents the set of symmetric matrices over $\mfo$ of size $n$ so as to be an affine space of dimension $n(n+1)/2$.
Then $\varphi^{-1}(f)$ is the orthogonal group scheme stabilizing $f$, which is not  smooth in general. 
The volume of $\varphi^{-1}(f)(\mfo)$ is known to be the local density of the quadratic form $f$, which is the local factor of \textit{the celebrated Smith-Minkowski-Siegel mass formula}. 

Since $\varphi$ is highly non-smooth, we need to do ``smoothening'' of the morphism $\varphi$.
A methodology of smoothening is to assign  certain congruence conditions (at the level of $\mfo$-points) on both $\mathfrak{gl}_{n, \mfo}$ and $\mathrm{Sym}_{n, \mfo}$ to produce ``subgroups'' such that the restriction of $\varphi$,  denoted by $\varphi'$, turns out to be smooth. 
Here, ``subgroup'' means a smooth group scheme which is a subsheaf of $\mathfrak{gl}_{n, \mfo}$ on the small fppf site on $\mfo$. 
Two main points of the morphism $\varphi'$ are
\[
\left\{
\begin{array}{l}
\textit{the generic fiber of $\varphi'$ is the same as that of $\varphi$};\\
\textit{$\varphi'$ and $\varphi$ are the same at the level of $R$-points with any  \'etale $\mfo$-algebra $R$}.
\end{array}\right.
\]

We can now use Weil's formula to describe the volume of $(\varphi')^{-1}(f)(\mfo)$.
The desired volume $\varphi^{-1}(f)(\mfo)$ is then the product of the volume of $(\varphi')^{-1}(f)(\mfo)$ and the difference between two volume forms arising from $\varphi$ and $\varphi'$.

This observation was first discovered by Gan and Yu in \cite{GY}, to find the formula for the local density when $q$ is odd.
When $F$ is unramified over $\mathbb{Q}_2$, it  was completed by one of our authors in \cite{C1} for  quadratic forms and in \cite{C2} and \cite{C2'} for  hermitian forms. 
A conceptual nature of the construction of a smooth morphism $\varphi'$ associated to any bilinear form and any local field 
was provided in \cite{C3}. 
\\

\subsubsection{\textbf{Siegel series: Stratification}}

Consider
\[
\varphi: \underline{M}_{\mfo}^{\ast}(L,H_k) \longrightarrow \mathrm{Sym}_{n, \mfo}, ~~~~~~~~~~~~m \mapsto q_k\circ m,
 \]
where $(L,f)$ is a quadratic $\mfo$-lattice of rank  $n$ with nonzero discriminant, $(H_k, q_k)$ is the $k$-copies of the hyperbolic quadratic plane over $\mfo$, and $\underline{M}_{\mfo}^{\ast}(L,H_k)$ represents the set of injective maps from $L$ to $H_k$ (at the level of $R$-points with a flat $\mfo$-algebra $R$) by forgetting the inherent quadratic forms.
Then the volume of $\varphi^{-1}(f)(\mfo)$ is  the Siegel series, which is   the local factor of the Siegel-Eisenstein series with full level structure.
This has an important role in Kudla's program.

As in the above case, the morphism  $\varphi$ is  highly non-smooth so that we need to do ``smoothening'' of the morphism $\varphi$. 
Instead of ``assigning congruence conditions'', 
we stratify $\underline{M}_{\mfo}^{\ast}(L,H_k)(\mfo)$ into 
$$\underline{M}_{\mfo}^{\ast}(L,H_k)(\mfo)=\bigsqcup\limits_{L\subset L'} \underline{M}_{\mfo}^{prim}(L',H_k)(\mfo),$$
where $L'$ runs over  quadratic $\mfo$-lattices of rank $n$ including $L$ and 
$\underline{M}_{\mfo}^{prim}(L',H_k)$ represents the set of injective maps from $L'$ to $H_k$ whose reduction over $\kappa$ is  injective as well.
Then $\underline{M}_{\mfo}^{prim}(L',H_k)$ is an open subscheme of $\underline{M}_{\mfo}^{\ast}(L',H_k)$ which is a subsheaf of $\underline{M}_{\mfo}^{\ast}(L,H_k)$ on the small fppf site on $\mfo$. 

The restriction of $\varphi$ to $\underline{M}_{\mfo}^{prim}(L',H_k)$, denoted by $\varphi_{L'}$, turns out to be smooth so that Weil's formula is applicable. 
Using this process, one of the authors and Yamauchi produced a formula of the Siegel series in \cite{CY}, which serves as a key ingredient in the proof of Kudla-Rapoport conjecture in 
\cite{LZ1} and \cite{LZ2}.
\\

\subsubsection{\textbf{Stable orbital integral: Stratification $+$ Congruence conditions}}

Let $\mathfrak{g}$ be a reductive group scheme or its Lie algebra over $\mfo$ whose generic fiber $\mathfrak{g}_F$ is a simply connected reductive group or its Lie algebra over $F$, respectively. 
Let $n$ be the dimension of a maximal torus of $\mathfrak{g}_F$.
Then the Chevalley restriction theorem defines a Chevalley morphism
\[
\rho_n: \mathfrak{g}_F \longrightarrow \mathbb{A}_F^n.
\]
If $\mathfrak{g}_F$ is a classical algebraic group or its Lie algebra of type $A_n, B_n, C_n$, then $\rho_n(m)$ for $m\in \mathfrak{g}_F$ is nothing but the coefficients of the characteristic polynomial, denoted by $\chi_m(x)$, of $m$ (up to certain symmetries).

If $\mathfrak{g}_F$ is an unramified reductive group, then it allows an integral model which is a reductive group scheme $\mathfrak{g}$. 
We further assume that $\mathfrak{g}_F$ is a classical group of type $A_n, B_n, C_n$.
Since $\rho_n(m)$ is basically the coefficients of the characteristic polynomial,  
the morphism $\rho_n$ defined over $F$ is lifted to the morphism defined over $\mfo$ described as follows:

\[
\varphi_n: \mathfrak{g} \longrightarrow \mathbb{A}^n_{\mfo}, ~~~~~~ m \mapsto \textit{coefficients of $\chi_m(x)$}.
 \]
 Then the stable orbital integral for a regular and semisimple element $\gamma\in \mathfrak{g}(\mfo)$ and for the unit element of the Hecke algebra is the volume of $\varphi_n^{-1}(\chi_{\gamma}(x))(\mfo)$.
 Note that this description of the  stable orbital integral holds for any characteristic of $F$.
 We denote it by 
 \[
 \sog.
 \]
Note that our definition of $\sog$ follows that of \cite{FLN} which is different from the classical definition. 
The difference between two definitions is described in  \cite[Proposition 3.29]{FLN}  when the characteristic of $F$ is either $0$ or $>n$. 
We explain about it precisely in Section \ref{sscotn} for $\mathfrak{gl}_{n,\mfo}$ and in Section \ref{section:comparison} for $\mathfrak{u}_{n,\mfo}$ and $\mathfrak{sp}_{2n,\mfo}$.

As in the above two cases, $\varphi_n$ is highly non-smooth.
In this paper, we will work with the case  $\mathfrak{g}=\mathfrak{gl}_{n,\mfo}, \mathfrak{u}_{n,\mfo}, \mathfrak{sp}_{2n,\mfo}$, which are the Lie algebras of the general linear group, the unitary group, and the symplectic group over $\mfo$ respectively. 
Our method to do ``smoothening'' consists of both \textbf{stratification} and assigning \textbf{congruence conditions}. 
\begin{enumerate}
\item{\textbf{Stratification}}
We will first apply the method of stratification,  based on the theory of  a finitely generated module over PID. This yields a new summation formula for $\sog$.

\item{\textbf{Congruence conditions}}
After that we will assign certain congruence conditions on each stratum.
\end{enumerate}

Then we will prove that these two considerations are enough to do ``smoothening'' of $\varphi_n$ when $\mathfrak{g}=\mathfrak{gl}_{2,\mfo}, \mathfrak{gl}_{3,\mfo}, \mathfrak{u}_{2,\mfo}$, so that we provide  a closed formula for $\sog$.
For a general $n$, we  provide a lower bound for $\sog$ by using these two without any restriction in $\mathfrak{gl}_{n,\mfo}$ case and  with a few restrictions in $\mathfrak{u}_{n,\mfo}$ and  $\mathfrak{sp}_{2n,\mfo}$ cases.
Our method and main results will be explained more precisely in the following subsection.
\\

\subsection{Main results of Part 1: for \texorpdfstring{$\mathfrak{gl}_{n,\mfo}$}{gln}}\label{subsection1.2}
Let $\mathfrak{g}$ be  $\mathfrak{gl}_{n,\mfo}$ 
and let $\gamma\in \mathfrak{gl}_{n,\mfo}$ be regular and semisimple. 
Note that this condition is equivalent that the characteristic polynomial $\chi_{\gamma}(x)$ is separable.
Write $\chi_{\gamma}(x)=x^n+c_1x^{n-1}+\cdots + c_{n-1}x+c_n$.
Due to    \cite[Section 4]{Yun13}, we may and do assume that $\chi_{\gamma}(x)$ is irreducible whose reduction modulo $\pi$ is $x^n$ (cf. Proposition \ref{cor410} and Section \ref{reductionss}).
\\

\subsubsection{\textbf{Stratification}}\label{subsubsection1.2.1}
Put $\underline{G}_{\gamma}:=\varphi_n^{-1}(\chi_{\gamma}(x))$ and let $d$ be the exponential order of $c_n$.
Let us identify $\mathfrak{gl}_{n, \mfo}$ with $\End_\mfo(L)$ for a free $\mfo$-module $L$ of rank $n$. 
An element  $f \in \underline{G}_{\gamma}(\mfo)$ can be viewed as an endomorphism of $L$.
Then the image of $f$  is a sublattice $M$ of $L$ such that $[L:M]=d$. 

This observation motivates us to define the following two sets associated to any sublattice $M$ with  $[L:M]=d$;
\[
\left\{
\begin{array}{l}
\cL(L,M)(\mfo)=\{f|\textit{$f:L\rightarrow M$ is surjective}\}\subset \mathrm{End}(L);\\
O_{\gamma, \cL(L,M)}=\uGr(\mfo)\cap \cL(L,M)(\mfo)=\{f\in \cL(L,M)(\mfo)| \vpi_n(f)=\vpi_n(\gamma)\}.
\end{array}\right.
\]
The set $O_{\gamma, \cL(L,M)}$ is non-empty and open  in $\uGr(\mfo)$ and so its volume, denoted by $\mathcal{SO}_{\gamma, \cL(L,M)}$, is nontrivial.
We can also assign the type $(k_1, \cdots, k_{n-m})$ to $M$ such that $L/M \cong \mfo/\pi^{k_1}\mfo \oplus \cdots \oplus \mfo/\pi^{k_{n-m}}\mfo$, where $k_i$ is an  integer such that $1\leq k_1\leq \cdots \leq k_{n-m}$.
Then we have the following stratification:
\begin{equation}\label{stra121}
\uGr(\mfo) =\bigsqcup_{\substack{(k_1, \cdots, k_{n-m}),  \\ \sum k_i=\mathrm{ord}(c_n),\\ n-m\leq n }}
\left(\bigsqcup_{\substack{M:\textit{type(M)=} \\ (k_1, \cdots, k_{n-m})}} O_{\gamma, \cL(L,M)}\right).
\end{equation}
We prove that if $M$ and $M'$ have the same type, then the volume of $O_{\gamma, \cL(L,M)}$ is the same as the volume of $O_{\gamma, \cL(L,M')}$ in Lemma \ref{lemred1}. 
Thus we can define $\mathcal{SO}_{\gamma, (k_1, \cdots, k_{n-m})}$ to be the volume of $O_{\gamma, \cL(L,M)}$ such that the type of $M$ is $(k_1, \cdots, k_{n-m})$.\
The stratification directly yields a summation formula for $\sog$ as follows:
\begin{proposition}(Proposition \ref{stra1})\label{stra11}
We have
\[
\mathcal{SO}_{\gamma}= \sum\limits_{\substack{(k_1, \cdots, k_{n-m}),  \\ \sum k_i=\mathrm{ord}(c_n),\\ m \geq 0 }}
c_{(k_1, \cdots, k_{n-m})}\cdot
\mathcal{SO}_{\gamma, (k_1, \cdots, k_{n-m})}.
\]
Here, $c_{(k_1, \cdots, k_{n-m})}$ is the number of sublattices of type $(k_1, \cdots, k_{n-m})$ in $L$.
\end{proposition}

In Lemma \ref{counttype}, we provide a closed formula for  $c_{(k_1, \cdots, k_{n-m})}$.
\\

\subsubsection{\textbf{Refined stratification}}
Since $f\in O_{\gamma, \cL(L,M)}$ defines an $\mfo$-linear morphism from $L$ to $M$, it induces a $\kappa$-linear morphism $\bar{f}: \overline{M}\rightarrow \overline{M}$, where $\overline{M}=M/M\cap \pi L$. 
Let $V$ be the image of $\bar{f}$ as a  subspace of $\overline{M}$.
This observation also motivates us to define the following two sets associated to a subspace $V$:
\[
\left\{
\begin{array}{l}
\cL(L,M,V)(\mfo)=\{f\in \cL(L,M)(\mfo)|\bar{f}(\overline{M})=V\};\\
O_{\gamma, \cL(L,M,V)}=\uGr(\mfo)\cap \cL(L,M,V)(\mfo)=\{f\in \cL(L,M,V)(\mfo)| \vpi_n(f)=\vpi_n(\gamma)\}.
\end{array}\right.
\]

We prove that if $V$ and $V'$ are subspaces of $\overline{M}$ having the same dimension, then the volume of $O_{\gamma, \cL(L,M,V)}$ is the same as the volume of $O_{\gamma, \cL(L,M,V')}$ in Lemma \ref{lem412}.
Thus it makes sense to define $\mathcal{SO}_{\gamma, (k_1, \cdots, k_{n-m}),t}$ to be the volume of $O_{\gamma, \cL(L,M,V)}$, where $m-t$ is the dimension of $V$ as a  $\kappa$-vector space.
Here $m$ is the dimension of $\overline{M}$, where the type of $M$ is $(k_1, \cdots, k_{n-m})$. 
This yields a refined stratification of $\cL(L,M)(\mfo)$ in terms of $\cL(L,M,V)(\mfo)$ for a subspace $V$ in $\overline{M}$. 
Therefore we have a refined summation formula for $\sog$ as follows:
\begin{proposition}(Proposition \ref{stra2})\label{stra12}
We have 
\[
\mathcal{SO}_{\gamma}=\sum\limits_{\substack{(k_1, \cdots, k_{n-m}),  \\ \sum k_i=\mathrm{ord}(c_n),\\ m\geq 0 }}
c_{(k_1, \cdots, k_{n-m})}\cdot
\left(
\sum\limits_{t=1}^{n-m}d_t\cdot \mathcal{SO}_{\gamma, (k_1, \cdots, k_{n-m}),t}
\right).
\]
Here, $d_t$ is the number of subspaces $V$ of $\overline{M}$ of dimension $m-t$.
\end{proposition}
The integer $d_t$ is the cardinality of a certain  Grassmannian over a finite field.
In  Lemma \ref{lem415}, we provide a closed formula for $d_t$.

Due to the above summation formulas in Propositions \ref{stra11}-\ref{stra12}, it suffices to analyze $\mathcal{SO}_{\gamma, (k_1, \cdots, k_{n-m})}$ or $\mathcal{SO}_{\gamma, (k_1, \cdots, k_{n-m}),t}$ in order to investigate $\sog$.
In Section \ref{sec51}, we assign scheme structures to $O_{\gamma, \cL(L,M)}$ and $O_{\gamma, \cL(L,M,V)}$.
Especially,  we re-describe  $\mathcal{SO}_{\gamma, (k_1, \cdots, k_{n-m}),t}$ based on the measure arising from the inherent scheme structure in Proposition \ref{prop52}.
\\

\subsubsection{\textbf{Formula for $\sog$ with $n=2$ or $n=3$}}
For $\gamma \in \mathfrak{gl}_{n}(\mfo)$ with $ n=2 \textit{ or } n=3$, 
we assign certain congruence conditions to $O_{\gamma,\cL(L,M)}$ to make it smooth in Sections 6-7. Then Weil's formula given in  \cite[Theorem 2.2.5]{Weil} is applicable so that we obtain a series of formulas for $\sog$ with $n=2$ or $n=3$.
Before the statement of the result, we introduce two notions in the following:
\[
\left\{\begin{array}{l}
F_{\chi_{\gamma}}:=F[x]/(\chi_{\gamma}(x));\\
 S(\gamma):=[\mathfrak{o}_{F_{\chi_{\gamma}}}:\mathfrak{o}[x]/(\chi_{\gamma}(x))].
    \end{array}\right.
\]
Here, $[\mathfrak{o}_{F_{\chi_{\gamma}}}:\mathfrak{o}[x]/(\chi_{\gamma}(x))]$ is 
the relative $\mfo$-length   (see \cite[Section 2.1]{Yun13}).



\begin{theorem}(Theorems \ref{result2}-\ref{thm66})\label{copy2}
The formula for $\sog$ in $\gl_{2}(\mfo)$ is as follows.
\begin{enumerate}
    \item For an elliptic regular semisimple element $\gamma \in \gl_{2}(\mfo)$,
\[
    \mathcal{SO}_{\gamma}=\left\{\begin{array}{l l}
\frac{q+1}{q}-\frac{2}{q^{S(\gamma)+1}}  &\textit{if $F_{\chi_{\gamma}}/F$ is unramified};\\
\frac{q+1}{q}-\frac{q+1}{q^{S(\gamma)+2}}
     & \textit{if $F_{\chi_{\gamma}}/F$ is ramified}.
    \end{array}\right.
\]
    \item For a hyperbolic regular semisimple element $\gamma \in \mathfrak{gl}_{2}(\mfo)$,
\[
    \mathcal{SO}_{\gamma}=\frac{q+1}{q}.
\]
\end{enumerate}
Here, (1) holds for an arbitrary local field $F$ of any characteristic, and (2)  holds when $char(F)=0$ or $char(F)>2$.
\end{theorem}

\begin{theorem}(Theorems \ref{result3}-\ref{result4})\label{copy3}
The formula for $\sog$ in $\gl_{3}(\mfo)$ is as follows.
\begin{enumerate}
\item For an elliptic regular semisimple element $\gamma \in \gl_{3}(\mfo)$, 
\[\mathcal{SO}_{\gamma}=\left\{\begin{array}{l l}
\frac{(q+1)(q^{2}+q+1)}{q^{3}}-\frac{3(q^{2}+q+1)}{q^{d'+3}}+\frac{3}{q^{3d'+3}}&\textit{if }F_{\chi_{\gamma}}/F\textit{ is unramified and }S(\gamma)=3d';\\
\frac{(q+1)(q^{2}+q+1)}{q^{3}}-\frac{(2q+1)(q^{2}+q+1)}{q^{d'+4}}+\frac{q^{2}+q+1}{q^{3d'+5}} & \textit{if }F_{\chi_{\gamma}}/F\textit{ is ramified and }S(\gamma)=3d';\\
\frac{(q+1)(q^{2}+q+1)}{q^{3}}-\frac{(q+2)(q^{2}+q+1)}{q^{d'+4}}+\frac{q^{2}+q+1}{q^{3d'+6}} & \textit{if }F_{\chi_{\gamma}}/F\textit{ is ramified and }S(\gamma)=3d'+1.\\
\end{array}\right.
\]
Here $S(\gamma)$ cannot be of the form $3d'+2$ (see Proposition \ref{da3}).
\\

\item If $\chi_{\gamma}(x)$ involves an irreducible quadratic polynomial as a factor,  then the irreducible quadratic  factor of $\chi_{\gamma}(x)$ is expressed as $\chi_{\gamma'}(x)\in \mfo[x]$ for a certain  $\gamma'\in \mathfrak{gl}_{2}(\mfo)$.  
We then have the following formula:
\[
\mathcal{SO}_{\gamma}=\left\{\begin{array}{l l}
\frac{(q+1)(q^{2}+q+1)}{q^{3}}-\frac{2(q^{2}+q+1)}{q^{S(\gamma')+3}}& \textit{if a direct summand of $F_{\chi_{\gamma}}/F$ involves unramified quadratic};\\
\frac{(q+1)(q^{2}+q+1)}{q^{3}}-\frac{(q+1)(q^2+q+1)}{q^{S(\gamma')+4}} & \textit{if a direct summand of $F_{\chi_{\gamma}}/F$ involves ramified quadratic}.
\end{array}\right.
\]
The relation between $S(\gamma')$ and $S(\gamma)$ is explained in Theorem \ref{result4}.
\\
\item For a hyperbolic regular semisimple element $\gamma \in \mathfrak{gl}_{3}(\mfo)$,
\[
\mathcal{SO}_{\gamma}=\frac{(q+1)(q^{2}+q+1)}{q^{3}}.
\]
\end{enumerate}
Here, (1) holds for an arbitrary local field $F$ of any characteristic, and (2) and (3) hold when $char(F)=0$ or $char(F)>n$.
\end{theorem}

\begin{remark}
In \cite{CL}, the above formulas play a key role to the explicit formula for orbital integrals associated with elliptic regular semisimple elements in $\mathfrak{gl}_{3}(\mfo)$  and associated with  arbitrary elements of the spherical Hecke algebra of $\mathrm{GL}_{3}(F)$.
\end{remark}

\subsubsection{\textbf{Lower bound for $\sog$ with general $n$}}
With a general $n$, it seems quite complicated to assign  congruence conditions  to $O_{\gamma, \cL(L,M,V)}$ to make it smooth.
Instead, we could find congruence conditions to make both $O_{\gamma, \cL(L,M)}$, where the type of $M$ is $(k_{1})$, and a certain open subset of $O_{\gamma, \cL(L,M,V)}$, where the type of $M$ is $(k_{1},k_{2})$ and $t=1$, smooth. This consideration yields a lower bound for $\sog$.
We also translate our lower bound with respect to another measure $d\mu$ used in \cite{Yun13}, whose associated  stable orbital integral will be denoted by $\mathcal{SO}_{\gamma,d\mu}$ (cf. Section \ref{sec321}).

\begin{theorem}(Theorems \ref{lb1}-\ref{genlb1})\label{copylb1}
Suppose  $char(F)=0$ or $char(F)>n$.
For a  regular semisimple element $\gamma \in \mathfrak{gl}_{n}(\mfo)$, a lower bound for $\sog$ is provided as follows:
\[\left\{
\begin{array}{l}
    \mathcal{SO}_{\gamma} >\frac{\#\mathrm{GL}_{n}(\kappa)}{q^{n^{2}}}\prod\limits_{i\in B(\gamma)}N'_{\gamma_{i}}(q^{d_{i}});\\
    \mathcal{SO}_{\gamma,d\mu} > q^{\rho(\gamma)}\prod\limits_{i\in B(\gamma)}N'_{\gamma_{i},d\mu}(q^{d_{i}}),
\end{array}
\right.
\]
where for an elliptic regular semisimple element $\gamma_{i}\in\mathfrak{gl}_{n}(\mfo)$,
\[\left\{
\begin{array}{l}
N'_{\gamma_{i}}(x)=\frac{x}{x-1}(1+2x^{-1}+\cdots+2x^{-\lfloor\frac{\bar{d}_{\gamma_{i}}}{2}\rfloor+1}+\varepsilon(\bar{d}_{\gamma_{i}})x^{-\lfloor\frac{\bar{d}_{\gamma_{i}}}{2}\rfloor});\\
N'_{\gamma_{i},d\mu}(x)=(x^{S_{R_{i}}(\gamma_{i})}+x^{S_{R_{i}}(\gamma_{i})-1}+\cdots+x^{S_{R_{i}}(\gamma_{i})-r_i+1})(1+2x^{-1}+\cdots+2x^{-\lfloor\frac{\bar{d}_{\gamma_{i}}}{2}\rfloor+1}+\varepsilon(\bar{d}_{\gamma_{i}})x^{-\lfloor\frac{\bar{d}_{\gamma_{i}}}{2}\rfloor}).
\end{array}\right.\]

\end{theorem}
In this formula, we use the following notations:
\begin{equation}\label{nota12}
\left\{
\begin{array}{l}
    \textit{$B(\gamma)$: an index set in bijection with the irreducible factors of $\chi_{\gamma}(x)$}~  (\textit{see Section \ref{sec322}}); \\
    \gamma=(\gamma_{i})\in Lie(L)\textit{ with }\gamma_{i}\in Lie(L_{i})~(\textit{see Section \ref{sec322}});\\
    R_{i}=\mfo[x]/(\chi_{\gamma_{i}}(x))\textit{ and }F_{\chi_{\gamma_{i}}}=F[x]/(\chi_{\gamma_{i}}(x));\\
    \textit{$\kappa_{R_{i}}$ and $\kappa_{F_{\chi_{\gamma_{i}}}}$ are the residue fields of $R_{i}$ and $F_{\chi_{\gamma_{i}}}$, respectively};\\
    d_{i}=[\kappa_{R_{i}}:\kappa], ~~~~~ r_{i}=[\kappa_{F_{\chi_{\gamma_{i}}}}:\kappa_{R_{i}}];\\
    \rho(\gamma)=S(\gamma)-\sum\limits_{i\in B(\gamma)}S(\gamma_{i}) ;\\
    S_{R_{i}}(\gamma_{i})=[\mfo_{F_{\chi_{\gamma_{i}}}}:R_{i}]_{R_{i}}, ~~~~~ \textit{where $[\mfo_{F_{\chi_{\gamma_{i}}}}:R_{i}]_{R_{i}}$ is the relative $R_{i}$-length};\\
     \varepsilon(\bar{d}_{\gamma_i})=\left \{ \begin{array}{l l} 1 &\textit{if }\bar{d}_{\gamma_i}\textit{ is even};\\2 & \textit{if }\bar{d}_{\gamma_i}\textit{ is odd};
\end{array}
\right.
\\
\textit{$\bar{d}_{\gamma_{i}}$ is the integer defined in the paragraph just before Remark  \ref{rmk89}}.
\end{array}
\right.
\end{equation}


\textit{}

\subsection{Comparison of the lower bound with the  bounds of \texorpdfstring{\cite{Yun13}}{Yun's}:  for $\mathfrak{gl}_{n,\mfo}$}

\subsubsection{Yun's bounds}

 One of the main results in \cite{Yun13} is to give lower and  upper bounds for $\mathcal{SO}_{\gamma,d\mu}$. 
 His method is based on a functional equation of the Dedekind Zeta function for an order.
We will describe Yun's bounds following of  \cite[Theorem 1.5]{Yun13}, in order to compare them with our lower bound.

For a regular semisimple element $\gamma \in \mathfrak{gl}_{n}(\mfo)$, Yun's bounds are given as follows:
\[
q^{\rho(\gamma)}\prod_{i \in B(\gamma)}N_{\gamma_{i},d\mu}(q^{d_{i}})\leq \mathcal{SO}_{\gamma,d\mu} \leq q^{\rho(\gamma)}\prod_{i \in B(\gamma)}M_{\gamma_{i},d\mu}(q^{d_{i}}),
\]
where
\[\left\{
\begin{array}{l }
    N_{\gamma_{i},d\mu}(x)=\left\{
\begin{array}{l l}
    x^{S_{R_{i}}(\gamma_i)}+x^{S_{R_{i}}(\gamma_i)-1}+\cdots+x^{S_{R_{i}}(\gamma_i)-r_i+1}+r_i&\textit{if }r_i\leq S_{R_{i}}(\gamma_i);\\
    x^{S_{R_{i}}(\gamma_i)}+x^{S_{R_{i}}(\gamma_i)-1}+\cdots+x+S_{R_{i}}(\gamma_i)+1 &\textit{if }r_i>S_{R_{i}}(\gamma_i);
\end{array}\right.\\
M_{\gamma_{i},d\mu}(x)=\sum\limits_{|\lambda|\leq S_{R_{i}}(\gamma_i), m_{1}(\lambda)<r_i}x^{S_{R_{i}}(\gamma_i)-l(\lambda)}+\sum\limits_{S_{R_{i}}(\gamma_i)-r_i\leq |\lambda|<S_{R_{i}}(\gamma_i)}x^{|\lambda|-l(\lambda)}.
\end{array}\right.
\]
Here notations are as explained in (\ref{nota12}) and unexplained notations in $M_{\gamma_{i},d\mu}(x)$ can be found at Section 1.4 in \cite{Yun13}.

Yun's bounds are controlled by two polynomials $N_{\gamma_{i},d\mu}(x)$ and $M_{\gamma_{i},d\mu}(x)$. Both have the same leading term 
$x^{S_{R_{i}}(\gamma_i)}$ but the coefficients of the second leading terms are quite different.

\subsubsection{\textbf{Comparison; the second leading terms}}\label{sec132}
\begin{enumerate}
\item{Our lower bound}

Our lower bound in Theorem \ref{copylb1} is also controlled by a rational function  $N'_{\gamma_{i},d\mu}(x)$.
The leading term of $N'_{\gamma_{i},d\mu}(x)$ is 
$x^{S_{R_{i}}(\gamma_i)}$, which is the same as the case of Yun's bound \footnote{It should be since Yun's leading term is optimal.}.

Since the case with $S_{R_{i}}(\gamma_i)=0$ is trivial, 
we suppose that $S_{R_{i}}(\gamma_i)\geq 1$.
We define $K_i$ to be the unramified extension over $F$ in $F_{\chi_{\gamma_i}}$ which corresponds to the residue field extension $\kappa_{R_i}/\kappa$ (cf. Section \ref{reductionss}), and so $F_{\chi_{\gamma_i}}/K_i$ is totally ramified if and only if $r_i=1$.
The second leading term of  $N'_{\gamma_{i},d\mu}(x)$ is then described as follows:
\begin{itemize}
    \item When $\bar{d}_{\gamma_{i}}=2$,
        \[\left\{
        \begin{array}{l l}
        x^{S_{R_{i}}(\gamma_i)-1} & \textit{if $r_i=1$, or equivalently  $F_{\chi_{\gamma_{i}}}/K_i$ is totally ramified};\\
        2x^{S_{R_{i}}(\gamma_i)-1}  & \textit{otherwise}.
        \end{array}\right.
        \]    
    \item When $\bar{d}_{\gamma_{i}}\geq 3$
        \[\left\{
        \begin{array}{l l}
         2x^{S_{R_{i}}(\gamma_i)-1} & \textit{if $r_i=1$, or equivalently  $F_{\chi_{\gamma_{i}}}/K_i$ is totally ramified};\\
        3x^{S_{R_{i}}(\gamma_i)-1}  & \textit{otherwise}.
        \end{array}\right.
        \]

\item If $\bar{d}_{\gamma_{i}}=1$, then $S_{R_{i}}(\gamma_{i})=0$ and so  $N_{\gamma_{i},d\mu}'=N_{\gamma_{i},d\mu}=1$ by (\ref{r=1}).
\end{itemize}

Note that the second leading term   is due to the presence of   $\left\{
        \begin{array}{l l}
        \varepsilon(\bar{d}_{\gamma_{i}})x^{-1} & \textit{when $d_{\gamma_{i}}=2$};\\
        2x^{-1}  & \textit{when $d_{\gamma_{i}}\geq 3$}
        \end{array}\right.
        $ in the description of  $N'_{\gamma_{i},d\mu}(x)$, which appears in the right side of Theorem \ref{copylb1}.
\\

\item{Yun's lower bound}

The second leading term of $N_{\gamma_{i},d\mu}(x)$ in Yun's lower bound is  described as follows:
$$\left\{
        \begin{array}{l l}
        0 & \textit{if $r_i=1$, or equivalently  $F_{\chi_{\gamma_{i}}}/K_i$ is totally ramified};\\
        x^{S_{R_{i}}(\gamma_i)-1}  & \textit{otherwise}.
        \end{array}\right.$$
Therefore, our result improves the second leading term in \cite{Yun13}.

By using a similar argument used in the comparison of the second leading terms, we can also deduce that other remaining terms  of $N'_{\gamma_{i},d\mu}(x)$ in Theorem \ref{copylb1} improve the corresponding terms of $N_{\gamma_{i},d\mu}(x)$ in \cite{Yun13}.
\\

\item{Yun's upper bound}

The second leading term  of $M_{\gamma_{i},d\mu}(x)$  in Yun's upper bound is described as follows:
\begin{itemize}
    \item When $S_{R_{i}}(\gamma_{i})=1 \textit{ or }2$
    \[\left\{
    \begin{array}{l l}
        x^{S_{R_{i}}(\gamma_{i})-1} & \textit{if $r_i=1$, or equivalently  $F_{\chi_{\gamma_{i}}}/K_i$ is totally ramified};\\
        2x^{S_{R_{i}}(\gamma_{i})-1} & \textit{otherwise}.
    \end{array}
    \right.
    \]
    \item When $S_{R_{i}}(\gamma_{i})\geq 3$
    \[\left\{
    \begin{array}{l l}
        (S_{R_{i}}(\gamma_{i})-1)x^{S_{R_{i}}(\gamma_{i})-1} &  \textit{if $r_i=1$, or equivalently  $F_{\chi_{\gamma_{i}}}/K_i$ is totally ramified};\\
        S_{R_{i}}(\gamma_{i})x^{S_{R_{i}}(\gamma_{i})-1} & \textit{otherwise}.
    \end{array}
    \right.
    \]
\end{itemize}
By comparing this with our lower bound, we have the following interpretations:
\begin{enumerate}
\item{The case of $S_{R_{i}}(\gamma_{i})=\bar{d}_{\gamma_{i}}=2$}

If $S_{R_{i}}(\gamma_{i})=\bar{d}_{\gamma_{i}}=2$, then the second leading term in our lower bound coincides with the second leading term of Yun's upper bound, which is bigger than that of Yun's lower bound. 
Therefore, our second leading term in this case is \textbf{optimal}.

\item{The general case}

 By the direct calculation for $S_{R_{i}}(\gamma_{i})$, we can deduce that $S_{R_{i}}(\gamma_{i})\geq \bar{d}_{\gamma_{i}}$ with $n\geq 4$.
For example, if $n$ is prime, then $S_{R_{i}}(\gamma_{i})$ is bigger than $\frac{n-1}{2}\cdot \bar{d}_{\gamma_{i}}$ (cf. Remark \ref{rmk74}).
Therefore, $S_{R_{i}}(\gamma_{i})$ grows faster than 
 $\bar{d}_{\gamma_{i}}$.

 In conclusion, the second leading term of Yun's upper bound is much bigger than  the second leading term of our lower bound. 
Indeed we expect that our second leading term  is optimal.
This will be explained  in Section \ref{subsec1.5} more carefully.
\end{enumerate}

\end{enumerate}

\subsection{Main results of Part 2: for \texorpdfstring{$\mathfrak{u}_{n,\mfo}$}{un} and \texorpdfstring{$\mathfrak{sp}_{2n,\mfo}$}{sp2n}}
In Part 2, we extend our argument used in part 1 (that is, reduction, stratification, and smoothening of each stratum) to classical Lie algebras. 
This argument requires two conditions:
\begin{itemize}
    \item the stable orbital  should be determined by the characteristic polynomial $\chi_{\gamma}(x)$;
    \item the reduction of $\chi_{\gamma}(x)$ modulo $\pi$ should be of the form $x^n$ (or $x^{2n}$ for $\mathfrak{sp}_{2n,\mfo}$ case). 
\end{itemize}
The first condition holds when $\mathfrak{g}$ is of type $A_n, B_n$, or $C_n$ since the stable orbit for $\mathfrak{so}_{2n}$ of type $D_n$ requires the additional invariant, called Phaffian (when $char(F)\neq 2$).
See \cite[Theorems 1.2-1.3]{CR}.
On the other hand, our stratification is based on the determinant of $\gamma$, regular and semisimple element in $\mathfrak{g}(\mfo)$. 
Since the determinant of an element in $\mathfrak{so}_{2n+1}$ is zero, we will exclude the case of type $B_n$ so as to  work with $\mathfrak{g}=\mathfrak{u}_{n,\mfo}$ or $\mathfrak{g}=\mathfrak{sp}_{2n,\mfo}$ and with a regular and semisimple element $\gamma\in \mathfrak{g}(\mfo)$.
Note that this condition is equivalent that the characteristic polynomial $\chi_{\gamma}(x)$ is separable (cf. \cite[Section 3.1]{Gor22} or \cite[Sections 6-7]{Gro05}).

For the second condition, it is necessary to have reduction as in the case of $\mathfrak{gl}_{n,\mfo}$, which is provided in Proposition \ref{cor410} and Section \ref{reductionss}.
However, the parabolic descent yields that the stable orbital integral is reduced to the case when $\gamma$ is elliptic whose characteristic polynomial is not necessarily irreducible.
We explain the parabolic descent explicitly in terms of the factorization of a characteristic polynomial in 
Propositions \ref{lem:par_des_sorb}-\ref{pro:par_des_sorb}.

When $\mathfrak{g}=\mathfrak{u}_{n,\mfo}$ and $\chi_{\gamma}(x)$ is irreducible, we prove that the stable orbital integral is reduced to the case that the reduction of $\chi_{\gamma}(x)$ modulo $\pi$ is $x^n$ in Proposition \ref{red:result1}, Corollary \ref{red:result2}, and Lemma \ref{lem:invarianct_translation}. 

\begin{remark}
    We initially obtained  result for $\mathfrak{u}_{n,\mfo}$ but realized that majority of arguments  work for $\mathfrak{sp}_{2n,\mfo}$ with a few restrictions without serious modification.
This is why the case for $\mathfrak{sp}_{2n,\mfo}$  requires additional assumptions.

For example, it is necessary to prove that $\sog$ is invariant under the translation (i.e. $\mathcal{SO}_{\gamma}=\mathcal{SO}_{\gamma+c}$ for a constant matrix $c\in \mathfrak{g}(\mfo)$).
This is proved in Lemma \ref{lem:invarianct_translation} for $\mathfrak{u}_{n,\mfo}$ but is nonsense for $\mathfrak{sp}_{2n,\mfo}$ since a constant matrix is not an element of $\mathfrak{sp}_{2n,\mfo}$.
This is why $\sog$  for the $\mathfrak{u}_{n,\mfo}$ case when $\chi_{\gamma}(x)$ is irreducible is reduced to the case when the reduction of $\chi_{\gamma}(x)$ modulo $\pi$ is $x^n$, whereas we need extra assumption beyond irreducibility of $\chi_{\gamma}(x)$ that the reduction of $\chi_{\gamma}(x)$ modulo $\pi$ is $x^{2n}$ for the $\mathfrak{sp}_{2n,\mfo}$ case. 

A more general situation for the $\mathfrak{sp}_{2n,\mfo}$ case will be  treated in our subsequent work.
\end{remark}

\subsubsection{\textbf{Stratification}}
To simplify notation, 
let $\mathfrak{g}=\mathfrak{u}_{n,\mfo}$ associated with a unimodular hermitian lattice $(L, h)$.
Here $L$ is a free $\mfo_E$-module having the hermitian form $h$, where $E$ is an unramified and quadratic field extension of $F$.
Write $\chi_{\gamma}(x)=x^n+c_1x^{n-1}+\cdots + c_{n-1}x+c_n$ which is irreducible in $\mfo_E[x]$ and whose reduction modulo $\pi$ is $x^n$.
Let $d_i$ be the exponential order of $c_i$.

We still have the stratification of $\uGr(\mfo)$ as in Equation (\ref{stra121}). 
However, the volumes of two strata $O_{\gamma, \cL(L,M)}$ and $O_{\gamma, \cL(L,M')}$ could be different  although $type(M)=type(M')$.
Thus  Proposition \ref{stra11} does not hold anymore. 

To solve it, we introduce additional invariant of $M$, called \textit{the Jordan type} and denoted by $\JT(M)$,   which comes from the exponential orders of the Jordan splitting (i.e. diagonalization of the hermitian form $h$) associated to the hermitian lattice $(M, h|_M)$ in Definition \ref{def:tjt}.
Let $\mathcal{SO}_{\gamma, M}$ be the volume of $O_{\gamma, \cL(L,M)}$ with respect to the measure associated to the Chevalley morphism (cf. Definition \ref{def:stableorbital} and Equation (\ref{equation:som})).

\begin{conjecture}
For sublattices  $M$ and $M'$ of $L$, if they are of the same type and of the same Jordan type,
then $\mathcal{SO}_{\gamma, M}=\mathcal{SO}_{\gamma, M'}$.
\end{conjecture}

A major issue, which was not appeared in the $\mathfrak{gl}_n$ case, is that it is difficult to choose a basis of $L$ which reflects both properties of a type and of a Jordan type.
For example, If the type of $M$ is $(d_n)$, then there exists a basis $(e_1, \cdots, e_n)$ of $L$ such that $(e_1, \cdots, e_{n-1}, \pi^{d_n} e_n)$ is a basis for $M$. But this basis  does not yield a Jordan splitting of $h|_{M}$. 

Nonetheless, when the type of $M$ is $(d_n)$, then we could find a suitable basis for $L$ in Section \ref{section:jordan}. 
Using this basis, we prove exactly when $O_{\gamma, \cL(L,M')}$ is non-empty, purely in terms of the Jordan type of $M$ in Propositions  \ref{lem:Jordan_of_type_m}-\ref{corcounting} and Theorem \ref{theorem:smoothness}. 
Furthermore, we prove the explicit formula for $\mathcal{SO}_{\gamma, M}$ in  Proposition \ref{cor:610}, which only depends on $d_{n-1}$ and $d_n$.
This formula directly yields that the above conjecture is true.

In Proposition \ref{cor:Sab}, we prove the formula for the number of sublattices of $L$ whose type is $(d_n)$ and whose Jordan type is fixed such that $O_{\gamma, \cL(L,M')}$ is non-empty.
This directly yields a lower bound for $\sog$ as in the case of $\mathfrak{gl}_{n, \mfo}$.

\begin{remark}
 The condition about the characteristic of a field or of the residue field is quite sensitive in each step. For example, the parabolic descent for $\mathfrak{u}_{n,\mfo}$ with respect to the quotient measure requires that the residue characteristic is not $2$, whereas there is no restriction for $\mathfrak{sp}_{2n,\mfo}$ (cf. Section \ref{subsection12.3}).
 The translation of the parabolic descent into the geometric measure requires further assumption that $char(F) = 0$ or $char(F)>n$ (cf. Proposition \ref{pro:par_des_sorb}).
 
    We thoroughly checked this and explained  which place requires a specific assumption on the characteristic. 
\end{remark}

\subsubsection{\textbf{Lower bound for $\sog$  and the formula when $\mathfrak{g}=\mathfrak{u}_{2,\mfo}$}}

Let $\mathfrak{g}=\mathfrak{u}_{n,\mfo}$ for simplication. 
To explain our main results, we need a factorization of $\chi_{\gamma}(x)$ over $F$ as the $\mathfrak{gl}_n$ case which was explained at the first paragraph of Section \ref{subsection1.2}.
But,  $\chi_{\gamma}(x)\in \mfo_E[x]$ and thus we need a modification of $\chi_{\gamma}(x)$ as an object over $\mfo$.

To do that, we first understand the factorization in terms of rings and the Galois descent.
It is easy to see that $E[x]/(\chi_{\gamma}(x))$ is stable under the nontrivial Galois action $\sigma$ on $E/F$ so that the ring $\left(E[x]/(\chi_{\gamma}(x)\right)^{\sigma}$ makes sense. 
Then we can do the factorization of $\left(E[x]/(\chi_{\gamma}(x)\right)^{\sigma}$ as the product of  finite field extensions of $F$ by the Chinese remainder theorem. 

This discussion is  translated with respect to the factorization of a certain polynomial defined over $\mfo_F$ associated with  $\chi_{\gamma}(x)$.
Choose $\alpha \in \mathfrak{o}_E^\times$ such that $\alpha + \sigma(\alpha) = 0$ (for the existence, see  the proof of Lemma \ref{lem:c+sigma(c)_bij}).
Note that $\alpha\gamma$ is not contained in $\mathfrak{u}_n(\mfo)$. 
Nonetheless, the characteristic polynomial of $\alpha\gamma$, denoted by $\chi_{\alpha\gamma}(x)$, makes sense and has coefficients in $\mfo$. Furthermore we have that  $\left(E[x]/(\chi_{\gamma}(x)\right)^{\sigma}=F[x]/(\chi_{\alpha\gamma}(x))$ (cf. Equation (\ref{equation:F_gamma_iso})). 

Now, we consider the following factorization (cf. Section \ref{desc_field}):
\[
\chi_{\alpha\gamma}(x)=\prod_{i\in B(\gamma)}\psi_i(x)=\left(\prod_{i\in B(\gamma)^{irred}}\psi_i(x)\right) \times \left(\prod_{i\in B(\gamma)^{split}}\psi_i(x)\right),
\]
where $\psi_i(x)(\in \mfo[x])$ is irreducible over $F$, and  where $i\in B(\gamma)^{irred}$ if and only if $\psi_i(x)$ is irreducible over $E$ as well.
Then the degree of $\psi_i(x)$ with $i \in B(\gamma)^{split}$ is always even by the Galois descent. 
We write 
\[
m=\sum\limits_{i \in B(\gamma)^{irred}}deg(\psi_i(x)) ~~~  \textit{ and  } ~~~~  2l_i=deg(\psi_i(x)) \textit{ for } i \in B(\gamma)^{split}.
\]
With respect to this factorization, then we may and do write $\gamma$ as follows (cf. Section \ref{section:matdescofr}):
\[
\gamma = \gamma_0 \oplus \bigoplus\limits_{i \in B(\gamma)^{split}}\begin{pmatrix} g_i & 0 \\ 0 & -{}^t\sigma(g_i)   \end{pmatrix},  ~~~~  \textit{$\gamma_0\in \mathfrak{u}_{m, \mfo}(\mfo)$ and $g_i\in \mathfrak{gl}_{l_i, \mfo_E}(\mfo_E)$ such that } 
\]
\[
\left\{
\begin{array}{l}
\textit{the characteristic polynomial of } \alpha\gamma_0 = \prod\limits_{i\in B(\gamma)^{irred}}\psi_i(x);\\
\textit{the characteristic polynomial of }\alpha\begin{pmatrix}  g_i & 0 \\ 0 & -{}^t\sigma(g_i)   \end{pmatrix}  = \psi_i(x).
 \end{array}
\right.
\]
This is an interpretation of $\gamma$ as an element of the Levi subalgebra (cf. Section \ref{sec:desc_Levi}). 
Here we need to suppose that  $char(\kappa)>2$ if $n$ is even  due to the existence of a Kostant section (cf. \cite{Lee23}). If $\chi_{\alpha\gamma}(x)$ is irreducible over $F$, i.e. $B(\gamma)$ is a singleton, or if $n$ is odd,  then this condition  is unnecessary.


When  $\mathfrak{g}=\mathfrak{sp}_{2n,\mfo}$ with no restriction on the characteristic of a field, we can also define index sets $B(\gamma)^{irred}$ and $B(\gamma)^{split}$ together with a  factorization, the involution $\sigma$ on $F[x]/(\chi_{\gamma}(x))$ so that $\left(F[x]/(\chi_{\gamma}(x))\right)^\sigma$ makes sense, and so on.
For a detailed discussion, see Sections \ref{desc_field} and \ref{section:matdescofr}-\ref{sec:desc_Levi}, which explain all contents in a uniform way combining $\mathfrak{u}_{n,\mathfrak{o}}$ and $\mathfrak{sp}_{2n,\mathfrak{o}}$. 
We state our main result on a lower bound as follows:

\begin{theorem}\label{thm:intro_Part2}(Theorems \ref{thm:lowerbound_geom}-\ref{thm:lowerbound_SO})
        Suppose that $char(F)=0$ or $char(F)>n$ and  that $B(\gamma)^{irred}$ is a singleton.
        Then  we have the following bounds:
        \[ 
\mathcal{SO}_{\gamma}>
\begin{cases}
    \frac{\#\mathrm{U}_n(\kappa)}{(1+q^{-l})q^{n^2}} \cdot \prod\limits_{i \in B(\gamma)^{split}} N'_{g_{i}}(q^{2d_{i}}) &\textit{if $\mathfrak{g} = \mathfrak{u}_{n,\mathfrak{o}}$};\\
\frac{\#\mathrm{Sp}_{2n}(\kappa)}{q^{n(2n+1)}} \cdot \prod\limits_{i \in B(\gamma)^{split}}N'_{g_{i}}(q^{d_{i}})   &\textit{if $\mathfrak{g} = \mathfrak{sp}_{2n,\mathfrak{o}}$ and $\overline{\chi}_\gamma(x) = x^{2n}$},
\end{cases}
\]
        Here for $\mathfrak{u}_{n,\mathfrak{o}}$,  $l$ is the degree of the irreducible factor of $\overline{\chi}_{\alpha\gamma_0}(x)$, and the other notations $d_i$ and $N'_{g_i}(-)$  with $i \in B(\gamma)^{split}$ are referred to Theorem \ref{copylb1}  (by replacing $\gamma \in \mathfrak{gl}_n(\mathfrak{o})$ with an elliptic regular semisimple element $g_i \in \mathfrak{gl}_{l_i}(\mathfrak{o}_{E})$ so that $B(g_i)$ is a singleton).

         We refer to Sections \ref{sec:trans_herm}-\ref{desc_field} for the notations of $\mathfrak{sp}_{2n, \mfo}$.
\end{theorem}
        When $\mathfrak{g} = \mathfrak{u}_{n,\mathfrak{o}}$, we need to suppose that  $char(\kappa)>2$ if $n$ is even. 
But this assumption is unnecessary if $B(\gamma)$ is a singleton (cf. Section \ref{section:matdescofr}) or  $n$ is odd.

We  also translate our lower bound with respect to another measure $d\mu$ used in \cite[Section 3.2]{Yun16}, whose associated  stable orbital integral is denoted by $\mathcal{SO}_{\gamma,d\mu}$ (cf. Section \ref{section:comparison}).
The result is stated in Theorem \ref{thm:lowerbound_SO}.

When $n=2$, we have the following closed formula.
\begin{theorem}(Theorem \ref{theorem12.1sorn=2})
    Suppose that $char(F)=0$ or $char(F)>2$. 
For $\gamma \in \mathfrak{u}_2(\mfo)$ such that $\left(E[x]/(\chi_{\gamma}(x)\right)^{\sigma}$ is a quadratic field extension of $F$, we have the following formula:
\[
\left(\mathcal{SO}_{\gamma},  \mathcal{SO}_{\gamma,d\mu}\right)=
\left\{\begin{array}{l l}
\left(\frac{q+1}{q},  q^{S(\gamma)}\right) & \textit{ if $\left(E[x]/(\chi_{\gamma}(x)\right)^{\sigma}/F$ is unramified};\\
\left(\frac{(q+1)(q^{S(\gamma)+1}-1)}{q^{S(\gamma)+2}}, \frac{q^{S(\gamma)+1}-1}{q-1}\right) & \textit{ if $\left(E[x]/(\chi_{\gamma}(x)\right)^{\sigma}/F$ is ramified}.
\end{array}\right.
\]
\end{theorem}
For the notion of $S(\gamma)$ see Definition \ref{def:Serre_inv}.
Here the first case is when $\chi_{\gamma}(x)$ is not irreducible over $E$ but $\chi_{\alpha\gamma}(x)$ is irreducible over $F$.
The second  case is when both $\chi_{\gamma}(x)$ and  $\chi_{\alpha\gamma}(x)$ are irreducible over $E$ and $F$, respectively.

\subsection{Conjecture of \texorpdfstring{$\mathcal{SO}_{\gamma}$}{SO}}\label{subsec1.5}

Let us firstly consider the $\mathfrak{gl}_n$ case, to have intuition about $\mathcal{SO}_{\gamma}$.
In the context of Equations (\ref{matrixform})-(\ref{matrixform2}), we can say that congruence conditions are most accumulated in the case of type $(k_1)$.
The secondly most accumulated case is of type $(k_1, k_2)$ with $t=1$ in a sense of our refined stratification.
Especially, $(1,k_{2})$ type is most accumulated among the types of the form $(k_{1},k_{2})$.
Furthermore,  congruence conditions assigned on the open subset of $\mathcal{SO}_{\gamma, (k_1, k_2),1}$
treated in Theorem \ref{copylb1} are most accumulated inside the case of type $(k_1, k_2)$ with $t=1$.

It seems that the more accumulated congruence conditions are, the bigger the volume is.
This is supported by our results  about the type $(k_1)$ and the type $(k_{1},k_{2})$.

More precisely,  the volume in the case of type $(k_1)$ is much bigger than the volume in the case of type $(k_1, k_2)$ (cf. Corollaries \ref{cornm1} and \ref{ineqm2}).
Furthermore, in the case of  the type $(k_{1},k_{2})$ with $k_{1}\leq k_{2}$,  the smaller  $k_{1}$ is, the bigger the volume of (an open subset of) $\mathcal{SO}_{\gamma,(k_{1},k_{2}),1}$ is (cf. Corollary \ref{ineqm2}).

Because of such observation, our lower bound might be close to the explicit value of $\sog$ for a general $n$.
In this context, we propose the following conjecture.

\begin{conjecture}\label{conj1}
        Suppose  $char(F)=0$ or $char(F)>n$.
For a  regular semisimple element $\gamma\in \mathfrak{gl}_{n}(\mfo)$ with $n\geq 3$,
\[\left\{
\begin{array}{l}
\mathcal{SO}_{\gamma}=\frac{\#\mathrm{GL}_{n}(\kappa)}{q^{n^{2}}}\prod\limits_{i\in B(\gamma)}\frac{q^{d_{i}}}{q^{d_{i}}-1}(1+\alpha(\bar{d}_{\gamma_i})q^{-d_{i}}+O(q^{-2d_{i}}))\\
\mathcal{SO}_{\gamma,d\mu}=
q^{\rho(\gamma)}\prod\limits_{i\in B(\gamma)}(q^{S(\gamma_{i})d_{i}}+q^{(S(\gamma_{i})-1)d_{i}}+\cdots+q^{(S(\gamma_{i})-r_{i}+1)d_{i}})(1+\alpha(\bar{d}_{\gamma_i})q^{-d_{i}}+O(q^{-2d_{i}})).
\end{array}
\right.
\]
where $\alpha(\bar{d}_{\gamma_i})=\left\{\begin{array}{l l}
    0 & \textit{if $\bar{d}_{\gamma_i}=0$ or $1$};\\
    1 & \textit{if $\bar{d}_{\gamma_i}=2$};\\
    2 & \textit{if $\bar{d}_{\gamma_i}\geq 3$}.
\end{array}\right.$
Here we use the notation in (\ref{nota12}).
\end{conjecture}

The conjecture is equivalent to saying that   $\alpha(\bar{d}_{\gamma_i})x^{-1}$ in the description of  $N'_{\gamma_{i},d\mu}(x)$, which appears in the right side of Theorem \ref{copylb1}, is optimal. Equivalently, the second leading term of our lower bound is optimal.

The conjecture is true in the following two cases, which serve as evidences:
\begin{enumerate}
\item{the case $n=3$;}
Theorem \ref{copy3} directly yields that 
 the conjecture is true when $n=3$.

\item{the case $S(\gamma_{i})=d_{\gamma_{i}}=2$;}
  In Section \ref{sec132}.(3), we already confirm it.
\end{enumerate}

In the $\mathfrak{u}_{n,\mfo}$ and $\mathfrak{sp}_{2n,\mfo}$ cases, we treated sublattices of type $(d_n)$. 
Based on the above  observation, we propose the following conjecture.

\begin{conjecture}\label{conj:Intro_Part2}
        Suppose that $char(F)=0$ or $char(F)>n$.
        Suppose that $B(\gamma)^{irred}$ is a singleton for a regular semisimple element $\gamma \in \mathfrak{u}_{n,\mfo}(\mfo)$ or $\gamma \in \mathfrak{sp}_{2n,\mfo}(\mfo)$.
        Then we have the followings:
        \[ 
\mathcal{SO}_{\gamma}=
\begin{cases}
    \frac{\#\mathrm{U}_n(\kappa)}{q^{n^2}(1+q^{-l})} \cdot(1+O(q^{-l}))  \prod\limits_{i \in B(\gamma)^{split}}\frac{q^{2d_{i}}}{q^{2d_{i}-2}}(1+\alpha(\bar{d}_{g_i})q^{-2d_i} + O(q^{-4d_i})) &\textit{if $\mathfrak{g} = \mathfrak{u}_{n,\mathfrak{o}}$};\\
    \frac{\#\mathrm{Sp}_{2n}(\kappa)}{q^{2n^2+n}} \cdot (1+O(q^{-1})) \cdot\prod\limits_{i \in B(\gamma)^{split}}\frac{q^{d_{i}}}{q^{d_{i}}-1}(1+\alpha(\bar{d}_{g_i})q^{-d_i} + O(q^{-2d_i}))   &\textit{if $\mathfrak{g} = \mathfrak{sp}_{2n,\mathfrak{o}}$ and $\overline{\chi}_\gamma(x) = x^{2n}$}.
    \end{cases}
    \]
           We refer to Theorem \ref{thm:intro_Part2} and Conjecture \ref{conj1} for the notations.
    \end{conjecture}   
In this conjecture, 
the appearance of $O(q^{-l})$ and $O(q^{-1})$ means that the first leading term of our lower bound is optimal.
The appearance of the big $O$-notion for $i \in B(\gamma)^{split}$ follows Conjecture \ref{conj1}.

\begin{remark}
  We translate the above conjecture in terms of the another measure $d\mu$.
        \begin{enumerate}
            \item For the case $\mathfrak{g} = \mathfrak{u}_{n,\mathfrak{o}}$, we have
            \begin{align*}
            \mathcal{SO}_{\gamma,d\mu}
            &= q^{\rho(\gamma)} \cdot \frac{1 + q^{-d}}{1+q^{-l}} \cdot (1+ O(q^{-l})) \\
            &\times\prod\limits_{i\in B(\gamma)^{split}}(q^{S(g_{i})2d_{i}}+q^{(S(g_{i})-1)2d_{i}}+\cdots+q^{(S(g_{i})-r_{i}+1)2d_{i}})(1+\alpha(\bar{d}_{g_{i}})q^{-2d_{i}}+O(q^{-4d_{i}})).
            \end{align*}
           \item For the case $\mathfrak{g} = \mathfrak{sp}_{2n,\mathfrak{o}}$,  $\overline{\chi}_{\gamma_0}(x) = x^{2n}$, and $\widetilde{F}_{\chi_{\gamma_0}}/F^\sigma_{\chi_{\gamma_0}}$ is unramified, we have
            \begin{align*}
            \mathcal{SO}_{\gamma,d\mu}
            &= q^{\rho(\gamma) - S(\psi)} \cdot (1 + q^{-d}) \cdot (1+O(q^{-1}))\\
            &\times\prod\limits_{i\in B(\gamma)^{split}}(q^{S(g_{i})d_{i}}+q^{(S(g_{i})-1)d_{i}}+\cdots+q^{(S(g_{i})-r_{i}+1)d_{i}})(1+\alpha(\bar{d}_{g_{i}})q^{-d_{i}}+O(q^{-2d_{i}})).
            \end{align*}

            \item For the case $\mathfrak{g} = \mathfrak{sp}_{2n,\mathfrak{o}}$,  $\overline{\chi}_{\gamma_0}(x) = x^{2n}$, and $\widetilde{F}_{\chi_{\gamma_0}}/F^\sigma_{\chi_{\gamma_0}}$ is ramified, we have
            \begin{align*}
            \mathcal{SO}_{\gamma,d\mu}
            &= q^{\rho(\gamma) - S(\psi)} \cdot 2 \cdot (1+O(q^{-1}))\\
            &\times\prod\limits_{i\in B(\gamma)^{split}}(q^{S(g_{i})d_{i}}+q^{(S(g_{i})-1)d_{i}}+\cdots+q^{(S(g_{i})-r_{i}+1)d_{i}})(1+\alpha(\bar{d}_{g_{i}})q^{-d_{i}}+O(q^{-2d_{i}})).
            \end{align*}
        \end{enumerate}
        Here the notations for $\mathfrak{u}_{n,\mathfrak{o}}$ are as follows:
        \[
        \begin{cases}
            E_{\chi_{\gamma_0}} = E[x]/(\chi_{\gamma_0}(x)),
            F^\sigma_{\chi_{\gamma_0}} = \left(E_{\chi_{\gamma_0}}\right)^\sigma, \textit{ and }
            R = \left(\mathfrak{o}_E[x]/(\chi_{\gamma_0}(x))\right)^\sigma; \\
            l = [\kappa_{R}:\kappa] ~ (=\textit{the degree of the irreducible factor of } \overline{\chi}_{\alpha \gamma_0}(x)); \\
            d : \textit{the inertial degree of } F^\sigma_{\chi_{\gamma_0}}/F;\\
            \rho(\gamma) = S(\gamma) - \sum\limits_{i \in B(\gamma)^{split}} S(g_i) \textit{ where } S(\gamma):=[\mathfrak{o}_{E_{\chi_\gamma}}:\mathfrak{o}_E[x]/(\chi_{\gamma}(x))] \textit{ is the relative } \mathfrak{o}_E\textit{-length}.
        \end{cases}
        \] 
        We refer to Sections \ref{sec:trans_herm}-\ref{desc_field} for the notations of $\mathfrak{sp}_{2n, \mfo}$ and Definition \ref{def:Serre_inv} for $S(\psi)$.
        The other notations $S(g_i), d_i, \alpha(\bar{d}_{g_{i}})$ for $i \in B(\gamma)^{split}$ are referred to Theorem \ref{copylb1} and Conjecture \ref{conj1}  (by replacing $\gamma$ with an elliptic regular semisimple element $g_i$ so that $B(g_i)$ is a singleton).
\end{remark}

\begin{remark}
\begin{enumerate}
\item In the $\mathfrak{sp}_{2n, \mfo}$ case, $\mathcal{SO}_{\gamma,d\mu}$ is divided into two cases whether $\widetilde{F}_{\chi_{\gamma_0}}/F_{\chi_{\gamma_0}}^\sigma$ is unramified or ramified, whereas such distinction is not appeared in $\mathcal{SO}_{\gamma}$.
This is due to the comparison between two measures (cf. Section \ref{section:comparison}).

    \item 
If $\mathfrak{g} = \mathfrak{sp}_{2n,\mathfrak{o}}$,  $\widetilde{F}_{\chi_\gamma}/F_{\chi_\gamma}^\sigma$ is ramified, and $char(\kappa) > 2$, then $\overline{\chi}_\gamma(x) = x^{2n}$ always so that the above conjecture is applicable (cf. Remark \ref{rmk:reduction_cond}).
\end{enumerate}    
\end{remark}

\subsection{Heuristics}
Orbital integrals and the associated smoothenings are much more difficult than the cases of the local density of a bilinear form and the Siegel series. The reason is that schemes appearing in the latter two cases are determined by homogeneous polynomials of degree  $\leq 2$. In the case of orbital integrals, the schemes we need to handle are determined by homogeneous polynomials of degree  $\leq n$. 

On the other hand, the smoothening process is, in some sense, determined by the differentials which are polynomials of degree reduced by $1$.
Thus the differentials appearing in the two cases of the local density of a bilinear form and the Siegel series  are linear forms so that we can use the theory of linear algebras (the theory of modules over PID, in our case). In the case of orbital integrals, the differentials are determined by polynomials of degree  $\leq n-1$.
This proves that the formula for  $n=3$ is  quite difficult since the differentials are quadratic forms so that there seems no way to apply linear algebra.

The stratification introduced in this paper has a role to reduce the degree of polynomials by $1$.
This is why we could use  the theory of modules over PID in smoothening process and  find a closed formula  when $n=3$.
To treat higher rank cases, one might need more geometric observations which enable  to reduce the degree of polynomials by $>1$. 

Our method is applicable in other cases including classical Lie algebras. This is  pursued in the subsequent work in Part \ref{part2}.

We hope that our argument might be used to Diophantine geometry of cubic polynomials of certain types.
\\

\subsection{Organizations}
This manuscript is organized into two parts: Part 1 treating $\mathfrak{gl}_{n, \mfo}$ and Part 2 treating $\mathfrak{u}_{n, \mfo}$ and $\mathfrak{sp}_{2n, \mfo}$.

After fixing notations at the beginning of Part 1, we explain the measure associated to $\underline{G}_{\gamma}$ to define the  stable orbital integral $\sog$ and  the difference between our measure and the measure used in \cite{Yun13}, following \cite{FLN},  in Section \ref{sec3}.
Section \ref{sec4} contains a few reduction steps 
and stratifications to $\underline{G}_{\gamma}(\mfo)$ by introducing $O_{\gamma, \cL(L,M)}(\mfo)$ and $O_{\gamma, \cL(L,M,V)}(\mfo)$. 
In Section \ref{sec5}, we explain a scheme theoretic description of $O_{\gamma, \cL(L,M)}$ and $O_{\gamma, \cL(L,M,V)}$ and discuss two extreme cases of type $(k_1)$ and $(k_1, \cdots, k_n)$. 
We provide a formula for $\sog$ with $n=1,2$ in Section \ref{sec6} and a formula for $\sog$ with $n=3$ in Section \ref{sec7}. 
In Section \ref{sec8} we provide  a lower bound for $\sog$ with all $n$.
  In Appendix \ref{App:AppendixA}, we explain a  proof of Theorem   \ref{thm71}, which is  a detailed process of smoothening for each stratum in the case $n=3$.
  In Appendix \ref{appc}, we provide  proofs of Theorems \ref{result3}-\ref{result4} which give the closed formula for $\sog$ with $n=3$.

The structure of Part 2 is similar to that of Part 1.
After fixing  notations at the beginning of Part 2,  we explain the measure to define the stable orbital integral  $\sog$ and the difference between two measures in Section \ref{section10soi}. 
 Section \ref{sec:parabolicdescent} is devoted to  explain a parabolic descent  in terms of the factorization of the characteristic polynomial.
 Section \ref{sec:reduction} contains a few reduction steps and stratifications.
 In Section \ref{sectypem},
we explain a scheme theoretic formulation of $\sog$ and provide a lower bound for $\sog$ with all $n$.
In Section  \ref{sectionformulaforu2}, we provide a closed formula for $\sog$ with $\mathfrak{u}_{2, \mfo}$ when $B(\gamma)$ is a singleton. 
\\

\textbf{Acknowledgments.} 
We sincerely thank Jungin Lee  to provide a better proof of Lemma \ref{counttype} and helpful comments.
We deeply thank Professor Benedict Gross and Professor Jiu-Kang Yu for their fruitful comments and encouragement on this project.

\part{Stable orbital integrals for \texorpdfstring{$\mathfrak{gl}_n$}{gln} and smoothening}\label{part1}

\section*{Notations}\label{sectionnss}

\begin{itemize}
\item Let $F$ be a  non-Archimedean local field  of any characteristic with $\mathfrak{o}_F$  its ring of integers and $\kappa$  its residue field.
Let $\pi$ be a uniformizer in $\mathfrak{o}_F$.
Let $q$ be the cardinality of the finite field $\kappa$.
If there is no confusion, we sometimes use $\mathfrak{o}$ to stand for $\mathfrak{o}_F$. 

More generally, for a finite field extension $F'$ of $F$, the ring of integers in $F'$ is denoted by $\mfo_{F'}$ and the residue field of $\mfo_{F'}$ is denoted by $\kappa_{F'}$.

\item For an element $x\in F$, the exponential order of $x$ with respect to the maximal ideal in $\mathfrak{o}$ is written by $\mathrm{ord}(x)$.

\item For an element $x\in F$, the value of $x$ is $|x|_F:=q^{-\mathrm{ord}(x)}$.
If there is no confusion, then we sometimes omit $F$ so that the value of $x$ is written as $|x|$.

\item Let $\mathrm{GL}_{n, A}$ be  the general linear group scheme defined over $A$,
 $\mathfrak{gl}_{n, A}$ be the Lie algebra of  $\mathfrak{gl}_{n, A}$,
  and $\mathbb{A}^n_A$ be the affine space of dimension $n$ defined over $A$,
   where $A$ is a commutative $\mfo$-algebra. 
If there is no confusion then we sometimes omit $A$ in the subscript to express schemes over $\mfo$. 
Thus $\mathrm{GL}_{n}$ and $\mathfrak{gl}_{n}$ stand for $\mathrm{GL}_{n, \mathfrak{o}}$ and $\mathfrak{gl}_{n, \mathfrak{o}}$, respectively.

Similarly $\mathrm{M}_n$ is the scheme over $\mfo$ representing the set of $n \times n$ matrices.



\item For $a\in R$ or $f(x)\in R[x]$ with a flat $\mfo$-algebra $R$, $\bar{a}\in R\otimes \kappa$ or $\bar{f}(x)\in R\otimes \kappa[x]$ is the reduction of $a$ or $f(x)$ modulo $\pi$, respectively.


\item Let $\chi_{\gamma}(x)\in \mathfrak{o}[x]$  be the characteristic polynomial of $\gamma \in \mathfrak{gl}_{n, \mfo}(\mfo)$.
We always write $$\chi_{\gamma}(x)=x^n+c_1x^{n-1}+\cdots + c_{n-1}x+c_n$$ with $c_i\in \mfo$.

\item Let $\Delta_{\gamma}$ be the discriminant of $\chi_{\gamma}(x)$.
 
\item By saying $\gamma\in \mathfrak{gl}_{n, \mfo}(\mfo)$ regular, we mean that the identity component of the centralizer of $\gamma$ in $\mathrm{GL}_{n, F}$ is a maximal torus.
 In particular,  $\gamma$ of being  regular and semisimple  is equivalent that $\chi_{\gamma}(x)$ has distinct roots in the algebraic closure of $F$,
equivalently  $\Delta_{\gamma}\neq 0$ (cf. \cite[Section 3.1]{Gor22} or \cite[Sections 6-7]{Gro05}).

\item For a matrix $M$, the transpose of $M$ is denoted by ${}^tM$.

\item For a rational number $a\in \mathbb{Q}$, 
$\lfloor a\rfloor$ is the largest integer which is less than or equal to $a$ and 
 $\lceil a \rceil$ is the smallest integer which is greater than or equal to $a$.

\end{itemize}


\section{Stable orbital integral}\label{sec3}
In this section, we define the stable orbital integral for a regular element in $\mathfrak{gl}_{n, \mfo}(\mfo)$ and for the unit element in the Hecke algebra, following \cite{FLN}. Note that our definition is different from the classical one which is used in \cite{Yun13}.
We will also explain the exact difference between them.

\subsection{Measure}\label{measure}

Let $\omega_{\mathfrak{gl}_{n, \mathfrak{o}}}$ and $\omega_{\mathbb{A}^n_{\mathfrak{o}}}$ be nonzero,  translation-invariant forms on   $\mathfrak{gl}_{n, F}$ and $\mathbb{A}^n_F$,
 respectively, with normalizations
$$\int_{\mathfrak{gl}_{n, \mathfrak{o}}(\mathfrak{o})}|\omega_{\mathfrak{gl}_{n, \mathfrak{o}}}|=1 \mathrm{~and~}  \int_{\mathbb{A}^n_{\mathfrak{o}}(\mathfrak{o})}|\omega_{\mathbb{A}^n_{\mathfrak{o}}}|=1.$$

Define a map 
\[
\rho_n : \gl_{n,F} \longrightarrow \mathbb{A}_F^n, ~~~ \gamma\mapsto 
\textit{coefficients of $\chi_{\gamma}(x)$}.
\]

That is,  $\rho_n(\gamma)=(c_{1}, \cdots, c_n)$ for $\chi_{\gamma}(x)=x^n+c_1x^{n-1}+\cdots + c_{n-1}x+c_n$ with $c_i\in F$.
The morphism $\rho_n$ is then representable as a morphism of schemes over $F$. 
It is well known that  $\rho_n$ is smooth at $\gamma$ if and only if $\gamma$ is regular (cf. \cite[Theorem 4.20]{Hum}).
Let $\gl_{n,F}^{\ast}$ be the smooth locus of $\rho_n$. 
 It is non-empty and a nonsingular variety since the smooth locus of a morphism  is open (and non-empty for $\rho_n$). 

From now on until the end of Part 1, $\gamma$ is a regular and semisimple element in $\glno(\mathfrak{o})$.
We define $G_{\gamma}$ to be $\rho_n^{-1}(\chi_{\gamma})$.
Since $\rho_n|_{\gl_{n,F}^{\ast}}$ is smooth  and  $G_{\gamma}$ is a subvariety of  $\gl_{n,F}^{\ast}$, 
$G_{\gamma}$ is also smooth over $F$ since smoothness is stable under base change.

\begin{definition}\label{diff1}
We will define a differential  $\omega_{\gamma}^{\mathrm{ld}}$ on $G_{\gamma}$
associated to $\omega_{\mathfrak{gl}_{n, \mathfrak{o}}}$ and $\omega_{\mathbb{A}^n_{\mathfrak{o}}}$.
This is taken from \cite[Section 3]{GY} (or  \cite[Definition 3.1]{CY}). 
Smoothness of the morphism $\rho_n : \gl_{n,F}^{\ast} \rightarrow \mathbb{A}_F^n$ induces the following short exact sequence of locally free sheaves on $\gl_{n,F}^{\ast}$ (cf.  \cite[Proposition II.5]{BRL}):
\begin{equation*}\label{eqshort}
0\rightarrow \rho_n^{\ast}\Omega_{\mathbb{A}_F^n\slash F}\rightarrow \Omega_{\gl_{n,F}^{\ast}\slash F} \rightarrow \Omega_{\gl_{n,F}^{\ast}\slash \mathbb{A}_F^n} \rightarrow 0.
\end{equation*}

This gives rise to an isomorphism
\begin{equation*}\label{eqtop}
\rho_n^{\ast}\left(\bigwedge\limits^{\mathrm{top}}\Omega_{\mathbb{A}_F^n\slash F}\right)\otimes
\bigwedge\limits^{\mathrm{top}}\Omega_{\gl_{n,F}^{\ast}\slash \mathbb{A}_F^n}\simeq 
\bigwedge\limits^{\mathrm{top}}\Omega_{\gl_{n,F}^{\ast}\slash F}.
\end{equation*}

Let $\omega_{\chi_{\gamma}}\in \bigwedge\limits^{\mathrm{top}}\Omega_{\gl_{n,F}^{\ast}\slash \mathbb{A}_F^n}(\gl_{n,F}^{\ast})$ be such that 
$\rho_n^{\ast}\omega_{\mathbb{A}^n_{\mathfrak{o}}}\otimes \omega_{\chi_{\gamma}}=\omega_{\mathfrak{gl}_{n, \mathfrak{o}}}|_{\gl_{n,F}^{\ast}}$. 
We then denote by $\omega_{\chi_{\gamma}}^{\mathrm{ld}}$  the restriction of $\omega_{\chi_{\gamma}}$ to $G_{\gamma}$.
We sometimes write $\omega_{\chi_{\gamma}}^{\mathrm{ld}}=\omega_{\mathfrak{gl}_{n, \mathfrak{o}}}/\rho_n^{\ast}\omega_{\mathbb{A}^n_{\mathfrak{o}}}$.
\end{definition}

We consider the following morphism of schemes defined over $\mfo$:
\[
\varphi_n: \mathfrak{gl}_{n, \mathfrak{o}} \longrightarrow \mathbb{A}_{\mathfrak{o}}^n, \gamma \mapsto 
\textit{coefficients of $\chi_{\gamma}(x)$}.
\]
Thus the  generic fiber of $\varphi_n$ over $F$ is   $\rho_n$.
If we put $\underline{G}_{\gamma}=\varphi_n^{-1}(\chi_{\gamma})$, then $\underline{G}_{\gamma}(\mfo)$ is an open subset of $G_{\gamma}(F)$ with respect to the $\pi$-adic topology.

\begin{definition}\label{defsoi}
The stable orbital integral for $\gamma$ and  for the unit element in the Hecke algebra, denoted by $\mathcal{SO}_{\gamma}$, is defined to be 
\[
\mathcal{SO}_{\gamma}=\int_{\uGr(\mfo)}|\omega_{\chi_{\gamma}}^{\mathrm{ld}}|.
\]
\end{definition}

We emphasize that all of $\mathcal{SO}_{\gamma}$, $\omega_{\chi_{\gamma}}^{\mathrm{ld}}$,  $\uGr$, and $G_{\gamma}$ depend  on the characteristic polynomial $\chi_{\gamma}$, not the element $\gamma$.
We sometimes use the notation $\mathcal{SO}_{\chi_{\gamma}}$ for $\sog$ when we need to emphasize the role of $\chi_{\gamma}$ (cf. Sections \ref{sec6}-\ref{sec7}).

\begin{lemma}( \cite[Lemma 3.2]{CY})\label{volumeform}
We have the following equation: 
\[
\int_{\uGr(\mfo)}|\omega_{\chi_{\gamma}}^{\mathrm{ld}}|=\lim_{N\rightarrow \infty} q^{-N\cdot \mathrm{dim} G_{\gamma}}
\#\uGr(\mfo/\pi^N \mfo),
\]
where the limit stabilizes for $N$ sufficiently large.
Here, $\mathrm{dim} G_{\gamma}=n^2-n$.
\end{lemma}
The proof is identical to that of \cite[Lemma 3.2]{CY} and so we omit it.
Although we do not use the lemma in the paper, 
this is a formulation of the orbital integral in a sense of local densities.
Gekeler used this formulation in order to compute the orbital integral for $\mathfrak{gl}_2$ in \cite{Gek}.

\subsection{Comparison of two normalizations}\label{sscotn}
Our normalization of a Haar measure follows that introduced in \cite{FLN}.
This   is different from the classical one used in  \cite{Yun13}.
In this subsection, we explain the normalization of \cite{Yun13} and then compare it with ours.

\subsubsection{Normalization of \cite{Yun13}}\label{sec321}
We first explain the normalization used in \cite{Yun13} by reproducing Section 1.3 of loc.cit. 
Let
\[
\left\{
\begin{array}{l}
\textit{$\mathrm{T}_{\gamma}$ be the centralizer of  $\gamma$ via the adjoint action in $\mathrm{GL}_{n,F}$ which is a  maximal torus};\\
\textit{$\mathrm{T}_c$ be the maximal compact subgroup of $\mathrm{T}_{\gamma}(F)$};\\
\textit{$dt$ be the Haar measure on $\mathrm{T}_{\gamma}(F)$ such that $\mathrm{vol}(dt, \mathrm{T}_c)=1$};\\
\textit{$dg$ be the Haar measure on $\mathrm{GL}_{n}(F)$ such that $\mathrm{vol}(dg, \mathrm{GL}_n(\mathfrak{o}))=1$};\\
\textit{$d\mu=\frac{dg}{dt}$ be the quotient measure defined on $\mathrm{T}_{\gamma}(F)\backslash \mathrm{GL}_n(F)$.}
\end{array}\right.
\]
The  stable  orbital integral of \cite{Yun13} uses the quotient measure $d\mu$. We denote it by $\mathcal{SO}_{\gamma, d\mu}$ to emphasize the role of the Haar measure $d\mu$.
This  is formulated as follows:
\[
\mathcal{SO}_{\gamma, d\mu}=\int_{\mathrm{T}_{\gamma}(F)\backslash \mathrm{GL}_n(F)} \mathbf{1}_{\mathfrak{gl}_n(\mathfrak{o})}(g^{-1}\gamma g) d\mu(g).
\]
Here, $\mathbf{1}_{\mathfrak{gl}_n(\mathfrak{o})}$ is the characteristic function of $\mathfrak{gl}_n(\mathfrak{o})\subset \mathfrak{gl}_n(F)$.

\subsubsection{Explicit description of \texorpdfstring{$\mathrm{T}_{\gamma}$}{Tg}}\label{sec322}
To compare two measures, we need a precise description of $\mathrm{T}_{\gamma}$ and $\mathrm{T}_c$. 
This is essentially explained at the beginning of \cite[Section 4.8]{Yun13} or \cite[Section 3.2.4]
{Yun16}. We will reproduce the relevant part following \cite[Sections 4.1 and 4.8]{Yun13}.
Let
\[
\left\{
\begin{array}{l}
\textit{$B(\gamma)$ be an index set in bijection with the irreducible factors $\chi_{\gamma, i}(x)$ of $\chi_{\gamma}(x)$};\\
\textit{$F_i$ be the finite field extension of $F$ obtained by adjoining a root  of $\chi_{\gamma, i}(x)$.}
\end{array}\right.
\]
Then the centralizer $\mathrm{T}_{\gamma}$ is identified with
\[
\mathrm{T}_{\gamma}\cong \prod_{i\in B(\gamma)} \mathrm{Res}_{F_i/F}\mathbb{G}_m.
\]

Let $\mathrm{L}=\prod_{i\in B(\gamma)}\mathrm{L}_i$ be a Levi subgroup of $\mathrm{GL}_{n, F}$ such that 
\[
\left\{
\begin{array}{l}
\textit{$\mathrm{Res}_{F_i/F}\mathbb{G}_m$ is a maximal torus of $\mathrm{L}_i$ 
 (so $\mathrm{L}_i\cong \mathrm{GL}_{n_i, F}$ with $n_i=[F_i:F]$)};\\
 \textit{$\gamma=(\gamma_i)\in Lie(L)$ with $\gamma_i\in Lie(L_i)$};\\
\textit{$\mathrm{Res}_{F_i/F}\mathbb{G}_m$ is the centralizer of $\gamma_i$ in $\mathrm{L}_i$.}
\end{array}\right.
\]
Note that this $n_{i}=[F_{i}:F]$ is different to $n_{i}$ defined in \cite{Yun13}.
 Here, $Lie(\mathrm{G})$ is the Lie algebra of $\mathrm{G}$ for any algebraic group $\mathrm{G}$.
Due to the last condition in the above description of $\mathrm{L}$,  we can denote  $\mathrm{Res}_{F_i/F}\mathbb{G}_m$ by $\mathrm{T}_{\gamma_i}$.
Then $\mathrm{T}_{\gamma_i}$ has a natural integral structure given by $\mathrm{Res}_{\mathfrak{o}_{F_i}/\mathfrak{o}_F}\mathbb{G}_m$, which we also denote by $\mathrm{T}_{\gamma_i}$.
Note that $\mathrm{Res}_{\mathfrak{o}_{F_i}/\mathfrak{o}_F}\mathbb{G}_m$ is a smooth group scheme over $\mfo_F$ 
and that $\mathrm{Res}_{\mathfrak{o}_F}^{\mathfrak{o}_{F_i}}\mathbb{G}_m(\mfo_F)=\mfo_{F_i}^{\times}$ is the maximal compact subgroup of $\mathrm{T}_{\gamma_i}(F)=F_i^{\times}$.
Therefore, 
\[
\mathrm{T}_c = \prod_{i\in B(\gamma)} \mfo_{F_i}^{\times}.
\]

\subsubsection{Comparison}
 The difference between two measures $\omega_{\chi_{\gamma}}^{\mathrm{ld}}$ and $d\mu$ is described in  \cite[Proposition 3.29]{FLN}.
 This is  also given in  \cite[Proposition 3.9]{Gor22} with more detailed explanation. 
 A group version of the comparison is described in  \cite[Equation (3.31)]{FLN}, which is precisely explained in  \cite[Theorem 3.11]{Gor22}. 
 
\begin{proposition}(\cite[Proposition 3.29]{FLN} or \cite[Proposition 3.9]{Gor22})\label{proptrans}
Suppose that $char(F) = 0$ or $char(F)>n$.
The difference between two stable orbital integrals
$\mathcal{SO}_{\gamma}$ and $\mathcal{SO}_{\gamma, d\mu}$
 is described by the equation:
\[
\mathcal{SO}_{\gamma}=
|\Delta_{\gamma}|^{1/2}\cdot
\frac{\#\mathrm{GL}_n(\kappa)q^{-\mathrm{dim}\mathrm{GL_n}}}
{\#\mathrm{T}_{\gamma}(\kappa)q^{-\mathrm{dim}\mathrm{T}_{\gamma}}}
\cdot \left(\prod_{i\in B(\gamma)}|\Delta_{F_i/F}|^{-1/2}\right)
\cdot \mathcal{SO}_{\gamma, d\mu}.
\]
Here, $\Delta_{F_i/F}$ is the discriminant of the field extension $F_i/F$.
\end{proposition}

Indeed, \cite{FLN} and \cite{Gor22} used the \textit{Weyl discriminant} of $\gamma$ in the place of $\Delta_{\gamma}$.
 In the case of a regular semisimple element $\gamma$ of $\mathfrak{gl}_n$, the Weyl discriminant of $\gamma$ is the same as the discriminant of the characteristic polynomial of $\gamma$ (cf.   \cite[Section 3.3]{Gor22}).
 

\begin{proof}
This is a reformulation of  \cite[Proposition 3.29]{FLN} (or see  \cite[Proposition 3.9]{Gor22}). Since the presentation of loc. cit. is different from our formulation, we will provide a proof of this, following \cite{Gor22}.

When $char(F) = 0$ or $char(F)>n$ (see Proposition \ref{prop:comparison_measures} for a detailed explanation), \cite[Proposition 3.9]{Gor22} gives the following relation:
\[
|\omega_{\chi_{\gamma}}^{\mathrm{ld}}|=|\Delta_{\gamma}|^{1/2}\cdot |\omega_{\mathrm{T}_{\gamma}\backslash \mathrm{GL}_n}|.
\]
Here, $\omega_{\GL_n}$ is given in \cite[Section 2.3]{Gor22},  $\omega_{\mathrm{T}_{\gamma}}$ is given in  \cite[Section 2.2]{Gor22}, and $\omega_{\mathrm{T}_{\gamma}\backslash \mathrm{GL}_n}$ is given in  \cite[Section 3.4.3]{Gor22}. 

The difference between  $\omega_{\mathrm{GL}_n}$ and $dg$, following \cite[Equation (17)]{Gor22}, is
\[
|\omega_{\mathrm{GL}_n}|=\#\mathrm{GL}_n(\kappa)q^{-\mathrm{dim}\mathrm{GL_n}}\cdot dg
\] 
and the difference between $\omega_{\mathrm{T}_{\gamma}}$ and $dt$, by  \cite[Example 2.8 and Section 2.2.2]{Gor22}, is 
\[
|\omega_{\mathrm{T}_{\gamma}}|=\left(\prod_{i\in B(\gamma)}|\Delta_{F_i/F}|^{1/2}\right)\cdot \#\mathrm{T}_{\gamma}(\kappa)q^{-\mathrm{dim}\mathrm{T}_{\gamma}}\cdot dt.
\]
The desired equation directly follows by combining these three relations.
\end{proof}

In the case that $\chi_{\gamma}(x)$ is irreducible, let $F_{\chi_{\gamma}}=F[x]/(\chi_{\gamma}(x))$ and let $\Delta_{F_{\chi_{\gamma}}/F}$ be the discriminant of the field extension $F_{\chi_{\gamma}}/F$.
The relation between $|\Delta_{\gamma}|$ and $\Delta_{F_{\chi_{\gamma}}/F}$ is described in the following proposition:

\begin{proposition}\label{propserre}
If $\chi_{\gamma}(x)$ is irreducible, then 
$|\Delta_{\gamma}|$ and $\Delta_{F_{\chi_{\gamma}}/F}$ are related by the following equation:
\[
|\Delta_{F_{\chi_{\gamma}}/F}|^{1/2}=q^{S(\gamma)}\cdot |\Delta_{\gamma}|^{1/2},
\]
 where $S(\gamma)$ is the Serre invariant, which is the relative $\mfo$-length $[\mathfrak{o}_{F_{\chi_{\gamma}}}:\mathfrak{o}[x]/(\chi_{\gamma}(x))]$ 
 (see \cite[Section 2.1]{Yun13}).
\end{proposition}
\begin{proof}
Let $R=\mfo[x]/(\chi_{\gamma}(x))$. 
Then $R$ is a free $\mfo$-module of rank $n$ having a basis  $\{1,\bar{x},\bar{x}^{2}, \cdots, \bar{x}^{n-1}\}$. For simplicity, let $\beta_{i+1}=\bar{x}^{i}$ for $0\leq i \leq n-1$ so that $\{1,\bar{x}, \cdots, \bar{x}^{n-1}\}=\{\beta_{1},\beta_{2},\cdots, \beta_{n}\}$. Since $\mfo_{F_{\chi_{\gamma}}}$ is also a free $\mfo$-module of rank $n$, we can choose a basis  $\{\alpha_{1},\alpha_{2},\cdots \alpha_{n}\}$.

Then
\[\Delta_{\gamma}=D(\beta_{1}, \cdots, \beta_{n})\textit{ and }\Delta_{F_{\chi_{\gamma}}/F}=D(\alpha_{1}, \cdots, \alpha_{n}),\] 
where $D(t_{1},\cdots,t_{n})=\mathrm{det}(\mathrm{Tr}_{\mfo_{F_{\chi_{\gamma}}}/\mfo}(t_{i}\cdot t_{j}))$. 

Note that both  $\{\alpha_{1}, \cdots, \alpha_{n} \}$ and $\{\beta_{1},  \cdots, \beta_{n}\}$ are bases of $F_{\chi_{\gamma}}$ as an $F$-vector space. 
Thus   \cite[Lemma 2.23]{Mil} yields the following relation:
\begin{align*}
  |D(\beta_{1}, \cdots, \beta_{n})|&=|\mathrm{det}(a_{ij})^{2} \cdot D(\alpha_{1}, \cdots, \alpha_{n})|\\
  &=|(\mfo_{F_{\chi_{\gamma}}}:R)|^{2}\cdot|D(\alpha_{1}, \cdots, \alpha_{n})|,
\end{align*}
where $\beta_{i}=\sum\limits_{j=1}^{n} a_{ij}\alpha_{j}$ for $1\leq i\leq n$. 
Since $q^{S(\gamma)}=(\mfo_{F_{\chi_{\gamma}}}:R)$, we finally have 
\[
\frac{|\Delta_{F_{\chi_{\gamma}}/F}|^{1/2}}{|\Delta_{\gamma}|^{1/2}}
=\frac{|D(\alpha_{1}, \cdots, \alpha_{n})|^{1/2}}{|D(\beta_{1}, \cdots, \beta_{n})|^{1/2}}
=\frac{1}{|(\mfo_{F_{\chi_{\gamma}}}:R)|}=q^{S(\gamma)}.
\]
\end{proof}

\subsubsection{Parabolic descent}\label{sec324}
\cite[Corollary 4.10]{Yun13} explains an inductive structure of $\mathcal{SO}_{\gamma, d\mu}$ with respect to irreducible factors of $\chi_{\gamma}(x)$.
 We rephrase it in terms of $\mathcal{SO}_{\gamma}$.

We define
\[
\mathrm{Res(\gamma)}=\prod_{\{i,j\}\subset B(\gamma), i\neq j}\mathrm{Res}(\chi_{\gamma,i}, \chi_{\gamma,j}),
\]
where $\mathrm{Res}(\cdot, \cdot)$ is the resultant of two polynomials. 
Then the Weyl discriminant is formulated as follows:
\begin{equation}\label{eq324}
|\Delta_{\gamma}|=
|\mathrm{Res(\gamma)}|^2\cdot \prod\limits_{i\in B(\gamma)}|\textit{disc}(\chi_{\gamma, i})|=
|\mathrm{Res(\gamma)}|^2\cdot \prod\limits_{i\in B(\gamma)}|\Delta_{\gamma_i}|
.    
\end{equation}
  \cite[Corollary 4.10]{Yun13} is then rephrased as follows:

\begin{proposition}(\cite[Corollary 4.10]{Yun13})\label{cor410}
We have the following inductive formula;
\[
\mathcal{SO}_{\gamma}=
\frac{\#\mathrm{GL}_n(\kappa)q^{-n^2}}{\prod\limits_{i\in B(\gamma)}\#\mathrm{GL}_{n_i}(\kappa)q^{-n_i^2}}\prod_{i\in B(\gamma)}\mathcal{SO}_{\gamma_i}.
\]

\end{proposition}

\section{Reductions and stratification}\label{sec4}
 Recall that we defined the followings;
\[
\left\{
\begin{array}{l}
\varphi_n: \mathfrak{gl}_{n, \mathfrak{o}} \longrightarrow \mathbb{A}_{\mathfrak{o}}^n, \gamma \mapsto 
\textit{coefficients of $\chi_{\gamma}(x)$};\\
\uGr=\varphi_n^{-1}(\chi_{\gamma}).
\end{array}\right.
\]




\subsection{Reductions}\label{reductionss} 
 Due to Proposition \ref{cor410}, 
 we may and do assume that $\gamma \in \mfo$ is elliptic, that is,  $\chi_{\gamma}(x)$ is irreducible.
 Then $\overline{\chi}_{\gamma}(x)\in \kappa[x]$, which is the reduction of $\chi_{\gamma}(x)$ to the residue field $\kappa$, has only one irreducible factor due to Hensel's lemma.
 
 Z. Yun mentioned in the sentence just below \cite[Corollary 4.10]{Yun13} that we may reduce to the case that the irreducible factor of $\overline{\chi}_{\gamma}(x)$ is linear, without proof. Although this could be well known, we have not found a proof in literature. Thus we provide a detailed proof of it in this subsection.

 \subsubsection{Interpretation of $\gamma$ in $\mathfrak{gl}_n(\mfo)$ as an element in $\mathfrak{gl}_{n/l}(\mfo_K)$}
 We first introduce notions used in \cite{Yun13}. For an elliptic regular semisimple element $\gamma \in \mathfrak{gl}_{n}( \mathfrak{o})$, let $R=\mfo [x]/(\chi_{\gamma}(x))$. 
Recall from the paragraph just before Proposition \ref{propserre} that 
 $F_{\chi_{\gamma}}=F[x]/(\chi_{\gamma}(x))$, which is the field of fractions of $R$, so that $\mfo_{F_{\chi_{\gamma}}}$ is the ring of integers of ${F_{\chi_{\gamma}}}$.

We denote by $\kappa_{R}$ and $\kappa_{F_{\chi_{\gamma}}}$  the residue fields of $R$ and $\mfo_{F_{\chi_{\gamma}}}$, respectively, and by $d$  the degree of the field extension $\kappa_{R}/\kappa$.
This setting can be visualized as follows:

\begin{equation}\label{diag}
    \begin{tikzcd}
    F \arrow[d, no head] \arrow[rr, "n"] & &{F_{\chi_{\gamma}}} \arrow[ld, no head] \arrow[d, no head]\\
\mathfrak{o} \arrow[d, no head] \arrow[r, "n"] & R \arrow[d, no head] \arrow[r] & \mathfrak{o}_{F_{\chi_{\gamma}}} \arrow[d, no head] \\
\kappa \arrow[r, "d"]                          & \kappa_{R} \arrow[r]           & \kappa_{F_{\chi_{\gamma}}}.                         
\end{tikzcd}
\end{equation}
Here, the letter over the arrow stands for the degree (the case of fields) or the rank (the case of free modules). 

Recall that $\overline{\chi}_{\gamma}(x)\in \kappa[x]$ has the only one irreducible factor by Hensel's lemma. We denote it by   $\phi(x)\in \kappa[x]$ so that we have
\[
\overline{\chi}_{\gamma}(x)= (\phi(x))^{m} \textit{ modulo }\pi
\]
where $0<m\leq n$.

By  \cite[Lemma I.4]{Ser}, the maximal ideal of the local ring $R=\mfo[x]/(\chi_{\gamma}(x))$ is $(g(x), \pi)$ where $g(x)$ is an irreducible polynomial in $\mfo[x]$ such that $\bar{g}(x)=\phi(x)$ in $\kappa[x]$. 
Then we have 
\[
\kappa_R\cong \kappa[x]/(\phi(x)).
\]
Therefore, the degree of $\phi(x)$ is $l$ and thus $\overline{\chi}_{\gamma}(x)=(\phi(x))^{n/l}$ and  $m=n/l$.

By \cite[Corollary 2 of Theorem 2 in Chapter 3]{Ser}, there exists a unique unramified extension $K$ over $F$ in ${F_{\chi_{\gamma}}}$ which corresponds to the residue field extension $\kappa_{R}/\kappa$. The ring of integers of $K$ is of the form $\mfo_{K}\cong \mfo[x]/(g(x))$ by  \cite[Proposition I.15]{Ser}.

 \begin{lemma}\label{unramadd}
 $\mfo_{K}$ is contained in $R$.
 \end{lemma}

The lemma yields that 
we can improve Diagram (\ref{diag}) as follows:
\begin{equation}\label{diag2_gl}
\begin{tikzcd}
F \arrow[r,"l"] \arrow[d, no head]               & K \arrow[rr,"n/l"] \arrow[d, no head]     &                        & {F_{\chi_{\gamma}}} \arrow[ld, no head] \arrow[d, no head]    \\
\mfo \arrow[r, "l"] \arrow[d, no head] & \mfo_{K} \arrow[r, "n/l"] \arrow[d, no head]      & R \arrow[r] \arrow[ld, no head] & \mfo_{F_{\chi_{\gamma}}} \arrow[d, no head] \\
\kappa \arrow[r,"l"]                     & \kappa_{R} \arrow[rr] & {}                     & \kappa_{F_{\chi_{\gamma}}}.            
\end{tikzcd}
\end{equation}

 \begin{proof}
It suffices to find an injective ring homomorphism from $\mfo_{K}$ into $R$. Note that the finite $\mfo$-module $R$ is $\pi$-adically complete.
 Since $\mfo$ is a noetherian $\pi$-adically complete ring,
 $\widehat{R} \cong \widehat{\mfo} \otimes R \cong \mfo \otimes R \cong R$, where $\widehat{\cdot}$ means $\pi$-adic completion.
 
 Suppose that $R$ is $(\pi, g(x))$-adically complete. Then we obtain the injective homomorphism as follows:
 \begin{itemize}
    \item{Equal characteristic case, equivalently,  $F=\kappa((t))$ and $\mfo=\kappa[[t]]$}

    In this case $K\cong \kappa_{R}((t))$ and $\mfo_{K}\cong \kappa_{R}[[t]]$. From the inclusion $\kappa \subset \kappa_{R}$ of the residue fields, we obtain the inclusion $\kappa[[t]]\subset \kappa_{R}[[t]]$. Since $R$ is $(\pi, g(x))$-adically complete, we have $f:\kappa_{R}\rightarrow R$ defined in \cite[Proposition II.8.(1)]{Ser}. Thus, we have the inclusion
\[        \kappa_{R}[[t]]\longrightarrow R, ~~~~~~~~~
        t\mapsto t ~~\textit{and}~~  \lambda \mapsto f(\lambda) \textit{ for }\lambda \in \kappa_{R}.
  \]
    \item{Unequal characteristic case}
    
    Since $R$ is $(\pi, g(x))$-adically complete, $\mfo_{K}$ and $R$ are $p$-rings defined in  \cite[Chapter 2.5]{Ser} where $p$ is the characteristic of the residue field $\kappa_{R}$. Especially, since $\mfo_{K}$ is the ring of integers of an unramified extension of $F$, the descending filtration $\{a_i\}$ of ideals, which gives the $p$-ring structure to $\mfo_{K}$, is exactly the same as $\{(\pi)^i\}$. Thus it is a strict $p$-ring.
    
    Now, we can use \cite[Chapter 2.5]{Ser} to obtain the unique homomorphism $\varphi$ satisfying the following   commutative diagram;
    \[
        \begin{tikzcd}  
        \mathfrak{o}_{K} \arrow[r, "\varphi"] \arrow[d] &  R \arrow[d] \\
        \kappa_{R} \arrow[r, "id"]                &  \kappa_{R}.
        \end{tikzcd}
    \]
    From the construction of $\varphi$ in  \cite[Proposition II.10]{Ser}, $\varphi$ is an injective ring homomorphism since $char(\mfo_{K})=char(R)\neq p$ and the restriction of $\varphi$ to the residue field $\kappa_{R}$ is the identity map.
    
 \end{itemize}

 
 The remaining part    is to prove the assumption; $R$ is a $(\pi, g(x))$-adically complete ring. We already know that $R$ is $\pi$-adically complete. From the definition of $g(x)$,
 \[
 \overline{\chi}_{\gamma}(x)=(\phi(x))^{m}=(\bar{g}(x))^{m},\ \chi_{\gamma}(x)\equiv (g(x))^{m} \textit{ modulo }\pi
 \]where $m=n/l$.
 Thus $(g(x))^{m} \in (\pi)$ and $(\pi, g(x))^{m} \subset (\pi)$ as ideals of $R$. We obtain the following relation between the $(\pi)$-adic topology and the $(\pi, g(x))$-adic topology  on $R$;
 \[
\left\{
\begin{array}{l}
(\pi, g(x))^{mi} \subset (\pi)^{i}\textit{ for arbitrary integer }i>0;\\
(\pi)^{j} \subset (\pi, g(x))^{j}\textit{ for arbitrary integer }j>0.
\end{array} \right.
\]
These two relations directly imply the equivalence of two topologies.
\end{proof}
We identify  $\mathrm{End}_{\mfo} (R)=\mathfrak{gl}_{n}(\mfo)$, where $R\left(=\mfo[x]/(\chi_{\gamma}(x))\right)$ is considered as a free $\mfo$-module of rank $n$.
 The element $\gamma$ can then be considered as an element of $R$ through the following ring isomorphism:
 \[
\mfo[\gamma]\cong R, ~~~~~~~~~~~~~~~~~~~~~~~~~~~~~~~~  \gamma \mapsto x.  
 \]
Then $R$ is embedded into $\mathrm{End}_{\mfo} (R)$ such that the action of $\gamma$ on $R$ is the multiplication by $x$.

By Lemma \ref{unramadd},  $R$ can also be considered as  a free $\mfo_K$-module of rank $n/l$.
Let  $\chi_{\gamma}^{K}(x)\in \mfo_{K}[x]$ be the minimal polynomial of $\gamma$ over $\mfo_K$.
We then have the following identification of rings:
\[R\cong \mfo[\gamma]\cong\mfo_K[\gamma]\cong \mfo_{K}[x]/(\chi_{\gamma}^{K}(x)).
\]
Thus we may consider $\gamma$ as an element of $\mathrm{End}_{\mfo_K} (R)$, which is identified with $\mathfrak{gl}_{n/l}(\mfo_K)$. We denote $\gamma$ by $\gamma_K$ in this case.
 
\subsubsection{Comparison between $\mathcal{SO}_{\gamma,d\mu}$ and $\mathcal{SO}_{\gamma_K,d\mu_K}$}
We define the stable orbital integral $\mathcal{SO}_{\gamma_K,d\mu_K}$ where $d\mu_K$ is the quotient measure on $\mathrm{T}_{\gamma_K}(K)\backslash \mathrm{GL}_{n/l}(K)$ defined as in Section \ref{sec321}.

 
Yun interpreted  the orbital integral  as follows (cf. \cite[Section 4.2]
{Yun13}):
 \begin{equation}\label{yunobs}
 \mathcal{SO}_{\gamma,d\mu}=\#(\Lambda \backslash X_{R}),
 \end{equation}
 where $X_{R}$ is the set of fractional $R$-ideals and $\Lambda \left(\subset F_{\chi_{\gamma}}^{\times}\right)$ is $\varpi^{\mathbb{Z}}$.
 Here $\varpi$ is a uniformizer in ${F_{\chi_{\gamma}}}$.


\begin{lemma}\label{lem42}
For an elliptic regular semisimple element $\gamma \in \mathfrak{gl}_{n}(\mfo)$, two orbital integrals for $\gamma$ are identified as follows:
\[\mathcal{SO}_{\gamma,d\mu}=\mathcal{SO}_{\gamma_K,d\mu_K}.\]
 \end{lemma}
 \begin{proof}
Considering $\gamma$ as an element of $\mathfrak{gl}_{n/l}(\mfo_K)$, the minimal polynomial $\chi_{\gamma}^K (x)$ of $\gamma$ is of degree $n/l$ and thus the characteristic polynomial of $\gamma$ is also $\chi_{\gamma}^K (x)$.
Note that  $\gamma$ is an elliptic element in $\mathfrak{gl}_{n/l}(\mfo_K)$ since  $\chi_{\gamma}^{K}(x)$ is irreducible. Furthermore, $\gamma$ is a regular semisimple element in $\mathfrak{gl}_{n/l}(\mfo_K)$ since $\gamma$ is in $\mathfrak{gl}_{n}(\mfo)$.
 By Equation (\ref{yunobs}),  we have
  \[\mathcal{SO}_{\gamma,d\mu}=\#(\Lambda \backslash X_{R})~~~~~~~\textit{ and } ~~~~~~~~~~~\mathcal{SO}_{\gamma_K,d\mu_K}=\#(\Lambda \backslash X_{R}).
  \]
This completes the proof.
\end{proof}
With respect to the geometric measure, Proposition \ref{proptrans} yields the following reduction:

\begin{corollary}\label{corred}
We have the following relation between two orbital integrals 
$ \mathcal{SO}_{\gamma}$ and $\mathcal{SO}_{\gamma_K}$:
 \[
 \mathcal{SO}_{\gamma}=\frac{\#\mathrm{GL}_n(\kappa)}{\#\mathrm{GL}_{n/l}(\kappa_R)\cdot q^{n^2-n^2/l}} \cdot\mathcal{SO}_{\gamma_K}.
 \]
 \end{corollary}

\subsubsection{Invariance under the translation}
 
By the virtue of Corollary $\ref{corred}$, we may and do assume that $[\kappa_{R}:\kappa]=l=1$. 
Equivalently, we may and do assume that 
 the irreducible factor of $\overline{\chi}_{\gamma}(x)$ is linear so that $\overline{\chi}_{\gamma}(x)=(x-a)^n$ for some $a\in \kappa$. 
 
\begin{lemma}\label{constantlemma}
 $\mathcal{SO}_{\gamma}=\mathcal{SO}_{\gamma+c}$ for a constant matrix $c\in \mathfrak{gl}_n(\mfo)$.
\end{lemma}
\begin{proof}

It suffices to show that there exists an automorphism $\iota_c$ of $\mathbb{A}_{\mathfrak{o}}^n$ , depending on the choice of $c$, which makes the following diagram commutes;

\[\xymatrixcolsep{5pc}\xymatrix{
\mathfrak{gl}_{n, \mathfrak{o}}  \ar[d]^{m \mapsto m+c} \ar[r]^{\varphi_n} &  \mathbb{A}_{\mathfrak{o}}^n \ar[d]^{\iota_c} \\ \mathfrak{gl}_{n, \mathfrak{o}} \ar[r]^{\varphi_n} & \mathbb{A}_{\mathfrak{o}}^n.
}\]

We will first construct a morphism $\iota_c$ which makes the diagram commute.
Choose $\underline{a}=(a_{1}, \cdots,  a_n)\in \mathbb{A}_{\mathfrak{o}}^n$.
Let $f$ be the polynomial of degree $n$ having $\underline{a}$ as coefficients.
Let $\underline{a}'=(a'_{1}, \cdots,  a'_n)\in \mathbb{A}_{\mathfrak{o}}^n$ such that 
$\underline{a}'$ is the coefficients of $f(x-c)$. 
Then each $a'_{i}$ is  the sum of a linear combination of $a_j$'s whose coefficients are elements of $\mathbb{Z}[c]$.
Define the morphism $\iota_c$ which sends  $a_i$'s to $a_i'$'s respectively.
Then $\iota_c$ satisfies the desired property.

The morphism $\iota_c$ satisfies the following property:
 $$\iota_c\circ \iota_{-c}=\iota_{-c}\circ \iota_{c}=\iota_0.$$ 
 Since $\iota_0$ is the identity, the morphism  $\iota_c$ is an automorphism.
\end{proof}

 

  


\subsection{Stratification}\label{sec42st}
The goal of this subsection is to impose a certain linear algebraic data to the orbit $\underline{G}_{\gamma}(\mfo)$ of $\gamma$. 
Let us introduce the following notations and related facts:
\begin{itemize}
\item Let $L$ be a free $\mfo$-module $L$ of rank $n$ and let $M$ be a sublattice of $L$ of the same rank $n$.

\item Define a functor $\cL(L,M)$ on the category of flat $\mfo$-algebras to the category of sets such that 
\[
\cL(L,M)(R)=\{f\mid\textit{$f:L\otimes R\rightarrow M\otimes R $ is $R$-linear and surjective}\}
\]
for a flat $\mfo$-algebra $R$. 
The functor $\cL(L,M)$ is then represented by an open subscheme of the affine space over $\mfo$ of dimension $n^2$ so as to be smooth over $\mfo$.
Here, `open subscheme' structure is due to surjectivity of elements of $\cL(L,M)(R)$. 

\item 
We can assign the $\mfo$-scheme structure on $\mathrm{End}(L)$ by defining the functor on the category of flat $\mfo$-algebras to the category of  sets as follows:
\[
\End(L)(R)=\End_{R}(L\otimes R)
\]
for an $\mfo$-algebra $R$. 

Then $\mathrm{End}(L)$ is represented by an affine space over $\mfo$ of dimension $n^2$.
As schemes, $\cL(L,L)$ is an open subscheme of  $\mathrm{End}(L)$. 
As sets,  $\cL(L,L)(\mfo)$ is an open (in terms of $\pi$-adic  topology) and proper subset of $\mathrm{End}(L)(\mfo)$. 

However, $\cL(L,M)$ is no longer  a subscheme  of $\End(L)$ unless $M= L$. 
On the other hand,  $\cL(L,M)$ is a subfunctor of $\End(L)$ on the category of flat $\mfo$-algebras.
This defines a morphism of schemes $\cL(L,M) \rightarrow \End(L)$.

We sometimes use  $\End(L)$ to stand for the set $\End(L)(\mfo)$ if there is no confusion.


\item We express $\gl_{n}(\mfo)=\End(L)$.
Then  we may and do regard $\cL(L,M)(\mfo)$  as  an open subset  of $\gl_{n}(\mfo)$.

\item Define the orbit of $\gamma$ inside $\cL(L,M)(\mfo)$ as follows:
\[O_{\gamma, \cL(L,M)}=\uGr(\mfo)\cap \cL(L,M)(\mfo)=\{f\in \cL(L,M)(\mfo)| \vpi_n(f)=\vpi_n(\gamma)\}\] 
where the intersection is taken inside $\End(L)$.
For example,  $O_{\gamma, \End(L)}=\uGr(\mfo)$.
Then $O_{\gamma, \cL(L,M)}$ is an open subset of $\uGr(\mfo)$.


\item Let $\varphi_{n,M}$ be the restriction of $\varphi_n$ to $\cL(L,M)$ so that 
\[
\varphi_{n,M}:  \cL(L,M) \longrightarrow \mathbb{A}_{\mathfrak{o}}^n.
\]
\end{itemize}

We now have the following stratification;
\begin{equation}\label{st1}
O_{\gamma, \End(L)} \left(=\uGr(\mfo)\right)  =\bigsqcup_M O_{\gamma, \cL(L,M)}
\end{equation}
where $M$ runs over all sublattices of $L$ such that $[L:M]=\mathrm{ord}(c_n)$ where $c_{n}$ is the constant term of the characteristic polynomial $\chi_{\gamma}(x)$.
Here, $[L:M]$ is the index as $\mfo$-modules.
Thus there are finitely many such $M$'s and so the above disjoint union is finite.

\begin{definition}
Let $M$ be a sublattice of $L$ with rank $n$. The type of $M$ is defined to be $(k_1, \cdots, k_{n-m})$, where $k_i\in \Z$ and $1\leq k_i\leq k_j$ for $i\leq j$, such that 
\[
L/M \cong \mfo/\pi^{k_1}\mfo \oplus \cdots \oplus \mfo/\pi^{k_{n-m}}\mfo.
\]
\end{definition}
Note that $k_1+\cdots +k_{n-m}=[L:M]$. Using this,  the above stratification (\ref{st1}) can be refined as follows:
\begin{equation}\label{st2}
O_{\gamma, \End(L)} \left(=\uGr(\mfo)\right) =\bigsqcup_{\substack{(k_1, \cdots, k_{n-m}),  \\ \sum k_i=\mathrm{ord}(c_n),\\ m\geq 0 }}
\left(\bigsqcup_{\substack{M:\textit{type(M)=} \\ (k_1, \cdots, k_{n-m})}} O_{\gamma, \cL(L,M)}\right).
\end{equation}

Let\[
\mathcal{SO}_{\gamma, \cL(L,M)}=\int_{O_{\gamma, \cL(L,M)}}|\omega_{\chi_{\gamma}}^{\mathrm{ld}}|.
\]

\begin{lemma}\label{lemred1}
If the type of $M$ is the same as that of $M'$, then $\mathcal{SO}_{\gamma, \cL(L,M)}=\mathcal{SO}_{\gamma, \cL(L,M')}$.
\end{lemma}
\begin{proof}
Since all lattices  are  modules over PID, we can choose a basis $(e_1, \cdots, e_n)$ (resp. $(e'_1, \cdots, e'_n)$)  of $L$ such that $(\pi^{k_1}e_1, \cdots, \pi^{k_{n-m}}e_{n-m}, e_{n-m+1}, \cdots, e_n)$ (resp. $(\pi^{k_1}e'_1, \cdots, \pi^{k_{n-m}}e'_{n-m}, e'_{n-m+1}, \cdots, e'_n)$) is a basis of $M$ (resp. $M'$). 
Choose $g$ in $\End(L)$ which sends $e_i$ to $e_i'$. Then $g$ is invertible as an endomorphism of $L$ and defines a bijection from $M$ to $M'$.
This also yields 
$\cL(L,M)=g^{-1}\cL(L,M')g$ as a scheme over $\mfo$.

As in the proof of Lemma \ref{constantlemma},
it suffices to show the existence of  an automorphism $\iota$ of $\mathbb{A}_{\mathfrak{o}}^n$ which the following diagram commutes;

\[\xymatrixcolsep{5pc}\xymatrix{
\cL(L,M')  \ar[d]^{f \mapsto g^{-1}fg} \ar[r]^{\varphi_{n,M'}} &  \mathbb{A}_{\mathfrak{o}}^n \ar[d]^{\iota} \\ \cL(L,M) \ar[r]^{\varphi_{n,M}} & \mathbb{A}_{\mathfrak{o}}^n.
}\]
Since $\vpi_n(f)=\vpi_n(g^{-1}fg)$, letting $\iota$ be the identity map makes the above diagram commute.
 This completes the proof.
\end{proof}


This lemma enables us to define the following two notations:
 \[
\left\{
\begin{array}{l l}
\mathcal{SO}_{\gamma, (k_1, \cdots, k_{n-m})}:=\mathcal{SO}_{\gamma, \cL(L,M)}\textit{ where the type of $M$ is $(k_1, \cdots, k_{n-m})$};\\
c_{(k_1, \cdots, k_{n-m})}:=\#\{M\subset L \mid \textit{the type of $M$ is $(k_1, \cdots, k_{n-m})$}\}.
\end{array} \right.
\]

Then the stratification (\ref{st2}) yields the following formula:
\begin{proposition}\label{stra1}
We have
\[
\mathcal{SO}_{\gamma}= \sum\limits_{\substack{(k_1, \cdots, k_{n-m}),  \\ \sum k_i=\mathrm{ord}(c_n),\\ m \geq 0 }}
c_{(k_1, \cdots, k_{n-m})}\cdot
\mathcal{SO}_{\gamma, (k_1, \cdots, k_{n-m})}.
\]
\end{proposition}

\subsubsection{Formula for $c_{(k_1,\cdots,k_{n-m})}$}

We will provide a precise formula for each coefficient $c_{(k_1, \cdots, k_{n-m})}$.
To do that, 
We rewrite $(k_1, \cdots, k_{n-m})=(r_{1}, \cdots, r_{1}, r_{2}, \cdots, r_{2}, r_{3}, \cdots, r_{3}, \cdots,r_{l}, \cdots, r_{l})$, where $r_i<r_j$ if $i<j$.
Let  $m_{i}$  be the number of occurrence of  $r_{i}$ in the above type  and let $m_{0}=m$ so that 
 $\sum\limits_{i=0}^{l}m_{i}=n$.

\begin{lemma}\label{counttype}

 The formula for $c_{(k_1, \cdots, k_{n-m})}$ $\left(=c_{(r_{1},\cdots ,r_{1},\cdots,r_{l},\cdots,r_{l})}\right)$ is given as follows:
\[
c_{(r_{1},\cdots ,r_{1},\cdots,r_{l},\cdots,r_{l})}=\prod_{i=1}^{l}\#\mathrm{Gr}_{\kappa}(m_{i},s_{i})\cdot q^{\left((r_{i}-r_{i-1}-1)(n-s_{i-1})+(n-s_{i})\right)s_{i-1}},
\]
where $r_{0}=0$, $s_{i}=\sum\limits_{j=0}^{i}m_{j}$, and $\mathrm{Gr}_{\kappa}(m,n)$ is the Grassmannian variety classifying $m$-dimensional subspaces of the $n$-dimensional vector space over $\kappa$.
\end{lemma}
\begin{proof}\footnote{Jungin Lee informed us a simplified proof based on \cite{Mac} and \cite{HM}.}
Define $L_{0}:=(\mfo/\pi^{r_{l}}\mfo)^{n}$. Then we have
\[
    c_{(r_{1},\cdots ,r_{1},\cdots,r_{l},\cdots,r_{l})}=\#\left\{N\subset L_{0}:L_{0}/N\cong \prod_{j=0}^{l}(\mfo/\pi^{r_{j}}\mfo)^{m_{j}} \right\}
    =\#\left\{N\subset L_{0}:N\cong \prod_{j=0}^{l}(\mfo/\pi^{r_{j}}\mfo)^{m_{j}} \right\},
\]
where the last equality is due to the duality of finite $\mfo$-modules. ( cf. \cite[II.(1.5)]{Mac})

For integers $n_{1}>\cdots>n_{l}\geq 1$, let $(\overset{(t_{1})}{n_{1}},\cdots ,\overset{(t_{l})}{n_{l}})$ be the $(t_1+\cdots +t_l)$-tuple such that each number $n_{i}$ appears $t_{i}$ times.
If we consider tuples $\lambda=(\overset{(n)}{r_{l}})$ and $\mu=(\overset{(m_{l})}{r_{l}},\cdots,\overset{(m_{1})}{r_{1}},\overset{(m_{0})}{r_{0}})$ as partitions of $nr_{l}$ and $\sum\limits_{i=0}^{l}m_{i}r_{i}$, respectively, then their conjugate partitions are given by $\lambda'=(\overset{(r_{l})}{n})$ and $\mu'=(\overset{(r_{1})}{n-s_{0}},\overset{(r_{2}-r_{1})}{n-s_{1}},\cdots, \overset{(r_{l}-r_{l-1})}{n-s_{l-1}})$, respectively.
We then apply 
 \cite[Theorem 31]{HM} to these conjugate partitions so that we have
\[\#\left\{N\subset L_{0}:N\cong \prod_{j=0}^{l}(\mfo/\pi^{r_{j}}\mfo)^{m_{j}} \right\}
    =\prod_{i=1}^{r_{l}}\#\mathrm{Gr}_{\kappa}(\lambda'_{i}-\mu'_{i+1},\mu'_{i}-\mu'_{i+1})\cdot q^{\mu'_{i+1}(\lambda'_{i}-\mu'_{i})}.\]
    Here, $\lambda_i'$ is the $i$-th component of $\lambda'$ (so that it is $n$) and 
    $\mu_i'$ is the $i$-th component of $\mu'$.
The right hand side is equal to    
\begin{align*}
    &\prod_{t=1}^{l}q^{(r_{t}-r_{t-1}-1)(n-s_{t-1})s_{t-1}}~~~~(\textit{when }\mu'_{i}=\mu'_{i+1}=n-s_{t-1})\\
    &\times\prod_{t-1}^{l}\#\mathrm{Gr}_{\kappa}(m_{t},s_{t})q^{(n-s_{t})s_{t-1}}~~~~(\textit{when }\mu'_{t}=n-s_{t-1}, \mu'_{t+1}=n-s_{t})\\
&=\prod_{t=1}^{l}\#\mathrm{Gr}_{\kappa}(m_{t},s_{t})\cdot q^{((r_{t}-r_{t-1}-1)(n-s_{t-1})+(n-s_{t}))s_{t-1}}.
\end{align*}    
This completes the proof.
\end{proof}
\begin{corollary}\label{cor49}
We have
\[
\left\{
\begin{array}{l}
c_{(k_{1})}={\frac{q^{n}-1}{q-1}}\cdot q^{(n-1)(k_{1}-1)};\\
c_{(k_{1},k_{1})}
=\frac{(q^{n}-1)(q^{n-1}-1)}{(q^{2}-1)(q-1)}\cdot q^{2(n-2)(k_{1}-1)};\\
c_{(k_{1},k_{2})}
=\frac{(q^{n}-1)(q^{n-1}-1)}{(q^{2}-1)(q-1)}\cdot (q+1) \cdot q^{(n-3)k_{1}+(n-1)k_{2}-2n+3}\textit{ where }k_{1}<k_{2}.
\end{array} \right.
\]
\end{corollary}
\begin{proof}
From  Lemma \ref{counttype}, we have the following formula;
\[
\left\{
\begin{array}{l}
c_{(k_{1})}=\#\mathrm{Gr}_{\kappa}(1,n)\cdot q^{(n-1)(k_{1}-1)};\\
c_{(k_{1},k_{1})}=\#\mathrm{Gr}_{\kappa}(2,n)\cdot q^{(n-2)\cdot 2 \cdot (k_{1}-1)};\\
c_{(k_{1},k_{2})}=\#\mathrm{Gr}_{\kappa}(1,n-1)\cdot q^{((k_{1}-1)\cdot 2+1)(n-2)}\cdot \#\mathrm{Gr}_{\kappa}(1,n)\cdot q^{(k_{2}-k_{1}-1)(n-1)}.
\end{array} \right.
\]
Plugging in
\[
\#\mathrm{Gr}_{\kappa}(d,n)=\frac{(q^{n}-1)(q^{n}-q)\cdots (q^{n}-q^{d-1})}{(q^{d}-1)(q^{d}-q)\cdots(q^{d}-q^{d-1})}
\]
into the above formula completes the proof.
\end{proof}
\begin{remark}\label{rmk410}
\begin{enumerate}
\item
If $n=2$, then 
\[
c_{(k_{1})}=(q+1) q^{k_{1}-1}.
\]

\item
If $n=3$, then 
\[
\left\{
\begin{array}{l}
c_{(k_{1})}=(q^{2}+q+1) q^{2(k_{1}-1)};\\
c_{(k_{1},k_{1})}=(q^{2}+q+1)q^{2(k_{1}-1)};\\
c_{(k_{1},k_{2})}=(q^{2}+q+1)(q+1)q^{2k_{2}-3}\textit{ where }k_{1}<k_{2}.
\end{array} \right.
\]
\end{enumerate}
\end{remark}

 \subsection{Refined stratification}\label{subsec43}
By using the reduction steps (e.g. Lemma \ref{lem42} (or Corollary \ref{corred}) and Lemma \ref{constantlemma}) introduced in Section \ref{reductionss},
we may and do assume that $\chi_{\gamma}(x)$ is irreducible  whose reduction $\overline{\chi}_{\gamma}(x)$ modulo $\pi$ is $x^n$.
For a lattice $M$ of type $(k_1, \cdots, k_{n-m})$ with $m \geq 1$, we choose  a basis $(e_1, \cdots, e_n)$  of $L$ such that $(e_1, \cdots, e_{m}, \pi^{k_{1}}e_{m+1}, \cdots, \pi^{k_{n-m}}e_n)$ is a basis of $M$. With respect to this basis, the matrix representation of an element $X$ of $\cL(L,M)(\mfo)$ is given as follows:

\begin{equation}\label{matrixform}
X=\begin{pmatrix} &&&x_{1, m+1}& \cdots &x_{1, n}\\ & X_{m,m} &&\vdots&  &\vdots \\&&&x_{m, m+1}& \cdots &x_{m, n}
\\ \pi^{k_1}x_{m+1, 1}&\cdots&\pi^{k_1}x_{m+1, m}& \pi^{k_1}x_{m+1, m+1}& \cdots &\pi^{k_1}x_{m+1, n}\\ 
 \vdots &&\vdots &\vdots & &\vdots \\  
 \pi^{k_{n-m}}x_{n, 1}&\cdots&\pi^{k_{n-m}}x_{n, m}& \pi^{k_{n-m}}x_{n, m+1}& \cdots &\pi^{k_{n-m}}x_{n, n}
\end{pmatrix}
\end{equation}
such that
\[
X':=\begin{pmatrix} &&&x_{1, m+1}& \cdots &x_{1, n}\\ & X_{m,m} &&\vdots&  &\vdots \\&&&x_{m, m+1}& \cdots &x_{m, n}
\\ x_{m+1, 1}&\cdots&x_{m+1, m}& x_{m+1, m+1}& \cdots &x_{m+1, n}\\ 
 \vdots &&\vdots &\vdots & &\vdots \\  
 x_{n, 1}&\cdots&x_{n, m}& x_{n, m+1}& \cdots &x_{n, n}
\end{pmatrix}
\]
is invertible in $\mathrm{M}_n(\mfo)$ with respect to the multiplication.
Here
 the size of $X_{m,m}$ is $m\times m$.
 
 If we further restrict $x$ to be an element of $O_{\gamma, \cL(L,M)}$, then  the characteristic polynomial of the matrix $\overline{X}_{m,m}$, which is the reduction of the matrix $X_{m,m}$ modulo $\pi$, is $x^{m}$. 

\begin{lemma}\label{lem411}
The rank of $\overline{X}_{m,m}$ is at least $2m-n$ and at most $m-1$.
\end{lemma}
\begin{proof}
Since the characteristic polynomial of $\overline{X}_{m,m}$ is $x^m$, its determinant is zero, which gives an upper bound of the rank of $\overline{X}_{m,m}$ of being $m-1$.

For a lower bound of the rank of $\overline{X}_{m,m}$, we express $\overline{X}'$ as the matrix
$\begin{pmatrix}
\overline{X}_{m,m} & \bar{y} \\ \bar{c} & \bar{d}
\end{pmatrix} \in \mathrm{M}_n(\kappa)$. 
Since $\overline{X}'$ is invertible, the (row) rank of $\begin{pmatrix}
\overline{X}_{m,m} & \bar{y} \end{pmatrix}$ should be $m$.
Since the (column) rank of $\bar{y}$ is at most $n-m$, the rank of $\overline{X}_{m,m}$ is at least $2m-n$.
\end{proof}

Let $\overline{L}=L/\pi L$ and let $\overline{M}=M/ M\cap \pi L$ so that $\overline{M}$ is a subspace of $\overline{L}$ over $\kappa$. 
Note that the dimension of $\overline{M}$ is $m$. 
The linear map $\overline{X}$ induces a linear map from $\overline{M}$ to $\overline{M}$ which is represented by the matrix $\overline{X}_{m,m}$. 

Let $V$ be a subspace of $\overline{M}$ whose dimension is at least $2m-n$ and at most $m-1$.
Define a functor $\cL(L,M,V)$ on the category of flat $\mfo$-algebras to the category of sets such that 
\[
\cL(L,M,V)(R)=\{f\in \cL(L,M)(R)\mid \bar{f}(\overline{M}\otimes R/\pi R)=V\otimes R/\pi R\}
\]
for a flat $\mfo$-algebra $R$. 
\begin{remark}\label{refinedmatrix}
A matrix representation of an element $x$ of $\cL(L,M,V)(\mfo)$, such that  the dimension of $V$ is $m-t$, is given by the matrix form (\ref{matrixform}) with the extra congruent condition;
\begin{equation}\label{matrixform2}
X_{m,m}=\begin{pmatrix}
\pi x_1 & \cdots &\pi x_t&x_{t+1}&\cdots &x_{m}
\end{pmatrix},
\end{equation}
such that 
$\begin{pmatrix}\bar{x}_{t+1}&\cdots &\bar{x}_{m}\end{pmatrix}$ is of rank $m-t$ over $\kappa$, 
equivalently, at least one $(m-t)\times (m-t)$ minor of $X_{m,m}$ is a unit in $\mfo$.
Here, $x_i$ is the $i$-th column vector of $X_{m,m}$.
\end{remark}

\begin{lemma}\label{lem413}
The functor $\cL(L,M,V)$ is represented by an open subscheme of an affine space over $\mfo$ of dimension $n^2$ so as to be smooth over $\mfo$.
\end{lemma}
\begin{proof}
Define a functor $\widetilde{\cL}(L,M,V)$ on the category of flat $\mfo$-algebras to the category of additive abelian groups such that 
\[
\widetilde{\cL}(L,M,V)(R)=\{f|\textit{$f:L\otimes R\rightarrow M\otimes R$ is $R$-linear and $\bar{f}(\overline{M}\otimes R/\pi R)\subset V\otimes R/\pi R$}
\}
\]
for a flat $\mfo$-algebra $R$. 
Then each element of $\widetilde{\cL}(L,M,V)(R)$ is represented by a matrix of the form described in Equation (\ref{matrixform}) with $x_{i,j}\in R$ for $(i,j)\neq (m,m)$, by omitting the condition that $x'$ is invertible, and a matrix of the form described in Equation (\ref{matrixform2}) with $x_i$ a column matrix in $R$, by omitting the condition that $\begin{pmatrix}\bar{x}_{t+1}&\cdots &\bar{x}_{m}\end{pmatrix}$ is of rank $m-t$.
Then $\widetilde{\cL}(L,M,V)$ is uniquely represented by a flat $\mfo$-algebra which is a polynomial ring over $\mfo$ of $n^2$ variables.

The functor $\cL(L,M,V)$ is then represented by an open subscheme of $\widetilde{\cL}(L,M,V)$ and thus is smooth. 
This completes the proof.
\end{proof}
Note that  $\cL(L,M,V)$ is not necessarily affine, but has finite affine covers given by each $(m-t)\times (m-t)$ minor of $X_{m,m}$. 
Since $\cL(L,M,V)$ is representable,  we can now talk of  $\cL(L,M,V)(R)$ for any (not necessarily flat) $\mfo$-algebra $R$.

As in the previous section, we define
\[
\left\{
  \begin{array}{l }
O_{\gamma, \cL(L,M,V)}=\uGr(\mfo)\cap \cL(L,M,V)(\mfo)=\{f\in \cL(L,M,V)(\mfo)| \vpi_n(f)=\vpi_n(\gamma)\};\\
\mathcal{SO}_{\gamma, \cL(L,M,V)}=\int_{O_{\gamma, \cL(L,M,V)}}|\omega_{\chi_{\gamma}}^{\mathrm{ld}}|.
    \end{array} \right.
\]

\begin{lemma}\label{lem412}
If $V$ and $W$ are subspaces of $\overline{M}$ having  the same dimension,  then 
\[
\mathcal{SO}_{\gamma, \cL(L,M,V)}=\mathcal{SO}_{\gamma, \cL(L,M,W)}.
\]
\end{lemma}
\begin{proof}
As in the proof of Lemma \ref{lemred1}, there exists  an invertible element $g\in \End(L)$  such that $g$ sends $M$ to $M$ and $\bar{g}$ sends $W$ to $V$ isomorphically.
Then the following diagram is commutative:
\[\xymatrixcolsep{5pc}\xymatrix{
 \cL(L,M,V) \ar[d]^{f\mapsto g^{-1}fg} \ar[r]^{\vpi_n} &  \mathbb{A}_{\mathfrak{o}}^n \ar[d]^{identity} \\ \cL(L,M,W) \ar[r]^{\varphi_n} &  \mathbb{A}_{\mathfrak{o}}^n
}\]
This completes the proof.
\end{proof}

Thus if the dimension of $V$ is $m-t$, then we can let 
\[
\mathcal{SO}_{\gamma, (k_1, \cdots, k_{n-m}), t}=\mathcal{SO}_{\gamma, \cL(L,M,V)}.
\]

Here the range of $t$ is  $1\leq t \leq n-m$.

\begin{lemma}\label{lem415}
Let $d_t=\#\{V\subset \overline{M}\mid \textit{the dimension of $V$ is $m-t$}\}$.
Then
\[
d_{t}=\#\mathrm{Gr}_{\kappa}(t,m).
\]
Here, $\mathrm{Gr}_{\kappa}(t,m)$ is the Grassmannian variety classifying $t$-dimensional subspaces of $\overline{M}$. 


\end{lemma}
\begin{proof}
The dimension of $\overline{M}$, as a $\kappa$-vector space,  is $m$. 
Therefore the number of $(m-t)$-dimensional subspaces of $\overline{M}$ is $\#\mathrm{Gr}_{\kappa}(m-t,m)(\kappa)=\#\mathrm{Gr}_{\kappa}(t,m)(\kappa)$.
\end{proof}

By plugging the above two lemmas into Proposition \ref{stra1}, we have a more refined reduction formula as follows:

\begin{proposition}\label{stra2}
We have 
\[
\mathcal{SO}_{\gamma}=\sum\limits_{\substack{(k_1, \cdots, k_{n-m}),  \\ \sum k_i=\mathrm{ord}(c_n),\\ m\geq 0 }}
c_{(k_1, \cdots, k_{n-m})}\cdot
\left(
\sum\limits_{t=1}^{n-m}d_t\cdot \mathcal{SO}_{\gamma, (k_1, \cdots, k_{n-m}),t}
\right).
\]
\end{proposition}





\subsection{Reduction on \texorpdfstring{$\mathcal{SO}_{\gamma,\mathrm{Hom}(L,\pi^{k}L)}$}{SO}}\label{sec44}
In this subsection, we will make a relation between $\mathcal{SO}_{\gamma,\mathrm{Hom}(L,\pi^{k}L)}$ and $\mathcal{SO}_{\gamma^{(k)}}$ for a certain element $\gamma^{(k)}\in \End(L)$ with an integer $k\geq 1$. 

\subsubsection{Description of $\mathcal{SO}_{\gamma, \mathrm{Hom}(L,\pi^k L)}$}
We endow the $\mfo$-scheme structure on $\mathrm{Hom}(L,\pi^k L)$ as follows:
Define a functor  $\mathrm{Hom}(L,\pi^k L)$ on the category of flat $\mfo$-algebras to the category of sets such that 
\[
\mathrm{Hom}(L, \pi^k L)(R)=\mathrm{Hom}_R(L\otimes R, \pi^k L \otimes R). 
\]
Then $\mathrm{Hom}(L,\pi^k L)$ is uniquely represented by the affine space over $\mfo$ of dimension $n^2$.
We can then define $\mathcal{SO}_{\gamma, \mathrm{Hom}(L,\pi^k L)}$ by using the same argument used in the definition of $\mathcal{SO}_{\gamma, \cL(L,M)}$ (cf. Section \ref{sec42st}) as follows;
\[
\mathcal{SO}_{\gamma, \mathrm{Hom}(L,\pi^k L)}=\int_{O_{\gamma, \mathrm{Hom}(L,\pi^k L)}}|\omega_{\chi_{\gamma}}^{\mathrm{ld}}|
\]
where \[
O_{\gamma, \mathrm{Hom}(L,\pi^k L)}=\bigsqcup_{M \subset \pi^{k}L}O_{\gamma,\cL(L,M)}.
\]

\subsubsection{Another description of $\mathcal{SO}_{\gamma, \mathrm{Hom}(L,\pi^k L)}$ using a different measure}\label{subsub442}
We keep assuming that $\chi_{\gamma}(x)$ is irreducible  whose reduction $\overline{\chi}_{\gamma}(x)$ modulo $\pi$ is $x^n$.
 We suppose that $\mathrm{ord}(c_n)\geq nk$ in this subsection.
Then $\mathrm{ord}(c_i)\geq ik$ for all $i$ by Newton polygon. 
We define another functor $\mathbb{A}_{\mathrm{Hom}(L, \pi^k L)}$  on the category of flat $\mfo$-algebras to the category of sets such that
\[
\mathbb{A}_{\mathrm{Hom}(L, \pi^{k} L)}(R)=\pi^k R \times \pi^{2k} R \times \cdots \times \pi^{nk} R
\]
for a flat $\mfo$-algebra $R$. Then $\mathbb{A}_{\mathrm{Hom}(L, \pi^k L)}$ is represented by the affine space over $\mfo$ of dimension $n$ and $\chi_{\gamma}(x)\in \mathbb{A}_{\mathrm{Hom}(L, \pi^k L)}(\mfo)$.
 Note that $\mathbb{A}_{\mathrm{Hom}(L, \pi^k L)}$ is a subfunctor of $\mathbb{A}_{\mfo}^n$ on the category of flat $\mfo$-algebras, but not a subscheme of $\mathbb{A}_{\mfo}^n$.
Nonetheless, being a subfunctor induces a morphism of schemes $\mathbb{A}_{\mathrm{Hom}(L, \pi^k L)} \longrightarrow \mathbb{A}_{\mfo}^n$.
If we restrict the morphism $\vpi_n$ to $\mathrm{Hom}(L, \pi^k L)$, then we have a well-defined morphism of schemes over $\mfo$ as follows;
\[
\varphi_{\mathrm{Hom}(L, \pi^{k} L)}: \mathrm{Hom}(L, \pi^{k} L) \longrightarrow \mathbb{A}_{\mathrm{Hom}(L, \pi^k L)}.
\]

Let $\omega_{\mathrm{Hom}(L, \pi^{k} L)}$ and $\omega_{\mathbb{A}_{\mathrm{Hom}(L, \pi^{k} L)}}$ be nonzero,  translation-invariant forms on   $\mathfrak{gl}_{n, F}$ and $\mathbb{A}^n_F$,
 respectively, with normalizations
$$\int_{\mathrm{Hom}(L, \pi^{k} L)(\mathfrak{o})}|\omega_{\mathrm{Hom}(L, \pi^{k} L)}|=1 \mathrm{~and~}  \int_{\mathbb{A}_{\mathrm{Hom}(L, \pi^{k} L)}(\mathfrak{o})}|\omega_{\mathbb{A}_{\mathrm{Hom}(L, \pi^{k} L)}}|=1.$$
Let $\omega^{ld}_{(\chi_{\gamma},\mathrm{Hom}(L, \pi^{k} L))}=\omega_{\mathrm{Hom}(L, \pi^{k} L)}/\rho_n^{\ast}\omega_{\mathbb{A}_{\mathrm{Hom}(L, \pi^{k} L)}}$ be a differential on $G_{\gamma}$. 
Here, we refer to Section \ref{measure} for the notion of the quotient of two forms.
 Then we have the following comparison among differentials;
 
 \[
\left\{
  \begin{array}{l }
|\omega_{\mathfrak{gl}_{n, \mathfrak{o}}}|=|\pi|^{kn^2}\cdot|\omega_{\mathrm{Hom}(L, \pi^{k} L)}|;\\
|\omega_{\mathbb{A}^n_{\mathfrak{o}}}|=
|\pi|^{kn(n+1)/2}\cdot |\omega_{\mathbb{A}_{\mathrm{Hom}(L, \pi^{k} L)}}|;\\
|\omega_{\chi_{\gamma}}^{\mathrm{ld}}|=|\pi|^{kn(n-1)/2}|\omega^{ld}_{(\chi_{\gamma},\mathrm{Hom}(L, \pi^{k} L))}|.
    \end{array} \right.
\]
and thus we have
\begin{equation}\label{eq477}
\mathcal{SO}_{\gamma, \mathrm{Hom}(L,\pi^k L)}=\int_{O_{\gamma, \mathrm{Hom}(L,\pi^k L)}}|\omega_{\chi_{\gamma}}^{\mathrm{ld}}|=q^{-kn(n-1)/2}\int_{O_{\gamma, \mathrm{Hom}(L,\pi^k L)}}|\omega^{ld}_{(\chi_{\gamma},\mathrm{Hom}(L, \pi^{k} L))}|.
\end{equation}
Here $\gamma$ does not need to be contained in $\mathrm{Hom}(L,\pi^k L)(\mfo)$. 
As we mentioned in the sentence just below Definition \ref{defsoi},  $\mathcal{SO}_{\gamma, \mathrm{Hom}(L,\pi^k L)}$ depends on $\chi_{\gamma}(x)$, not $\gamma$.

\subsubsection{Comparison  between $\mathcal{SO}_{\gamma, \mathrm{Hom}(L,\pi^k L)}$ and $\mathcal{SO}_{\gamma^{(k)}}$}\label{sec443}
Two schemes $\mathrm{Hom}(L, \pi^{k} L)$   and $\mathbb{A}_{\mathrm{Hom}(L, \pi^k L)}$ can be identified with  $\End(L)$ and $\mathbb{A}^n_{\mfo}$ as $\mfo$-schemes, respectively.
Precisely, for a flat $\mfo$-algebra $R$, 
define
 \[
\left\{
  \begin{array}{l }
\iota_1: \mathrm{Hom}(L, \pi^{k} L)(R)\cong \End(L)(R), ~~~~~~~ f \mapsto \frac{1}{\pi^k} f;\\
\iota_2: \mathbb{A}_{\mathrm{Hom}(L, \pi^{k} L)}(R) \cong \mathbb{A}^n_{\mfo}(R), ~~~~~~~ (\pi^k a_1, \cdots, \pi^{nk} a_n) \mapsto (a_1, \cdots, a_n).
    \end{array} \right.
\]
Then both $\iota_1$ and $\iota_2$ are isomorphisms of $\mfo$-schemes.
The following diagram is now commutative:
\begin{equation}\label{diag48}
\xymatrixcolsep{5pc}\xymatrix{
\mathrm{Hom}(L, \pi^{k} L) \ar[d]^{\iota_1, \cong} \ar[r]^{\vpi_{\mathrm{Hom}(L, \pi^{k} L)}} &  \mathbb{A}_{\mathrm{Hom}(L, \pi^{k} L)} \ar[d]^{\iota_2, \cong} \\ \End(L) \ar[r]^{\varphi_n} &  \mathbb{A}_{\mathfrak{o}}^n
}
\end{equation}

 Let 
 \[
\left\{
  \begin{array}{l }
\chi_{\gamma}(\pi^k x)/\pi^{nk}=x^n+c_{1}^{(k)}x^{n-1}+\cdots +c_{n-1}^{(k)}x+c_n^{(k)} \textit{ with } c_i^{(k)}:=c_i/\pi^{ik}\in \mfo;\\
\textit{$\gamma^{(k)}\in \End(L)$ whose characteristic polynomial (denoted by $\chi_{\gamma^{(k)}}(x)$) is $\chi_{\gamma}(\pi^k x)/\pi^{nk}$}.
    \end{array} \right.
\]
Here, the choice of $\gamma^{(k)}$ is not unique but it does not matter in our situation since our orbital integral  depends on the characteristic polynomial, not an element.

By considering $\chi_{\gamma}(x)$ as an element of $\mathbb{A}_{\mathrm{Hom}(L, \pi^{k} L)}(\mfo)$, 
the image of $\chi_{\gamma}(x)$ under the morphism $\iota_2$ is $\chi_{\gamma^{(k)}}(x)\left(\in \mathbb{A}^n_{\mfo}(\mfo)  \right)$.
The above commutative diagram (\ref{diag48}) then induces  
\begin{equation}\label{eq48}
\int_{O_{\gamma, \mathrm{Hom}(L,\pi^k L)}}|\omega^{ld}_{(\chi_{\gamma},\mathrm{Hom}(L, \pi^{k} L))}|=\int_{O_{\gamma^{(k)}, \End(L)}}|\omega^{ld}_{(\chi_{\gamma^{(k)}})}|\left(=:\mathcal{SO}_{\gamma^{(k)}}\right).
\end{equation}

If we combine Equation (\ref{eq48})  with Equation (\ref{eq477}), then we have the following formula:

\begin{proposition}\label{propredend}
Suppose that $\mathrm{ord}(c_n)\geq nk$ for an integer $k\geq 1$.
Let $\gamma^{(k)}$ be an element of $\End(L)$ whose characteristic polynomial is 
$$\chi_{\gamma^{(k)}}(x):=\chi_{\gamma}(\pi^k x)/\pi^{nk}=x^n+c_{1}^{(k)}x^{n-1}+\cdots +c_{n-1}^{(k)}x+c_n^{(k)},$$
where $c_i^{(k)}:=c_i/\pi^{ik}\in \mfo$.
Then we have
\[
\mathcal{SO}_{\gamma, \mathrm{Hom}(L,\pi^k L)}=q^{-k\cdot \frac{n(n-1)}{2}}\mathcal{SO}_{\gamma^{(k)}}.
\]
\end{proposition}

Here, the Newton polygon of an irreducible polynomial $\chi_{\gamma}(x)$  confirms that $\mathrm{ord}(c_i)\geq ik$ so that $\chi_{\gamma^{(k)}}(x)\in \mathbb{A}^n_{\mfo}(\mfo)$.

\begin{remark}\label{remarkdisc}
Since
\[
\mathrm{ord}(\Delta_{\gamma^{(k)}})=\mathrm{ord}(\Delta_{\gamma})-k\cdot n(n-1),
\]
the above formula gives a reduction of the orbital integral with respect to the order of the discriminant of the characteristic polynomial of $\gamma$.

\end{remark}

\section{Geometric formulation of \texorpdfstring{$\sog$}{SO}}\label{sec5}
The goal of this section is to give a geometric formulation of each stratum and then to give another formulation of $\sog$.
We keep assuming that $\chi_{\gamma}(x)$ is irreducible  whose reduction $\overline{\chi}_{\gamma}(x)$ modulo $\pi$ is $x^n$ in this section.

\subsection{Geometric formulation of each stratum}\label{sec51}

Recall from Section \ref{subsec43} that we defined the scheme $\cL(L,M,V)$, where the type of $M$ is $\left(k_1, \cdots, k_{n-m}  \right)$ and the dimension of $V$ is $\left( m-t\right)$, and that \[
O_{\gamma, \cL(L,M,V)}=\uGr(\mfo)\cap \cL(L,M,V)(\mfo)=\{f\in \cL(L,M,V)(\mfo)| \vpi_n(f)=\vpi_n(\gamma)\}.
\]
We define another functor $\mathbb{A}_{M,V}$  on the category of flat $\mfo$-algebras to the category of sets such that
\[
\mathbb{A}_{M,V}(R)=R^{m-t}\times \prod_{i=1}^{t} \pi^{\min_{0\leq i'\leq i/2}\{\left(\sum\limits_{l=1}^{i'}k_l\right)+i-2i'\}}R \times \prod_{j=1}^{t}\pi^{\left(\sum\limits_{i=1}^{j}k_i\right) +t-j}R\times 
\prod_{j=t+1}^{n-m}\pi^{\sum\limits_{i=1}^{j}k_i} R
\]
for a flat $\mfo$-algebra $R$. Then $\mathbb{A}_{M,V}$ is represented by an affine space over $\mfo$ of dimension $n$. Note that $\mathbb{A}_{M,V}$ is a subfunctor of $\mathbb{A}_{\mfo}^n$ on the category of flat $\mfo$-algebras, but not a subscheme of $\mathbb{A}_{\mfo}^n$.

\begin{remark}\label{rmkchangemea}
In this remark, we will describe a morphism $\varphi_{n,M,V}: \cL(L,M,V) \longrightarrow \mathbb{A}_{M,V}$.
\begin{enumerate}
\item 
If we restrict the morphism $\vpi_n$ to $\cL(L,M,V)$ on the category of flat $\mfo$-algebras, then it induces a well-defined morphism
\[
\varphi_{n,M,V}: \cL(L,M,V)(R) \longrightarrow \mathbb{A}_{M,V}(R).
\]
This morphism is represented by a morphism of schemes over $\mfo$. 
Furthermore, $O_{\gamma, \cL(L,M,V)}\neq \emptyset$ if and only if $\chi_{\gamma}(x)\in \mathbb{A}_{M,V}(\mfo)$. Thus, if $O_{\gamma, \cL(L,M,V)}\neq \emptyset$, then
\[
\varphi_{n,M,V}^{-1}(\chi_{\gamma})(\mfo)=O_{\gamma, \cL(L,M,V)}.
\]

\item Let us explain how to describe the morphism $\varphi_{n,M,V}$ in terms of matrices. 
For simplicity, we suppose that $t=1$. 
We write $m\in \cL(L,M,V)(R)$ as a matrix of the form described in  Equations (\ref{matrixform}) and (\ref{matrixform2}), for an arbitrary $\mfo$-algebra  $R$. 
Then compute  $\varphi_{n,M,V}(m)$ formally. It is of the form 
\[
\begin{pmatrix}r_1, \cdots, r_{m-1}, \pi^1r_{m},  \pi^{k_1}r_{m+1},   \cdots, \pi^{\sum\limits_{i=1}^{n-m}k_i}r_n\end{pmatrix},
\]
where $r_i\in R$.
The image of $m$, under the morphism $\varphi_{n,M,V}$, is then $\begin{pmatrix}
r_1, \cdots, r_n\end{pmatrix}$.  
\end{enumerate}
\end{remark}

We now define an inherent volume form on  $O_{\gamma, \cL(L,M,V)}$ based on the morphism $\varphi_{n,M,V}$ by using a similar argument used in Section \ref{subsub442}.
Recall that we defined  $\widetilde{\cL}(L,M,V)$ in the proof of Lemma \ref{lem413} such that  $\cL(L,M,V)$ is an open subscheme of  $\widetilde{\cL}(L,M,V)$. 
Then $\varphi_{n,M,V}$ can be defined on $\widetilde{\cL}(L,M,V)$ so that we have 
\[
\tilde{\varphi}_{n,M,V}: \widetilde{\cL}(L,M,V) \longrightarrow \mathbb{A}_{M,V}.
\]
Here, we emphasize that the generic fibers of $\tilde{\varphi}_{n,M,V}$, $\widetilde{\cL}(L,M,V)$, and  $\mathbb{A}_{M,V}$  are   $\rho_n$, $\mathfrak{gl}_{n,F}$, and $\mathbb{A}_{F}^n$, respectively.

Let $\omega_{\widetilde{\cL}(L,M,V)}$ and $\omega_{\mathbb{A}_{M,V}}$ be nonzero,  translation-invariant forms on   $\mathfrak{gl}_{n, F}$ and $\mathbb{A}^n_F$,
 respectively, with normalizations
$$\int_{\widetilde{\cL}(L,M,V)(\mathfrak{o})}|\omega_{\widetilde{\cL}(L,M,V)}|=1 \mathrm{~and~}  \int_{\mathbb{A}_{M,V}(\mathfrak{o})}|\omega_{\mathbb{A}_{M,V}}|=1.$$
Let $\omega^{ld}_{(\chi_{\gamma},\widetilde{\cL}(L,M,V))}=\omega_{\widetilde{\cL}(L,M,V)}/\rho_n^{\ast}\omega_{\mathbb{A}_{M,V}}$ be a differential on $G_{\gamma}$. 
Here, we refer to Section \ref{measure} for the notion of the quotient of two forms.
 Then we have the following comparison among differentials;
 
 \[
\left\{
  \begin{array}{l }
|\omega_{\mathfrak{gl}_{n, \mathfrak{o}}}|=|\pi|^{n(k_1+\cdots +k_{n-m})+mt}|\omega_{\widetilde{\cL}(L,M,V)}|;\\
|\omega_{\mathbb{A}^n_{\mathfrak{o}}}|=
|\pi|^{t^2+\sum\limits_{i=1}^{n-m}(n-m-i+1)k_i}|\omega_{\mathbb{A}_{M,V}}|;\\
|\omega_{\chi_{\gamma}}^{\mathrm{ld}}|=|\pi|^{t(m-t)+\sum\limits_{i=1}^{n-m}(m+i-1)k_i}|\omega^{ld}_{(\chi_{\gamma},\widetilde{\cL}(L,M,V))}|.
    \end{array} \right.
\]
This yields the following formulation:
\begin{proposition}\label{prop52}
$\mathcal{SO}_{\gamma, (k_1, \cdots, k_{n-m}),t}$, which is the same as $\mathcal{SO}_{\gamma, \cL(L,M,V)}$, is formulated as follows:
\[
\mathcal{SO}_{\gamma, (k_1, \cdots, k_{n-m}),t}
=\int_{O_{\gamma, \cL(L,M,V)}}|\omega_{\chi_{\gamma}}^{\mathrm{ld}}|=q^{-t(m-t)-\sum\limits_{i=1}^{n-m}(m+i-1)k_i}\int_{O_{\gamma, \cL(L,M,V)}}|\omega^{ld}_{(\chi_{\gamma},\widetilde{\cL}(L,M,V))}|.
\]
\end{proposition}


\begin{remark}\label{rmk53}
We can also work with $\cL(L,M)$ in order to reformulate $\mathcal{SO}_{\gamma, \cL(L,M)}$ by using the same argument used in the above. Since the procedure is  identical but simpler than the above case, we just list constructions and results.

\begin{enumerate}
\item The functor $\widetilde{\cL}(L,M)$ with 
$$\widetilde{\cL}(L,M)(R)=\{f|\textit{$f:L\otimes R\rightarrow M\otimes R $ is  $R$-linear}\}$$
 for a flat $\mfo$-algebra $R$ is represented by an affine space over $\mfo$ of dimension $n^2$ such that $\cL(L,M)$ is an open subscheme of $\widetilde{\cL}(L,M)$.
 
 \item The functor $\mathbb{A}_{M}$ with
\[
\mathbb{A}_{M}(R)=R^{m}\times  
\prod_{j=1}^{n-m}\pi^{\sum\limits_{i=1}^{j}k_i} R
\]
for a flat $\mfo$-algebra $R$
is represented by an affine space over $\mfo$ of dimension $n$.

\item The morphisms
 \[
\left\{
  \begin{array}{l }
\varphi_{n,M}: \cL(L,M)(R) \longrightarrow \mathbb{A}_{M}(R);\\
\tilde{\varphi}_{n,M}: \widetilde{\cL}(L,M)(R) \longrightarrow \mathbb{A}_{M}(R)
    \end{array} \right.
\]
for a flat $\mfo$-algebra $R$
are represented by  morphisms of schemes over $\mfo$ such that 
 \[
\left\{
  \begin{array}{l }
\varphi_{n,M}^{-1}(\chi_{\gamma})(\mfo)=O_{\gamma, \cL(L,M)};\\
\textit{the generic fibers of $\tilde{\varphi}_{n,M}$, $\widetilde{\cL}(L,M)$, and $\mathbb{A}_{M}$ are $\rho_n$, $\mathfrak{gl}_{n,F}$, and $\mathbb{A}_{F}^n$, respectively.}
    \end{array} \right.
\]

\item Let $\omega_{\widetilde{\cL}(L,M)}$ and $\omega_{\mathbb{A}_{M}}$ be nonzero,  translation-invariant forms on   $\mathfrak{gl}_{n, F}$ and $\mathbb{A}^n_F$,
 respectively, with normalizations
$$\int_{\widetilde{\cL}(L,M)(\mathfrak{o})}|\omega_{\widetilde{\cL}(L,M)}|=1 \mathrm{~and~}  \int_{\mathbb{A}_{M}(\mathfrak{o})}|\omega_{\mathbb{A}_{M}}|=1.$$
Let $\omega^{ld}_{(\chi_{\gamma},\widetilde{\cL}(L,M))}=\omega_{\widetilde{\cL}(L,M)}/\rho_n^{\ast}\omega_{\mathbb{A}_{M}}$ be a differential on $G_{\gamma}$. 
 Then we have the following comparison among differentials;
 
 \[
\left\{
  \begin{array}{l }
|\omega_{\mathfrak{gl}_{n, \mathfrak{o}}}|=|\pi|^{n(k_1+\cdots +k_{n-m})}|\omega_{\widetilde{\cL}(L,M)}|;\\
|\omega_{\mathbb{A}^n_{\mathfrak{o}}}|=
|\pi|^{\sum\limits_{i=1}^{n-m}(n-m-i+1)k_i} |\omega_{\mathbb{A}_{M}}|;\\
|\omega_{\chi_{\gamma}}^{\mathrm{ld}}|=|\pi|^{\sum\limits_{i=1}^{n-m}(m+i-1)k_i}|\omega^{ld}_{(\chi_{\gamma},\widetilde{\cL}(L,M))}|.
    \end{array} \right.
\]
\end{enumerate}
\end{remark}
We finally have the following formulation:
\begin{proposition}\label{prop54}
$\mathcal{SO}_{\gamma, (k_1, \cdots, k_{n-m})}$, which is the same as $\mathcal{SO}_{\gamma, \cL(L,M)}$, is formulated as follows:
\[\mathcal{SO}_{\gamma, (k_1, \cdots, k_{n-m})}
=\int_{O_{\gamma, \cL(L,M)}}|\omega_{\chi_{\gamma}}^{\mathrm{ld}}|=q^{-\sum\limits_{i=1}^{n-m}(m+i-1)k_i}\int_{O_{\gamma, \cL(L,M)}}|\omega^{ld}_{(\chi_{\gamma},\widetilde{\cL}(L,M))}|.
\]
\end{proposition}

\begin{remark}
Based on Propositions \ref{stra1}, \ref{stra2}  and Propositions \ref{prop52}, \ref{prop54},
if we can compute $\int_{O_{\gamma, \cL(L,M,V)}}|\omega^{ld}_{(\chi_{\gamma},\widetilde{\cL}(L,M,V))}|$ or $\int_{O_{\gamma, \cL(L,M)}}|\omega^{ld}_{(\chi_{\gamma},\widetilde{\cL}(L,M))}|$, then we obtain the precise formula for $\sog$.

If the scheme $\varphi_{n,M}^{-1}(\chi_{\gamma})$ or $\varphi_{n,M,V}^{-1}(\chi_{\gamma})$ is smooth, then this integral  is $\frac{\#\varphi_{n,M}^{-1}(\chi_{\gamma})(\kappa)}{q^{n^2-n}}$ or 
$\frac{\#\varphi_{n,M,V}^{-1}(\chi_{\gamma})(\kappa)}{q^{n^2-n}}$, respectively, by \cite[Theorem 2.2.5]{Weil}.
However, these two schemes are far from smoothness in the general situation.
In order to obtain a smooth morphism, there will probably be required more congruence conditions beyond our two stratifications.

\end{remark}

In the following two subsections, we treat two extreme cases; when $m=n-1$ and when $m=0$.

\subsection{The case when \texorpdfstring{$m=n-1$}{m=n-1}}\label{subsec52}
Suppose that the type of $M$ is $(k_1)$ so that $m=n-1$.
In this case, $\mathbb{A}_M(R)=R^{n-1}\times \pi^{k_1} R$ for an $\mfo$-algebra $R$. 
\begin{theorem}\label{thm55}
The scheme $\vpi_{n,M}^{-1}(\chi_{\gamma})$ is smooth over $\mfo$.
\end{theorem}
\begin{proof}
It suffices to show that $\vpi_{n,M}^{-1}(\chi_{\gamma})$ is contained in the smooth locus of the morphism $\varphi_{n,M}$ since smoothness is stable under the base change.
By \cite[Lemma 5.5.1]{GY} combined with the fact that $\cL(L,M)$ is flat over $\mfo$, it suffices to show that both the generic fiber and the special fiber of $\vpi_{n,M}$ are smooth at all points of $\vpi_{n,M}^{-1}(\chi_{\gamma})$.

The generic fiber of $\vpi_{n,M}$ is the same as the generic fiber of $\vpi_n$, which is  $\rho_n$. 
It suffices to show that $\vpi_{n,M}\otimes F$ is smooth at any element of $\vpi_{n,M}^{-1}(\chi_{\gamma})(\bar{F})$.
Since an element of $\vpi_{n,M}^{-1}(\chi_{\gamma})(\bar{F})$ is conjugate to $\gamma$ so as to be regular, $\rho_n$ is smooth by \cite[Theorem 4.20]{Hum}. 
The case of the special fiber is treated in the following lemma, which completes the proof of the theorem.
\end{proof}

\begin{lemma}\label{lemsm}
The special fiber of $\vpi_{n,M}$, formulated as follows,
\[\varphi_{n,M}\otimes \kappa: \cL(L,M)\otimes \kappa \longrightarrow \mathbb{A}_{M}\otimes \kappa\]
 is smooth at all elements of $\vpi_{n,M}^{-1}(\chi_{\gamma})\otimes \kappa$.
\end{lemma}

\begin{proof}
It suffices to show that for any $m\in \varphi_{n,M}^{-1}(\chi(\gamma))(\bar{\kappa})$,
the induced map on the Zariski tangent space 
\[
d(\vpi_{n,M})_{\ast, m}: T_m \longrightarrow T_{\vpi_{n,M}(m)}
\] 
is surjective, where $T_m$ is the Zariski tangent space of $ \cL(L,M)\otimes \bar{\kappa}$ at $m$ and $T_{\vpi_{n,M}(m)}$ is the Zariski tangent space of $\mathbb{A}_{M}\otimes \bar{\kappa}$ at $\vpi_{n,M}(m)$.

Using the matrix description given in Equation (\ref{matrixform}), 
we write $m\in \varphi_{n,M}^{-1}(\chi(\gamma))(\bar{\kappa})$ and $X\in T_m$ as the following matrices formally;
 \[
m=\begin{pmatrix}
x &y \\ \pi^{k_1} z & \pi^{k_1} w
\end{pmatrix}   \textit{   and   } X=\begin{pmatrix}
a &b \\ \pi^{k_1} c & \pi^{k_1} d
\end{pmatrix},
\]
 where $x$ and $a$ are matrices of size $(n-1)\times (n-1)$  with entries in $\bar{\kappa}$, 
  $w$ and $d$ are matrices of size $1\times 1$ with entries in $\bar{\kappa}$ and so on, and the matrix $\begin{pmatrix}
x &y \\ z & w
\end{pmatrix}$ is invertible as a matrix of $\mathfrak{gl}_n(\bar{\kappa})$. 
Note that the matrix  $\begin{pmatrix}
a &b \\ c & d
\end{pmatrix}$ does not need to be invertible since $X$ is an element of the  tangent space $T_{m}$.
Since $\cL(L,M)$ is an open subscheme of $\widetilde{\cL}(L,M)$, $T_m$ is also the tangent space of $\widetilde{\cL}(L,M)\otimes \tilde{\kappa}$ at $m$.

Since the characteristic polynomial of $x$ is of the form $x^{n-1}$ and the rank of $x$ is $n-2$ as a matrix in $M_{n-1}(\bar{\kappa})$ by Lemma \ref{lem411}, $x$ is conjugate to the Jordan canonical form having one Jordan component with eigenvalue $0$. 
By change of variables, we may and do assume that $x$ is the Jordan canonical form  described as follows;
\[
x=\begin{pmatrix}
0 &1&\cdots & 0\\
\vdots &\vdots   &\ddots &\vdots  \\
0 &0&\cdots &1\\
0 &0&\cdots &0
 \end{pmatrix}.
\] 
Here we emphasize that such change of basis does not affect on congruence conditions on $m$ and $X$.
Express $z=\begin{pmatrix}
z_1, \cdots, z_{n-1}
\end{pmatrix}$ and $y={}^t\begin{pmatrix}
y_1, \cdots, y_{n-1}
\end{pmatrix}$. Then both $z_1$ and $y_{n-1}$ are nonzero in $\bar{\kappa}$ since 
the matrix  $\begin{pmatrix}
x &y \\ z & w
\end{pmatrix}$ is invertible.

We explain how to compute the image of $X$ under the morphism $d(\vpi_{n,M})_{\ast, m}$ explicitly, based on Remark \ref{rmkchangemea}.(2).
Let us  identify $T_m$ with the subset of $\cL(L,M)(\bar{\kappa}[\epsilon])$, where $\epsilon^2=0$, whose element is of the form $m+\epsilon X$. 
We formally compute $\vpi_{n,M}(m+\epsilon X)$. It consists of $n$ entries; the first is the sum of $1\times 1$ principal  minors of $m+\epsilon X$, the second  is the sum of $2 \times 2$ principal  minors of $m+\epsilon X$, and so on, and the last is the determinant of $m+\epsilon X$, up to the sign $\pm 1$. 
Then it is of   the form $\vpi_{n,M}(m)+\epsilon Y$, where $Y=\begin{pmatrix} f_1, \cdots, f_{n-1}, \pi^{k_1}f_n \end{pmatrix}$ with $f_i\in \bar{\kappa}$, as an element of $\mathbb{A}_{M}(\bar{\kappa}[\epsilon])$. 
The image of $X$ is then $\begin{pmatrix}
f_1, \cdots, f_{n-1}, f_n
\end{pmatrix}\in \bar{\kappa}^n$. 

Our method to prove the surjectivity of $d(\vpi_{n,M})_{\ast, m}$ is to choose a certain subspace of $T_m$ mapping onto $T_{\vpi_{n,M}(m)}$.

Take $X$ such that $b=0, d=0$ and all entries of $a$ and $c$ are zero except for the first column. 
Write ${}^t\begin{pmatrix} a_1, \cdots, a_{n-1}\end{pmatrix}$ (resp. $c_1$) for the first column vector of $a$ (resp. $c$). The image of such $X$ is then $\begin{pmatrix}
a_1, \cdots, a_{n-1}, c_1\cdot y_{n-1} \end{pmatrix}$, where each entry is up to the sign $\pm 1$. 
Since $y_{n-1}$ is a unit in $\bar{\kappa}$, this  shows the surjectivity of the map  $d(\vpi_{n,M})_{\ast, m}$.
\end{proof}

\begin{corollary}\label{corcounting_1}
The order of the set  $\vpi_{n,M}^{-1}(\chi_{\gamma})(\kappa)$ is 
\[
\#\vpi_{n,M}^{-1}(\chi_{\gamma})(\kappa)=
\#\mathrm{GL}_{n-1}(\kappa)\cdot q^{n-1}.
\]
\end{corollary}
\begin{proof}
We use  Remark \ref{rmkchangemea}.(2) to find the equations defining the smooth variety $\vpi_{n,M}^{-1}(\chi_{\gamma})$. 
We write
 \[
\left\{
  \begin{array}{l }
m=\begin{pmatrix}
x &y \\ \pi^{k_1} z & \pi^{k_1} w
\end{pmatrix}\in \cL(L,M)(\kappa) \textit{ such that } det\begin{pmatrix}
x &y \\ z & w
\end{pmatrix}\neq 0\in\kappa;\\
det(\gamma)=\pi^{k_1} u \textit{ with } u\in \mfo^{\times}.
    \end{array} \right.
\]

Then  
 \[
m\in \vpi_{n,M}^{-1}(\chi_{\gamma})(\kappa)\textit{ if and only if }\left\{
  \begin{array}{l }
\textit{the characteristic polynomial of $x$ is $x^{n-1}$};\\
\textit{the rank of $x$ is $n-2$};\\
det\begin{pmatrix}
x &y \\ z & w
\end{pmatrix}=\bar{u}\in \kappa^{\times}.
    \end{array} \right.
\]
 
In particular the matrix $x$ is conjugate to the Jordan canonical form consisting of  one Jordan block of size $n-1$, which is $J_{n-1}:=
\begin{pmatrix}
0 &1&\cdots & 0\\
\vdots &\vdots   &\ddots &\vdots  \\
0 &0&\cdots &1\\
0 &0&\cdots &0
 \end{pmatrix}.$
If $x$ is the Jordan canonical form  $J_{n-1}$, then 
the only condition that $m$ is   an element of $\vpi_{n,M}^{-1}(\chi_{\gamma})(\kappa)$ is that $z_1\cdot y_{n-1}=\bar{u}$. 
Thus there are $q-1$ choices for $z_1$ and it determines $y_{n-1}$.
All of the other entries in $y, z, w$ are arbitrary. This shows that 
\[
\#\vpi_{n,M}^{-1}(\chi_{\gamma})(\kappa)=(\textit{the number of matrices conjugate to $J_{n-1}$}) \times \#\left(\mathbb{G}_m\times \mathbb{A}_{\kappa}^{2n-3}\right)(\kappa).
\]
Here $\mathbb{A}_{\kappa}^{2n-3}$ is the affine space of dimension $2n-3$ over $\kappa$.
The latter number is $(q-1)q^{2n-3}$. 

We claim that the former number is $\frac{\#\mathrm{GL}_{n-1}(\kappa)}{q^{n-2}(q-1)}$.
 The group $\mathrm{GL}_{n-1}(\kappa)$ acts transitively on the set of matrices conjugate to $J_{n-1}$. 
The stabilizer of $J_{n-1}$ is isomorphic to the group of units in $\kappa[x]/(x^{n-1})$ (since it is the group of $\kappa$ points of the Weil restriction of $\mathbb{G}_m$ with respect to a totally ramified extension of $\mfo$ of degree $n-1$). Thus the order of the stabilizer is $q^{n-1}-q^{n-2}$. This completes our claim.
\end{proof}

This corollary combined with Remark \ref{rmk53}.(4), which is $|\omega_{\chi_{\gamma}}^{\mathrm{ld}}|=|\pi|^{(n-1)k_1}|\omega^{ld}_{(\chi_{\gamma},\widetilde{\cL}(L,M))}|$, yields the following formula:

\begin{corollary}\label{cornm1}
We have
\[c_{(k_1)} \cdot \mathcal{SO}_{\gamma, (k_1)}=\frac{\#\mathrm{GL}_{n}(\kappa)}{(q-1)\cdot q^{n^{2}-1}}.\]
\end{corollary}

\begin{proof}
\cite[Theorem 2.2.5]{Weil} says that 
\[
\int_{O_{\gamma, \cL(L,M)}}|\omega^{ld}_{(\chi_{\gamma},\widetilde{\cL}(L,M))}|=
\frac{\#\vpi_{n,M}^{-1}(\chi_{\gamma})(\kappa)}{q^{n^2-n}}=
\frac{\#\mathrm{GL}_{n-1}(\kappa)}{q^{(n-1)^2}}
\]
since $\vpi_{n,M}^{-1}(\chi_{\gamma})$ is smooth over $\mfo$ with 
$\vpi_{n,M}^{-1}(\chi_{\gamma})(\mfo)=O_{\gamma, \cL(L,M)}$ and thus 
$\omega^{ld}_{(\chi_{\gamma},\widetilde{\cL}(L,M))}$ is a differential of top degree on $\vpi_{n,M}^{-1}$ over $\mfo$.
Here, the dimension of the special fiber of $\vpi_{n,M}^{-1}(\chi_{\gamma})$ is  $n^2-n$.

Therefore, Proposition \ref{prop54} yields
\[
\mathcal{SO}_{\gamma, (k_1)}=\int_{O_{\gamma, \cL(L,M)}}|\omega_{\chi_{\gamma}}^{\mathrm{ld}}|=q^{-(n-1)k_1}\int_{O_{\gamma, \cL(L,M)}}|\omega^{ld}_{(\chi_{\gamma},\widetilde{\cL}(L,M))}|=
\frac{\#\mathrm{GL}_{n-1}(\kappa)}{ q^{(n-1)^2+k_1(n-1)}}.
\]
From  Corollary \ref{cor49}, we have 
\[
c_{(k_1)}=\frac{q^{n}-1}{q-1}\cdot q^{(n-1)(k_{1}-1)}=\frac{\#\mathrm{GL}_{n}(\kappa)}{\#\mathrm{GL}_{n-1}(\kappa)\cdot (q-1)}\cdot q^{(n-1)(k_{1}-2)}.
\]
The desired formula is the product of these two.
\end{proof}

\begin{remark}
\begin{enumerate}
\item If $\mathrm{ord}(c_{n})=k_1=1$, then $c_{(k_1)}\cdot \mathcal{SO}_{\gamma, (k_1)}$ in the above corollary is the same as $\sog$. The formula for $\sog$ in this case is well known in the literature, such as \cite{Kot05}, which is the same as our formula.

\item \cite[Theorem 4.20]{Hum} proves that the smooth locus of $\rho_n$, over a field, is the same as the regular locus. 
Indeed, this theorem can also be proved by using the same argument used in the proof of  Lemma \ref{lemsm} at least in the case of $\mathfrak{gl}_n$.
Therefore, the proof of  Lemma \ref{lemsm} provides another proof of the well known fact, \cite[Theorem 4.20]{Hum} in the case of  $\mathfrak{gl}_n$.
\end{enumerate}
\end{remark}

\subsection{The case when \texorpdfstring{$m=0$}{m=0}}
Suppose that $m=0$ and thus the type of $M$ is $(k_1, \cdots, k_n)$ with $k_1\geq 1$.
In this case, $\mathbb{A}_M(R)= \pi^{\sum\limits_{i=1}^{1}k_i} R\times \cdots \times \pi^{\sum\limits_{i=1}^{n}k_i}R$ for an $\mfo$-algebra $R$ with $k_1\leq \cdots \leq k_n$,  as given in Remark \ref{rmk53}.(2).
Recall from Proposition \ref{propredend} that 
 \[
\left\{
  \begin{array}{l }
\chi_{\gamma^{(k_1)}}(x):=\chi_{\gamma}(\pi^{k_1} x)/\pi^{k_1n}=x^n+c_{1}^{(k_1)}x^{n-1}+\cdots +c_{n-1}^{(k_1)}x+c_n^{(k_1)} \textit{ with }c_i^{(k_1)}:=c_i/\pi^{ik_1}\in \mfo;\\
\textit{$\gamma^{(k_1)}\in\glno(\mfo)$ whose characteristic polynomial is $\chi_{\gamma^{(k_1)}}(x)$.}
    \end{array} \right.
\]

\begin{proposition}\label{propendred}
We have
\[
\mathcal{SO}_{\gamma, (k_1, \cdots, k_n)}=q^{-k_1\cdot \frac{n(n-1)}{2}}\mathcal{SO}_{\gamma^{(k_1)}, (k_1-k_1, \cdots, k_n-k_1)}.
\]
\end{proposition}
\begin{proof}
Since $M\subset \pi^{k_1}L$, $\cL(L,M)(\mfo)$ is an open subset of $\mathrm{Hom}(L,\pi^{k_1}L)(\mfo)$ with respect to the $\pi$-adic topology.  
The rest of the proof is the same as that of Proposition \ref{propredend} and so we skip it.
\end{proof}

As mentioned in Remark \ref{remarkdisc}, 
$
\mathrm{ord}(\Delta_{\gamma^{(k_1)}})=\mathrm{ord}(\Delta_{\gamma})-k_1\cdot n(n-1).
$

\section{Application 1: Formula for \texorpdfstring{$\sog$}{SO} in \texorpdfstring{$\gl_1$}{gl1} and \texorpdfstring{$\gl_2$}{gl2}}\label{sec6}

Suppose that  $\chi_{\gamma}(x)$ is irreducible of degree $n$, not restricted to $2$.
We first introduce the following notation and related facts:
\begin{itemize}
\item Let $\chi_{\gamma, a}(x):=\chi_{\gamma}(x+a) \in \mfo[x]$, which is irreducible,  for $a\in \mfo$. 
Note that $\mathcal{SO}_{\chi_{\gamma}}=\mathcal{SO}_{\chi_{\gamma, a}}$ by Lemma \ref{constantlemma}.

\item Let $d_a=\ord(\chi_{\gamma, a}(0))=\ord(\chi_{\gamma}(a))$.

\item Let $\Delta_{\gamma,a}$ be the discriminant of $\chi_{\gamma,a}(x)$. Since the discriminant of a polynomial is invariant under the change of variable $x$ to $x+a$, $\Delta_{\gamma,a}=\Delta_{\gamma}$, not depending on $a$.
Then $\ord (\Delta_{\gamma,a})$ gives an upper bound for $d_a$ given by 
\[
(n-1)d_a\leq \ord (\Delta_{\gamma,a})
\]
for any $a\in \mfo$, since roots of $\chi_{\gamma, a}(x)$ in $\bar{F}$ have the order $d_a/n$.  

\item Write 
\[
\chi_{\gamma, a}(x)=x^n+\pi^{a_1}c_{1,a}x^{n-1}+\cdots+\pi^{a_{n-1}}c_{n-1,a}x+\pi^{d_a}c_{n,a},\]
where $a_i=\lceil \frac{d_a i}{n} \rceil$ for $1\leq i<n$ and where $c_{1,a},\cdots, c_{n-1,a}\in \mfo$ and $c_{n,a}\in \mfo^{\times}$. 
  This is justified by  Newton polygon of  $\chi_{\gamma, a}(x)$, which is irreducible.
Here, $\lceil \frac{d_a}{n} \rceil$ is the smallest integer which is greater than or equal to $\frac{d_a}{n}$, and the same principle for  $\lceil \frac{d_a i}{n} \rceil$.

\item Recall from the paragraph just before Proposition \ref{propserre} that   
$$F_{\chi_{\gamma}} = F[x]/(\chi_{\gamma}(x))$$
is a finite field extension of $F$ of degree $n$. 
Note that $F_{\chi_{\gamma}}\cong F_{\chi_{\gamma, a}}$ as an $F$-algebra.

\end{itemize}

\begin{proposition}\label{charred3}
For a prime number $n$, there exists $a\in \mfo$ such that 
\[
\left\{
\begin{array}{l l}
    d_{a}\equiv 0 \textit{ modulo }n   &\textit{ if $F_{\chi_{\gamma}}$ is unramified over $F$};\\
    d_{a}\not\equiv 0 \textit{ modulo }n  &\textit{ if $F_{\chi_{\gamma}}$ is (totally) ramified over $F$}
\end{array}
\right.
\]
and in the first case, the reduction of  $x^n+c_{1,a}x^{n-1}+\cdots+c_{n-1,a}x+c_{n,a}$ modulo $\pi$ is irreducible over $\kappa$.
\end{proposition}

\begin{proof}
If $d_{0}\not\equiv 0 \textit{ modulo }n$, then set $a=0$.
Now suppose that $d_{0}\equiv 0 \textit{ modulo }n$ and put $d_0=nt$.
Since $\chi_{\gamma}(\pi^{t}x)/\pi^{nt}=x^n+c_{1,0}x^{n-1}+\cdots+c_{n-1,0}x+c_{n,0}\in \mfo[x]$ is  irreducible over $\mfo$ with $c_{n,0}\in \mfo^{\times}$,  the reduction modulo $\pi$  is either irreducible or a power of one linear polynomial over  $\kappa$ by Hensel's lemma.
If the reduction is irreducible, then set $a=0$.
Otherwise, if the reduction is a power of a linear poloynomial over $\kappa$, then choose $\alpha\in \mfo^{\times}$ such that $\bar{\alpha}\in \kappa^{\times}$ is a root of the reduction of $x^n+c_{1,0}x^{n-1}+\cdots+c_{n-1,0}x+c_{n,0}$.
Working with   $\chi_{\gamma, \pi^t \alpha}(x)$,
we have that $d_{\pi^t \alpha}\geq d_0+1$. 
If $d_{\pi^t \alpha} \not\equiv 0$ modulo $n$ or $d_{\pi^{t}\alpha}\equiv 0$ and the reduction of $x^n+c_{1,\pi^{t}\alpha}x^{n-1}+\cdots+c_{n-1,\pi^{t}\alpha}x+c_{n,\pi^{t}\alpha}$ is irreducible over $\kappa$, then we take $a=\pi^t \alpha$.
Otherwise, we repeat this process with $\chi_{\gamma, \pi^t \alpha}(x)$ instead of $\chi_{\gamma}(x)$. 
Since $d_a\leq ord(\Delta_{\gamma,a})$ and $\Delta_{\gamma,a}=\Delta_{\gamma}$  for any $a\in \mfo$ as explained just before the proposition, the process is finalized in a finitely many steps.



Suppose that $d_a \equiv 0 $ modulo $n$ and the reduction of  
$x^n+c_{1,a}x^{n-1}+\cdots+c_{n-1,a}x+c_{n,a}$ modulo $\pi$ is irreducible over $\kappa$ for some $a\in \mfo$.
Then $x^n+c_{1,a}x^{n-1}+\cdots+c_{n-1,a}x+c_{n,a}$ is irreducible over $\mfo$ by Hensel's lemma and the field extension $F[x]/(x^n+c_{1,a}x^{n-1}+\cdots+c_{n-1,a}x+c_{n,a})$ is unramified over $F$.
Since
 \[
 x^n+c_{1,a}x^{n-1}+\cdots+c_{n-1,a}x+c_{n,a}=\frac{\chi_{\gamma, a }(\pi^{\frac{d_{a}}{n}} x)}{\pi^{d_a}},\] 
we have that $F[x]/(x^n+c_{1,a}x^{n-1}+\cdots+c_{n-1,a}x+c_{n,a}) \cong F_{\chi_{\gamma, a}}\cong F_{\chi_{\gamma}}$ as an $F$-algebra. 
This proves that the field extension $F_{\chi_{\gamma}}$ over $F$ is unramified. 

On the other hand, if $d_a\not\equiv 0 $ modulo $n$, 
then the order of a root of $\chi_{\gamma, a}(x)$ which is taken in $\bar{F}$ is $d_a/n$.
Since  $d_a/n$ is not an integer,  
the field extension $F[x]/(\chi_{\gamma, a}(x))$ over $F$ is not unramified so as to be totally ramified since $[F[x]/(\chi_{\gamma, a}(x)):F]=n$ is a prime integer.
\end{proof}

\subsection{Formula for \texorpdfstring{$\sog$}{SO} in  \texorpdfstring{$\gl_2$}{gl2}}\label{sec61}
Recall that $\chi_\gamma(x)$ is assumed to be irreducible over $\mfo$.
 Suppose the degree of $\chi_{\gamma}(x)$ is $2$ and  that the reduction of the characteristic polynomial $\chi_{\gamma}(x)=x^2+c_1x+c_2$  modulo $\pi$ is $x^2$. 
Then $F_{\chi_{\gamma}}=F[x]/(\chi_{\gamma}(x))$ is a quadratic extension over $F$. By Proposition \ref{charred3}, there exists $a\in \mfo$ such that
\[
\left\{
\begin{array}{l l}
    d_{a}\equiv 0 \textit{ modulo }2   &\textit{ if $F_{\chi_{\gamma}}$ is unramified over $F$};\\
    d_{a}\equiv 1 \textit{ modulo }2  &\textit{ if $F_{\chi_{\gamma}}$ is (totally) ramified over $F$}
\end{array}
\right.
\]
and in the first case, the reduction of $\frac{\chi_{\gamma,a}(\pi^{\frac{d_{a}}{2}}x)}{\pi^{d_{a}}}=x^{2}+c_{1,a}x+c_{2,a}$ is irreducible over $\kappa$.
Since $\mathcal{SO}_{\chi_{\gamma}}=\mathcal{SO}_{\chi_{\gamma, a}}$ by Lemma \ref{constantlemma}, we may and do work with $\chi_{\gamma, a}$ instead of $\chi_{\gamma}$.

\begin{proposition}\label{da2}
The integer $d_a$ is characterized canonically as follows:
\[
S(\gamma)=
\left\{
\begin{array}{l l}
    \frac{d_{a}}{2} &\textit{if }F_{\chi_{\gamma}}\textit{ is unramified over }F;  \\
    \frac{d_{a}-1}{2} &\textit{if }F_{\chi_{\gamma}}\textit{ is ramified over }F.
\end{array}\right.
\]
Here, $S(\gamma)$ is the Serre invariant, which is the relative $\mfo$-length $[\mathfrak{o}_{F_{\chi_{\gamma}}}:\mathfrak{o}[x]/(\chi_{\gamma}(x))]$ explained in Proposition \ref{propserre}. 
\end{proposition}
The proof of this proposition is identical to that of  Proposition \ref{da3} and so we skip it since Proposition \ref{da3} treats a more complicated case.

\begin{corollary}
When $n=2$, $d_{a}$ is independent of any choice of $a\in\mfo$ which satisfies the conditions in Proposition \ref{charred3} for $\chi_{\gamma}(x)$. 
\end{corollary}

From now on until the end of this section, we denote the integer $d_{a}$ by $d_{\gamma}$.

\begin{theorem}\label{result2}
For an elliptic regular semisimple element $\gamma \in \gl_{2}(\mfo)$, the orbital integral for $\gamma$ is as follows:

\[
\mathcal{SO}_{\gamma}=\left\{\begin{array}{l l}
    q^{-S(\gamma)}\cdot \frac{q-1}{q}(1+(q+1)\cdot{q^{S(\gamma)}-1\over q-1})=\frac{q +1}{q}-\frac{2}{q^{S(\gamma)+1}} & \textit{if $F_{\chi_{\gamma}}/F$ is unramified};\\
 q^{-S(\gamma)}\cdot {q^{2}-1 \over q^{2}}\cdot{q^{S(\gamma)+1}-1\over q-1}=\frac{q+1}{q}-\frac{q+1}{q^{S(\gamma)+2}}
 & \textit{if $F_{\chi_{\gamma}}/F$ is ramified}.
\end{array}\right.
\]
\end{theorem}

\begin{proof}
As we mentioned at the beginning of Section \ref{sec61}, we may and do work with $\chi_{\gamma,a}$ instead of $\chi_{\gamma}$ where $a\in\mfo$ satisfies the conditions in Proposition \ref{charred3} for $\chi_{\gamma}(x)$.
By Proposition \ref{stra1}, 
\begin{equation}\label{2sum}
\mathcal{SO}_{\gamma}=\sum\limits_{k_1=0}^{\lfloor d_{\gamma}/2\rfloor}c_{(k_1, d-k_1)}\cdot \mathcal{SO}_{\gamma, (k_1, d-k_1)}.
\end{equation}

By Proposition \ref{propendred}, $\mathcal{SO}_{\gamma, (k_1, d_{\gamma}-k_1)}=q^{-k_1}\cdot \mathcal{SO}_{\gamma^{(k_1)}, (0, d_{\gamma}-2k_1)}$. 
Note that $c_{(k_1, d_{\gamma}-k_1)}=c_{(0, d_{\gamma}-2k_1)}$.
If $d_{\gamma}-2k_1>0$, then Corollary \ref{cornm1} says that 
\[
{c_{(k_1, d_{\gamma}-k_1)}\cdot\mathcal{SO}_{\gamma, (k_1, d_{\gamma}-k_1)}=q^{-k_1}\cdot \frac{\#\mathrm{GL}_{2}(\kappa)}{(q-1)\cdot q^{3}}=q^{-k_{1}}\cdot \frac{q^{2}-1}{q^{2}}.}
\]

     If $F_{\chi_{\gamma}}$ is unramified over $F$, 
    then we can set $d_{\gamma}=2S(\gamma)$ by Proposition \ref{da2}. By Summation (\ref{2sum}),
    \[
    \mathcal{SO}_{\gamma}=\sum\limits_{k_1=0}^{S(\gamma)-1}q^{-k_{1}}\cdot {q^{2}-1 \over q^{2}}+q^{-S(\gamma)}\cdot \mathcal{SO}_{\gamma^{(S(\gamma))}}.
    \]
    Since the characteristic polynomial of $\gamma^{(S(\gamma))}$ is an irreducible polynomial in $\mfo[x]$, by the inductive formula in Corollary \ref{corred}, 
    \[
    \mathcal{SO}_{\gamma^{(S(\gamma))}}={q-1 \over q}.
    \]
    Therefore,
    \begin{align*}
        \mathcal{SO}_{\gamma}&={q^{2}-1 \over q^{2}}\cdot \sum\limits_{k_1=0}^{S(\gamma)-1}q^{-k_{1}}+q^{-S(\gamma)}\cdot {q-1 \over q}
        =q^{-S(\gamma)}\cdot {q-1 \over q}(1+(q+1)\cdot{q^{S(\gamma)}-1\over q-1}).
    \end{align*}
     
     If $F_{\chi_{\gamma}}$ is ramified over $F$, 
    then we can set $d_{\gamma}=2S(\gamma)+1$ by Proposition \ref{da2}. By Summation (\ref{2sum}),
    \begin{align*}
    \mathcal{SO}_{\gamma}&=\sum\limits_{k_{1}=0}^{S(\gamma)}q^{-k_{1}}\cdot {q^{2}-1 \over q^{2}}
    =q^{-S(\gamma)}\cdot {q^{2}-1 \over q^{2}}\cdot{q^{S(\gamma)+1}-1\over q-1}.
    \end{align*}
\end{proof}


\begin{theorem}\label{thm66}


For a hyperbolic regular semisimple element $\gamma \in \mathfrak{gl}_{2}(\mfo)$, 
\[
\mathcal{SO}_{\gamma}=\frac{q+1}{q}.
\]
\end{theorem}
\begin{proof}

Since $\gamma \in \mathfrak{gl}_{2}(\mfo)$ is a hyperbolic regular semisimple element, $\chi_{\gamma}(x)$ is the product of two distinct linear monic polynomials over $F$. We can then use the inductive formula of Proposition \ref{cor410}, together with the result of Section \ref{sec62}, so as to have the following formula:
\[
\mathcal{SO}_{\gamma}=\frac{\# \mathrm{GL}_{2}(\kappa)\cdot q^{-4}}{(\#\mathrm{GL}_{1}(\kappa)\cdot q^{-1})^{2}}\cdot 1^{2} = \frac{q+1}{q}.
\]
\end{proof}

\begin{remark}\label{rmk67}
Note that in \cite[Example 4.1]{Gor22}, we can easily deduce that $d_{\gamma}$ used in loc. cit. is equal to $S(\gamma)$. (Note that this is not equal to our $d_{\gamma}$.) Then, our result exactly coincides with the result in  \cite[Equation (40)]{Gor22}.
\end{remark}
\begin{remark}
Using Propositions \ref{proptrans}-\ref{propserre}, we translate our formula into  the orbital integral with respect to the measure $d\mu$. 
Suppose that $char(F) = 0$ or $char(F)>2$. Then 
\[
\left\{
\begin{array}{l l}
\mathcal{SO}_{\gamma, d\mu}=q^{S(\gamma)}&\textit{if $\gamma$ is hyperbolic};\\
\mathcal{SO}_{\gamma, d\mu}=1+(q+1)\cdot{q^{S(\gamma)}-1\over q-1}&\textit{if $\gamma$ is elliptic and $F_{\chi_{\gamma}}$ is unramified over $F$};\\
\mathcal{SO}_{\gamma, d\mu}={q^{S(\gamma)+1}-1\over q-1} &\textit{if $\gamma$ is elliptic and $F_{\chi_{\gamma}}$ is ramified over $F$}.
\end{array} \right.
\]
Note that $S(\gamma)=ord(\Delta_{\gamma})$ if $\gamma$ is hyperbolic.  
\end{remark}

\subsection{Formula for \texorpdfstring{$\sog$}{SO} in \texorpdfstring{$\gl_1$}{gl1}}\label{sec62}
When $n=1$, $\vpi_1$ is smooth over $\mfo$ and so we can directly obtain the formula
\[
\sog =1.
\]

\section{Application 2: Formula for \texorpdfstring{$\sog$}{SOg} in \texorpdfstring{$\gl_3$}{gl3}}\label{sec7}
The goal of this section is to provide an explicit formula for $\sog$ with $n=3$. We continue to use the notation in Section 6.
Recall that $\chi_{\gamma}(x)$ is assumed to be irreducible over $\mfo$.
 Suppose that the reduction of the characteristic polynomial $\chi_{\gamma}(x)=x^3+c_1x^{2}+c_2 x+c_3$  modulo $\pi$ is $x^3$. 
  Then $F_{\chi_{\gamma}}(:=F[x]/(\chi_{\gamma}(x)))$ is a cubic field extension of $F$. 
Recall from  Proposition \ref{charred3} that there exists $a\in \mfo$ such that
\[
\left\{
\begin{array}{l l}
    d_{a}\equiv 0 \textit{ modulo }3 &\textit{ if $F_{\chi_{\gamma}}$ is unramified over $F$};\\
    d_{a}\equiv 1 \textit{ or } 2 \textit{ modulo }3  &\textit{ if $F_{\chi_{\gamma}}$ is (totally) ramified over $F$}
\end{array}
\right.
\]
and in the first case, the reduction of $\frac{\chi_{\gamma,a}(\pi^{\frac{d_{a}}{3}}x)}{\pi^{d_{a}}}=x^{3}+c_{1,a}x^{2}+c_{2,a}x+c_{3,a}$ modulo  $\pi$ is irreducible over $\kappa.$
Since $\mathcal{SO}_{\chi_{\gamma}}=\mathcal{SO}_{\chi_{\gamma, a}}$ by Lemma \ref{constantlemma}, we may and do work with $\chi_{\gamma, a}$ instead of $\chi_{\gamma}$.

By Propositions \ref{stra1} and \ref{propendred}, we have
\begin{equation}\label{eqsor}
\mathcal{SO}_{\gamma}=\sum\limits_{s=0}^{\lfloor\frac{d_{a}}{3}\rfloor}q^{-3s}\sum\limits_{k_1=0}^{\lfloor\frac{d_{a}-3s}{2}\rfloor}c_{(k_1, d_{a}-3s-k_1)}\cdot \mathcal{SO}_{\gamma^{(s)}, (k_1, d_{a}-3s-k_1)}=\sum\limits_{\substack{k_1+k_2=d_{a}-3s;\\ 0\leq k_1\leq k_2;\\ 0\leq s \leq \lfloor\frac{d_{a}}{3}\rfloor}}q^{-3s}c_{(k_1, k_2)}\cdot \mathcal{SO}_{\gamma^{(s)}, (k_1, k_2)}.
\end{equation}

Here, $\lfloor\frac{d_{a}}{3}\rfloor$ is the largest integer which is less than or equal to $\frac{d_{a}}{3}$, and the same principle for $\lfloor\frac{d_{a}-3s}{2}\rfloor$.
Note that this equation holds without any restriction on $q$. 

\subsection{Smoothening of stratum}
By  virtue of Equation (\ref{eqsor}), it suffices to find an explicit formula for $\mathcal{SO}_{\gamma, (k_1, k_2)}$ for an irreducible polynomial $\chi_{\gamma}(x)=x^3+c_1x^{2}+c_2 x+c_3$.
In this subsection, we will compute it  based on  the argument used in Section \ref{subsec52}.

Let $M$ be a sublattice of $L$ of type $(k_1, k_2)$ with $k_1>0$ and let $d=k_{1}+k_{2}$. 
Recall 
\[
\mathcal{SO}_{\gamma, (k_1, k_2)}=\mathcal{SO}_{\gamma, \cL(L,M)}=\int_{O_{\gamma, \cL(L,M)}}|\omega_{\chi_{\gamma}}^{\mathrm{ld}}|,
\]
where 
 \[
\left\{
  \begin{array}{l }
  \textit{$\omega_{\chi_{\gamma}}^{\mathrm{ld}}=\omega_{\mathfrak{gl}_{n, \mathfrak{o}}}/\rho_n^{\ast}\omega_{\mathbb{A}^n_{\mathfrak{o}}}$ in Section \ref{measure}};\\
O_{\gamma, \cL(L,M)}=\varphi_{3,M}^{-1}(\chi_{\gamma})(\mfo) \textit{ in Remark \ref{rmk53}.(3)};\\
\varphi_{3,M}: \cL(L,M) \longrightarrow \mathbb{A}_{M} \textit{ in Remark \ref{rmk53}.(3)};\\
\textit{$\cL(L,M)$ is an open subscheme of $\widetilde{\cL}(L,M)$ in Remark \ref{rmk53}.(1)};\\
\textit{an element of $\widetilde{\cL}(L,M)(R)$ is of the form }\begin{pmatrix}
x_{1,1} &x_{1,2} & x_{1,3} \\
\pi^{k_1}x_{2,1} &\pi^{k_1}x_{2,2} & \pi^{k_1}x_{2,3} \\
\pi^{k_2}x_{3,1} &\pi^{k_2}x_{3,2} & \pi^{k_2}x_{3,3}
\end{pmatrix} \textit{ in Equation (\ref{matrixform})};\\
\mathbb{A}_{M}(R)=R\times \pi^{k_1}R\times \pi^d R \textit{ in Remark \ref{rmk53}.(2)}.
    \end{array} \right.
\]

Let $t=\lceil\frac{d}{3}\rceil$ and let $s=\lfloor \frac{d}{3}\rfloor$ (so that $d$ is $3t$, $3t-1$, or $3t-2$).
The relation between $s$ and $t$ is given  below:
\begin{equation}\label{eq72}
\left\{
\begin{array}{l l}
s=t  & \textit{if $d=3t$};\\
s=t-1  & \textit{if $d=3t-1$ or $d=3t-2$}.
\end{array} \right.
\end{equation}

We separate the condition $0<k_1\leq k_2$ with $k_1+k_2=d$ into four cases;
\[
\left\{
\begin{array}{l}
\textit{Case 1: $k_1< k_2$ and $k_1<t$};\\
\textit{Case 2: $k_1< k_2$ and $k_1>s$};\\
\textit{Case 3: $k_1=k_2$};\\
\textit{Case 4: $k_1=t$ and $d=3t$}.
\end{array} \right.
\]
Here, \textit{Case 4} occurs only when $F_{\chi_{\gamma}}$ is unramified over $F$ and the reduction of $\chi_{\gamma}(\pi^tx)/\pi^{3t}$ modulo $\pi$ is irreducible over $\kappa$ by Proposition \ref{charred3}.

\begin{theorem}\label{thm71}
The formula of $c_{(k_1, k_2)}\mathcal{SO}_{\gamma, (k_1, k_2)}$ is given as follows:
\[
c_{(k_1, k_2)}\cdot \mathcal{SO}_{\gamma, (k_1, k_2)}=\left\{
  \begin{array}{l l}
(q^{3}-1)(q^{2}-1)q^{k_{2}}\cdot \frac{(k_1+1)q-(k_1-1)}{q^{6+d}}   & \textit{in Case 1};\\
(q^{3}-1)(q^{2}-1)q^{k_{2}}\cdot \frac{(k_2-k_1+1)q-(k_2-k_1-1)}{q^{6+d}}   & \textit{in Case 2};\\
\frac{(q^{3}-1)(q^{2}-1)}{q^{5+d}}\cdot q^{k_{1}}   & \textit{in Case 3};\\
 (q^{3}-1)(q^{2}-1)q^{k_{2}}\cdot \frac{(t+1)q-(t-1)}{q^{6+d}}
  & \textit{in Case 4}.
    \end{array} \right.
\]
Here, $d=k_1+k_2$.
\end{theorem}
We postpone the proof to Appendix \ref{App:AppendixA}.


\subsection{Formula for \texorpdfstring{$\sog$}{SOg} in \texorpdfstring{$\gl_3$}{gl3}} For a given $\gamma$, we first clarify invariance of $d_{a}$ to $a$ where $a\in \mfo$ satisfies that
\[
\left\{
\begin{array}{l l}
    d_{a}\equiv 0 \textit{ modulo }3 &\textit{ if $F_{\chi_{\gamma}}$ is unramified over $F$};\\
    d_{a}\equiv 1 \textit{ or } 2 \textit{ modulo }3  &\textit{ if $F_{\chi_{\gamma}}$ is ramified over $F$}
\end{array}
\right.
\]
and that in the first case, the reduction of $\frac{\chi_{\gamma,a}(\pi^{\frac{d_{a}}{3}}x)}{\pi^{d_{a}}}=x^{3}+c_{1,a}x^{2}+c_{2,a}x+c_{3,a}$ is irreducible over $\kappa$ (cf. Proposition \ref{charred3}).
\begin{proposition}\label{da3}
The integer $d_a$ is characterized canonically as follows:
\[
S(\gamma)=\left\{
\begin{array}{l l}
    d_{a} &\textit{ if $F_{\chi_{\gamma}}$ is unramified over $F$};\\
    d_{a}-1 &\textit{ if $F_{\chi_{\gamma}}$ is ramified over $F$}.
\end{array}
\right.
\]
Here, $S(\gamma)$ is the Serre invariant, which is the relative $\mfo$-length $[\mathfrak{o}_{F_{\chi_{\gamma}}}:\mathfrak{o}[x]/(\chi_{\gamma}(x))]$ explained in Proposition \ref{propserre}. 
\end{proposition}
\begin{proof}
We will explain the relation between $d_{a}$ and $S(\gamma)$ by direct computation of  the relative $\mfo$-length $[\mfo_{F_{\chi_{\gamma}}}:\mfo[x]/(\chi_{\gamma}(x))]$.
\begin{itemize}
    \item If $F_{\chi_{\gamma}}/F$ is unramified, then we can set $d_{a}=3d'$. Choose a root of $\chi_{\gamma,a}(x)$ denoted by $\alpha$.
    Then its order is $d'$ and $\{1,\alpha,\alpha^{2}\}$ is a basis of $R$ as an $\mfo$-module. 
    Note that $\frac{\alpha}{\pi^{d'}}$ is a root of $f(x):=\frac{\chi_{\gamma,a}(\pi^{d'} x)}{\pi^{3d'}}$. 
    Here, $F_{\chi_{\gamma}}\cong F_{f}$ and $\mfo_{F_{\chi_{\gamma}}}\cong\mfo[x]/(f(x))$, since $\bar{f}(x)$ is irreducible over $\kappa$.
    Therefore, $\{1, \frac{\alpha}{\pi^{d'}},(\frac{\alpha}{\pi^{d'}})^{2}\}$ forms a basis of $\mfo_{F_{\chi_{\gamma}}}$ as $\mfo$-module and we have
    \[
    S(\gamma)=d'+2d'=3d'=d_a.
    \]
    \item If $F_{\chi_{\gamma}}/F$ is ramified, then we can set $d_{a}=3d'+1$ or $d_{a}=3d'+2$. 
    Choose a root of $\chi_{\gamma,a}(x)$ denoted by $\alpha$. Then its order is $d'+\frac{1}{3}$ (resp. $d'+\frac{2}{3}$) and $\{1, \alpha,\alpha^{2}\}$ is a basis of $R$ as an $\mfo$-module.
    Note that an Eisenstein polynomial defines the ring of integers of totally ramified extension over $F$, thus let $f(x)$ be the Eisenstein polynomial such that $\mfo_{F_{\chi_{\gamma}}}=\mfo[x]/(f(x))$.
    Choose a root of $f(x)$, denote it by $\alpha_{F_{\chi_{\gamma}}}$. Then its order is $\frac{1}{3}$ and $\{1,\alpha_{F_{\chi_{\gamma}}},\alpha_{F_{\chi_{\gamma}}}^{2}\}$ is a basis of $\mfo_{F_{\chi_{\gamma}}}$ as $\mfo$-module.
    We then have
        \[
    S(\gamma)=
    \left\{
    \begin{array}{l l}
    (d'+\frac{1}{3}+2d'+\frac{2}{3})-(\frac{1}{3}+\frac{2}{3})=3d'&\textit{ when }d_{a}=3d'+1;\\
    (d'+\frac{2}{3}+2d'+\frac{4}{3})-(\frac{1}{3}+\frac{2}{3})=3d'+1&\textit{ when }d_{a}=3d'+2.
    \end{array}\right.
    \]
\end{itemize}
\end{proof}

\begin{corollary}\label{cor73}
When $n=3$, $d_{a}$ is independent of any choice of  $a\in\mfo$ which satisfies the conditions in Proposition \ref{charred3} for $\chi_{\gamma}(x)$.
\end{corollary}
\begin{remark}\label{rmk74}
If $n$ is a prime integer, then we have the following relation between $S(\gamma)$ and $d_a$:
\begin{enumerate}
    \item Suppose that $F_{\chi_{\gamma}}$ is unramified over $F$.
    We write  $d_{a}=nd'$. Then we have the following relation:
        \[
    S(\gamma)=d'+2d'+\cdots+(n-1)d'=\frac{n(n-1)}{2}d'. 
    \]
    \item Suppose that  $F_{\chi_{\gamma}}$ is (totally) ramified over $F$.
    We write $d_{a}=nd'+r'$ with $0<r'<n$. Then we have the following relation: 
    \begin{align*}
        S(\gamma)&=(d'+\frac{r'}{n}+2d'+\frac{2r'}{n}+\cdots+(n-1)d'+\frac{(n-1)r'}{n})-(\frac{1}{n}+\frac{2}{n}+\cdots+\frac{n-1}{n})\\
        &=\frac{n(n-1)}{2}d'+\frac{(r'-1)(n-1)}{2}=\frac{(nd'+r'-1)(n-1)}{2}.
    \end{align*}
\end{enumerate}
The proof is the same as that of  Proposition \ref{da3} and so we skip it.
Therefore, $d_{a}$ is independent of any choice of $a\in \mfo$ which satisfies the conditions in Proposition \ref{charred3} for $\chi_{\gamma}(x)$. 

\end{remark}

\begin{theorem}\label{result3}
For an elliptic regular semisimple element $\gamma \in \gl_{3}(\mfo)$, the orbital integral for $\gamma$ is as follows:

\begin{enumerate}
\item If $F_{\chi_{\gamma}}$ is an unramified extension over $F$, then we have
\[
\mathcal{SO}_{\gamma}=\frac{(q+1)(q^{2}+q+1)}{q^{3}}-\frac{3(q^{2}+q+1)}{q^{d'+3}}+\frac{3}{q^{3d'+3}}
\]
where $S(\gamma)=3d'$.

\item If $F_{\chi_{\gamma}}$ is a ramified extension over $F$, then we have
\[
\mathcal{SO}_{\gamma}=\left\{\begin{array}{l l}
\frac{(q+1)(q^{2}+q+1)}{q^{3}}-\frac{(2q+1)(q^{2}+q+1)}{q^{d'+4}}+\frac{q^{2}+q+1}{q^{3d'+5}} & \textit{if $S(\gamma)=3d'$};\\
\frac{(q+1)(q^{2}+q+1)}{q^{3}}-\frac{(q+2)(q^{2}+q+1)}{q^{d'+4}}+\frac{q^{2}+q+1}{q^{3d'+6}} & \textit{if $S(\gamma)=3d'+1$}.
\end{array}\right.
\]

\end{enumerate}

\end{theorem}

In the theorem, $S(\gamma)$ cannot be of the form $3d'+2$ by Proposition \ref{da3}.

\begin{theorem}\label{result4}
For other regular semisimple elements $\gamma \in \gl_{3}(\mfo)$, the orbital integral for $\gamma$ is as follows:

\begin{enumerate}
\item If $\chi_{\gamma}(x)$ involves an irreducible quadratic polynomial as a factor,  then the irreducible quadratic  factor of $\chi_{\gamma}(x)$ is expressed as $\chi_{\gamma'}(x)\in \mfo[x]$ for a certain  $\gamma'\in \mathfrak{gl}_{2}(\mfo)$.  
We then have the following formula:
\[
\mathcal{SO}_{\gamma}=\left\{\begin{array}{l l}
\frac{(q+1)(q^{2}+q+1)}{q^{3}}-\frac{2(q^{2}+q+1)}{q^{S(\gamma')+3}}& \textit{if a direct summand of $F_{\chi_{\gamma}}/F$ involves unramified quadratic};\\
\frac{(q+1)(q^{2}+q+1)}{q^{3}}-\frac{(q+1)(q^{2}+q+1)}{q^{S(\gamma')+4}} & \textit{if a direct summand of $F_{\chi_{\gamma}}/F$ involves ramified quadratic}.
\end{array}\right.
\]
Here  
$S(\gamma')=S(\gamma)-\frac{\ord (\Delta_{\gamma})-\ord (\Delta_{\gamma'})}{2}$.

\item For a hyperbolic regular semisimple element $\gamma \in \mathfrak{gl}_{3}(\mfo)$,
\[
\mathcal{SO}_{\gamma}=\frac{(q+1)(q^{2}+q+1)}{q^{3}}.
\]

\end{enumerate}

\end{theorem}

We postpone the proofs of Theorems \ref{result3}-\ref{result4} to Appendix \ref{appc}.

\begin{remark}\label{rmk77}
In this remark, we will describe our formula provided in Theorems \ref{result3}-\ref{result4} with respect to another measure $d\mu$ used in \cite{Yun13}, using  Propositions \ref{proptrans}-\ref{propserre}.
Suppose that $char(F) = 0$ or $char(F)>3$. 
The orbital integral  for a regular semisimple element $\gamma \in \mathfrak{gl}_{3}(\mfo)$ with respect to $d\mu$, which was denoted by $\mathcal{SO}_{\gamma,d\mu}$ in Section \ref{sscotn}, is given as follows: 
\begin{enumerate}
\item
 For an elliptic regular semisimple element $\gamma \in \gl_{3}(\mfo)$, 
\[\mathcal{SO}_{\gamma,d\mu}=\left\{\begin{array}{l l}
\frac{(q^{2}+q+1)}{(q-1)^2}\left(q^{3d'} - 3 \cdot \frac{q^{2d'}-1}{q+1}-1\right)+1 & \textit{if $F_{\chi_{\gamma}}/F$ is unramified and $S(\gamma)=3d'$};\\
\frac{q^2}{(q-1)^2}\left(q^{3d'}-(2q+1)\frac{q^{2d'}-1}{q(q+1)}-1\right)+1& \textit{if $F_{\chi_{\gamma}}/F$ is ramified and $S(\gamma)=3d'$};\\
\frac{q^2}{(q-1)^2}\left(q^{3d'+1} - (q+2)\frac{q^{2d'}-1}{q+1}-q\right)+q+1& \textit{if $F_{\chi_{\gamma}}/F$ is ramified and  $S(\gamma)=3d'+1$}.\\
\end{array}\right.
\]

\item If $\chi_{\gamma}(x)$ involves an irreducible quadratic polynomial as a factor,  then the irreducible quadratic  factor of $\chi_{\gamma}(x)$ is expressed as $\chi_{\gamma'}(x)\in \mfo[x]$ for a certain  $\gamma'\in \mathfrak{gl}_{2}(\mfo)$.  
We then have the following formula:
\[
\mathcal{SO}_{\gamma,d\mu}=\left\{\begin{array}{l l}
q^{S(\gamma)-S(\gamma')}\cdot (1+(q+1)\frac{q^{S(\gamma')}-1}{q-1})& \textit{if a direct summand of $F_{\chi_{\gamma}}/F$ involves unramified quadratic};\\
q^{S(\gamma)-S(\gamma')}\cdot \frac{q^{S(\gamma')+1}-1}{q-1} & \textit{if a direct summand of $F_{\chi_{\gamma}}/F$ involves ramified quadratic}.
\end{array}\right.
\]

\item For a hyperbolic regular semisimple element $\gamma \in \mathfrak{gl}_{3}(\mfo)$,
\[
\mathcal{SO}_{\gamma,d\mu}=q^{S(\gamma)}.
\]
\end{enumerate}
\end{remark}

\section{Application 3: Lower bound for \texorpdfstring{$\mathcal {SO}_{\gamma}}{SOg}$ in $\mathfrak{gl}_n$}\label{sec8}
In this section, we provide a lower bound for  $\sog$, when $n\geq 4$.
Our strategy is to compute two easy terms in the formula of Proposition \ref{stra2}, when $m=n-1$ and when $m=n-2$  and $t=1$.
The case $m=n-1$ is already treated in Section \ref{subsec52} and thus we will compute the latter case.

\subsection{The case of \texorpdfstring{$m=n-2$}{m = n-2} and \texorpdfstring{$t=1$}{t=1}}
In this subsection, we will provide a formula for a lower bound of $\mathcal{SO}_{\gamma, (k_1, k_2),1}$.
Let $M$ be a sublattice of $L$ of type $(k_1, k_2)$ with $k_1>0$
and let $V$ be a subspace of $M/ M\cap \pi L$ of dimension $n-3$ over $\kappa$.  
Recall 
\[
\mathcal{SO}_{\gamma, (k_1, k_2),1}=\mathcal{SO}_{\gamma, \cL(L,M,V)}=\int_{O_{\gamma, \cL(L,M,V)}}|\omega_{\chi_{\gamma}}^{\mathrm{ld}}|,
\]
where 
 \[
\left\{
  \begin{array}{l }
  \textit{$\omega_{\chi_{\gamma}}^{\mathrm{ld}}=\omega_{\mathfrak{gl}_{n, \mathfrak{o}}}/\rho_n^{\ast}\omega_{\mathbb{A}^n_{\mathfrak{o}}}$ in Section \ref{measure}};\\
O_{\gamma, \cL(L,M,V)}=\varphi_{n,M,V}^{-1}(\chi_{\gamma})(\mfo) \textit{ in Remark \ref{rmkchangemea}.(1)};\\
\varphi_{n,M,V}: \cL(L,M,V) \longrightarrow \mathbb{A}_{M,V} \textit{ in Remark \ref{rmkchangemea}.(1)};\\
\textit{$\cL(L,M,V)$ is an open subscheme of $\widetilde{\cL}(L,M,V)$ in Lemma \ref{lem413}};\\
\mathbb{A}_{M,V}(R)=R^{n-3}\times \pi R \times \pi^{k_1} R \times  \pi^{k_1+k_2}R \textit{ in Section \ref{sec51}}.
    \end{array} \right.
\]

By Equations (\ref{matrixform})-(\ref{matrixform2}), each element $x$ of $\cL(L,M,V)(R)$ for a flat $\mfo$-algebra $R$ is expressed as the following matrix:
\begin{equation}\label{eq81}
x=\begin{pmatrix}
\pi x_{1,1} &x_{1,2}&\cdots & x_{1,n-2}&x_{1,n-1}&x_{1,n}\\
\vdots &\vdots & &\vdots &\vdots &\vdots \\
\pi x_{n-2,1} &x_{n-2,2}&\cdots & x_{n-2,n-2}&x_{n-2,n-1}&x_{n-2,n}\\
\pi^{k_1} x_{n-1,1} &\pi^{k_1} x_{n-1,2}&\cdots & \pi^{k_1} x_{n-1,n-2}&\pi^{k_1} x_{n-1,n-1}&\pi^{k_1} x_{n-1,n}\\
\pi^{k_2} x_{n,1} &\pi^{k_2} x_{n,2}&\cdots & \pi^{k_2} x_{n,n-2}&\pi^{k_2} x_{n,n-1}&\pi^{k_2} x_{n,n}
\end{pmatrix}
\end{equation}
such that
 \[
\left\{
  \begin{array}{l }
x':=\begin{pmatrix}
\pi x_{1,1} &x_{1,2}&\cdots & x_{1,n-2}&x_{1,n-1}&x_{1,n}\\
\vdots &\vdots & &\vdots &\vdots &\vdots \\
\pi x_{n-2,1} &x_{n-2,2}&\cdots & x_{n-2,n-2}&x_{n-2,n-1}&x_{n-2,n}\\
 x_{n-1,1} &  x_{n-1,2}&\cdots &  x_{n-1,n-2}&  x_{n-1,n-1}&  x_{n-1,n}\\
 x_{n,1} & x_{n,2}&\cdots &  x_{n,n-2}&  x_{n,n-1}&  x_{n,n}
\end{pmatrix}
\textit{ is invertible in }\mathrm{M}_n(R);\\
\textit{for } X_{n-2,n-2}:= \begin{pmatrix}
\pi x_{1,1} &x_{1,2}&\cdots & x_{1,n-2}\\
\vdots &\vdots & &\vdots  \\
\pi x_{n-2,1} &x_{n-2,2}&\cdots & x_{n-2,n-2}
\end{pmatrix},
\textit{$\overline{X}_{n-2,n-2}\in \mathrm{M}_{n-2}(R\otimes \kappa)$ is of rank $n-3$}.
    \end{array} \right.
\] 

Here, the rank of a matrix in $\mathrm{M}_{n-2}(R\otimes \kappa)$ means the following;
a matrix in $\mathrm{M}_{n-2}(R\otimes \kappa)$ is an endomorphism of a free $R\otimes \kappa$-module of rank $n-2$. 
If the image is free of rank $n-3$ and its cokernel is free of rank $1$, then we say that such matrix is of rank $n-3$.

\begin{remark}\label{rmk81}
\begin{enumerate}
\item 
As in the former cases treating $\gl_2$ and $\gl_3$, it would be the best to show that 
 the scheme $\vpi_{n,M, V}^{-1}(\chi_{\gamma})$ is smooth over $\mfo$.
 However, this is no longer true if $k_1<k_2$. 
 \\

\item Instead, we will prove that a certain open subscheme of  $\vpi_{n,M, V}^{-1}(\chi_{\gamma})$ is smooth over $\mfo$ in the  proposition provided below. 
To define such ``open'' subscheme, we need a bit more  setting.

\begin{enumerate}
\item Suppose that $k_1<k_2$.
Let $\mathrm{M}_{n-1}^0$ be the closed subscheme of $\mathrm{M}_{n-1}$ defined over $\mfo$ representing 
$$\mathrm{M}_{n-1}^0(R)=\begin{pmatrix}
0 &x_{1,2}&\cdots & x_{1,n-2}&x_{1,n-1}\\
\vdots &\vdots & &\vdots &\vdots  \\
0 &x_{n-2,2}&\cdots & x_{n-2,n-2}&x_{n-2,n-1}\\
 x_{n-1,1} &  x_{n-1,2}&\cdots &  x_{n-1,n-2}&  x_{n-1,n-1}
\end{pmatrix}$$ for any $\mfo$-algebra $R$, where $x_{i,j}\in R$. 
Define the projection from $\cL(L,M,V)$ to $\mathrm{M}_{n-1}^0$ as follows:
\[pr:\cL(L,M,V)\rightarrow \mathrm{M}_{n-1}^0, ~~~~~ x \mapsto \begin{pmatrix}
0 &x_{1,2}&\cdots & x_{1,n-2}&x_{1,n-1}\\
\vdots &\vdots & &\vdots &\vdots  \\
0 &x_{n-2,2}&\cdots & x_{n-2,n-2}&x_{n-2,n-1}\\
 x_{n-1,1} &  x_{n-1,2}&\cdots &  x_{n-1,n-2}&  x_{n-1,n-1}
\end{pmatrix}.\]
Here, $x$ is a matrix described in  Equation (\ref{eq81}).
Then $pr$ is obviously smooth. 

Consider the determinant morphism $det: \mathrm{M}_{n-1}^0 \rightarrow \mathbb{A}^1_{\mfo}$. The singular locus of the special fiber of the morphism $det$ is then a closed subscheme of $\mathrm{M}_{n-1}^0$ since the special fiber is determined by the ideal $(\pi)$.
Let $\tilde{\mathrm{M}}_{n-1}^0$ be its complement inside $\mathrm{M}_{n-1}^0$ so as to be an open subscheme of  $\mathrm{M}_{n-1}^0$. 

Caution that $\tilde{\mathrm{M}}_{n-1}^0$ is not the smooth locus of the morphism $det$ since we do not consider any property over the generic fiber. Only the special fiber of $\tilde{\mathrm{M}}_{n-1}^0$ is the smooth locus of the special fiber of $det$ over $\kappa$.

We now define $\cL(L,M,V)_1$ to be the inverse image of $\tilde{\mathrm{M}}_{n-1}^0$ along the morphism $pr$. 
Then  $\cL(L,M,V)_1$  is an open subscheme of $\cL(L, M, V)$ having the same generic fiber.
It is visualized as the following commutative diagram:

\[\xymatrixcolsep{5pc}\xymatrix{
\cL(L, M, V)   \ar[r]^{pr} & \mathrm{M}_{n-1}^0 \ar[r]^{det} & \mathbb{A}^1_{\mfo} \\ 
\cL(L,M,V)_1 \ar@{^{(}->}[u]^{open} \ar[r]^{pr}    & \tilde{\mathrm{M}}_{n-1}^0 \ar@{^{(}->}[u]^{open}  \ar[ur]^{det}
}
\]

Indeed  the special fiber of the composite $det\circ pr: \cL(L, M, V) \rightarrow \mathbb{A}^1_{\mfo}$ is the same as the $(n-1)$-th component of the morphism $\vpi_{n,M, V}\otimes 1$ over $\kappa$.
Let $\vpi_{n,M, V}^1$ be the  restriction of the morphism $\vpi_{n,M, V}$ to $\cL(L,M,V)_1$ so that 
\[
\vpi_{n,M, V}^1 : \cL(L,M,V)_1 \rightarrow  \mathbb{A}_{M,V}.
\]

\item If $k_1=k_2$, then we let $\cL(L,M,V)_1=\cL(L,M,V)$.
\\
\end{enumerate}

\item If $k_1<k_2$, then 
\[
(\vpi^1_{n,M, V})^{-1}(\chi_{\gamma})(\mfo) \overset{open}{\subsetneq} \vpi_{n,M, V}^{-1}(\chi_{\gamma})(\mfo).
\]
Furthermore, the complement of $(\vpi^1_{n,M, V})^{-1}(\chi_{\gamma})(\mfo)$ inside $\vpi_{n,M, V}^{-1}(\chi_{\gamma})(\mfo)$ is non-empty and open  as well.
Thus the volume of the right hand side is  strictly bigger than  the volume of the left hand side. 
Nonetheless, we will provide a formula of the volume of the left hand side. This will be used to provide a lower bound of $\sog$ with any $n\geq 4$.
\\

\item
We remark that  if $k_1<k_2$ then the  morphism $\vpi_{n,M, V}$ is no longer smooth at a point of $\vpi_{n,M, V}^{-1}(\chi_{\gamma}) \cap  \left(\cL(L,M,V) \backslash \cL(L,M,V)_1\right)$.  
Therefore the precise formula for $\mathcal{SO}_{\gamma, (k_1, k_2),1}$ (and so $\sog$) with $n\geq 4$ would require much more involved congruence conditions than those considered in this paper.
\end{enumerate}

\end{remark}

We will show that  $(\vpi_{n,M, V}^1)^{-1}(\chi_{\gamma})$ is smooth over $\mfo$.
As a preparation, we need the following lemma which explains a matrix description of an element of $x\in \cL(L,M,V)_1(\bar{\kappa})$.
\begin{lemma}\label{lem82}
Suppose that $k_1<k_2$.
Let $x\in \cL(L,M,V)(\bar{\kappa})$ such that  
$$\begin{pmatrix}
x_{1,2}&\cdots & x_{1,n-2}\\
\vdots & &\vdots  \\
x_{n-2,2}&\cdots & x_{n-2,n-2}
\end{pmatrix}=
\begin{pmatrix}
x_{1,2}&x_{1,3}&\cdots & x_{1,n-2}\\
0&1&\cdots & x_{2,n-2}\\
\vdots &\vdots &  &\vdots   \\
0&\cdots & 1&x_{n-4,n-2}\\
0&\cdots & 0&1\\
0&\cdots & 0&0
 \end{pmatrix}$$
 with $x_{1,2}\neq 0$.
Then $x\in \cL(L,M,V)_1(\bar{\kappa})$ if and only if  either $x_{n-1,1}$ or $x_{n-2,n-1}$ is nonzero.
\end{lemma}

\begin{proof}
For such $x\in \cL(L,M,V)(\bar{\kappa})$,    $pr(x)\left(\in \mathrm{M}_{n-1}^0(\bar{\kappa})\right)$ is 
$$pr(x)=\begin{pmatrix}
0 &x_{1,2}&x_{1,3}&\cdots & x_{1,n-2}&x_{1,n-1}\\
0 &0&1&\cdots & x_{2,n-2}&x_{2,n-1}\\
\vdots &\vdots & & &\vdots &\vdots  \\
0 & 0 &\cdots &0&1&x_{n-3,n-1}  \\
0 &0&\cdots & & 0&x_{n-2,n-1}\\
 x_{n-1,1} &  x_{n-1,2}& & \cdots &  x_{n-1,n-2}&  x_{n-1,n-1}
\end{pmatrix}.$$
It suffices to show that the special fiber of $det$ is smooth at $pr(x)$ if and only if either $x_{n-1,1}$ or $x_{n-2,n-1}$ is nonzero.

The morphism $det$ is smooth at $pr(x)$ only when the induced map on the Zariski tangent space is surjective. 
Let $X=\begin{pmatrix}
0 &a_{1,2}&\cdots & a_{1,n-2}&a_{1,n-1}\\
\vdots &\vdots & &\vdots &\vdots  \\
0 &a_{n-2,2}&\cdots & a_{n-2,n-2}&a_{n-2,n-1}\\
 a_{n-1,1} &  a_{n-1,2}&\cdots &  a_{n-1,n-2}&  a_{n-1,n-1}
\end{pmatrix}$ be an element of the Zariski tangent space of $\mathrm{M}^0_{n-1}(\bar{\kappa})$ at $pr(x)$.
\begin{enumerate}
\item If $x_{n-1,1}=x_{n-2,n-1}=0$, then it is easy to see that  the image of $X$ under the differential of $det$ is $0$.  

\item Suppose that $x_{n-1,1}\neq 0$. 
Let $a_{i,j}=0$ except for $a_{n-2, n-1}$. 
Then the image of $X$ under the differential of $det$ is $x_{n-1,1}x_{1,2}a_{n-2,n-1}$, which yields surjectivity.

\item Suppose that $x_{n-2,n-1}\neq 0$. 
Let $a_{i,j}=0$ except for $a_{n-1, 1}$. 
Then the image of $X$ under the differential of $det$ is $x_{n-2,n-1}x_{1,2}a_{n-1,1}$, which yields surjectivity.
\end{enumerate}
Here, the calculation of the image of $X$ under the differential of $det$ follows the argument described in the proof of Lemma \ref{lemsm}.
The above three considerations complete the proof.
\end{proof}

\begin{proposition}\label{prop83}
The scheme $(\vpi_{n,M, V}^1)^{-1}(\chi_{\gamma})$ is smooth over $\mfo$.
\end{proposition}

\begin{proof}
The proof is similar to that of Theorem \ref{thm55}.
It suffices to show that for any 
$m\in (\vpi_{n,M, V}^1)^{-1}(\chi_{\gamma})(\bar{\kappa})$,
the induced map on the Zariski tangent space 
\[
d(\vpi_{n,M,V}^1)_{\ast, m}: T_m \longrightarrow T_{\vpi_{n,M,V}^1(m)}
\] 
is surjective, where $T_m$ is the Zariski tangent space of $\cL(L,M,V)_1\otimes \bar{\kappa}$ at $m$ and $T_{\vpi_{n,M,V}^1(m)}$ is the Zariski tangent space of $\mathbb{A}_{M}\otimes \bar{\kappa}$ at $\vpi^1_{n,M,V}(m)$.

Using the matrix description given in Equation (\ref{eq81}), 
we write $m\in \cL(n,M,V)(\bar{\kappa})$ and $X\in T_m$ as the following matrices formally;
 \begin{equation}\label{eq82}
\left\{
  \begin{array}{l }
m=\begin{pmatrix}
\pi x_{1,1} &x_{1,2}&\cdots & x_{1,n-2}&x_{1,n-1}&x_{1,n}\\
\vdots &\vdots & &\vdots &\vdots &\vdots \\
\pi x_{n-2,1} &x_{n-2,2}&\cdots & x_{n-2,n-2}&x_{n-2,n-1}&x_{n-2,n}\\
\pi^{k_1} x_{n-1,1} &\pi^{k_1} x_{n-1,2}&\cdots & \pi^{k_1} x_{n-1,n-2}&\pi^{k_1} x_{n-1,n-1}&\pi^{k_1} x_{n-1,n}\\
\pi^{k_2} x_{n,1} &\pi^{k_2} x_{n,2}&\cdots & \pi^{k_2} x_{n,n-2}&\pi^{k_2} x_{n,n-1}&\pi^{k_2} x_{n,n}
\end{pmatrix};\\
X=\begin{pmatrix}
\pi a_{1,1} &a_{1,2}&\cdots & a_{1,n-2}&a_{1,n-1}&a_{1,n}\\
\vdots &\vdots & &\vdots &\vdots &\vdots \\
\pi a_{n-2,1} &a_{n-2,2}&\cdots & a_{n-2,n-2}&a_{n-2,n-1}&a_{n-2,n}\\
\pi^{k_1} a_{n-1,1} &\pi^{k_1} a_{n-1,2}&\cdots & \pi^{k_1} a_{n-1,n-2}&\pi^{k_1} a_{n-1,n-1}&\pi^{k_1} a_{n-1,n}\\
\pi^{k_2} a_{n,1} &\pi^{k_2} a_{n,2}&\cdots & \pi^{k_2} a_{n,n-2}&\pi^{k_2} a_{n,n-1}&\pi^{k_2} a_{n,n}
\end{pmatrix}.
    \end{array} \right.
 \end{equation} 
 
 where  
\[
\left\{
  \begin{array}{l }
\textit{$x_{ij}, a_{ij}\in \bar{\kappa}$};\\
m'=\begin{pmatrix}
0 &x_{1,2}&\cdots & x_{1,n-2}&x_{1,n-1}&x_{1,n}\\
\vdots &\vdots & &\vdots &\vdots &\vdots \\
0 &x_{n-2,2}&\cdots & x_{n-2,n-2}&x_{n-2,n-1}&x_{n-2,n}\\
 x_{n-1,1} &  x_{n-1,2}&\cdots &  x_{n-1,n-2}&  x_{n-1,n-1}&  x_{n-1,n}\\
 x_{n,1} & x_{n,2}&\cdots &  x_{n,n-2}&  x_{n,n-1}&  x_{n,n}
\end{pmatrix} \textit{ is invertible  in } \mathrm{M}_n(\bar{\kappa});\\
\textit{the rank of }m'_{n-2}:= 
\begin{pmatrix}
0 &x_{1,2}&\cdots & x_{1,n-2}\\
\vdots &\vdots & &\vdots  \\
0 &x_{n-2,2}&\cdots & x_{n-2,n-2}
\end{pmatrix} \textit{ is $n-3$  in } \mathrm{M}_{n-2}(\bar{\kappa}) .
    \end{array} \right.
\] 

Note that change of a basis does not affect on  smoothness (cf. Lemma \ref{lem412}).
Since $m\in (\vpi_{n,M, V}^1)^{-1}(\chi_{\gamma})(\bar{\kappa})$, the characteristic polynomial of $m'_{n-2}$ is $x^{n-2}$.
By choosing a  basis for $V$ inducing the Jordan canonical form, we may and do  assume that 
\[
\begin{pmatrix}
0 &x_{1,2}&\cdots & x_{1,n-2}\\
\vdots &\vdots & &\vdots \\
0 &x_{n-2,2}&\cdots & x_{n-2,n-2}
\end{pmatrix}=\begin{pmatrix}
0 &1&\cdots & 0\\
\vdots &\vdots   &\ddots &\vdots  \\
0 &0&\cdots &1\\
0 &0&\cdots &0
 \end{pmatrix}.
\]

Our method to prove the surjectivity of $d(\vpi^1_{n,M, V})_{\ast, m}$ is to choose a certain subspace of $T_m$ mapping onto $T_{\vpi_{n,M,V}^1(m)}$.

\begin{enumerate}
\item{The case that $x_{n-1,1}\neq 0$;}

Since $m'$ is invertible, either $x_{n-2,n-1}$ or $x_{n-2,n}$ is nonzero.
Let $A_l$ be the element of $T_m$ such that
\[
A_l=\left\{
  \begin{array}{l l}
\textit{$a_{i,j}=0$ for all $0\leq i,j\leq n$ except $a_{n-2,l}=1$} & \textit{if $1\leq l \leq n-2$};\\
\textit{$a_{i,j}=0$ for all $0\leq i,j\leq n$ except $a_{n-2,n-1}=1$} & \textit{if $l=n-1$}.
    \end{array} \right.
\] 
and
\[
A_n=\left\{
  \begin{array}{l l}
\textit{$a_{i,j}=0$ for all $0\leq i,j\leq n$ except $a_{n,n-1}=1$} & \textit{if $x_{n-2,n}\neq 0$};\\
\textit{$a_{i,j}=0$ for all $0\leq i,j\leq n$ except $a_{n,n}=1$} & \textit{if $x_{n-2,n}=0$ and $ x_{n-2,n-1}\neq 0$}.
    \end{array} \right.
\]

Then the images of $\left(A_{n-2}, \cdots, A_{1}, A_{n-1}, A_n\right)$ are 
\begin{multline*}
(1, 0, \cdots), (0, 1,  0, \cdots), \cdots, (0, \cdots, 0_{n-4}, 1,  0, \cdots), (0, \cdots, 0_{n-3}, \pi,  0, 0), \\
 (0_1, \cdots, 0_{n-3}, \pi^{k_1} x_{n-1,2}, \pi^{k_1} x_{n-1,1}, \ast), \\
\textit{and }
\left\{
  \begin{array}{l l}
   (0, \cdots, 0_{n-1}, \pi^{k_1+k_2} x_{n-1,1}x_{n-2,n}) & \textit{if $x_{n-2,n}\neq 0$};\\
(0, \cdots, 0_{n-1}, \pi^{k_1+k_2} x_{n-1,1}x_{n-2,n-1}) & \textit{if $x_{n-2,n}=0, x_{n-2,n-1}\neq 0$}.
    \end{array} \right.
\end{multline*}
Here, the subscript of $0$ stands for the position of the entry between $1$ and $n$. 
These $n$-vectors are linearly independent as elements of $T_{\vpi_{n,M,V}^1(m)}$ and so the map $d(\vpi^1_{n,M, V})_{\ast, m}$ is surjective.

\item{The case that $x_{n-1,1}=0$ and $x_{n-2, n-1}\neq 0$;}

Since $m'$ is invertible, $x_{n,1}$ is nonzero.
Let $A_l$ be the matrix of $T_m$ such that
\[
A_l=\left\{
  \begin{array}{l l}
\textit{$a_{i,j}=0$ for all $0\leq i,j\leq n$ except $a_{n-2,l}=1$} & \textit{if $1\leq l \leq n-2$};\\
\textit{$a_{i,j}=0$ for all $0\leq i,j\leq n$ except $a_{n-1,1}=1$} & \textit{if $l=n-1$};\\
\textit{$a_{i,j}=0$ for all $0\leq i,j\leq n$ except $a_{n-1,n}=1$} & \textit{if $l=n$}.
    \end{array} \right.
\]

Then the images of $\left(A_{n-2}, \cdots, A_{1}, A_{n-1}, A_n\right)$ are 
\begin{multline*}
(1, 0, \cdots), (0, 1,  0, \cdots), \cdots, (0, \cdots, 0_{n-4}, 1,  0, \cdots), (0, \cdots, 0_{n-3}, \pi,  0, 0),\\
  ((0_1, \cdots, 0_{n-2}, \pi^{k_1} x_{n-2,n-1}, \ast),  (0, \cdots, 0_{n-1}, \pi^{k_1+k_2} x_{n-2,n-1}x_{n,1}). 
\end{multline*}

These $n$-vectors are linearly independent as elements of $T_{\vpi_{n,M,V}^1(m)}$ and so the map $d(\vpi^1_{n,M, V})_{\ast, m}$ is surjective.
\\
\end{enumerate}

Now suppose that $m\in \cL(n,M,V)_1(\bar{\kappa})$.
If  $k_1<k_2$, then  either $x_{n-1,1}$ or $x_{n-2,n-1}$ is nonzero by Lemma \ref{lem82}.
The above two cases then complete  the proof.

The remaining case is when $k_1=k_2$. 
Since $m'$ is invertible,  either  $x_{n-1,1}$ or $x_{n,1}$ is nonzero.
By change of the last two vectors in a basis forming the matrix  in Equation (\ref{eq82}) (if necessary), we may and do assume that $x_{n-1,1}\neq 0$.
The proof is then completed by  the first case given above.
\end{proof}

\begin{corollary}\label{cor84}
The cardinality of the set $(\vpi_{n,M, V}^1)^{-1}(\chi_{\gamma})(\kappa)$ is
\[
\left\{
  \begin{array}{l l}
\frac{\#\mathrm{GL}_{n-3}(\kappa)}{(q-1)q^{n-4}}\cdot 2(q-1)^3q^{6n-15}
 & \textit{if $k_1<k_2$};\\
\frac{\#\mathrm{GL}_{n-3}(\kappa)}{(q-1)q^{n-4}}\cdot (q-1)^3(q+1)q^{6n-16}
 & \textit{if $k_1=k_2$}.
    \end{array} \right.
\]

\end{corollary}

\begin{proof}
Our strategy is to analyze the equations defining $(\vpi_{n,M, V}^1)^{-1}(\chi_{\gamma})(\kappa)$.
We use Equation (\ref{eq82}) for a matrix description of $m\in (\vpi_{n,M, V}^1)^{-1}(\chi_{\gamma})(\kappa)$.

\begin{enumerate}
\item 
The former $n-3$ entries of $\varphi_{n,M,V}^1(m)$, as an element of $\mathbb{A}_{M,V}(\kappa)$,  are just coefficients of the characteristic polynomial of the submatrix of $m$ consisting of entries from $x_{2,2}$ to $x_{n-2,n-2}$ of size $n-3$.
This characteristic polynomial is $x^{n-3}$ defined over $\kappa$. In addition, the rank of such matrix is $n-4$ by the third condition of Equation (\ref{eq82}).
Thus the set of such submatrices is isomorphic to the set of nilpotent matrices whose characteristic polynomial is $x^{n-3}$ and whose rank is $n-4$.
This is the set of matrices which are similar to the Jordan canonical form 
\[
J_{n-3}=\begin{pmatrix}
0 &1&\cdots & 0\\
\vdots &\vdots   &\ddots &\vdots  \\
0 &0&\cdots &1\\
0 &0&\cdots &0
 \end{pmatrix}.
\]

The stabilizer of $J_{n-3}$ in $\mathfrak{gl}_{n-3}(\kappa)$ with respect to the conjugation is the same as the set of  polynomials of $J_{n-3}$ having non trivial constant.
Since $(J_{n-3})^{n-3}=0$, this set is bijectively identified with the set of polynomials having non trivial constant of degree $\leq n-4$ over $\kappa$, whose cardinality is $(q-1)q^{n-4}$.
Therefore, the cardinality of the set of such submatrices of size $n-3$ is
\begin{equation}\label{eq83}
\frac{\#\mathrm{GL}_{n-3}(\kappa)}{(q-1)q^{n-4}}.
\end{equation}

\item 
Note that the scheme $(\vpi_{n,M, V}^1)^{-1}(\chi_{\gamma})$ is independent of the  change of a basis forming the matrix  in Equation (\ref{eq82}) by   Lemma \ref{lem412}.
Let $(e_1, \cdots, e_n)$ be a basis used in Equation (\ref{eq82}). 
Then a matrix description with respect to another basis  $(e_1, e_2', \cdots, e_{n-2}', e_{n-1}, e_n)$ such that  $e_i'$ with $2\leq i \leq n-2$ is a linear combination of $(e_2, \cdots, e_{n-2})$ is also as described in Equation (\ref{eq82}).

Suppose that   the submatrix of $m$ from $x_{2,2}$ to $x_{n-2,n-2}$ is the Jordan canonical form $J_{n-3}$.
The  third condition in Equation (\ref{eq82}) then implies that 
\begin{equation}\label{eq84}
x_{1,2}\neq 0.
\end{equation}

\item  
The $(n-2)$-th entry of $\varphi_{n,M,V}^1(m)$ is $(-1)^{n-2}$ times the sum of the determinant of the submatrix from $x_{1,1}$ to $x_{n-2,n-2}$ and the determinants of  submatrices of size $n-2$ including either the $(n-1)$-row or the $n$-th row.
Under the assumption that  the submatrix of $m$ from $x_{2,2}$ to $x_{n-2,n-2}$ is  $J_{n-3}$, it turns out to be 
\begin{equation}\label{eq85}
(-1)^{n-2}\pi(x_{n-2,1}x_{1,2}+\pi^{k_1-1}\ast+\pi^{k_2-1}\ast\ast).
\end{equation}

Its contribution  to the equations defining $(\vpi_{n,M, V}^1)^{-1}(\chi_{\gamma})$ is to eliminate the variable $x_{n-2,1}$ since $x_{1,2}\neq 0$.

\item
The $(n-1)$-th entry of $\varphi_{n,M,V}^1(m)$ is $(-1)^{n-1}$ times the sum of the determinant of the submatrix from $x_{1,1}$ to $x_{n-1,n-1}$, the determinant of the submatrix of size $n-1$ excluding the $n-1$-th row, and the determinant of the submatrices of size $n-1$ including both the $(n-1)$-th row and $n$-th row.
Under the assumption that  the submatrix of $m$ from $x_{2,2}$ to $x_{n-2,n-2}$ is  $J_{n-3}$, it turns out to be 
\[
(-1)^{n-1}\pi^{k_1}(x_{n-1,1}x_{n-2,n-1}x_{1,2}+\pi^{k_2-k_1}x_{n,1}x_{n-2,n}x_{1,2}+\pi\ast).
\]
The Newton polygon of an irreducible polynomial $\chi_{\gamma}(x)$ yields that  the $(n-1)$-th coefficient of $\chi_{\gamma}(x)$ is of order at least $k_1+1$. Since $x_{1,2}\neq 0$, its  contribution  to the equations defining $(\vpi_{n,M, V}^1)^{-1}(\chi_{\gamma})$ is 
\begin{equation}\label{eq86}
x_{n-1,1}x_{n-2,n-1}-\pi^{k_2-k_1}x_{n,1}x_{n-2,n}\in (\pi).
\end{equation}

\item 
If we express the determinant of $m$ formally as the form of $\pi^{k_1+k_2}\cdot d$ with $d\in \kappa$, then the $n$-th entry of $\varphi_{n,M,V}^1(m)$ is $d$.
Thus, we can assume that $x_{1,1}=\cdots = x_{n-2,1}=0$.
Since $d$ is invariant under row/column operations,
we may and do assume that the submatrix from $x_{1,1}$ to $x_{n-2,n-2}$ is the Jordan canonical form having 1 on the superdiagonal entries and $0$ on the rest, and all other entries, except for $x_{n-1,1}, x_{n,1}, x_{n-2,n-1}, x_{n-2,n}, x_{n-1,n-1}, x_{n-1,n}, x_{n,n-1}, x_{n,n}$, are zero.
Then the determinant turns out to be
\begin{equation}\label{eq87}
\pi^{k_1+k_2}\cdot \left(x_{n-2,n-1}\begin{vmatrix}  x_{n-1,1} &x_{n-1,n} \\
x_{n,1} &x_{n,n}
 \end{vmatrix}- x_{n-2,n}\begin{vmatrix} x_{n-1,1} &x_{n-1,n-1} \\
x_{n,1} &x_{n,n-1}
 \end{vmatrix}\right).
\end{equation}
\\

\end{enumerate}

We now collect Equations (\ref{eq83})-(\ref{eq87}) together with three conditions in Equation (\ref{eq82}) to compute the desired cardinality.
\begin{enumerate}
\item Suppose that $k_1<k_2$ so that Equation (\ref{eq86}) turns out to be $x_{n-1,1}x_{n-2,n-1}=0$.
\begin{enumerate}
\item If  $x_{n-2, n-1}\neq 0$, then $x_{n-1,1}=0$ by Equation (\ref{eq86}) so that $x_{n,1}\neq 0$  by the second condition of Equation (\ref{eq82}).
The contributions of Equations (\ref{eq85}) and (\ref{eq87}) are to eliminate the variables $x_{n-2,1}$ and $x_{n-1,n}$.
Therefore, the cardinality with the condition $x_{n-2, n-1}\neq 0$ is 
\[
\frac{\#\mathrm{GL}_{n-3}(\kappa)}{(q-1)q^{n-4}}\cdot (q-1)^3q^{6n-15}.
\]

\item If  $x_{n-2, n-1}= 0$, then $x_{n-1,1}\neq 0$ by Lemma \ref{lem82} and $x_{n-2,n}\neq 0$  by the second condition of Equation (\ref{eq82}).
The contributions of Equations (\ref{eq85}) and (\ref{eq87}) are to eliminate the variables $x_{n-2,1}$ and $x_{n,n-1}$.
Therefore, the cardinality with the condition $x_{n-2, n-1}= 0$ is 
\[
\frac{\#\mathrm{GL}_{n-3}(\kappa)}{(q-1)q^{n-4}}\cdot (q-1)^3q^{6n-15}.
\]

\end{enumerate}
The claimed formula is the sum of these two.
	
\item Suppose that $k_1=k_2$ so that Equation (\ref{eq86}) turns out to be $x_{n-1,1}x_{n-2,n-1}=x_{n,1}x_{n-2,n}$.
\begin{enumerate}
\item If  $x_{n-2, n-1}\neq 0$ and $x_{n-2,n}=0$,  then $x_{n-1,1}=0$ by Equation (\ref{eq86}) so that $x_{n,1}\neq 0$  by the second condition of Equation (\ref{eq82}).
The contributions of Equations (\ref{eq85}) and (\ref{eq87}) are to eliminate the variables $x_{n-2,1}$ and $x_{n-1,n}$.
Therefore, the cardinality in this case is 
\[
\frac{\#\mathrm{GL}_{n-3}(\kappa)}{(q-1)q^{n-4}}\cdot (q-1)^3q^{6n-16}.
\]

\item If  $x_{n-2, n-1}\neq 0$ and $x_{n-2,n}\neq 0$, then $x_{n,1}\neq 0$ so that $x_{n-1,1}$ is determined by Equation (\ref{eq86}).
The contributions of Equations (\ref{eq85}) and (\ref{eq87}) are to eliminate the variables $x_{n-2,1}$ and $x_{n-1,n}$.
Therefore, the cardinality in this case is 
\[
\frac{\#\mathrm{GL}_{n-3}(\kappa)}{(q-1)q^{n-4}}\cdot (q-1)^4q^{6n-16}.
\]

\item If  $x_{n-2, n-1}= 0$, then $x_{n-2,n}\neq 0$ by the second condition of Equation (\ref{eq82}). 
Then Equation (\ref{eq86}) yields that $x_{n,1}=0$.
Lemma \ref{lem82} yields that   $x_{n-1,1}\neq 0$.
The contributions of Equations (\ref{eq85}) and (\ref{eq87}) are to eliminate the variables $x_{n-2,1}$ and $x_{n,n-1}$.
Therefore, the cardinality with the condition $x_{n-2, n-1}\neq 0$ is 
\[
\frac{\#\mathrm{GL}_{n-3}(\kappa)}{(q-1)q^{n-4}}\cdot (q-1)^3q^{6n-16}.
\]
\end{enumerate}
The claimed formula is the sum of these three cases.
\end{enumerate}
\end{proof}

\begin{corollary}\label{ineqm2}
We have 
\[
\mathcal{SO}_{\gamma, (k_1, k_2),1}>\left\{
  \begin{array}{l l}
\frac{\#\mathrm{GL}_{n-3}(\kappa)}{(q-1)q^{n-4}}\cdot 2(q-1)^3\cdot \frac{1}{q^{n^2-6n+(n-2)k_1+(n-1)k_2+12}}
 & \textit{if $k_1<k_2$};\\
\frac{\#\mathrm{GL}_{n-3}(\kappa)}{(q-1)q^{n-4}}\cdot (q-1)^3(q+1)\cdot \frac{1}{q^{n^2-6n+(2n-3)k_1+13}}.
 & \textit{if $k_1=k_2$}.
    \end{array} \right.
\] 

\end{corollary}

\begin{proof}
The proof is parallel to that of Corollary \ref{cornm1}.
Let $\tilde{\cL}(L,M,V)$ be the affine space defined over $\mfo$ such that $\tilde{\cL}(L,M,V)(R)$ is the set of matrices of the form in Equation (\ref{eq81}) by ignoring two conditions about invertibility of $x'$ and the rank of $X_{n-2,n-2}$.
Since $\cL(L,M,V)$ is an open subscheme of $\tilde{\cL}(L,M,V)$, $\cL(L,M,V)_1$ is also an open subscheme of $\tilde{\cL}(L,M,V)$.

Let $\omega_{\tilde{\cL}(L,M,V)}$ and $\omega_{\mathbb{A}_{M,V}}$ be nonzero,  translation-invariant forms on   $\mathfrak{gl}_{n, F}$ and $\mathbb{A}^n_F$,
 respectively, with normalizations
$$\int_{\tilde{\cL}(L,M,V)(\mathfrak{o})}|\omega_{\tilde{\cL}(L,M,V)}|=1 \mathrm{~and~}  \int_{\mathbb{A}_{M,V}(\mathfrak{o})}|\omega_{\mathbb{A}_{M,V}}|=1.$$
Putting $\omega^{can}_{(\chi_{\gamma},L,M,V)}=\omega_{\tilde{\cL}(L,M,V)}/\rho_n^{\ast}\omega_{\mathbb{A}_{M,V}}$, 
  we have the following comparison among differentials;
 
 \[
\left\{
  \begin{array}{l }
|\omega_{\mathfrak{gl}_{n, \mathfrak{o}}}|=|\pi|^{n(k_1+k_2)+n-2}\cdot|\omega_{\tilde{\cL}(L,M,V)}|;\\
|\omega_{\mathbb{A}^n_{\mathfrak{o}}}|=
|\pi|^{1+2k_1+k_2}\cdot |\omega_{\mathbb{A}_{M,V}}|;\\
|\omega_{\chi_{\gamma}}^{\mathrm{ld}}|=|\pi|^{(n-2)k_1+(n-1)k_2+(n-3)}|\omega^{can}_{(\chi_{\gamma},L,M,V)}|.
    \end{array} \right.
\]
Recall that
\[
\mathcal{SO}_{\gamma, (k_1, k_2),1}=\int_{\vpi_{n,M, V}^{-1}(\chi_{\gamma})(\mfo)}|\omega_{\chi_{\gamma}}^{\mathrm{ld}}|>\int_{(\vpi^1_{n,M, V})^{-1}(\chi_{\gamma})(\mfo)}|\omega_{\chi_{\gamma}}^{\mathrm{ld}}|.
\]
Here the last inequality is explained in Remark \ref{rmk81}.(3).
Proposition \ref{prop83} combined with the above comparison of differentials then yields 
\[
\int_{(\vpi^1_{n,M, V})^{-1}(\chi_{\gamma})(\mfo)}|\omega_{\chi_{\gamma}}^{\mathrm{ld}}|=
\frac{1}{q^{(n-2)k_1+(n-1)k_2+(n-3)}}\cdot \frac{\#(\vpi_{n,M, V}^1)^{-1}(\chi_{\gamma})(\kappa)}{q^{n^2-n}}.
\]
The desired inequalities follow from 
Corollary \ref{cor84}.
\end{proof}
\textit{   }

\subsection{A lower bound for \texorpdfstring{$\sog$}{SO} with \texorpdfstring{$n\geq 4$}{n>=4}}\label{sec82}

In this subsection, we first compute a lower bound for an elliptic regular semisimple element $\gamma \in \mathfrak{gl}_{n}(\mfo)$ with $n\geq 4$ 
such that $\overline{\chi}_{\gamma}(x)$ is the power of one linear polynomial in $\kappa[x]$. After that, by using the inductive formulas in Proposition \ref{cor410} and Corollary \ref{corred}, we will obtain a lower bound for a general regular semisimple element $\gamma\in\mathfrak{gl}_{n}(\mfo)$.

Before we build a lower bound for $\mathcal{SO}_{\gamma}$, we will extend Proposition \ref{charred3} to a general $n>1$. 
Recall that  $\chi_{\gamma,a}(x)=\chi_{\gamma}(x+a)$ and $d_{a}=ord(\chi_{\gamma,a}(0))$ for $a\in \mfo$.

\begin{proposition}\label{dan}
Let $r$ be the inertial degree of $F_{\chi_{\gamma}}$ over $F$, where $\gamma$ is elliptic. 
For an integer $n>1$,  there exists $a\in \mfo$ such that 
\[
\left\{
\begin{array}{l l}
    d_{a}\equiv 0 \textit{ modulo }n   \textit{ or};\\
    d_{a}\equiv f \textit{ modulo }n  \textit{ where }r|f \textit{ and } 0<f<n,
\end{array}
\right.
\]
and in the first case, the reduction of $\frac{\chi_{\gamma,a}(\pi^{t}x)}{\pi^{nt}}$ modulo $\pi$ is a power of an irreducible polynomial of degree $>1$ in $\kappa[x]$, where $d_a=nt$.

\end{proposition}
\begin{proof}
Express $d_{a}=nt_a+f_a$ for $a\in \mfo$, where $t_a, f_a\in \mathbb{Z}$ with $0\leq  f_a < n$.
We  claim that either $f_a=0$ or $r \mid f_a$ if $f_a\neq 0$.

Suppose that $r\nmid f_a$ and $f_a\neq 0$.
By Newton's lemma the order of a root of $\chi_{\gamma,a}(x)$ is $t_a+\frac{f_a}{n}$. 
On the other hand, since the inertial degree of $F_{\chi_{\gamma}}$ over $F$ is $r$, the order of a uniformizer of $F_{\chi_{\gamma}}$ is $\frac{r}{n}$. 
Here the order means the exponential order with respect to $\pi$, a uniformizer of $F$.
Note that there exist integers $m_{1},m_{2}$ such that $m_{1}r+m_{2}f_a=(r,f_a)$, where $(r,f_a)$ is the greatest common divisor of $r$ and $f_a$. 
Since $r\nmid f_a$, $1 \leq (r,f_a)<r$. 
Therefore, there exists an element of order $\frac{(r,f_a)}{n}$  in $F_{\chi_{\gamma}}$, which is the product of the $m_1$-th power of a uniformizer of $F_{\chi_{\gamma}}$  and the $m_2$-th power of a root of $\chi_{\gamma,a}(x)$ divided by $\pi^{t_a}$.
 This contradicts to the assumption since the order of an element in $F_{\chi_{\gamma}}$ should be a $\mathbb{Z}$-multiple of $\frac{r}{n}$.

The rest of the proof is the same as that of Proposition \ref{charred3} and so we may skip it.
\end{proof}

\begin{remark}\label{dn}
In this remark, we will show that $d_{a}$ is independent of any choice of $a\in \mfo$ which satisfies the conditions in Proposition \ref{dan}.
In the following, the order means the   exponential order with respect to a uniformizer $\pi$ of $F$ as in the proof of Proposition \ref{dan}.

Suppose that $a,a'\in \mfo$ are distinct elements which satisfy the conditions in Proposition \ref{dan} for a given elliptic element $\gamma$. 
We can then write $a'$ as $a+b$ where $a'-a=b\in \mfo \backslash \{0\}$. 
Our goal is to show that 
$$d_a=d_{a+b}.$$
\begin{enumerate}
\item
If $ord(b)\neq \frac{d_{a}}{n}$, then $\frac{d_{a+b}}{n}=min\{ord(b),\frac{d_{a}}{n}\}$. Indeed, by Newton's lemma, the order of roots of $\chi_{\gamma,a+b}$ and $\chi_{\gamma,a}$ are $\frac{d_{a+b}}{n}$ and $\frac{d_{a}}{n}$, respectively. Since roots of $\chi_{\gamma,a+b}$ are translation of roots of $\chi_{\gamma,a}$ by $-b$, we conclude that  $\frac{d_{a+b}}{n}=min\{ord(b),\frac{d_{a}}{n}\}$. There are two cases whether $ord(b)<\frac{d_{a}}{n}$ or $ord(b)>\frac{d_{a}}{n}$.
We claim that  $ord(b)>\frac{d_{a}}{n}$. If it is, then we can directly see that $d_{a+b}=d_{a}$.

Suppose that  $ord(b)<\frac{d_{a}}{n}$ so that $\frac{d_{a+b}}{n}= ord(b)$. 
Note that $d_{a+b}$ satisfies the first condition of Proposition \ref{dan} so that the reduction of the polynomial $\frac{\chi_{\gamma,a+b}(\pi^{ord(b)}x)}{\pi^{n \cdot  ord(b)}}$ modulo $\pi$ should be a power of an irreducible polynomial of degree $>1$.

Note that $-\frac{b}{\pi^{ord(b)}}\in \mfo^{\times}$.
If we plug in $x=-\frac{b}{\pi^{ord(b)}}$ into the polynomial $\frac{\chi_{\gamma,a+b}(\pi^{ord(b)}x)}{\pi^{n  \cdot ord(b)}}$, then we have
\[
\frac{\chi_{\gamma,a+b}(-b)}{\pi^{n \cdot ord(b)}}=\frac{\chi_{\gamma,a}(0)}{\pi^{n \cdot ord(b)}}.
\]
The order of this element is $d_a-n\cdot ord(b)$, which is   $>0$ by the assumption that $ord(b)<\frac{d_{a}}{n}$. 
This implies that the reduction of the polynomial $\frac{\chi_{\gamma,a+b}(\pi^{ord(b)}x)}{\pi^{n  \cdot ord(b)}}$ modulo $\pi$ has a linear factor which is $x+\overline{b/\pi^{ord(b)}}$. This contradicts to the assumption.
\\


\item If $ord(b)=\frac{d_{a}}{n}$, then $d_{a}\equiv 0 \textit{ modulo }n$ and $d_{a+b}\geq d_{a}$.
Note that $d_a$ satisfies the  first condition of Proposition \ref{dan} so that the reduction of the polynomial $\frac{\chi_{\gamma,a}(\pi^{ord(b)}x)}{\pi^{n \cdot  ord(b)}}$ modulo $\pi$ should be a power of an irreducible polynomial of degree $>1$.

Suppose that $d_{a+b}>d_{a}$.
Then if we plug in $x=\frac{b}{\pi^{ord(b)}}\left(\in \mfo^{\times}\right)$ into the polynomial  $\frac{\chi_{\gamma,a}(\pi^{ord(b)}x)}{\pi^{n\cdot ord(b)}}$, then we have
\[ord(\frac{\chi_{\gamma,a}(b)}{\pi^{n\cdot ord(b)}}
)=ord(\frac{\chi_{\gamma,a+b}(0)}{\pi^{n\cdot ord(b)}})=d_{a+b}-n\cdot ord(b)>0.\] 
Thus $x-\overline{ b/{\pi^{ord(b)}}}\in \kappa^{\times}$ is a linear factor of  the reduction of $\frac{\chi_{\gamma,a}(\pi^{ord(b)}x)}{\pi^{n\cdot ord(b)}}$ modulo $\pi$, which is a contradiction. 
\end{enumerate}
\end{remark}

From now on, we can and do denote  $d_{a}$ in Proposition \ref{dan} by $d_{\gamma}$.
If we assume that  $\overline{\chi}_{\gamma}(x)=x^{n}$ for an elliptic element $\gamma$, then 
$d_{\gamma}$ cannot be $0$. Otherwise the reduction of $\overline{\chi}_{\gamma}(x) \textit{ modulo }\pi$ is a power of irreducible polynomial of degree $>1$ which contradicts to the assumption that $\overline{\chi}_{\gamma}(x)=x^{n}$.

\begin{theorem}\label{lb1}
For an elliptic regular semisimple element $\gamma\in \mathfrak{gl}_{n}(\mfo)$ such that $\overline{\chi}_{\gamma}(x)=x^{n}$, we have the following lower bound for the orbital integral for $\gamma$;
\[\left\{\begin{array}{l}
\mathcal{SO}_{\gamma}>\frac{\#\mathrm{GL}_{n}(\kappa)}{(q-1) q^{n^{2}-1}}(1+2q^{-1}+\cdots+2q^{-\lfloor\frac{d_{\gamma}}{2}\rfloor+1}+\varepsilon(d_{\gamma})q^{-\lfloor\frac{d_{\gamma}}{2}\rfloor});\\ 
\mathcal{SO}_{\gamma,d\mu}>q^{S(\gamma)}\cdot\frac{q^{r}-1}{q^{r-1}(q-1)}\cdot(1+2q^{-1}+\cdots+2q^{-\lfloor\frac{d_{\gamma}}{2}\rfloor+1}+\varepsilon(d_{\gamma})q^{-\lfloor\frac{d_{\gamma}}{2}\rfloor}) \textit{ if } char(F) = 0 \textit{ or } char(F)>n,
\end{array}\right .\]
where $d_{\gamma}$ is a positive integer defined in Remark \ref{dn}, $r$ is the inertial degree of $F_{\chi_{\gamma}}/F$, and \[\varepsilon(d_\gamma)=\left \{ \begin{array}{l l} 1 &\textit{if }d_\gamma\textit{ is even};\\2 & \textit{if }d_\gamma\textit{ is odd}.
\end{array}\right.\]
\end{theorem}

\begin{proof}
Firstly, we have the following formula in the case of $(k_{1})$-type in Corollary \ref{cornm1}:
\begin{align*}
    c_{(k_{1})}\cdot\mathcal{SO}_{\gamma,(k_{1})}=\frac{\#\mathrm{GL}_{n}(\kappa)}{(q-1)\cdot q^{n^{2}-1}}.
\end{align*}

Secondly, we will treat $d_{1} c_{(k_{1},k_{2})}\cdot \mathcal{SO}_{\gamma,(k_{1},k_{2}),1}$.
The value of $d_{1} c_{(k_{1},k_{2})}$ is described as follows using  Corollary \ref{cor49} and Lemma \ref{lem415}:
\[
d_{1}c_{(k_{1},k_{2})}=\left\{\begin{array}{c l}
    \frac{\#\mathrm{GL}_{n}(\kappa)\cdot q^{(n-3)k_{1}+(n-1)k_{2}-5n+9}}{\#\mathrm{GL}_{n-3}(\kappa)\cdot (q-1)^{3}} &\textit{ if }k_{1}< k_{2};\\
    \frac{\#\mathrm{GL}_{n}(\kappa)\cdot q^{(2n-4)k_{1}-5n+10}}{\#\mathrm{GL}_{n-3}(\kappa)\cdot (q-1)^{3}(q+1)} &\textit{ if }k_{1} = k_{2}.
\end{array}\right.
\]

From the result in Corollary \ref{ineqm2}, we can obtain a lower bound for $d_{1} c_{(k_{1},k_{2})}\cdot\mathcal{SO}_{\gamma,(k_{1},k_{2}),1}$ as follows:
\[d_{1}c_{(k_{1},k_{2})} \cdot \mathcal{SO}_{\gamma,(k_{1},k_{2}),1}>\left\{\begin{array}{l l}
    \frac{\#\mathrm{GL}_{n}(\kappa)}{(q-1)q^{n^{2}-1}}\cdot 2q^{-k_{1}}&\textit{ if }k_{1} < k_{2};\\
    \frac{\#\mathrm{GL}_{n}(\kappa)}{(q-1)q^{n^{2}-1}}\cdot q^{-k_{1}}&\textit{ if }k_{1} = k_{2}.
\end{array}\right.\]


Thus we have the following inequalities;
\[
\left\{
\begin{array}{l}
    c_{(k_{1})}\cdot\mathcal{SO}_{\gamma,(k_{1})}=\frac{\#\mathrm{GL}_{n}(\kappa)}{(q-1) q^{n^{2}-1}};\\
    c_{(k_{1},k_{2})}\cdot\mathcal{SO}_{\gamma,(k_{1},k_{2})}>\frac{\#\mathrm{GL}_{n}(\kappa)}{(q-1) q^{n^{2}-1}}\cdot 2q^{-k_{1}} ~~~~~~~~~~ \textit{       for $k_1<k_2$};\\
    c_{(k_{1},k_{1})}\cdot\mathcal{SO}_{\gamma,(k_{1},k_{1})}>\frac{\#\mathrm{GL}_{n}(\kappa)}{(q-1) q^{n^{2}-1}}\cdot q^{-k_{1}}.
\end{array}\right.
\]

We now plug in the above results into the following summation formula  due to  Proposition \ref{stra1}:
\[
\mathcal{SO}_{\gamma}=\sum\limits_{\substack{(k_{1}, \cdots, k_{n-m}),\\ \sum k_{i}=d_{\gamma},\\m\geq 0}}c_{(k_{1},\cdots,k_{n-m})}\cdot \mathcal{SO}_{\gamma,(k_{1},\cdots, k_{n-m})}
\]
where $d_{\gamma}>0$ is a positive integer defined in Remark \ref{dn}.

If $d_{\gamma}=1$, then $\chi_{\gamma}(x)$ is an Eisenstein polynomial over $\mfo$. 
Therefore $S(\gamma)=0$ and  $F_{\chi_{\gamma}}$ is totally ramified over $F$.
Since we only need to consider $(1)$-type lattices in the above summation, we have
\begin{equation}\label{r=1}
\mathcal{SO}_{\gamma}=c_{(1)}\cdot\mathcal{SO}_{\gamma, (1)}=\frac{\#\mathrm{GL}_{n}(\kappa)}{(q-1) q^{n^{2}-1}}\ \textit{ and }\ \mathcal{SO}_{\gamma,d\mu}=1.
\end{equation}
Suppose that $d_{\gamma}\geq 2$. We then have 
\begin{align*}
\mathcal{SO}_{\gamma}&>c_{(d_{\gamma})} \cdot\mathcal{SO}_{\gamma,(d_{\gamma})}+\sum\limits_{k_{1}=1}^{\lfloor \frac{d_{\gamma}}{2}\rfloor}c_{(k_{1},d_{\gamma}-k_{1})}\cdot \mathcal{SO}_{\gamma,(k_{1},d_{\gamma}-k_{1})}\\
&>\frac{\#\mathrm{GL}_{n}(\kappa)}{(q-1) q^{n^{2}-1}}+\frac{\#\mathrm{GL}_{n}(\kappa)}{(q-1) q^{n^{2}-1}}(2q^{-1}+\cdots+2q^{-\lfloor\frac{d_{\gamma}}{2}\rfloor+1}+\varepsilon(d_{\gamma})q^{-\lfloor\frac{d_{\gamma}}{2}\rfloor})\\
&=\frac{\#\mathrm{GL}_{n}(\kappa)}{(q-1) q^{n^{2}-1}}(1+2q^{-1}+\cdots+2q^{-\lfloor\frac{d_{\gamma}}{2}\rfloor+1}+\varepsilon(d_{\gamma})q^{-\lfloor\frac{d_{\gamma}}{2}\rfloor}).
\end{align*} 
This is the desired lower bound for $\sog$.
\\

For a lower bound for $\mathcal{SO}_{\gamma,d\mu}$, we use 
Propositions \ref{proptrans}-\ref{propserre} to have
\[
\mathcal{SO}_{\gamma,d\mu}=q^{S(\gamma)}\cdot\frac{\#\mathrm{T}_{\gamma}(\kappa)q^{-n}}{\#\mathrm{GL}_{n}(\kappa)q^{-n^{2}}}\cdot \mathcal{SO}_{\gamma}.
\]
Since $\gamma$ is an elliptic regular semisimple element and the inerial degree of $F_{\chi_{\gamma}}$ over $F$ is $r$, \[\# \mathrm{T}_{\gamma}(\kappa)=\# \mathrm{Res}_{\mfo_{F_{\chi_{\gamma}}}/\mathfrak{o}}\mathbb{G}_{m}(\kappa)=q^{n-r}(q^{r}-1).\]
Combining these two relations, we  obtain the desired lower bound for $\mathcal{SO}_{\gamma,d\mu}$.
\end{proof}

We now extend our result with a general regular semisimple element $\gamma \in \mathfrak{gl}_{n}(\mfo)$. By the inductive formula in Proposition \ref{cor410}, it suffices to consider the lower bound for an elliptic regular semisimple element.

To treat the case when $\gamma$ is elliptic without any restriction such as $\overline{\chi}_{\gamma}(x)=x^{n}$,
we recall that $R$ is $\mfo[x]/(\chi_{\gamma}(x))$ and $\kappa_{R}$ is the residue field of $R$ in Section \ref{reductionss} (or Diagrams (\ref{diag})-(\ref{diag2_gl})).
Define the integer $\bar{d}_{\gamma}$  as follows, which is  associated to the integer $d_{\gamma}$ defined in the paragraph just before Theorem \ref{lb1}.
\begin{enumerate}
    \item If $d_{\gamma}\neq 0$, then $[\kappa_{R}:\kappa]=1$. In this case, we define 
    $$\bar{d}_{\gamma}:=d_{\gamma}.$$
    \item If $d_{\gamma}=0$ then $[\kappa_{R}:\kappa]>1$.
    In this case, we use  the inductive formula in Corollary \ref{corred}.
    Namely, we  view $\gamma$ as an element of $\mathfrak{gl}_{n}(\mfo_{K})$ and work with the characteristic polynomial $\chi_{\gamma}^{K}(x)$ (where the notation $\mfo_{K}$ and $\chi_{\gamma}^{K}(x)$ follow those  explained in paragraphs before Lemma \ref{lem42}).
 \begin{enumerate}
 \item 
    If $\mathrm{deg}(\chi_{\gamma}^{K}(x))>1$, then we consider $d_{\gamma}$ for $\chi_{\gamma}^{K}(x)$ in Remark \ref{dn} and  denote it by $d_{\gamma}^{K}$, in order to distinguish it from $d_{\gamma}$ for $\chi_{\gamma}(x)$.
    Define 
    $$\bar{d}_{\gamma}:=d_{\gamma}^{K}.$$
\item
    If $\mathrm{deg}(\chi_{\gamma}^{K}(x))=1$, equivalently if $\overline{\chi}_{\gamma}(x)$ is irreducible over $\kappa$, then we define 
    $$\bar{d}_{\gamma}:=0.$$
    
 \end{enumerate}
\end{enumerate}
\begin{remark}\label{rmk89}
If  $n$ is a prime integer, then $\bar{d}_{\gamma}=d_{\gamma}$.
To show this, we only need to consider the case of $d_{\gamma}=0$.
In this case,    $\overline{\chi}_{\gamma}(x)$ is irreducible over $\kappa$ by Proposition \ref{charred3} and thus 
 $\bar{d}_{\gamma}=0$ as well by the above (b).
\end{remark}

Let
\begin{equation}\label{eq:Ngamma'}
\left\{
\begin{array}{l}
    N'_{\gamma}(x)=\frac{x}{x-1}(1+2x^{-1}+\cdots+2x^{-\lfloor\frac{\bar{d}_{\gamma}}{2}\rfloor+1}+\varepsilon(\bar{d}_{\gamma})x^{-\lfloor\frac{\bar{d}_{\gamma}}{2}\rfloor});\\
    N'_{\gamma,d\mu}(x)=(x^{S_{R}(\gamma)}+x^{S_{R}(\gamma)-1}+\cdots+x^{S_{R}(\gamma)-r+1})(1+2x^{-1}+\cdots+2x^{-\lfloor\frac{\bar{d}_{\gamma}}{2}\rfloor+1}+\varepsilon(\bar{d}_{\gamma})x^{-\lfloor\frac{\bar{d}_{\gamma}}{2}\rfloor})
\end{array}\right.
\end{equation}
for an elliptic regular semisimple element $\gamma \in \mathfrak{gl}_{n}(\mfo)$, where $S_{R}(\gamma)$ is the relative $R$-length $[\mfo_{F_{\chi_{\gamma}}}:R]_{R}$. Then 
$S(\gamma)=d\cdot S_{R}(\gamma)$ since $[\kappa_{R}:\kappa]=d$.

Combining inductive formulas in Proposition \ref{cor410} and Corollary \ref{corred} with the lower bound for Theorem \ref{lb1}, 
we have the following  lower bound for a regular semisimple element $\gamma \in \mathfrak{gl}_{n}(\mfo)$:
\begin{theorem}\label{genlb1}
Suppose that $char(F) = 0$ or $char(F)>n$.
For a regular semisimple element $\gamma \in \mathfrak{gl}_{n}(\mfo)$, we have
\[\left\{\begin{array}{l}
    \mathcal{SO}_{\gamma}>\frac{\#\mathrm{GL}_{n}(\kappa)}{q^{n^{2}}}\prod\limits_{i\in B(\gamma)}N'_{\gamma_{i}}(q^{d_{i}});\\
    \mathcal{SO}_{\gamma,d\mu} > q^{\rho(\gamma)}\prod\limits_{i\in B(\gamma)}N'_{\gamma_{i},d\mu}(q^{d_{i}})
\end{array}\right.\]
where $B(\gamma)$ and $\gamma_{i}$ are defined in Section \ref{sec322}.
Here   we use the following notation for an elliptic regular semisimple element $\gamma_{i}$:
\begin{equation}\label{eq:notations_rho_gamma}
\left\{
\begin{array}{l}
    \textit{$R_{i}=\mfo[x]/(\chi_{\gamma_{i}}(x))$ \textit{      and          } $
    F_{\chi_{\gamma_{i}}}=F[x]/(\chi_{\gamma_{i}}(x))$};\\
    \textit{$\kappa_{R_{i}}$ and $\kappa_{F_{\chi_{\gamma_{i}}}}$ are the residue fields of $R_{i}$ and $F_{\chi_{\gamma_{i}}}$, respectively};\\
    d_{i}=[\kappa_{R_{i}}:\kappa] \textit{      and          } r_{i}=[\kappa_{F_{\chi_{\gamma_{i}}}}:\kappa_{R_{i}}]; \\
    \rho(\gamma)=S(\gamma)-\sum\limits_{i\in B(\gamma)}S(\gamma_{i}).
\end{array}
\right.
\end{equation}
\end{theorem}

\begin{appendices}

\section{The proof of  Theorem  \ref{thm71}} \label{App:AppendixA}
 
In this Appendix, we provide the full proof of Theorem \ref{thm71}.
Each case will be handled independently in the following subsections.

\subsection{Case 1: \texorpdfstring{$k_1< k_2$}{k1<k2} and \texorpdfstring{$k_1<t$}{k1<t}}

Since $k_1+k_2=d$ and $k_1\leq t-1$, the Newton polygon of an irreducible polynomial $\chi_{\gamma}$ yields 
\begin{equation}\label{k12-1}
\left\{
  \begin{array}{l l}
\ord(c_1)\geq t, \ord(c_2)\geq 2t, k_2\geq 2t+1   & \textit{if $d=3t$};\\
\ord(c_1)\geq t, \ord(c_2)\geq 2t, k_2\geq 2t  & \textit{if $d=3t-1$};\\
\ord(c_1)\geq t, \ord(c_2)\geq 2t-1, k_2\geq 2t-1   & \textit{if $d=3t-2$}
    \end{array} \right.
\end{equation}
so that $\ord(c_1)>k_1$, $\ord(c_2)>2k_1$, and $k_2>2k_1$.
\\

Let  $m=\begin{pmatrix}
x_{1,1} &x_{1,2} & x_{1,3} \\
\pi^{k_1}x_{2,1} &\pi^{k_1}x_{2,2} & \pi^{k_1}x_{2,3} \\
\pi^{k_2}x_{3,1} &\pi^{k_2}x_{3,2} & \pi^{k_2}x_{3,3}
\end{pmatrix} $ be  an element of $\varphi_{3,M}^{-1}(\chi_{\gamma})(\mfo)$.
Since the characteristic polynomial of $m$ should satisfy Equation (\ref{k12-1}), both $x_{1,1}$ and $ x_{2,1}x_{1,2}$ are contained in the ideal $(\pi^{k_1})$.
Therefore we have the following stratification of $\varphi_{3,M}^{-1}(\chi_{\gamma})(\mfo)$ with respect to the valuation of $x_{2,1}$:
\begin{multline}\label{eq73}
\varphi_{3,M}^{-1}(\chi_{\gamma})(\mfo)=
\bigsqcup_{l=0}^{k_1-1}\Bigl\{\begin{pmatrix}
\pi^{k_1}x_{1,1} &\pi^{k_1-l}x_{1,2} & x_{1,3} \\
\pi^{k_1+l}x_{2,1} &\pi^{k_1}x_{2,2} & \pi^{k_1}x_{2,3} \\
\pi^{k_2}x_{3,1} &\pi^{k_2}x_{3,2} & \pi^{k_2}x_{3,3}
\end{pmatrix}\in \varphi_{3,M}^{-1}(\chi_{\gamma})(\mfo)|x_{2,1}\in \mfo^{\times}\Bigl\}\\
\bigsqcup \Bigl\{\begin{pmatrix}
\pi^{k_1}x_{1,1} &x_{1,2} & x_{1,3} \\
\pi^{2k_1}x_{2,1} &\pi^{k_1}x_{2,2} & \pi^{k_1}x_{2,3} \\
\pi^{k_2}x_{3,1} &\pi^{k_2}x_{3,2} & \pi^{k_2}x_{3,3}
\end{pmatrix}\in \varphi_{3,M}^{-1}(\chi_{\gamma})(\mfo)\Bigl\}.
\end{multline}
Here, the determinant of $\begin{pmatrix}
\pi^{k_1}x_{1,1} &\pi^{k_1-l}x_{1,2} & x_{1,3} \\
\pi^{l}x_{2,1} & x_{2,2} &  x_{2,3} \\
x_{3,1} & x_{3,2} &  x_{3,3}
\end{pmatrix}$ in each stratum for $0\leq l \leq k_1$ is a unit in $\mfo$, where the last case $l=k_1$ is about the last stratum.
In particular when $l=k_1$,
we claim that  $x_{1,2}\in \mfo^{\times}$.
Otherwise, $x_{1,2}\in (\pi)$ so that $x_{2,2}\in \mfo^{\times}$ since the determinant of the above matrix is a unit in $\mfo$.
On the other hand, Equation  (\ref{k12-1}) yields that $-x_{1,1}-x_{2,2},\  x_{1,1}x_{2,2}-x_{1,2}x_{2,1}\in (\pi)$. Thus $x_{1,1}$ is a unit as well, which contradicts to the fact that $x_{1,1}x_{2,2}-x_{1,2}x_{2,1}\in (\pi)$.
\\

Now, we assign a scheme structure on each stratum as in Section \ref{subsec52}.
Define   functors $\cL(L,M)_l$ for $0\leq l \leq k_1$ and $\mathbb{A}_{M,1}$ on the category of flat $\mfo$-algebras to the category of sets such that
 \[
\cL(L,M)_l(R)
=\Bigl\{\begin{pmatrix}
\pi^{k_1}x_{1,1} &\pi^{k_1-l}x_{1,2} & x_{1,3} \\
\pi^{k_1+l}x_{2,1} &\pi^{k_1}x_{2,2} & \pi^{k_1}x_{2,3} \\
\pi^{k_2}x_{3,1} &\pi^{k_2}x_{3,2} & \pi^{k_2}x_{3,3}
\end{pmatrix}\Bigl\}
   \textit{,    } ~~~~~~~~~~~ 
\mathbb{A}_{M,1}(R)=\pi^{k_1} R\times \pi^{2k_1} R\times \pi^d R
\]
where $x_{2,1}\in R^{\times}$ (resp. $x_{1,2}\in R^{\times}$)  if $0\leq l \leq k_1-1$ (resp. $l=k_1$) and $\begin{pmatrix}
\pi^{k_1}x_{1,1} &\pi^{k_1-l}x_{1,2} & x_{1,3} \\
\pi^{l}x_{2,1} & x_{2,2} &  x_{2,3} \\
x_{3,1} & x_{3,2} &  x_{3,3}
\end{pmatrix}$
is invertible as a matrix with entries in $R$ 
for  $0\leq l \leq k_1$.
Here, $R$ is a flat $\mfo$-algebra.

Let $\varphi_{3,M}^l$  be the restriction of $\varphi_{3,M}$ to $\cL(L,M)_l(R)$ for $0\leq l \leq k_1$.
Then the morphism
\[
\varphi_{3,M}^l: \cL(L,M)_l(R) \longrightarrow \mathbb{A}_{M,1}(R)
\]
is represented by a morphism of schemes over $\mfo$.
Note that each stratum in Equation (\ref{eq73}) is  $(\varphi_{3,M}^l)^{-1}(\chi_{\gamma})(\mfo)$.
Thus  the stratification in Equation (\ref{eq73}) turns out to be
\[
\varphi_{3,M}^{-1}(\chi_{\gamma})(\mfo)=\bigsqcup_{l=0}^{k_1}(\varphi_{3,M}^l)^{-1}(\chi_{\gamma})(\mfo).
\]

\begin{proposition}
The scheme $(\varphi_{3,M}^l)^{-1}(\chi_{\gamma})$ is smooth over $\mfo$.
\end{proposition}
\begin{proof}
The proof is similar to that of Theorem \ref{thm55}.
It suffices to show that for any $m\in (\varphi_{3,M}^l)^{-1}(\chi_{\gamma})(\bar{\kappa})$,
the induced map on the Zariski tangent space 
\[
d(\vpi_{3,M}^l)_{\ast, m}: T_m \longrightarrow T_{\varphi_{3,M}^l(m)}
\] 
is surjective, where $T_m$ is the Zariski tangent space of $ \cL(L,M)_l\otimes \bar{\kappa}$ at $m$ and $T_{\varphi_{3,M}^l(m)}$ is the Zariski tangent space of $\mathbb{A}_{M,1}\otimes \bar{\kappa}$ at $\varphi_{3,M}^l(m)$.

Write $m\in (\varphi_{3,M}^l)^{-1}(\chi_{\gamma})(\bar{\kappa})$ and $X\in T_m$ as the following matrix formally;
 \[
m=\begin{pmatrix}
\pi^{k_1}x_{1,1} &\pi^{k_1-l}x_{1,2} & x_{1,3} \\
\pi^{k_1+l}x_{2,1} &\pi^{k_1}x_{2,2} & \pi^{k_1}x_{2,3} \\
\pi^{k_2}x_{3,1} &\pi^{k_2}x_{3,2} & \pi^{k_2}x_{3,3}
\end{pmatrix}   \textit{   and   } X=
\begin{pmatrix}
\pi^{k_1}a_{1,1} &\pi^{k_1-l}a_{1,2} & a_{1,3} \\
\pi^{k_1+l}a_{2,1} &\pi^{k_1}a_{2,2} & \pi^{k_1}a_{2,3} \\
\pi^{k_2}a_{3,1} &\pi^{k_2}a_{3,2} & \pi^{k_2}a_{3,3}
\end{pmatrix},
\]
 where $x_{ij}, a_{ij}\in \bar{\kappa}$,  
 $x_{2,1}\in \bar{\kappa}^{\times}$ (resp. $x_{1,2}\in \bar{\kappa}^{\times})$  if $0\leq l \leq k_1-1$ (resp. $l=k_1$),  and
 $\begin{pmatrix}
0 &\pi^{k_1-l}x_{1,2} & x_{1,3} \\
\pi^{l}x_{2,1} & x_{2,2} &  x_{2,3} \\
x_{3,1} & x_{3,2} &  x_{3,3}
\end{pmatrix}$ is invertible  as a matrix of $\mathrm{M}_3(\bar{\kappa})$ for  $0\leq l \leq k_1$.
Thus 
 \[
\left\{
  \begin{array}{l l}
x_{2,1}, x_{1,3}\neq 0 &  \textit{if $l=0$};\\
x_{2,1}, x_{1,3}, x_{3,1}\neq 0 &  \textit{if $1\leq l \leq k_1-1$};\\
x_{1,2}, x_{3,1}\neq 0 & \textit{if $l=k_1$}.
    \end{array} \right.
\]

Our method to prove the surjectivity of $d(\vpi_{3,M}^l)_{\ast, m}$ is to choose a certain subspace of $T_m$ mapping onto $T_{\vpi_{3,M}^l(m)}$.
Let 
 $A_1=\begin{pmatrix}
\pi^{k_1}1 &\pi^{k_1-l}0 & 0 \\
\pi^{k_1+l}0 &\pi^{k_1}0 & \pi^{k_1}0 \\
\pi^{k_2}0 &\pi^{k_2}0 & \pi^{k_2}0
\end{pmatrix}$ for $0\leq l \leq k_1$ and let $A_2, A_3$ be
 \[
\left\{
  \begin{array}{l l}
  \begin{pmatrix}
\pi^{k_1}0 &\pi^{k_1}1 & 0 \\
\pi^{k_1}0 &\pi^{k_1}0 & \pi^{k_1}0 \\
\pi^{k_2}0 &\pi^{k_2}0 & \pi^{k_2}0
\end{pmatrix}, 
\begin{pmatrix}
\pi^{k_1}0 &\pi^{k_1}0 & 0 \\
\pi^{k_1}0 &\pi^{k_1}0 & \pi^{k_1}0 \\
\pi^{k_2}0 &\pi^{k_2}1 & \pi^{k_2}0
\end{pmatrix} &  \textit{if $l=0$};\\
\begin{pmatrix}
\pi^{k_1}0 &\pi^{k_1-l}1 & 0 \\
\pi^{k_1+l}0 &\pi^{k_1}0 & \pi^{k_1}0 \\
\pi^{k_2}0 &\pi^{k_2}0 & \pi^{k_2}0
\end{pmatrix}, 
\begin{pmatrix}
\pi^{k_1}0 &\pi^{k_1-l}0 & 0 \\
\pi^{k_1+l}0 &\pi^{k_1}1 & \pi^{k_1}0 \\
\pi^{k_2}0 &\pi^{k_2}0 & \pi^{k_2}0
\end{pmatrix} &  \textit{if $1\leq l \leq k_1-1$};\\
\begin{pmatrix}
\pi^{k_1}0 &0 & 0 \\
\pi^{2k_1}1 &\pi^{k_1}0 & \pi^{k_1}0 \\
\pi^{k_2}0 &\pi^{k_2}0 & \pi^{k_2}0
\end{pmatrix}, 
\begin{pmatrix}
\pi^{k_1}0 &0 & 0 \\
\pi^{2k_1}0 &\pi^{k_1}0 & \pi^{k_1}1 \\
\pi^{k_2}0 &\pi^{k_2}0 & \pi^{k_2}0
\end{pmatrix} & \textit{if $l=k_1$}
    \end{array} \right.
\]
as elements of $T_m$.
Then the images of $A_1, A_2, A_3$ are 
 \[
\left\{
  \begin{array}{l l}
\left(1, \ast, \ast\ast \right), \left(0, x_{2,1}, \ast \right), \left(0, 0, x_{2,1}x_{1,3} \right) &  \textit{if $l=0$};\\
\left(1, \ast, \ast\ast \right), \left(0, x_{2,1}, \ast \right), \left(0, 0, x_{3,1}x_{1,3} \right) &  \textit{if $1\leq l \leq k_1-1$};\\
\left(1,  \ast, \ast\ast \right), \left(0, x_{1,2}, \ast \right), \left(0, 0, x_{3,1}x_{1,2} \right) & \textit{if $l=k_1$}
    \end{array} \right.
\]
up to the sign $\pm 1$.
These three vectors are linearly independent as elements of $T_{\vpi_{3,M}^l(m)}$ and so the map $d(\vpi_{3,M}^l)_{\ast, m}$ is surjective.
\end{proof}

\begin{corollary}
The order of the set $(\varphi_{3,M}^l)^{-1}(\chi_{\gamma})(\kappa)$ is 
 \[
\#(\varphi_{3,M}^l)^{-1}(\chi_{\gamma})(\kappa)=\left\{
  \begin{array}{l l}
q^3(q-1)^3 &   \textit{if $1\leq l \leq k_1-1$};\\
q^4(q-1)^2 & \textit{if $l=0$ or $l=k_1$}.
    \end{array} \right.
\]
\end{corollary}

\begin{proof}
For $m\in \cL(L,M)_l(\kappa)$,  $\varphi_{3,M}^l(m)$  is 
 \[
\left\{
  \begin{array}{l l}
\left(\pi^{k_1}(-x_{1,1}-x_{2,2}), \pi^{2k_1}\begin{vmatrix}
x_{1,1} &x_{1,2} \\ x_{2,1} &x_{2,2}
\end{vmatrix}, \pi^d\left(-x_{1,3}\cdot 
\begin{vmatrix}
\pi^lx_{2,1} &x_{2,2} \\ x_{3,1} &x_{3,2}
\end{vmatrix}\right)\right) &   \textit{if $0\leq l \leq k_1-1$};\\
\left(\pi^{k_1}(-x_{1,1}-x_{2,2}), \pi^{2k_1}\begin{vmatrix}
x_{1,1} &x_{1,2} \\ x_{2,1} &x_{2,2}
\end{vmatrix}, \pi^d\left(-x_{3,1}\cdot 
\begin{vmatrix}
x_{1,2} &x_{1,3} \\ x_{2,2} &x_{2,3}
\end{vmatrix}\right)\right) & \textit{if $l=k_1$}.
    \end{array} \right.
\]

\begin{enumerate}
\item 
 Suppose that $l=0$ so that   $x_{2,1}, x_{1,3}\neq 0$ and $\begin{vmatrix}
x_{2,1} &x_{2,2} \\ x_{3,1} &x_{3,2}
\end{vmatrix} \neq 0$. 
If $m\in (\varphi_{3,M}^1)^{-1}(\chi_{\gamma})(\kappa)$, then $x_{1,1}, x_{1,2}, x_{1,3}$ are completely determined by all other entries satisfying $\begin{vmatrix}
x_{2,1} &x_{2,2} \\ x_{3,1} &x_{3,2}
\end{vmatrix} \neq 0$. Thus we have 
 \[
\#(\varphi_{3,M}^l)^{-1}(\chi_{\gamma})(\kappa)=
(\#\mathrm{GL}_2(\kappa)-q(q-1)^2)\cdot q^2=q^4(q-1)^2.
\]

\item Suppose that $1\leq l \leq k_1-1$ so that  $x_{2,1}, x_{1,3}, x_{3,1} \neq 0$ and $\begin{vmatrix}
0 &x_{2,2} \\ x_{3,1} &x_{3,2}
\end{vmatrix} \neq 0$. 
If $m\in (\varphi_{3,M}^1)^{-1}(\chi_{\gamma})(\kappa)$, then $x_{1,1}, x_{1,2}, x_{1,3}$ are completely determined by all other entries satisfying $\begin{vmatrix}
0 &x_{2,2} \\ x_{3,1} &x_{3,2}
\end{vmatrix} \neq 0$. Thus we have 
 \[
\#(\varphi_{3,M}^l)^{-1}(\chi_{\gamma})(\kappa)=
(q-1)^3q\cdot q^2=q^3(q-1)^3.
\]

\item Suppose that $l=k_1$ so that 
  $x_{1,2}, x_{3,1}\neq 0$ and $\begin{vmatrix}
x_{1,2} &x_{1,3} \\ x_{2,2} &x_{2,3}
\end{vmatrix}\neq 0$.
If $m\in (\varphi_{3,M}^1)^{-1}(\chi_{\gamma})(\kappa)$, then $x_{1,1}, x_{2,1}, x_{3,1}$ are completely determined by all other entries satisfying $\begin{vmatrix}
x_{1,2} &x_{1,3} \\ x_{2,2} &x_{2,3}
\end{vmatrix}\neq 0$.
 Thus we have
  \[
\#(\varphi_{3,M}^l)^{-1}(\chi_{\gamma})(\kappa)=
(\#\mathrm{GL}_2(\kappa)-q(q-1)^2)\cdot q^2=q^4(q-1)^2.
\]
\end{enumerate}
\end{proof}

\begin{corollary}
We have 
\[
c_{(k_{1},k_{2})}\cdot \mathcal{SO}_{\gamma, (k_1, k_2)}=(q^{3}-1)(q^{2}-1)q^{k_{2}}\cdot \frac{(k_1+1)q-(k_1-1)}{q^{6+d}}.
\]

\end{corollary}

\begin{proof}
The proof is parallel to that of Corollary \ref{cornm1}.
Let $\widetilde{\cL}(L,M)_l$ be the affine space defined over $\mfo$ such that 
$\widetilde{\cL}(L,M)_l(R)=\Bigl\{\begin{pmatrix}
\pi^{k_1}x_{1,1} &\pi^{k_1-l}x_{1,2} & x_{1,3} \\
\pi^{k_1+l}x_{2,1} &\pi^{k_1}x_{2,2} & \pi^{k_1}x_{2,3} \\
\pi^{k_2}x_{3,1} &\pi^{k_2}x_{3,2} & \pi^{k_2}x_{3,3}
\end{pmatrix}\Bigl\}$ for $0\leq l \leq k_1$. Then $\cL(L,M)_l$ is an open subscheme of $\widetilde{\cL}(L,M)_l$.

Let $\omega_{\widetilde{\cL}(L,M)_l}$ and $\omega_{\mathbb{A}_{M,1}}$ be nonzero,  translation-invariant forms on   $\mathfrak{gl}_{n, F}$ and $\mathbb{A}^n_F$,
 respectively, with normalizations
$$\int_{\widetilde{\cL}(L,M)_l(\mathfrak{o})}|\omega_{\widetilde{\cL}(L,M)_l}|=1 \mathrm{~and~}  \int_{\mathbb{A}_{M,1}(\mathfrak{o})}|\omega_{\mathbb{A}_{M,1}}|=1.$$
Putting $\omega^{can}_{(\chi_{\gamma},L,M)_l}=\omega_{\widetilde{\cL}(L,M)_l}/\rho_n^{\ast}\omega_{\mathbb{A}_{M,1}}$.
 Then we have the following comparison among differentials;
 
 \[
\left\{
  \begin{array}{l }
|\omega_{\mathfrak{gl}_{n, \mathfrak{o}}}|=|\pi|^{5k_1+3k_2}\cdot|\omega_{\widetilde{\cL}(L,M)_l}|;\\
|\omega_{\mathbb{A}^n_{\mathfrak{o}}}|=
|\pi|^{3k_1+d}\cdot |\omega_{\mathbb{A}_{M,1}}|;\\
|\omega_{\chi_{\gamma}}^{\mathrm{ld}}|=|\pi|^{d+k_2}|\omega^{can}_{(\chi_{\gamma},L,M)_l}|.
    \end{array} \right.
\]
Therefore, we have
\begin{multline*}
\mathcal{SO}_{\gamma, (k_1, k_2)}=\sum\limits_{l=0}^{k_1}\int_{(\varphi_{3,M}^l)^{-1}(\chi_{\gamma})(\mfo)}|\omega_{\chi_{\gamma}}^{\mathrm{ld}}|= 
\frac{(k_1-1)q^3(q-1)^3+2q^4(q-1)^2}{q^{6+d+k_2}}=\\
\frac{(q-1)^2((k_1+1)q-(k_1-1))}{q^{3+d+k_2}}.
\end{multline*}
Combing with Remark \ref{rmk410}, we have the final formula:
\begin{multline*}
c_{(k_{1},k_{2})}\cdot \mathcal{SO}_{\gamma, (k_1, k_2)}=(q^{2}+q+1)(q+1)\cdot q^{2k_{2}-3}\cdot\frac{(q-1)^2((k_1+1)q-(k_1-1))}{q^{3+d+k_2}}=\\
(q^{3}-1)(q^{2}-1)q^{k_{2}}\cdot \frac{(k_1+1)q-(k_1-1)}{q^{6+d}}.
\end{multline*}
\end{proof}

\subsection{Case 2: $k_1< k_2$ and $k_1>s$}
Recall that $t=\lceil\frac{d}{3}\rceil$ and $s=\lfloor \frac{d}{3}\rfloor$ whose relations are described in Equation (\ref{eq72}).
Since $k_1+k_2=d$ and $k_1\geq s+1$, the Newton polygon of an irreducible polynomial $\chi_{\gamma}$ yields
\begin{equation}\label{k12-2}
\left\{
  \begin{array}{l l}
\ord(c_1)\geq t, \ord(c_2)\geq 2t, k_2\leq 2t-1   & \textit{if $d=3t$};\\
\ord(c_1)\geq t, \ord(c_2)\geq 2t, k_2\leq 2t-1  & \textit{if $d=3t-1$};\\
\ord(c_1)\geq t, \ord(c_2)\geq 2t-1, k_2\leq 2t-2   & \textit{if $d=3t-2$}
    \end{array} \right.
\end{equation}
so that $k_1\geq t$, $t+k_1>k_2$, and $\ord(c_2)>k_2$.

Choose an element $m=\begin{pmatrix}
x_{1,1} &x_{1,2} & x_{1,3} \\
\pi^{k_1}x_{2,1} &\pi^{k_1}x_{2,2} & \pi^{k_1}x_{2,3} \\
\pi^{k_2}x_{3,1} &\pi^{k_2}x_{3,2} & \pi^{k_2}x_{3,3}
\end{pmatrix} $ of $\varphi_{3,M}^{-1}(\chi_{\gamma})(\mfo)$.
Since the characteristic polynomial of $m$ should satisfy Equation (\ref{k12-2}),  $x_{1,1}\in (\pi^{t})$ and $ x_{1,2}x_{2,1}\in (\pi^{k_2-k_1})$.

Therefore we have the following stratification of $\varphi_{3,M}^{-1}(\chi_{\gamma})(\mfo)$ with respect to the valuation of $x_{2,1}$:
\begin{multline}\label{eq76}
\varphi_{3,M}^{-1}(\chi_{\gamma})(\mfo)=
\bigsqcup_{l=0}^{k_2-k_1-1}\Bigl\{\begin{pmatrix}
\pi^{t}x_{1,1} &\pi^{k_2-k_1-l}x_{1,2} & x_{1,3} \\
\pi^{k_1+l}x_{2,1} &\pi^{k_1}x_{2,2} & \pi^{k_1}x_{2,3} \\
\pi^{k_2}x_{3,1} &\pi^{k_2}x_{3,2} & \pi^{k_2}x_{3,3}
\end{pmatrix}\in \varphi_{3,M}^{-1}(\chi_{\gamma})(\mfo)|x_{2,1}\in \mfo^{\times}\Bigl\}\\
\bigsqcup \Bigl\{\begin{pmatrix}
\pi^{t}x_{1,1} &x_{1,2} & x_{1,3} \\
\pi^{k_2}x_{2,1} &\pi^{k_1}x_{2,2} & \pi^{k_1}x_{2,3} \\
\pi^{k_2}x_{3,1} &\pi^{k_2}x_{3,2} & \pi^{k_2}x_{3,3}
\end{pmatrix}\in \varphi_{3,M}^{-1}(\chi_{\gamma})(\mfo)\Bigl\}.
\end{multline}
Here, the determinant of $\begin{pmatrix}
\pi^{t}x_{1,1} &\pi^{k_2-k_1-l}x_{1,2} & x_{1,3} \\
\pi^{l}x_{2,1} & x_{2,2} &  x_{2,3} \\
x_{3,1} & x_{3,2} &  x_{3,3}
\end{pmatrix}$ in each stratum for $0\leq l \leq k_2-k_1$ is a unit in $\mfo$, where the last case $l=k_2-k_1$ is about the last stratum.
In particular when $l=k_2-k_1$,
we claim that  $x_{1,2}\in \mfo^{\times}$.
Otherwise, $x_{1,2}\in (\pi)$ so that $x_{3,1}, x_{1,3}\in \mfo^{\times}$ since the determinant of the above matrix is a unit in $\mfo$.
On the other hand, Equation  (\ref{k12-2}) yields that $-x_{1,2}x_{2,1}-x_{3,1}x_{1,3}\in (\pi)$. This is a contradiction.
\\

Now, we assign a scheme structure on each stratum as in Section \ref{subsec52}.
Define   functors $\cL(L,M)_l$ for $0\leq l \leq k_2-k_1$ and $\mathbb{A}_{M,1}$ on the category of flat $\mfo$-algebras to the category of sets such that
 \[
\cL(L,M)_l(R)
=\begin{pmatrix}
\pi^{t}x_{1,1} &\pi^{k_2-k_1-l}x_{1,2} & x_{1,3} \\
\pi^{k_1+l}x_{2,1} &\pi^{k_1}x_{2,2} & \pi^{k_1}x_{2,3} \\
\pi^{k_2}x_{3,1} &\pi^{k_2}x_{3,2} & \pi^{k_2}x_{3,3}
\end{pmatrix}
   \textit{,    } 
\mathbb{A}_{M,1}(R)=\pi^{t} R\times \pi^{k_2} R\times \pi^d R
\]
where $x_{2,1}\in R^{\times}$ (resp. $x_{1,2}\in R^{\times}$) if $0\leq l \leq k_2-k_1-1$ (resp. $l=k_2-k_1$) and $\begin{pmatrix}
\pi^{t}x_{1,1} &\pi^{k_2-k_1-l}x_{1,2} & x_{1,3} \\
\pi^{l}x_{2,1} & x_{2,2} &  x_{2,3} \\
x_{3,1} &  x_{3,2} & x_{3,3}
\end{pmatrix}$
is invertible as a matrix with entries in $R$ for $0\leq l \leq k_2-k_1$.
Here, $R$ is a flat $\mfo$-algebra.

Let $\varphi_{3,M}^l$ be the restriction of $\varphi_{3,M}$ to $\cL(L,M)_l(R)$ for $0\leq l \leq k_2-k_1$.
Then the morphism
\[
\varphi_{3,M}^l: \cL(L,M)_l(R) \longrightarrow \mathbb{A}_{M,1}(R)
\]
is represented by a morphism of schemes over $\mfo$.
Note that each stratum in Equation (\ref{eq76}) is  $(\varphi_{3,M}^l)^{-1}(\chi_{\gamma})(\mfo)$.
Thus  the stratification in Equation (\ref{eq76}) turns out to be
\[
\varphi_{3,M}^{-1}(\chi_{\gamma})(\mfo)=\bigsqcup_{l=0}^{k_2-k_1}(\varphi_{3,M}^l)^{-1}(\chi_{\gamma})(\mfo).
\]

\begin{proposition}
The scheme $(\varphi_{3,M}^1)^{-1}(\chi_{\gamma})$ is smooth over $\mfo$.
\end{proposition}

\begin{proof}
The proof is similar to that of Theorem \ref{thm55}.
It suffices to show that for any $m\in (\varphi_{3,M}^l)^{-1}(\chi_{\gamma})(\bar{\kappa})$,
the induced map on the Zariski tangent space 
\[
d(\vpi_{3,M}^l)_{\ast, m}: T_m \longrightarrow T_{\varphi_{3,M}^l(m)}
\] 
is surjective, where $T_m$ is the Zariski tangent space of $ \cL(L,M)_l\otimes \bar{\kappa}$ at $m$ and $T_{\varphi_{3,M}^l(m)}$ is the Zariski tangent space of $\mathbb{A}_{M,1}\otimes \bar{\kappa}$ at $\varphi_{3,M}^l(m)$.

Write $m\in (\varphi_{3,M}^l)^{-1}(\chi_{\gamma})(\bar{\kappa})$ and $X\in T_m$ as the following matrix formally;
 \[
m=\begin{pmatrix}
\pi^{t}x_{1,1} &\pi^{k_2-k_1-l}x_{1,2} & x_{1,3} \\
\pi^{k_1+l}x_{2,1} &\pi^{k_1}x_{2,2} & \pi^{k_1}x_{2,3} \\
\pi^{k_2}x_{3,1} &\pi^{k_2}x_{3,2} & \pi^{k_2}x_{3,3}
\end{pmatrix}   \textit{   and   } X=\begin{pmatrix}
\pi^{t}a_{1,1} &\pi^{k_2-k_1-l}a_{1,2} & a_{1,3} \\
\pi^{k_1+l}a_{2,1} &\pi^{k_1}a_{2,2} & \pi^{k_1}a_{2,3} \\
\pi^{k_2}a_{3,1} &\pi^{k_2}a_{3,2} & \pi^{k_2}a_{3,3}
\end{pmatrix},
\]
where $x_{ij}, a_{ij}\in \bar{\kappa}$,  
 $x_{2,1}\in \bar{\kappa}^{\times}$ (resp. $x_{1,2}\in \bar{\kappa}^{\times})$  if $0\leq l \leq k_2-k_1-1$ (resp. $l=k_2-k_1$),  and
 $\begin{pmatrix}
0 &\pi^{k_2-k_1-l}x_{1,2} & x_{1,3} \\
\pi^{l}x_{2,1} & x_{2,2} &  x_{2,3} \\
x_{3,1} & x_{3,2} &  x_{3,3}
\end{pmatrix}$ is invertible  as a matrix of $\mathrm{M}_3(\bar{\kappa})$ for  $0\leq l \leq k_2-k_1$.
Thus 
 \[
\left\{
  \begin{array}{l l}
x_{2,1}, x_{1,3}\neq 0 &  \textit{if $l=0$};\\
x_{2,1}, x_{1,3}, x_{3,1}\neq 0 &  \textit{if $1\leq l \leq k_2-k_1-1$};\\
x_{1,2}, x_{3,1}\neq 0 & \textit{if $l=k_2-k_1$}.
    \end{array} \right.
\]

Our method to prove the surjectivity of $d(\vpi_{3,M}^l)_{\ast, m}$ is to choose a certain subspace of $T_m$ mapping onto $T_{\vpi_{3,M}^l(m)}$.
Let 
 $A_1=\begin{pmatrix}
\pi^{t}1 &\pi^{k_2-k_1-l}0 & 0 \\
\pi^{k_1+l}0 &\pi^{k_1}0 & \pi^{k_1}0 \\
\pi^{k_2}0 &\pi^{k_2}0 & \pi^{k_2}0
\end{pmatrix}$ for $0\leq k\leq k_2-k_1$ and let  $A_2, A_3$ be
 \[
\left\{
  \begin{array}{l l}
\begin{pmatrix}
\pi^{t}0 &\pi^{k_2-k_1}1 & 0 \\
\pi^{k_1}0 &\pi^{k_1}0 & \pi^{k_1}0 \\
\pi^{k_2}0 &\pi^{k_2}0 & \pi^{k_2}0
\end{pmatrix}, 
\begin{pmatrix}
\pi^{t}0 &\pi^{k_2-k_1}0 & 0 \\
\pi^{k_1}0 &\pi^{k_1}0 & \pi^{k_1}0 \\
\pi^{k_2}0 &\pi^{k_2}1 & \pi^{k_2}0
\end{pmatrix} &  \textit{if $l=0$};\\
\begin{pmatrix}
\pi^{t}0 &\pi^{k_2-k_1-l}1 & 0 \\
\pi^{k_1+l}0 &\pi^{k_1}0 & \pi^{k_1}0 \\
\pi^{k_2}0 &\pi^{k_2}0 & \pi^{k_2}0
\end{pmatrix}, 
\begin{pmatrix}
\pi^{t}0 &\pi^{k_2-k_1-l}0 & 0 \\
\pi^{k_1+l}0 &\pi^{k_1}1 & \pi^{k_1}0 \\
\pi^{k_2}0 &\pi^{k_2}0 & \pi^{k_2}0
\end{pmatrix} &  \textit{if $1\leq l \leq k_2-k_1-1$};\\
\begin{pmatrix}
\pi^{t}0 &0 & 0 \\
\pi^{k_2}1 &\pi^{k_1}0 & \pi^{k_1}0 \\
\pi^{k_2}0 &\pi^{k_2}0 & \pi^{k_2}0
\end{pmatrix}, 
\begin{pmatrix}
\pi^{t}0 &0 & 0 \\
\pi^{k_2}0 &\pi^{k_1}0 & \pi^{k_1}1 \\
\pi^{k_2}0 &\pi^{k_2}0 & \pi^{k_2}0
\end{pmatrix} & \textit{if $l=k_2-k_1$}
    \end{array} \right.
\]
as elements of $T_m$.
Then the images of $A_1, A_2, A_3$ are 
 \[
\left\{
  \begin{array}{l l}
\left(1, \ast, \ast\ast \right), \left(0, x_{2,1}, \ast \right), \left(0, 0, x_{2,1}x_{1,3} \right) &  \textit{if $l=0$};\\
\left(1, \ast, \ast\ast \right), \left(0, x_{2,1}, \ast \right), \left(0, 0, x_{3,1}x_{1,3} \right) &  \textit{if $1\leq l \leq k_2-k_1-1$};\\
\left(1,  \ast, \ast\ast \right), \left(0, x_{1,2}, \ast \right), \left(0, 0, x_{3,1}x_{1,2} \right) & \textit{if $l=k_2-k_1$}
    \end{array} \right.
\]
up to the sign $\pm 1$.
These three vectors are linearly independent as elements of $T_{\vpi_{3,M}^l(m)}$ and so the map $d(\vpi_{3,M}^l)_{\ast, m}$ is surjective.
\end{proof}

\begin{corollary}
The order of the set $(\varphi_{3,M}^l)^{-1}(\chi_{\gamma})(\kappa)$ is 
\[
\#(\varphi_{3,M}^l)^{-1}(\chi_{\gamma})(\kappa)=\left\{
  \begin{array}{l l}
q^3(q-1)^3 &   \textit{if $1\leq l \leq k_2-k_1-1$};\\
q^4(q-1)^2 & \textit{if $l=0$ or $l=k_2-k_1$}.
    \end{array} \right.
\]
\end{corollary}

\begin{proof}
For $m\in \cL(L,M)_l(\kappa)$,  $\varphi_{3,M}^l(m)$  is 
 \[
\left\{
  \begin{array}{l l}
\left(\pi^t(-x_{1,1}-\pi^{k_1-t}x_{2,2}), \pi^{k_2}(-x_{2,1}x_{1,2}-x_{3,1}x_{1,3}), \pi^d\left(-x_{1,3}\cdot 
\begin{vmatrix}
\pi^lx_{2,1} &x_{2,2} \\ x_{3,1} &x_{3,2}
\end{vmatrix}\right)\right) &   \textit{if $0\leq l \leq k_2-k_1-1$};\\
\left(\pi^t(-x_{1,1}-\pi^{k_1-t}x_{2,2}), \pi^{k_2}(-x_{2,1}x_{1,2}-x_{3,1}x_{1,3}), \pi^d\left(-x_{3,1}\cdot 
\begin{vmatrix}
x_{1,2} &x_{1,3} \\ x_{2,2} &x_{2,3}
\end{vmatrix}\right)\right) & \textit{if $l=k_2-k_1$}.
    \end{array} \right.
\]

\begin{enumerate}
\item 
 Suppose that $l=0$ so that   $x_{2,1}, x_{1,3}\neq 0$ and $\begin{vmatrix}
x_{2,1} &x_{2,2} \\ x_{3,1} &x_{3,2}
\end{vmatrix} \neq 0$. 
If $m\in (\varphi_{3,M}^1)^{-1}(\chi_{\gamma})(\kappa)$, then $x_{1,1}, x_{1,2}, x_{1,3}$ are completely determined by all other entries satisfying $\begin{vmatrix}
x_{2,1} &x_{2,2} \\ x_{3,1} &x_{3,2}
\end{vmatrix} \neq 0$. Thus we have 
 \[
\#(\varphi_{3,M}^l)^{-1}(\chi_{\gamma})(\kappa)=
(\#\mathrm{GL}_2(\kappa)-q(q-1)^2)\cdot q^2=q^4(q-1)^2.
\]

\item Suppose that $1\leq l \leq k_2-k_1-1$ so that  $x_{2,1}, x_{1,3}, x_{3,1} \neq 0$ and $\begin{vmatrix}
0 &x_{2,2} \\ x_{3,1} &x_{3,2}
\end{vmatrix} \neq 0$. 
If $m\in (\varphi_{3,M}^1)^{-1}(\chi_{\gamma})(\kappa)$, then $x_{1,1}, x_{1,2}, x_{1,3}$ are completely determined by all other entries satisfying $\begin{vmatrix}
0 &x_{2,2} \\ x_{3,1} &x_{3,2}
\end{vmatrix} \neq 0$. Thus we have 
 \[
\#(\varphi_{3,M}^l)^{-1}(\chi_{\gamma})(\kappa)=
(q-1)^3q\cdot q^2=q^3(q-1)^3.
\]

\item Suppose that $l=k_2-k_1$ so that 
  $x_{1,2}, x_{3,1}\neq 0$ and $\begin{vmatrix}
x_{1,2} &x_{1,3} \\ x_{2,2} &x_{2,3}
\end{vmatrix}\neq 0$.
If $m\in (\varphi_{3,M}^1)^{-1}(\chi_{\gamma})(\kappa)$, then $x_{1,1}, x_{2,1}, x_{3,1}$ are completely determined by all other entries satisfying $\begin{vmatrix}
x_{1,2} &x_{1,3} \\ x_{2,2} &x_{2,3}
\end{vmatrix}\neq 0$.
 Thus we have
  \[
\#(\varphi_{3,M}^l)^{-1}(\chi_{\gamma})(\kappa)=
(\#\mathrm{GL}_2(\kappa)-q(q-1)^2)\cdot q^2=q^4(q-1)^2.
\]
\end{enumerate}
\end{proof}

\begin{corollary}
We have 
\[
c_{(k_{1},k_{2})}\cdot \mathcal{SO}_{\gamma, (k_1, k_2)}=(q^{3}-1)(q^{2}-1)q^{k_{2}}\cdot \frac{(k_2-k_1+1)q-(k_2-k_1-1)}{q^{6+d}}.
\]

\end{corollary}

\begin{proof}
The proof is parallel to that of Corollary \ref{cornm1}.
Let $\widetilde{\cL}(L,M)_l$ be the affine space defined over $\mfo$ such that 
$\widetilde{\cL}(L,M)_l(R)=\Bigl\{\begin{pmatrix}
\pi^{t}x_{1,1} &\pi^{k_2-k_1-l}x_{1,2} & x_{1,3} \\
\pi^{k_1+l}x_{2,1} &\pi^{k_1}x_{2,2} & \pi^{k_1}x_{2,3} \\
\pi^{k_2}x_{3,1} &\pi^{k_2}x_{3,2} & \pi^{k_2}x_{3,3}
\end{pmatrix}\Bigl\}$ for $0\leq l \leq k_2-k_1$. Then $\cL(L,M)_l$ is an open subscheme of $\widetilde{\cL}(L,M)_l$.

Let $\omega_{\widetilde{\cL}(L,M)_l}$ and $\omega_{\mathbb{A}_{M,1}}$ be nonzero,  translation-invariant forms on   $\mathfrak{gl}_{n, F}$ and $\mathbb{A}^n_F$,
 respectively, with normalizations
$$\int_{\widetilde{\cL}(L,M)_l(\mathfrak{o})}|\omega_{\widetilde{\cL}(L,M)_l}|=1 \mathrm{~and~}  \int_{\mathbb{A}_{M,1}(\mathfrak{o})}|\omega_{\mathbb{A}_{M,1}}|=1.$$
Putting $\omega^{can}_{(\chi_{\gamma},L,M)_l}=\omega_{\widetilde{\cL}(L,M)_l}/\rho_n^{\ast}\omega_{\mathbb{A}_{M,1}}$.
 Then we have the following comparison among differentials;
 
 \[
\left\{
  \begin{array}{l }
|\omega_{\mathfrak{gl}_{n, \mathfrak{o}}}|=|\pi|^{2k_1+4k_2+t}\cdot|\omega_{\widetilde{\cL}(L,M)_l}|;\\
|\omega_{\mathbb{A}^n_{\mathfrak{o}}}|=
|\pi|^{t+k_2+d}\cdot |\omega_{\mathbb{A}_{M,1}}|;\\
|\omega_{\chi_{\gamma}}^{\mathrm{ld}}|=|\pi|^{d+k_2}|\omega^{can}_{(\chi_{\gamma},L,M)_l}|.
    \end{array} \right.
\]

Therefore, we have
\begin{multline*}
\mathcal{SO}_{\gamma, (k_1, k_2)}=\sum\limits_{l=0}^{k_2-k_1}\int_{(\varphi_{3,M}^l)^{-1}(\chi_{\gamma})(\mfo)}|\omega_{\chi_{\gamma}}^{\mathrm{ld}}|= 
\frac{(k_2-k_1-1)q^3(q-1)^3+2q^4(q-1)^2}{q^{6+d+k_2}}=\\
\frac{(q-1)^2((k_2-k_1+1)q-(k_2-k_1-1))}{q^{3+d+k_2}}.
\end{multline*}
Combing with Remark \ref{rmk410}, we have the final formula:
\begin{multline*}
c_{(k_{1},k_{2})}\cdot \mathcal{SO}_{\gamma, (k_1, k_2)}=(q^{2}+q+1)(q+1)\cdot q^{2k_{2}-3}\cdot\frac{(q-1)^2((k_2-k_1+1)q-(k_2-k_1-1))}{q^{3+d+k_2}}=\\
(q^{3}-1)(q^{2}-1)q^{k_{2}}\cdot \frac{(k_2-k_1+1)q-(k_2-k_1-1)}{q^{6+d}}.
\end{multline*}
\end{proof}

\subsection{Case 3: \texorpdfstring{$k_1=k_2$}{k1=k2}}
In this case $d=2k_1$.
The Newton polygon of the irreducible polynomial $\chi_{\gamma}$ yields 
\[
\ord(c_1)>0, ~~~~~~~~\ord(c_2)>k_1.
\]

\begin{proposition}\label{prop77}
The scheme $\vpi_{3,M}^{-1}(\chi_{\gamma})$ is smooth over $\mfo$.
\end{proposition}

\begin{proof}
The proof is similar to that of Theorem \ref{thm55}.
It suffices to show that for any $m\in \varphi_{3,M}^{-1}(\chi(\gamma))(\bar{\kappa})$,
the induced map on the Zariski tangent space 
\[
d(\vpi_{3,M})_{\ast, m}: T_m \longrightarrow T_{\vpi_{3,M}(m)}
\] 
is surjective, where $T_m$ is the Zariski tangent space of $ \cL(L,M)\otimes \bar{\kappa}$ at $m$ and $T_{\vpi_{3,M}(m)}$ is the Zariski tangent space of $\mathbb{A}_{M}\otimes \bar{\kappa}$ at $\vpi_{3,M}(m)$.

Using the matrix description given in Equation (\ref{matrixform}), 
we write $m\in \varphi_{3,M}^{-1}(\chi(\gamma))(\bar{\kappa})$ and $X\in T_m$ as the following matrix formally;
 \[
m=\begin{pmatrix}
x_{1,1} &x_{1,2} & x_{1,3} \\
\pi^{k_1}x_{2,1} &\pi^{k_1}x_{2,2} & \pi^{k_1}x_{2,3} \\
\pi^{k_1}x_{3,1} &\pi^{k_1}x_{3,2} & \pi^{k_1}x_{3,3}
\end{pmatrix}   \textit{   and   } X=\begin{pmatrix}
a_{1,1} &a_{1,2} & a_{1,3} \\
\pi^{k_1}a_{2,1} &\pi^{k_1}a_{2,2} & \pi^{k_1}a_{2,3} \\
\pi^{k_1}a_{3,1} &\pi^{k_1}a_{3,2} & \pi^{k_1}a_{3,3}
\end{pmatrix},
\]
 where $x_{ij}, a_{ij}\in \bar{\kappa}$, $x_{1,1}=0$, and
the matrix $\begin{pmatrix}
0 &x_{1,2} & x_{1,3} \\
x_{2,1} &x_{2,2} & x_{2,3} \\
x_{3,1} &x_{3,2} & x_{3,3}
\end{pmatrix}$ is invertible  as a matrix of $\mathrm{M}_3(\bar{\kappa})$.
Note that either $x_{2,1}$ or $x_{3,1}$ is nonzero. Then 
by change of a basis (if necessary), we may and do assume that $x_{2,1}\neq 0$  and either $x_{1,2}$ or $x_{1,3}$ is nonzero.

Our method to prove the surjectivity of $d(\vpi_{3,M})_{\ast, m}$ is to choose a certain subspace of $T_m$ mapping onto $T_{\vpi_{3,M}(m)}$.
Let $A_1, A_2, A_3$ be
 \[
\left\{
  \begin{array}{l l}
\begin{pmatrix}
1 &0&0 \\
\pi^{k_1}0 &\pi^{k_1}0 & \pi^{k_1}0 \\
\pi^{k_1}0 &\pi^{k_1}0 & \pi^{k_1}0
\end{pmatrix}, \begin{pmatrix}
0 &1&0 \\
\pi^{k_1}0 &\pi^{k_1}0 & \pi^{k_1}0 \\
\pi^{k_1}0 &\pi^{k_1}0 & \pi^{k_1}0
\end{pmatrix}, \begin{pmatrix}
0 &0&0 \\
\pi^{k_1}0 &\pi^{k_1}0 & \pi^{k_1}0 \\
\pi^{k_1}0 &\pi^{k_1}0 & \pi^{k_1}1
\end{pmatrix} &   \textit{if $x_{1,2}\neq 0$};\\
\begin{pmatrix}
1 &0&0 \\
\pi^{k_1}0 &\pi^{k_1}0 & \pi^{k_1}0 \\
\pi^{k_1}0 &\pi^{k_1}0 & \pi^{k_1}0
\end{pmatrix}, \begin{pmatrix}
0 &1&0 \\
\pi^{k_1}0 &\pi^{k_1}0 & \pi^{k_1}0 \\
\pi^{k_1}0 &\pi^{k_1}0 & \pi^{k_1}0
\end{pmatrix}, \begin{pmatrix}
0 &0&0 \\
\pi^{k_1}0 &\pi^{k_1}0 & \pi^{k_1}0 \\
\pi^{k_1}0 &\pi^{k_1}1 & \pi^{k_1}0
\end{pmatrix} & \textit{if $x_{1,2}=0, x_{1,3}\neq 0$}.
    \end{array} \right.
\]

Then the images of $A_1, A_2, A_3$ are 
\[
\left\{
  \begin{array}{l l}
\left( (1, \ast, \ast), (0, \pi^{k_1}x_{2,1}, \ast), (0, 0, \pi^{2k_1}x_{2,1}x_{1,2}) \right) &   \textit{if $x_{1,2}\neq 0$};\\
\left((1, \ast, \ast), (0, \pi^{k_1}x_{2,1}, \ast), (0, 0, \pi^{2k_1}x_{2,1}x_{1,3})\right) & \textit{if $x_{1,2}=0, x_{1,3}\neq 0$}
    \end{array} \right.
\]
up to the sign $\pm 1$.
These three vectors are linearly independently as elements of $T_{\vpi_{3,M}(m)}$ and so the map 
$d(\vpi_{3,M})_{\ast, m}$ is surjective.
\end{proof}

\begin{corollary}\label{cor78}
The order of the set  $\vpi_{3,M}^{-1}(\chi_{\gamma})(\kappa)$ is 
\[
\#\vpi_{3,M}^{-1}(\chi_{\gamma})(\kappa)=(q+1)(q-1)^2q^3.
\]
\end{corollary}

\begin{proof}
For $m\in \cL(L,M)(\kappa)$,  $\varphi_{3,M}(m)$  is 
\[
\left( -x_{1,1}, \pi^{k_1}\left(\begin{vmatrix} x_{1,1} &x_{1,2} \\
x_{2,1} &x_{2,2}
 \end{vmatrix}+ \begin{vmatrix} x_{1,1} & x_{1,3} \\

x_{3,1} & x_{3,3}
 \end{vmatrix}\right), \pi^{2k_1}\left(-\begin{vmatrix} x_{1,1} &x_{1,2} & x_{1,3} \\
x_{2,1} &x_{2,2} & x_{2,3} \\
x_{3,1} &x_{3,2} & x_{3,3}
 \end{vmatrix} \right)  \right).
 \]
We claim that the    number of all choices of  $x_{ij}$'s for  $m\in (\varphi_{3,M})^{-1}(\chi_{\gamma})(\kappa)$ is 
  \[
\left\{
  \begin{array}{l l}
(q-1)^2q^3 & \textit{if $x_{2,1}=0$};\\
(q-1)^2q^4 & \textit{if $x_{2,1}\neq 0$}.
    \end{array} \right.
\]

For $m\in (\varphi_{3,M})^{-1}(\chi_{\gamma})(\kappa)$, 
the first two components of $\varphi_{3,M}(m)$ yield that $x_{1,1}=0$ and $-x_{2,1}x_{1,2}-x_{3,1}x_{1,3}=0$ and the third component of $\varphi_{3,M}(m)$ is nonzero. 
\begin{enumerate}
\item 
If $x_{2,1}=0$, then $x_{3,1}\neq 0$ and so the second component of $\varphi_{3,M}(m)$ yields that  $x_{1,3}=0$ and $x_{1,2}\neq 0$.
Then the third component of $\varphi_{3,M}(m)$ is $x_{3,1}x_{1,2}x_{2,3}$. Thus $x_{2,3}$ is completely determined by 
$x_{3,1}$ and $x_{1,2}$. This gives the desired formula.

\item Suppose that  $x_{2,1}\neq 0$. If $x_{1,3}=0$, then the second component of $\varphi_{3,M}(m)$ implies that $x_{1,2}=0$, which is a contradiction that the third component of $\varphi_{3,M}(m)$ is nonzero. 
Thus $x_{1,3}\neq 0$. 

The second component of $\varphi_{3,M}(m)$ yields that $x_{1,2}$ is completely determined by $x_{2,1}, x_{1,3}, x_{3,1}$.
On the other hand, the third component of $\varphi_{3,M}(m)$ is of the form $x_{2,1}x_{1,3}x_{3,2}+\ast$.
Since $x_{2,1}x_{1,3}\neq 0$, $x_{3,2}$ is completely determined by the third component of $\varphi_{3,M}(m)$. 
This gives the desired formula.
\end{enumerate}
Then $\#\vpi_{3,M}^{-1}(\chi_{\gamma})(\kappa)$ is the sum of these two.
\end{proof}

\begin{corollary}
We have 
\[
c_{(k_{1},k_{2})}\cdot \mathcal{SO}_{\gamma, (k_1, k_2)}=\frac{(q^{3}-1)(q^{2}-1)}{q^{5+d}}\cdot q^{k_{1}}.
\]

\end{corollary}

\begin{proof}
The proof is parallel to that of Corollary \ref{cornm1}.
By Proposition \ref{prop54}, 
\[
\mathcal{SO}_{\gamma, (k_1, k_1)}=q^{-3k_1}\int_{O_{\gamma, \cL(L,M)}}|\omega^{ld}_{(\chi_{\gamma},\widetilde{\cL}(L,M))}|.
\]
By Proposition \ref{prop77} and Corollary \ref{cor78},
 \[
\int_{O_{\gamma, \cL(L,M)}}|\omega^{ld}_{(\chi_{\gamma},\widetilde{\cL}(L,M))}|=\frac{(q+1)(q-1)^2q^3}{q^6}.
\]
Combing with Remark \ref{rmk410}, we have the final formula:
\[
c_{(k_{1},k_{2})}\cdot \mathcal{SO}_{\gamma, (k_1, k_2)}=(q^{2}+q+1)\cdot q^{2k_{1}-2}\cdot \frac{(q+1)(q-1)^2}{q^{3+d+k_1}}=\frac{(q^{3}-1)(q^{2}-1)}{q^{5+d}}\cdot q^{k_{1}}.
\]
Here $d=2k_1$.
\end{proof}

\subsection{Case 4: \texorpdfstring{$k_1=t$}{k1=t} and \texorpdfstring{$d=3t$}{d=3t}}
In this case we suppose that the reduction of $\chi_{\gamma}(\pi^tx)/\pi^{3t}$ modulo $\pi$ is irreducible over $\kappa$, as mentioned at the paragraph following Equation (\ref{eq72}).
The Newton polygon of an irreducible polynomial $\chi_{\gamma}$ yields
\begin{equation}\label{eq77}
k_2=2t, ~~~~~~ \ord(c_1)\geq t, ~~~~~~~~~~~ \ord(c_2)\geq 2t.
\end{equation}

Choose an element $m=\begin{pmatrix}
x_{1,1} &x_{1,2} & x_{1,3} \\
\pi^{t}x_{2,1} &\pi^{t}x_{2,2} & \pi^{t}x_{2,3} \\
\pi^{2t}x_{3,1} &\pi^{2t}x_{3,2} & \pi^{2t}x_{3,3}
\end{pmatrix} $ of $\varphi_{3,M}^{-1}(\chi_{\gamma})(\mfo)$.
Since the characteristic polynomial of $m$ should satisfy Equation (\ref{eq77}),
$x_{1,1}, x_{2,1}x_{1,2}\in (\pi^t)$.
Therefore we have the following stratification of $\varphi_{3,M}^{-1}(\chi_{\gamma})(\mfo)$ with respect to the valuation of $x_{2,1}$:
\begin{multline}\label{eq78}
\varphi_{3,M}^{-1}(\chi_{\gamma})(\mfo)=
\bigsqcup_{l=0}^{t-1}\Bigl\{\begin{pmatrix}
\pi^t x_{1,1} &\pi^{t-l}x_{1,2} & x_{1,3} \\
\pi^{t+l}x_{2,1} &\pi^{t}x_{2,2} & \pi^{t}x_{2,3} \\
\pi^{2t}x_{3,1} &\pi^{2t}x_{3,2} & \pi^{2t}x_{3,3}
\end{pmatrix}\in \varphi_{3,M}^{-1}(\chi_{\gamma})(\mfo)|x_{2,1}\in \mfo^{\times}\Bigl\}\\
\bigsqcup \Bigl\{\begin{pmatrix}
\pi^tx_{1,1} &x_{1,2} & x_{1,3} \\
\pi^{2t}x_{2,1} &\pi^{t}x_{2,2} & \pi^{t}x_{2,3} \\
\pi^{2t}x_{3,1} &\pi^{2t}x_{3,2} & \pi^{2t}x_{3,3}
\end{pmatrix}\in \varphi_{3,M}^{-1}(\chi_{\gamma})(\mfo)\Bigl\}.
\end{multline}
Here, the determinant of $\begin{pmatrix}
\pi^t x_{1,1} &\pi^{t-l}x_{1,2} & x_{1,3} \\
\pi^{l}x_{2,1} &x_{2,2} & x_{2,3} \\
x_{3,1} &x_{3,2} & x_{3,3}
\end{pmatrix}$ in each stratum for $0\leq l \leq t$ is a unit in $\mfo$, where the last case $l=t$ is about the last stratum.
In particular when $l=t$,
we claim that  $x_{1,2}\in \mfo^{\times}$.
Otherwise, $x_{1,2}\in (\pi)$ and so $x_{2,2}, x_{1,3}, x_{3,1} \in \mfo^{\times}$ since the determinant of the above matrix is a unit in $\mfo$. 
Then the reduction of $\chi_m(\pi^tx)/\pi^{3t}$ modulo $\pi$ is 
\[
x^3-(\bar{x}_{1,1}+\bar{x}_{2,2})x^2+(\bar{x}_{1,1}\bar{x}_{2,2}-\bar{x}_{3,1}\bar{x}_{1,3})x+\bar{x}_{1,3}\bar{x}_{2,2}\bar{x}_{3,1}.
\]
This polynomial has a root $\bar{x}_{2,2}$, which is a contradiction to the irreducibility of the reduction of $\chi_m(\pi^tx)/\pi^{3t}$ modulo $\pi$.
\\

Now, we assign a scheme structure on each stratum as in Section \ref{subsec52}.
Define   functors $\cL(L,M)_l$ for $0\leq l \leq t$ and $\mathbb{A}_{M,1}$ on the category of flat $\mfo$-algebras to the category of sets such that
 \[
\cL(L,M)_l(R)
=\Bigl\{\begin{pmatrix}
\pi^t x_{1,1} &\pi^{t-l}x_{1,2} & x_{1,3} \\
\pi^{t+l}x_{2,1} &\pi^{t}x_{2,2} & \pi^{t}x_{2,3} \\
\pi^{2t}x_{3,1} &\pi^{2t}x_{3,2} & \pi^{2t}x_{3,3}
\end{pmatrix}\Bigl\}
   \textit{,    } ~~~~~~~~~~~ 
\mathbb{A}_{M,1}(R)=\pi^{t} R\times \pi^{2t} R\times \pi^{3t} R
\]
where $x_{2,1}\in R^{\times}$ (resp. $x_{1,2}\in R^{\times}$)  if $0\leq l \leq t-1$ (resp. $l=t$) and $\begin{pmatrix}
\pi^t x_{1,1} &\pi^{t-l}x_{1,2} & x_{1,3} \\
\pi^{l}x_{2,1} &x_{2,2} & x_{2,3} \\
x_{3,1} &x_{3,2} & x_{3,3}
\end{pmatrix}$
is invertible as a matrix with entries in $R$ 
for  $0\leq l \leq t$.
Here, $R$ is a flat $\mfo$-algebra.

Let $\varphi_{3,M}^l$  be the restriction of $\varphi_{3,M}$ to $\cL(L,M)_l(R)$ for $0\leq l \leq t$.
Then the morphism
\[
\varphi_{3,M}^l: \cL(L,M)_l(R) \longrightarrow \mathbb{A}_{M,1}(R)
\]
is represented by a morphism of schemes over $\mfo$.
Note that each stratum in Equation (\ref{eq78}) is  $(\varphi_{3,M}^l)^{-1}(\chi_{\gamma})(\mfo)$.
Thus  the stratification in Equation (\ref{eq78}) turns out to be
\[
(\varphi_{3,M})^{-1}(\chi_{\gamma})(\mfo)=\bigsqcup_{l=0}^{t}(\varphi_{3,M}^l)^{-1}(\chi_{\gamma})(\mfo).
\]

\begin{proposition}
The scheme $(\varphi_{3,M}^l)^{-1}(\chi_{\gamma})$ is smooth over $\mfo$.
\end{proposition}

\begin{proof}
The proof is similar to that of Theorem \ref{thm55}.
It suffices to show that for any $m\in (\varphi_{3,M}^l)^{-1}(\chi_{\gamma})(\bar{\kappa})$,
the induced map on the Zariski tangent space 
\[
d(\vpi_{3,M}^l)_{\ast, m}: T_m \longrightarrow T_{\varphi_{3,M}^l(m)}
\] 
is surjective, where $T_m$ is the Zariski tangent space of $ \cL(L,M)_l\otimes \bar{\kappa}$ at $m$ and $T_{\varphi_{3,M}^l(m)}$ is the Zariski tangent space of $\mathbb{A}_{M,1}\otimes \bar{\kappa}$ at $\varphi_{3,M}^l(m)$.

Write $m\in (\varphi_{3,M}^l)^{-1}(\chi_{\gamma})(\bar{\kappa})$ and $X\in T_m$ as the following matrix formally;
 \[
m=\begin{pmatrix}
\pi^tx_{1,1} &\pi^{t-l}x_{1,2} & x_{1,3} \\
\pi^{t+l}x_{2,1} &\pi^{t}x_{2,2} & \pi^{t}x_{2,3} \\
\pi^{2t}x_{3,1} &\pi^{2t}x_{3,2} & \pi^{2t}x_{3,3}
\end{pmatrix}   \textit{   and   } X=\begin{pmatrix}
\pi^ta_{1,1} &\pi^{t-l}a_{1,2} & a_{1,3} \\
\pi^{t+l}a_{2,1} &\pi^{t}a_{2,2} & \pi^{t}a_{2,3} \\
\pi^{2t}a_{3,1} &\pi^{2t}a_{3,2} & \pi^{2t}a_{3,3}
\end{pmatrix},
\]
 where $x_{ij}, a_{ij}\in \bar{\kappa}$,  
 $x_{2,1}\in \bar{\kappa}^{\times}$ (resp. $x_{1,2}\in \bar{\kappa}^{\times})$  if $0\leq l \leq t-1$ (resp. $l=t$),  and
 $\begin{pmatrix}
\pi^tx_{1,1} &\pi^{t-l}x_{1,2} & x_{1,3} \\
\pi^{l}x_{2,1} & x_{2,2} &  x_{2,3} \\
 x_{3,1} & x_{3,2} & x_{3,3}
\end{pmatrix}$ is invertible  as a matrix of $\mathrm{M}_3(\bar{\kappa})$ for  $0\leq l \leq t$.
Thus 
 \[
\left\{
  \begin{array}{l l}
x_{2,1}, x_{1,3}\neq 0 &  \textit{if $l=0$};\\
x_{2,1}, x_{1,3}, x_{3,1}\neq 0 &  \textit{if $1\leq l \leq t-1$};\\
x_{1,2}, x_{3,1}\neq 0 & \textit{if $l=t$}.
    \end{array} \right.
\]

Our method to prove the surjectivity of $d(\vpi_{3,M}^l)_{\ast, m}$ is to choose a certain subspace of $T_m$ mapping onto $T_{\vpi_{3,M}^l(m)}$.
Let 
 $A_1=\begin{pmatrix}
\pi^{t}1 &\pi^{t-l}0 & 0 \\
\pi^{t+l}0 &\pi^{t}0 & \pi^{t}0 \\
\pi^{2t}0 &\pi^{2t}0 & \pi^{2t}0
\end{pmatrix}$ for $0\leq l \leq t$ and let $A_2, A_3$ be
 \[
\left\{
  \begin{array}{l l}
  \begin{pmatrix}
\pi^{t}0 &\pi^{t}1 & 0 \\
\pi^{t}0 &\pi^{t}0 & \pi^{t}0 \\
\pi^{2t}0 &\pi^{2t}0 & \pi^{2t}0
\end{pmatrix}, 
\begin{pmatrix}
\pi^{t}0 &\pi^{t}0 & 0 \\
\pi^{t}0 &\pi^{t}0 & \pi^{t}0 \\
\pi^{2t}0 &\pi^{2t}1 & \pi^{2t}0
\end{pmatrix} &  \textit{if $l=0$};\\
\begin{pmatrix}
\pi^{t}0 &\pi^{t-l}1 & 0 \\
\pi^{t+l}0 &\pi^{t}0 & \pi^{t}0 \\
\pi^{2t}0 &\pi^{2t}0 & \pi^{2t}0
\end{pmatrix}, 
\begin{pmatrix}
\pi^{t}0 &\pi^{t-l}0 & 0 \\
\pi^{t+l}0 &\pi^{t}1 & \pi^{t}0 \\
\pi^{2t}0 &\pi^{2t}0 & \pi^{2t}0
\end{pmatrix} &  \textit{if $1\leq l \leq t-1$};\\
\begin{pmatrix}
\pi^{t}0 &0 & 0 \\
\pi^{2t}1 &\pi^{t}0 & \pi^{t}0 \\
\pi^{2t}0 &\pi^{2t}0 & \pi^{2t}0
\end{pmatrix}, 
\begin{pmatrix}
\pi^{t}0 & 0 & 0 \\
\pi^{2t}0 &\pi^{t}0 & \pi^{t}1 \\
\pi^{2t}0 &\pi^{2t}0 & \pi^{2t}0
\end{pmatrix} & \textit{if $l=t$}
    \end{array} \right.
\]
as elements of $T_m$.
Then the images of $A_1, A_2, A_3$ are 
 \[
\left\{
  \begin{array}{l l}
\left(1, \ast, \ast\ast \right), \left(0, x_{2,1}, \ast \right), \left(0, 0, x_{2,1}x_{1,3} \right) &  \textit{if $l=0$};\\
\left(1, \ast, \ast\ast \right), \left(0, x_{2,1}, 0 \right), \left(0, x_{1,1}, x_{3,1}x_{1,3} \right) &  \textit{if $1\leq l \leq k_1-1$};\\
\left(1,  \ast, \ast\ast \right), \left(0, x_{1,2}, \ast \right), \left(0, 0, x_{3,1}x_{1,2} \right) & \textit{if $l=k_1$}
    \end{array} \right.
\]
up to the sign $\pm 1$.
These three vectors are linearly independent as elements of $T_{\vpi_{3,M}^l(m)}$ and so the map $d(\vpi_{3,M}^l)_{\ast, m}$ is surjective.
\end{proof}

\begin{corollary}
The order of the set  $\vpi_{3,M}^{-1}(\chi_{\gamma})(\kappa)$ is 
 \[
\#(\varphi_{3,M}^l)^{-1}(\chi_{\gamma})(\kappa)=\left\{
  \begin{array}{l l}
q^3(q-1)^3 &   \textit{if $1\leq l \leq t-1$};\\
q^4(q-1)^2 & \textit{if $l=0$ or $l=t$}.
    \end{array} \right.
\]
\end{corollary}

\begin{proof}
For $m\in \cL(L,M)_l(\kappa)$,  $\varphi_{3,M}^l(m)$  is 
 \[
\left\{
  \begin{array}{l l}
\left(\pi^{t}(-x_{1,1}-x_{2,2}), \pi^{2t}\left(
\begin{vmatrix}
x_{1,1} &x_{1,2} \\ x_{2,1} &x_{2,2}
\end{vmatrix}-x_{3,1}x_{1,3}\right), \pi^{3t}\left(-x_{1,3}\cdot 
\begin{vmatrix}
\pi^lx_{2,1} &x_{2,2} \\ x_{3,1} &x_{3,2}
\end{vmatrix}\right)\right) &   \textit{if $0\leq l \leq t-1$};\\
\left(\pi^{t}(-x_{1,1}-x_{2,2}), \pi^{2t}\left(
\begin{vmatrix}
x_{1,1} &x_{1,2} \\ x_{2,1} &x_{2,2}
\end{vmatrix}-x_{3,1}x_{1,3}\right), \pi^{3t}\left(-x_{3,1}\cdot 
\begin{vmatrix}
x_{1,2} &x_{1,3} \\ x_{2,2} &x_{2,3}
\end{vmatrix}\right)\right) & \textit{if $l=t$}.
    \end{array} \right.
\]

\begin{enumerate}
\item 
 Suppose that $l=0$ so that   $x_{2,1}, x_{1,3}\neq 0$ and $\begin{vmatrix}
x_{2,1} &x_{2,2} \\ x_{3,1} &x_{3,2}
\end{vmatrix} \neq 0$. 
If $m\in (\varphi_{3,M}^1)^{-1}(\chi_{\gamma})(\kappa)$, then $x_{1,1}, x_{1,2}, x_{1,3}$ are completely determined by all other entries satisfying $\begin{vmatrix}
x_{2,1} &x_{2,2} \\ x_{3,1} &x_{3,2}
\end{vmatrix} \neq 0$. Thus we have 
 \[
\#(\varphi_{3,M}^l)^{-1}(\chi_{\gamma})(\kappa)=
(\#\mathrm{GL}_2(\kappa)-q(q-1)^2)\cdot q^2=q^4(q-1)^2.
\]

\item Suppose that $1\leq l \leq t-1$ so that  $x_{2,1}, x_{1,3}, x_{3,1} \neq 0$ and $\begin{vmatrix}
0 &x_{2,2} \\ x_{3,1} &x_{3,2}
\end{vmatrix} \neq 0$. 
If $m\in (\varphi_{3,M}^1)^{-1}(\chi_{\gamma})(\kappa)$, then $x_{1,1}, x_{1,2}, x_{1,3}$ are completely determined by all other entries satisfying $\begin{vmatrix}
0 &x_{2,2} \\ x_{3,1} &x_{3,2}
\end{vmatrix} \neq 0$. Thus we have 
 \[
\#(\varphi_{3,M}^l)^{-1}(\chi_{\gamma})(\kappa)=
(q-1)^3q\cdot q^2=q^3(q-1)^3.
\]

\item Suppose that $l=t$ so that 
  $x_{1,2}, x_{3,1}\neq 0$ and $\begin{vmatrix}
x_{1,2} &x_{1,3} \\ x_{2,2} &x_{2,3}
\end{vmatrix}\neq 0$.
If $m\in (\varphi_{3,M}^1)^{-1}(\chi_{\gamma})(\kappa)$, then $x_{1,1}, x_{2,1}, x_{3,1}$ are completely determined by all other entries satisfying $\begin{vmatrix}
x_{1,2} &x_{1,3} \\ x_{2,2} &x_{2,3}
\end{vmatrix}\neq 0$.
 Thus we have
  \[
\#(\varphi_{3,M}^l)^{-1}(\chi_{\gamma})(\kappa)=
(\#\mathrm{GL}_2(\kappa)-q(q-1)^2)\cdot q^2=q^4(q-1)^2.
\]
\end{enumerate}
\end{proof}

\begin{corollary}
We have 
\[
c_{(k_{1},k_{2})}\cdot \mathcal{SO}_{\gamma, (k_1, k_2)}=
(q^{3}-1)(q^{2}-1)q^{k_{2}}\cdot \frac{(t+1)q-(t-1)}{q^{6+d}}.
\]

\end{corollary}

\begin{proof}
The proof is parallel to that of Corollary \ref{cornm1}.
Let $\widetilde{\cL}(L,M)_l$ be the affine space defined over $\mfo$ such that 
$\widetilde{\cL}(L,M)_l(R)=\Bigl\{\begin{pmatrix}
\pi^t x_{1,1} &\pi^{t-l}x_{1,2} & x_{1,3} \\
\pi^{t+l}x_{2,1} &\pi^{t}x_{2,2} & \pi^{t}x_{2,3} \\
\pi^{2t}x_{3,1} &\pi^{2t}x_{3,2} & \pi^{2t}x_{3,3}
\end{pmatrix}\Bigl\}$ for $0\leq l \leq t$. Then $\cL(L,M)_l$ is an open subscheme of $\widetilde{\cL}(L,M)_l$.

Let $\omega_{\widetilde{\cL}(L,M)_l}$ and $\omega_{\mathbb{A}_{M,1}}$ be nonzero,  translation-invariant forms on   $\mathfrak{gl}_{n, F}$ and $\mathbb{A}^n_F$,
 respectively, with normalizations
$$\int_{\widetilde{\cL}(L,M)_l(\mathfrak{o})}|\omega_{\widetilde{\cL}(L,M)_l}|=1 \mathrm{~and~}  \int_{\mathbb{A}_{M,1}(\mathfrak{o})}|\omega_{\mathbb{A}_{M,1}}|=1.$$
Putting $\omega^{can}_{(\chi_{\gamma},L,M)_l}=\omega_{\widetilde{\cL}(L,M)_l}/\rho_n^{\ast}\omega_{\mathbb{A}_{M,1}}$. 
 Then we have the following comparison among differentials;
 
 \[
\left\{
  \begin{array}{l }
|\omega_{\mathfrak{gl}_{n, \mathfrak{o}}}|=|\pi|^{11t}\cdot|\omega_{\widetilde{\cL}(L,M)_l}|;\\
|\omega_{\mathbb{A}^n_{\mathfrak{o}}}|=
|\pi|^{6t}\cdot |\omega_{\mathbb{A}_{M,1}}|;\\
|\omega_{\chi_{\gamma}}^{\mathrm{ld}}|=|\pi|^{5t}|\omega^{can}_{(\chi_{\gamma},L,M)_l}|.
    \end{array} \right.
\]
Therefore, we have
\begin{multline*}
\mathcal{SO}_{\gamma, (k_1, k_2)}=\sum\limits_{l=0}^{t}\int_{(\varphi_{3,M}^l)^{-1}(\chi_{\gamma})(\mfo)}|\omega_{\chi_{\gamma}}^{\mathrm{ld}}|=
\frac{(t-1)q^3(q-1)^3+2q^4(q-1)^2}{q^{6+5t}}=\\
\frac{(q-1)^2((t+1)q-(t-1))}{q^{3+5t}}.
\end{multline*}
Combing with Remark \ref{rmk410}, we have the final formula:
\begin{multline*}
c_{(k_{1},k_{2})}\cdot \mathcal{SO}_{\gamma, (k_1, k_2)}=(q^{2}+q+1)(q+1)\cdot q^{2k_{2}-3}\cdot\frac{(q-1)^2((t+1)q-(t-1))}{q^{3+5t}}=\\
(q^{3}-1)(q^{2}-1)q^{k_{2}}\cdot \frac{(t+1)q-(t-1)}{q^{6+d}}.
\end{multline*}
\end{proof}

\section{The proofs of Theorems \ref{result3}-\ref{result4}}\label{appc}
\begin{proof}[The proof of Theorem \ref{result3}]
\textit{ }

Recall that the result in Corollary \ref{cor73}. When $n=3$, $d_{a}=\ord (\chi_{\gamma,a}(0))$ is independent of any choice of $a\in \mfo$ which satisfies the conditions in Proposition \ref{charred3} for $\chi_{\gamma}$. From now on until the end of this section, we denote the integer $d_{a}$ by $d_{\gamma}$.

Based on Equation (\ref{eqsor}), let
\[
l=d_{\gamma}-3s~~~~~~~~~~\textit{     for        }~~~~~~~ 0\leq s\leq \lfloor\frac{d_{\gamma}}{3}\rfloor.
\]
We will first analyze the following  summation
\begin{equation}\label{sum3}
\sum\limits_{\substack{k_1+k_2=l;\\ 0\leq k_1\leq k_2}}q^{-3s} c_{(k_1, k_2)}\cdot\mathcal{SO}_{\gamma^{(s)}, (k_1, k_2)}.
\end{equation}

We split Summation (\ref{sum3}) with respect to the cases indicated in Theorem \ref{thm71}. Note that 
\[\left\{
\begin{array}{l l}
l\equiv 0 \textit{ modulo }3&\textit{ if }F_{\chi_{\gamma}}\textit{ is unramified over }F;\\
l\equiv 1 \textit{ or }2 \textit{ modulo }3&\textit{ if }F_{\chi_{\gamma}}\textit{ is ramified over  }F.\end{array} 
\right.\]
\begin{enumerate}
    \item The case when $l>0$ is an odd integer;
    \begin{enumerate}
        \item 
        If $F_{\chi_{\gamma}}$ is unramified over $F$, then Summation (\ref{sum3}) splits  as follows;
        \[
        q^{-3s}\Bigg(c_{(l)}\cdot\mathcal{SO}_{\gamma^{(s)}, (l)}+\sum\limits_{0<k_{1}< \frac{l}{3}}c_{(k_1, k_2)}\cdot\mathcal{SO}_{\gamma^{(s)}, (k_1, k_2)}+c_{(\frac{l}{3}, \frac{2l}{3})}\cdot\mathcal{SO}_{\gamma^{(s)}, (\frac{l}{3}, \frac{2l}{3})}+\sum\limits_{ \frac{l}{3}< k_{1}\leq \frac{l-1}{2}}c_{(k_1, k_2)}\cdot\mathcal{SO}_{\gamma^{(s)}, (k_1, k_2)}\Bigg).
        \]
        \item If $F_{\chi_{\gamma}}$ is ramified over $F$, then Summation (\ref{sum3}) splits  as follows;
        \[
        q^{-3s}\Bigg(c_{(l)}\cdot\mathcal{SO}_{\gamma^{(s)}, (l)}+\sum\limits_{0<k_{1}\leq \lfloor \frac{l}{3}\rfloor}c_{(k_1, k_2)}\cdot\mathcal{SO}_{\gamma^{(s)}, (k_1, k_2)}+\sum\limits_{\lfloor \frac{l}{3}\rfloor< k_{1}\leq \frac{l-1}{2}}c_{(k_1, k_2)}\cdot\mathcal{SO}_{\gamma^{(s)}, (k_1, k_2)}\Bigg).
        \]
    \end{enumerate}
    From the results of Corollary \ref{cornm1} and Theorem \ref{thm71}, Summation (\ref{sum3}), whatever $F_{\chi_{\gamma}}$ is unramified or ramified,  is equal to
    \begin{align*}
    \frac{(q^{3}-1)(q^{2}-1)}{q^{6+d_{\gamma}}}\Bigg(q^{l+1}+\sum\limits_{0<k_{1}\leq \lfloor \frac{l}{3}\rfloor}\left((k_{1}+1)q-(k_{1}-1)\right)q^{l-k_{1}}+\sum\limits_{\lfloor \frac{l}{3}\rfloor<k_{1}\leq \frac{l-1}{2}}((l-2k_{1}+1)q-(l-2k_{1}-1))q^{l-k_{1}}\Bigg).
    \end{align*}
    \\
    \item The case when $l>0$ is an even integer;
    \begin{enumerate}
        \item 
        If $F_{\chi_{\gamma}}$ is unramified over $F$, then Summation (\ref{sum3}) splits  as follows;
        \begin{align*}
        q^{-3s}\Bigg(&c_{(l)}\cdot\mathcal{SO}_{\gamma^{(s)}, (l)}+\sum\limits_{0<k_{1}< \frac{l}{3}}c_{(k_1, k_2)}\cdot\mathcal{SO}_{\gamma^{(s)}, (k_1, k_2)}+c_{(\frac{l}{3}, \frac{2l}{3})}\cdot\mathcal{SO}_{\gamma^{(s)}, (\frac{l}{3}, \frac{2l}{3})}+\sum\limits_{ \frac{l}{3}< k_{1}\leq \frac{l}{2}-1}c_{(k_1, k_2)}\cdot\mathcal{SO}_{\gamma^{(s)}, (k_1, k_2)}\\&+c_{(\frac{l}{2}, \frac{l}{2})}\cdot\mathcal{SO}_{\gamma^{(s)}, (\frac{l}{2}, \frac{l}{2})}\Bigg).
        \end{align*}
        \item If $F_{\chi_{\gamma}}$ is ramified over $F$, then Summation (\ref{sum3}) splits  as follows;
        \[
        q^{-3s}\Bigg(c_{(l)}\cdot\mathcal{SO}_{\gamma^{(s)}, (l)}+\sum\limits_{0<k_{1}\leq \lfloor \frac{l}{3}\rfloor}c_{(k_1, k_2)}\cdot\mathcal{SO}_{\gamma^{(s)}, (k_1, k_2)}+\sum\limits_{\lfloor \frac{l}{3}\rfloor< k_{1}\leq \frac{l}{2}-1}c_{(k_1, k_2)}\cdot\mathcal{SO}_{\gamma^{(s)}, (k_1, k_2)}+c_{(\frac{l}{2}, \frac{l}{2})}\cdot\mathcal{SO}_{\gamma^{(s)}, (\frac{l}{2}, \frac{l}{2})}\Bigg).
        \]
    \end{enumerate}
     From the results of Corollary \ref{cornm1} and Theorem \ref{thm71}, Summation (\ref{sum3}), whatever $F_{\chi_{\gamma}}$ is unramified or ramified,  is equal to
     \begin{align*}
         \frac{(q^{3}-1)(q^{2}-1)}{q^{6+d_{\gamma}}} \Bigg(&q^{l+1}+\sum\limits_{0<k_{1}\leq \lfloor \frac{l}{3}\rfloor}((k_{1}+1)q-(k_{1}-1))q^{l-k_{1}}+\sum\limits_{\lfloor \frac{l}{3}\rfloor<k_{1}\leq \frac{l}{2}-1}((l-2k_{1}+1)q-(l-2k_{1}-1))q^{l-k_{1}}
         \\&+q^{\frac{l}{2}+1}\Bigg).
     \end{align*}
\end{enumerate}

 In particular, for any integer $l\geq 3$,
     \[
        \sum\limits_{0<k_{1}\leq \lfloor \frac{l}{3}\rfloor}((k_{1}+1)q-(k_{1}-1))q^{l-k_{1}}
        =2q^{l}+3q^{l-1}+3q^{l-2}+\cdots+3q^{l-\lfloor\frac{l}{3}\rfloor+1}-(\lfloor\frac{l}{3}\rfloor-1)q^{l-\lfloor\frac{l}{3}\rfloor}    
    \]and
    \begin{align*}
        \sum\limits_{\lfloor \frac{l}{3}\rfloor<k_{1}\leq \lfloor\frac{l-1}{2}\rfloor}
        ((l-2k_{1}+1)q-(l-2k_{1}-1))q^{l-k_{1}}
        =&(l-2\lfloor \frac{l}{3}\rfloor-1)q^{l-\lfloor \frac{l}{3}\rfloor}-(l-2\lfloor \frac{l-1}{2}\rfloor -1)q^{l-\lfloor\frac{l-1}{2}\rfloor}\\
        =&\left\{\begin{array}{l l}
        (l-2\lfloor \frac{l}{3}\rfloor-1)q^{l-\lfloor \frac{l}{3}\rfloor}     & \textit{ when }l\textit{ is odd};\\
        (l-2\lfloor \frac{l}{3}\rfloor-1)q^{l-\lfloor \frac{l}{3}\rfloor}-q^{\frac{l}{2}+1}     & \textit{ when }l\textit{ is even}.
        \end{array}\right.        
    \end{align*}
    
    Therefore, for $l\geq 3$, Summation (\ref{sum3}) turns to be
    \[
    \sum\limits_{\substack{k_1+k_2=l;\\ 0\leq k_1\leq k_2}}q^{-3s} c_{(k_1, k_2)}\cdot\mathcal{SO}_{\gamma^{(s)}, (k_1, k_2)}=
    \frac{(q^{3}-1)(q^{2}-1)}{q^{6+d_{\gamma}}}(q^{l+1}+2q^{l}+3q^{l-1}+\cdots+3q^{l-\lfloor \frac{l}{3}\rfloor+1}+(l-3\lfloor \frac{l}{3}\rfloor)q^{l-\lfloor \frac{l}{3}\rfloor}).
    \]
\\

The desired formula is the sum of (\ref{sum3})'s for all $l$'s such that $l=d_{\gamma}-3s$ for $0\leq s\leq \lfloor\frac{d_{\gamma}}{3}\rfloor$, by 
 Equation (\ref{eqsor}).
In the following, we do calculation case by case, whether $F_{\chi_{\gamma}}$ is unramified or ramified.

\begin{enumerate}
\item{The case when $F_{\chi_{\gamma}}$ is unramified over $F$;}

For $l=d_{\gamma}-3s=0$, by the inductive formula in Corollary \ref{corred},
\[
\sum\limits_{\substack{k_1+k_2=0;\\ 0\leq k_1\leq k_2}}q^{-d_{\gamma}} c_{(k_1, k_2)}\cdot\mathcal{SO}_{\gamma^{(s)}, (k_1, k_2)}=q^{-d_{\gamma}}\cdot\frac{(q^{2}-1)(q-1)}{q^{3}}=\frac{(q^{3}-1)(q^{2}-1)}{q^{6+d_{\gamma}}}\cdot \frac{q^{3}}{q^{2}+q+1}.
\]
 
Therefore, if we set $d_{\gamma}=3d'$ for $d' \geq 0$ and $l=3i$, then we have
\begin{align*}
    \mathcal{SO}_{\gamma}&=\sum\limits_{\substack{k_1+k_2=3i;\\ 0\leq k_1\leq k_2;\\  0\leq i \leq d'}}q^{3(i-d')} c_{(k_1, k_2)}\cdot\mathcal{SO}_{\gamma^{(s)}, (k_1, k_2)}\\
    &=\frac{(q^{3}-1)(q^{2}-1)}{q^{3d'+6}}(\frac{q^{3}}{q^{2}+q+1}+\sum\limits_{i=1}^{d'}(q^{3i+1}+2q^{3i}+3q^{3i-1}+\cdots+3q^{2i+1}))\\
    &=\frac{(q+1)(q^{2}+q+1)}{q^{3}}-\frac{3(q^{2}+q+1)}{q^{d'+3}}+\frac{3}{q^{3d'+3}}.
\end{align*}
By Proposition \ref{da3}, $S(\gamma)=d_{\gamma}=3d'$.

\textit{ }

\item{The case when $F_{\chi_{\gamma}}$ is ramified over $F$;} 

Recall that $d_{\gamma}\equiv 1\textit{ or }2$ modulo $3$.
\begin{enumerate}
\item 
Suppose that  $d_{\gamma}=3d'+1$ for $d'\geq 0$.
For $l(=d_{\gamma}-3s)=1$,
\[
\sum\limits_{\substack{k_1+k_2=1;\\ 0\leq k_1\leq k_2}}q^{-3d'} c_{(k_1, k_2)}\mathcal{SO}_{\gamma^{(s)}, (k_1, k_2)}=\frac{(q^{3}-1)(q^{2}-1)}{q^{d_{\gamma}+6}}q^{2}.
\]

Therefore, if we set $l=3i+1$, then we have
\begin{align*}
   \mathcal{SO}_{\gamma}&=\sum\limits_{\substack{k_1+k_2=3i+1;\\ 0\leq k_1\leq k_2;\\  0\leq i \leq d'}}q^{3(i-d')} c_{(k_1, k_2)}\mathcal{SO}_{\gamma^{(s)}, (k_1, k_2)}\\
    &=\frac{(q^{3}-1)(q^{2}-1)}{q^{3d'+7}}(q^{2}+\sum\limits_{i=1}^{d'}(q^{3i+2}+2q^{3i+1}+3q^{3i}+\cdots+3q^{2i+2}+q^{2i+1}))\\
    &=\frac{(q+1)(q^{2}+q+1)}{q^{3}}-\frac{(2q+1)(q^{2}+q+1)}{q^{d'+4}}+\frac{q^{2}+q+1}{q^{3d'+5}}.
\end{align*}
By Proposition \ref{da3}, $S(\gamma)=d_{\gamma}-1=3d'$.

\item 
Suppose that $d_{\gamma}=3d'+2$ for $d'\geq 0$. For $l(=d_{\gamma}-3s)=2$,
\[
\sum\limits_{\substack{k_1+k_2=2;\\ 0\leq k_1\leq k_2}}q^{-3d'} c_{(k_1, k_2)}\mathcal{SO}_{\gamma^{(s)}, (k_1, k_2)}=\frac{(q^{3}-1)(q^{2}-1)}{q^{d_{\gamma}+6}}(q^{3}+q^{2}).
\]

Therefore, if we set $l=3i+2$,  then we have
\begin{align*}
    \mathcal{SO}_{\gamma}&=\sum\limits_{\substack{k_1+k_2=3i+2;\\ 0\leq k_1\leq k_2;\\  0\leq i \leq d'}}q^{3(i-d')} c_{(k_1, k_2)}\mathcal{SO}_{\gamma^{(s)}, (k_1, k_2)}\\
    &=\frac{(q^{3}-1)(q^{2}-1)}{q^{3d'+8}}(q^{3}+q^{2}+\sum\limits_{i=1}^{d'}(q^{3i+3}+2q^{3i+2}+3q^{3i+1}+\cdots+3q^{2i+3}+2q^{2i+2}))\\
    &=\frac{(q+1)(q^{2}+q+1)}{q^{3}}-\frac{(q+2)(q^{2}+q+1)}{q^{d'+4}}+\frac{q^{2}+q+1}{q^{3d'+6}}.
\end{align*}
By Proposition \ref{da3}, $S(\gamma)=d_{\gamma}-1=3d'+1$.
\end{enumerate}
\end{enumerate}
\end{proof}

\begin{proof}[The proof of Theorem \ref{result4}]
\textit{}

\begin{enumerate}
\item
Suppose that  $\chi_{\gamma}(x)$ involves an irreducible quadratic polynomial as a factor. 
Then the irreducible quadratic  factor of $\chi_{\gamma}(x)$ is expressed as $\chi_{\gamma'}(x)\in \mfo[x]$ for a certain  $\gamma'\in \mathfrak{gl}_{2}(\mfo)$.
 From the inductive formula in Proposition \ref{cor410}, we have
\[
\mathcal{SO}_{\gamma}=\frac{q^{3}-1}{q^{2}(q-1)}\cdot \mathcal{SO}_{\gamma'}.
\]
Combining it with  the result in Theorem \ref{result2} for $\gamma'\in \mathfrak{gl}_{2}(\mfo)$, we obtain the desired formula.

For the relation between $S(\gamma)$ and $S(\gamma')$, we bring the equation of Section 4.1 in \cite{Yun13}\footnote{Yun uses $\delta_{i}=[\mfo _{F_{\chi_{\gamma_{i}}}}:R_{i}]_{R_i}$ (which is  the relative $R_{i}$-length)
and $d_i=[\kappa_{R_i}:\kappa]$,
where $R_{i}=\mfo[x]/(\chi_{\gamma_{i}}(x))$ and $\kappa_{R_i}$ is the residue field of $R_i$.
Note that $\delta_i$  is denoted by $S_{R_{i}}(\gamma_{i})$ in the paragraph following Remark \ref{rmk89}. 
Since  we use $S(\gamma_{i})$ for the relative $\mfo$-length $[\mfo _{F_{\chi_{\gamma_{i}}}}:R_{i}]$, 
we have $d_{i}\delta_{i}=S(\gamma_{i})$.}, which is described as follows:
\[S(\gamma)-\sum_{i\in B(\gamma)}S(\gamma_{i})=\sum_{\{i,j\}\subset B(\gamma),i\neq j}\ord (Res(\chi_{\gamma_{i}},\chi_{\gamma_{j}})),
\]
where the notations are in Section \ref{sec322} and Theorem \ref{genlb1}.  
By Equation (\ref{eq324}), we have \[S(\gamma)-\sum_{i\in B(\gamma)}S(\gamma_{i})=\frac{1}{2}(\ord (\Delta_{\gamma})-\sum\limits_{i\in B(\gamma)}\ord (\Delta_{\gamma_{i}})).\]
If we apply this equation for $\gamma$ whose characteristic polynomial is the product of an irreducible quadratic polynomial $\chi_{\gamma'}(x)$ and a linear polynomial in $\mfo[x]$, then we have 
\[
S(\gamma)-S(\gamma')=\frac{\ord (\Delta_{\gamma})-\ord (\Delta_{\gamma'})}{2}.
\]

\item For a hyperbolic regular semisimple element $\gamma \in \mathfrak{gl}_{3}(\mfo)$, $\chi_{\gamma}(x)$ is the product of three distinct linear polynomials in $\mfo[x]$. The inductive formula of Proposition \ref{cor410}, together with the result in Section \ref{sec62}, then yields
\[
\mathcal{SO}_{\gamma}=\frac{\# \mathrm{GL}_{3}(\kappa)\cdot q^{-9}}{(\#\mathrm{GL}_{1}(\kappa)\cdot q^{-1})^{3}}\cdot 1^{3} = \frac{(q+1)(q^2 +q+1)}{q^{3}}.
\]
\end{enumerate}
\end{proof}

\end{appendices}

\part{Stable orbital integrals for $\mathfrak{u}_n$ and $\mathfrak{sp}_{2n}$ and smoothening}\label{part2}

\section*{Notations}

We keep using notations from Part \ref{part1}, unless otherwise stated in this section. We introduce the following notations.
\begin{itemize}

\item  Let $E$ be an unramified quadratic field extension of $F$ with $\mathfrak{o}_E$ its ring of integers and $\kappa_E$ its residue field.

\item For an element  $x\in E$, the exponential order of $x$ with respect to the maximal ideal in   $\mfo_E$   is written by $\mathrm{ord}(x)$.

\item 
For an element $x\in E$, the value of $x$ is $|x|_E:=q^{-2\mathrm{ord}(x)}$. 


\item Let $\widetilde{F}$ be either $E$ or $F$ with involution $\sigma$. Here when $\widetilde{F}=E$, we understand $\sigma$ to be the non-trivial automorphism of $E/F$. When $\widetilde{F}=F$, we understand $\sigma$ to be the identity.

\item 
Let $L$ be a free $\mfo_{\widetilde{F}}$-module of  rank $\widetilde{n}$.
Let $\epsilon$ be either $1$ or $-1$.

By a hermitian form $h: L \times L \longrightarrow \mfo_{\widetilde{F}}$, we always mean a hermitian form depending on $(\widetilde{F}, \epsilon)$ which is formulated as follows:
\[
h(av, bw)=\sigma(a)b\cdot  h(v,w), ~~~~~~~~ \textit{     }  ~~~~~   h(w,v)=\epsilon \sigma(h(v,w))
\]
for $a,b\in \mfo_{\widetilde{F}}$ and $v,w\in L$.
We consider the following two specific cases:
\begin{itemize}
    \item $(\widetilde{F}, \epsilon)=(E, 1)$. In this case, $h$ is a usual hermitian form. Let $\widetilde{n}=n$ be the rank of $L$.
    \item $(\widetilde{F}, \epsilon)=(F, -1)$. In this case, $h$ is an alternating form. Let $\widetilde{n}=2n$ be the rank of $L$.
\end{itemize}
 We denote by a pair $(L, h)$ a hermitian lattice.

Let $V$ be an $\widetilde{F}$-vector space $L\otimes_{\mfo} F$ equipped with a hermitian form $h$. We then denote by a pair $(V, h)$ a hermitian space.

\item 
We define the dual lattice of $L$, denoted by $L^{\perp}$, as  
$$L^{\perp}=\{x \in V : h(x, L) \subset \mfof \}.$$
From now on, we assume that $(L, h)$ is unimodular, that is, $L^{\perp}=L$.

\item For a commutative $\mfo$-algebra $A$, let
\[
\left\{
\begin{array}{l}
\textit{$\mathfrak{gl}_{\wn,A}$ be the general linear group of dimension $\wn^2$ defined over $A$};\\
\textit{$\mathfrak{gl}_{\wn,A}$ be the Lie algebra of $\mathfrak{gl}_{\wn,A}$};\\
\textit{$\mathrm{G}_{A}$ be the (smooth) reductive group scheme  defined over $A$ stabilizing $(L\otimes_{\mfo}A, h\otimes 1)$};\\
\textit{$\mathfrak{g}_{A}$ be the Lie algebra of $\mathrm{G}_{A}$};\\
\textit{$\mathbb{A}^{n}_A$ be the affine space of dimension $n$ defined over $A$}.
 \end{array}\right.
\]
\end{itemize}

If there is no confusion then we sometimes omit $\mfo$ in the subscript to express schemes over $\mfo$.
Thus $\mathrm{G}$ and $\mfu$ stand for $\mathrm{G}_{\mfo}$ and $\mfu_{\mfo}$, respectively.
Both $\mathrm{G}$ and $\mathfrak{g}$ are closed subschemes of $\mathrm{Res}_{\mfof/\mfo}(\mathfrak{gl}_{\wn,\mfof})$ and $\mathrm{Res}_{\mfof/\mfo}(\mathfrak{gl}_{\wn,\mfof})$ respectively. Here $\mathrm{Res}$ stands for the Weil restriction.
Note that \[
(\mathrm{G}, \mathfrak{g})=\left\{
\begin{array}{l l}
(\mathrm{U}_n, \mathfrak{u}_n)  & \textit{if $(\widetilde{F}, \epsilon)=(E, 1)$};\\
(\mathrm{Sp}_{2n}, \mathfrak{sp}_{2n})  & \textit{if $(\widetilde{F}, \epsilon)=(F, -1)$}.
 \end{array}\right.
\]

\begin{itemize}
\item For a flat $\mfo$-algebra $R$, the reduction of $f(x)\in R\otimes_{\mfo}\mfof[x]$ modulo $\pi$ is denoted by $\bar{f}(x)\in R\otimes_{\kappa} \kf[x]$. 
The same for $X\in \mathfrak{gl}_{n, \mfof}(R\otimes_{\mfo}\mfof)$ so that $\bar{X}\in \mathfrak{gl}_{n, \kf}(R\otimes_{\kappa} \kf)$.

\item Define $\chi_{\gamma}(x)\in \mfof[x]$ such that
$\chi_{\gamma}(x)$  is the characteristic polynomial of  $\gamma \in \mathfrak{g}(\mfof)$.
Since $    h \gamma + \sigma({}^t\gamma) h = 0$, 
We always write 
\begin{equation}\label{equation:chir}
    \chi_{\gamma}(x)=\left\{
\begin{array}{l l}
x^n+c_1x^{n-1}+\cdots + c_{n-1}x+c_n \textit{ with } c_i\in \mfo_E  & \textit{if $(\widetilde{F}, \epsilon)=(E, 1)$};\\
x^{2n}+c_1x^{2n-2}+c_2x^{2n-4}+\cdots + c_{n-1}x^2+c_n \textit{ with } c_i\in \mfo  & \textit{if $(\widetilde{F}, \epsilon)=(F, -1)$}.
 \end{array}\right.
\end{equation}
Here $c_i + (-1)^{i+1} \sigma(c_i) = 0$ in the first case.

\item Let $d_i:=ord(c_i)$.

\item Let $\Delta_{\gamma}$ be the discriminant of $\chi_\gamma(x)$.
Note that $\Delta_{\gamma}$ is contained in $\mfo$ even when $\mfof=\mfo_E$. 

\item Let $\gamma\in \mfu(\mfo)$ be regular and semisimple.
By a regular and semisimple element, we mean that the identity component of the centralizer of $\gamma$ in $\mathrm{G}_{F}$ via the adjoint action is a maximal torus.
This is equivalent that $\chi_{\gamma}(x)$ is separable over $\widetilde{F}$ (cf. \cite[Section 3.1]{Gor22} or \cite[Sections 6-7]{Gro05}).


\end{itemize}

\begin{remark}\label{remark:reim}
In this remark, we suppose that $\widetilde{F}=E$. 
We have that  
 $\mfo_E\cong \mfo[x]/(p(x))$ for an irreducible quadratic polynomial $p(x)$ such that $\kappa_E\cong \kappa[x]/(\overline{p}(x))$. Let $a, b\in \mfo_E$ be two solutions of $p(x)$ so that $\sigma(a)=b$. Note that $a,b,a-b \in \mfo_E^{\times}$.

For $r\in \mfo_E\otimes_{\mfo}R$ with a commutative $\mfo$-algebra $R$, the real part of $r$, denoted by $Re(r)$, is defined to be 
\[
Re(r)=
r+\sigma(r).
\]
The imaginary part of $r$, denoted by $Im(r)$, is defined to be 
\[
Im(r)=\left\{
\begin{array}{l l}
r-\sigma(r) & \textit{if char($\kappa >2$)};\\
r-Re(r)\cdot \frac{a}{a+b}
 & \textit{if char($\kappa) =2$}.
 \end{array}\right.
\]
Then $Re(r)$ and $Im(r)$ satisfy the following relations:
\[\sigma(Re(r))=Re(r), ~~~~~~~~~~\textit{       } ~~~~~~~\sigma(Im(r))=-Im(r).
\]

Thus 
\[
r=\left\{
\begin{array}{l l}
Re(r)\cdot \frac{1}{2}+Im(r)\cdot \frac{1}{2} & \textit{if char($\kappa) >2$};\\
Re(r)\cdot \frac{a}{a+b}+Im(r)
 & \textit{if char($\kappa) =2$}.
 \end{array}\right.
\]
Here,  $\frac{1}{2} \in \mfo_E^{\times}$ when char($\kappa) >2$ and 
$\frac{a}{a+b}\in \mfo_E^{\times}$ when char($\kappa) =2$.
\end{remark}

\section{Stable orbital integral}\label{section10soi}

In this section, we define the stable orbital integral for a regular semisimple element in $\mfu(\mfo)$ and for the unit element in the Hecke algebra, which is parallel to Section \ref{sec3}.

\subsection{Measure}\label{subsubsec:measure}
We will first construct the Chevalley morphism whose fiber gives a stable orbit in $\mfu(\mfo)$.

\subsubsection{The Chevalley morphism $\varphi_n$ on $\mfu$}\label{subsubsection:chevalley}
When $(\widetilde{F}, \epsilon)=(F, -1)$, the Chevalley morphism of $\mathfrak{g}$ is described in \cite[Theorem 1.2]{CR}  as follows:
\[
    \varphi_{n} : \mathfrak{g} \to \mathbb{A}^n, ~~~~ \textit{    }
    \gamma \mapsto (c_1, \cdots, c_n).
\]
Note that the morphism $\varphi_{n}$ is representable as a morphism of schemes over $\mfo$.

When $(\widetilde{F}, \epsilon)=(E, 1)$, we will construct it  by using  the Chevalley morphism  of $\mathfrak{gl}_{n,\mfo_E}$ over $\mathfrak{o}_E$ with the Galois descent.
The Chevalley morphism of $\mathfrak{gl}_{n,\mfo_E}$ over $\mathfrak{o}_E$ is defined by
\[
    \varphi_{n, \mathfrak{o}_E} : \mathfrak{gl}_{n,\mfo_E} \to \mathbb{A}^n_{\mathfrak{o}_E}, 
   ~~~  \textit{   } \gamma \mapsto \textit{coefficients of $\chi_{\gamma}(x)$}=(c_1, \cdots, c_n).
\]
The morphism $ \varphi_{n, \mathfrak{o}_E}$ is also  representable as a morphism of schemes over $\mfo_E$.

Define the functor $\mathcal{A}_i$, for an integer $i$ with  $1\leq i\leq n$, on the category of commutative $\mfo$-algebras to the category of abelian groups such that 
    \begin{equation}\label{def:mathcala}    
        \mathcal{A}_i(R)=\{c \in \mfo_E \otimes_{\mfo}R  : c + (-1)^{i+1} \sigma(c) = 0\}
    \end{equation}
for an $\mfo$-algebra $R$.
Here $\sigma(a\otimes b):=\sigma(a)\otimes b$.

\begin{lemma} \label{lem:c+sigma(c)_bij}
The functor  $\mathcal{A}_i$   is represented by an affine space  $\mathbb{A}_{\mfo}^1$ over $\mfo$ of dimension $1$ for any integer $i$ with $1\leq i\leq n$.
\end{lemma}

\begin{proof}
It suffices to show that  $\mathcal{A}_i(R)$ is a free $R$-submodule in $\mfo_E \otimes_{\mfo}R$ of rank 1 and that $\mathcal{A}_i$ is functorial.
If there is $\alpha \in \mathfrak{o}_E^\times$ such that $\alpha + \sigma(\alpha) = 0$, then for $c \in \mathcal{A}_i(R)$ we have $\sigma((\alpha^i \otimes 1)c) = (\alpha^i\otimes 1) c$.
In particular, $(\alpha^i\otimes 1) c$ is contained in $R$.
Since $\alpha$ is a unit, we have the following $R$-module isomorphism:
\[
\mathcal{A}_i(R) \to R, \quad c \mapsto (\alpha^i\otimes 1) c.
\]
This completes the proof.

It remains to find such $\alpha$.
In the split case, put $\alpha=(-1, 1)$. 
In the unramified case, choose an eigenvector of $\sigma:E\rightarrow E$ as an $F$-vector space with the eigenvalue $-1$.
Multiplying this eigenvector by a suitable power of $\pi$, we obtain a desired element $\alpha$.
\end{proof}

Define \[
\mathcal{A}^n := \mathcal{A}_1 \times \cdots \times \mathcal{A}_n
\]
which is isomorphic to an affine space $\mathbb{A}^n$.

 For the nontrivial $\sigma \in \operatorname{Gal}(E/F)$, we define its action on $\mathbb{A}^n_{\mfo_E}$ by
 \[
     \sigma \cdot (c_1, \cdots, c_i, \cdots, c_n) = (-c_1, \cdots, (-1)^i c_i, \cdots, (-1)^n c_n).
 \]
 The action of $\sigma$ on $\mathfrak{gl}_{n,\mfo_E}$ is $\sigma\cdot m:= -h^{-1}\sigma({}^tm)h$ so that $\mfu_n$ is fixed by $\sigma$.
Then  $\varphi_{n, \mathfrak{o}_E} : \mathfrak{gl}_{n,\mfo_E} \to \mathbb{A}^n_{\mfo_E}$ is compatible with the Galois action so as to  descend to a morphism defined over $\mfo$.

This is defined to be the Chevalley morphism of  $\mathfrak{u}_{n,\mfo}$, which is denoted by 
 $\varphi_n$.
 For a precise description, we observe that  the scheme $\mathbb{A}^n_{\mfo_E}$ descends to $\mathcal{A}^n_{\mfo}$. 
 The morphism $\varphi_n$ is then described as follows:
 
 \begin{equation}\label{eq:chevalleymap}
    \varphi_n : \mathfrak{u}_{n} \longrightarrow \mathcal{A}^n, ~~~~~~~~~~ \textit{      } ~~~~~~~~~~~~~~
    \gamma \mapsto (c_1, \cdots, c_n).     
 \end{equation}

We denote by $\rho_n$ the generic fiber of $\varphi_n$ so that
 \begin{equation}\label{eq:chevalleymapp}
    \rho_n : \mathfrak{u}_{n,F} \longrightarrow \mathcal{A}^n_{F}, ~~~~~~~~~~ \textit{      } ~~~~~~~~~~~~~~
    \gamma \mapsto (c_1, \cdots, c_n).     
 \end{equation}

\begin{remark}\label{remark:chevnotation}
When $(\widetilde{F}, \epsilon)=(F, -1)$,
we let $\mathcal{A}_i=\mathbb{A}^1$ so that $\mathcal{A}^n=\mathbb{A}^n$.
Thus 
\[
\varphi_n : \mathfrak{g} \longrightarrow \mathcal{A}^n, ~~~~~~~  \textit{  and }  ~~~~ 
   \rho_n : \mathfrak{g}_{F} \longrightarrow \mathcal{A}^n_{F}
\]
stand for the Chevalley morphisms over $\mfo$ and $F$ respectively,  in all cases. 
\end{remark}



\subsubsection{Measure and stable orbital integral}\label{subsubsection:measure}
Put
\[
 \underline{G}_{\gamma}:=\varphi_n^{-1}(\chi_{\gamma}) ~~~~~~   \textit{    and   }   ~~~~~~~~~~~~~~
 G_{\gamma}:=\rho_n^{-1}(\chi_{\gamma}).
\]
\begin{proposition}\label{prop:git}
The set $\underline{G}_{\gamma}(\mfo)$ is the stable conjugacy class of $\gamma$ in $\mfu(\mfo)$.
More precisely, \[\underline{G}_{\gamma}(\mfo)=\{g\gamma g^{-1}\in \mfu(\mfo)|g\in \mathrm{G}_{F}(\bar{F})\}.\]
\end{proposition}
\begin{proof}
It suffices to show that $\mfu_F/\mathrm{G}_F$ is an affine space consisting of coefficients of the characteristic polynomial. 
When $\mfu_F=\mathfrak{gl}_{n,F}$, this is explained in \cite[Remark 2.1.4]{Kac} for an arbitrary field, which  also yields the argument for $\mathrm{u}_{n,F}$. 
When $\mfu_F=\mathrm{sp}_{2n,F}$, this is explained in \cite[Theorem 1.3]{CR} for an arbitrary field.    
\end{proof}

Let $\omega_{\mfu_{\mfo}}$ and $\omega_{\mathcal{A}^n_{\mfo}}$ be nonzero, translation-invariant forms on $\mfu_{F}$ and $\mathcal{A}^n_{F}$, respectively, with normalizations
\[
\int_{\mfu_{\mfo}(\mfo)} |\omega_{\mfu_{\mfo}}| = 1
\text{ and }
\int_{\mathcal{A}^n_{\mfo}(\mfo)} |\omega_{\mathcal{A}^n_{\mfo}}| = 1.
\]
Then we define  $\omega_{\chi_{\gamma}}^{\mathrm{ld}}$ to be $\omega_{\mfu_{\mfo}}/\rho_n^{\ast}\omega_{\mathcal{A}^n_{\mfo}}$ as in Definition \ref{diff1}.




\begin{definition}\label{def:stableorbital}
The stable orbital integral for $\gamma$ and  for the unit element in the Hecke algebra, denoted by $\mathcal{SO}_{\gamma}$, is defined to be 
\[
\mathcal{SO}_{\gamma}=\int_{\uGr(\mfo)}|\omega_{\chi_{\gamma}}^{\mathrm{ld}}|.
\]
Here $\underline{G}_{\gamma}(\mfo)$ is an open subset of $G_{\gamma}(F)$ with respect to the $\pi$-adic topology.
\end{definition}
As in Definition \ref{defsoi}, 
 all of $\mathcal{SO}_{\gamma}$, $\omega_{\chi_{\gamma}}^{\mathrm{ld}}$,  $\uGr$, and $G_{\gamma}$ depend  on the characteristic polynomial $\chi_{\gamma}$, not the element $\gamma$.

\subsection{Transport of hermitian forms}\label{sec:trans_herm}
Define an étale $\widetilde{F}$-algebra 
\[\widetilde{F}_{\chi_{\gamma}} :=  \widetilde{F}[x]/(\chi_{\gamma}(x)).\]
In this subsection, we will identify $\widetilde{F}_{\chi_{\gamma}}$  with $V$ as hermitian  spaces by transporting a hermitian form $h$ defined on $V$ to that defined on $\widetilde{F}_{\chi_{\gamma}}$. This is based on a  reformulation of (2.3.1) and (2.4.1) in \cite{Yun11}.

\subsubsection{Identification of two $\widetilde{F}$-vector spaces  $\widetilde{F}_{\chi_{\gamma}}$ and $V$}\label{sec:field_identification}

Remark that the Lie algebra $\mathfrak{g}(\mfo)$ is defined with respect to a hermitian lattice $(L, h)$. Recall that  $V=L\otimes_{\mfo} F$.
Then the action of $\gamma$ on $V$ yields  the action of $\widetilde{F}_{\chi_{\gamma}}$ on $V$
via the ring isomorphism
    $\widetilde{F}_{\chi_{\gamma}}\cong \widetilde{F}[\gamma]$, $x\mapsto \gamma$.

\begin{proposition}\label{lemma:existenceofe0}
    There exists $e_0\in L$ such that $\{e_{0},\gamma e_{0},\cdots,\gamma^{n-1} e_{0}\}$ is an  $\widetilde{F}$-basis for $V$. 
\end{proposition}
\begin{proof}
We claim that the  complement of the set 
\[
U:=\{v\in V\mid v, \gamma v, \cdots, \gamma^{n-1}v \textit{ are $\widetilde{F}$-linearly dependent}\}
\]
in $V$ is  open and dense in $V$ as a $\pi$-adic topology.
For $v\in U$, we consider the following set $I_{\gamma ,v}=\{\varphi(x)\in \widetilde{F}[x]\mid \varphi(\gamma)(v)=0 \}$. 
It is well-known in linear algebra that $I_{\gamma ,v}$ is a proper ideal of $\widetilde{F}[x]$ with a generator $\mu_{\gamma,v}(x)$ of degree $>0$ dividing the characteristic polynomial $\chi_{\gamma}(x)$, and the kernel of the linear morphism  $\mu_{\gamma,v}(\gamma): V \longrightarrow V, ~~~~~~ w \mapsto \mu_{\gamma,v}(\gamma)(w)$, is an $\widetilde{F}$-subspace of $V$ of dimension $\mathrm{deg}(\mu_{\gamma,v}(x))$, which is less than $n$.
Here the last argument about the kernel of $\mu_{\gamma,v}(\gamma)$ follows from our condition that  $\chi_{\gamma}(x)$ is also a minimal polynomial of $\gamma$.
In particular, we have that 
$v\in \mathrm{Ker}(\mu_{\gamma,v}(\gamma))$ and  that the polynomial  $\mu_{\gamma,v}(x)$ is of degree $<n$ dividing $\chi_{\gamma}(x)$.

Conversely, if $\varphi(x) \in \widetilde{F}[x]$ divides $\chi_{\gamma}(x)$ such that $\mathrm{deg}(\varphi(x))<n$, then the kernel of the linear map $\varphi(\gamma) : V \to V$ is contained in $U$.
Thus we have the following characterization for the set $U$:
\[
U=\bigcup_{\substack{\varphi(x)|\chi_\gamma(x),\ \mathrm{deg}(\varphi(x))<n}} \mathrm{Ker}(\varphi(\gamma)),
\]
where each term $\mathrm{Ker}(\varphi(\gamma))$ is an $\widetilde{F}$-subspace in $V$ of dimension $\mathrm{deg}(\varphi(x))$.
Therefore, the complement of $U$ in $V$ is an open and dense subset of $V$ as a $\pi$-adic topology.
\end{proof}

For $e_0\in L$ as in the above proposition, 
the following  $\widetilde{F}$-linear map turns to be an isomorphism:
\begin{align}\label{eq:def_iota_gamma}
\iota_\gamma : \widetilde{F}_\gamma \to V, \quad x^i \mapsto \gamma^i e_0.
\end{align}

\subsubsection{Transport of the hermitian form $h$ on $V$ to $E_{\chi_{\gamma}}$}\label{sec:transport_herm}

Under the $\widetilde{F}$-isomorphism $\iota_\gamma : \widetilde{F}_{\chi_\gamma} \cong V$ defined in Equation (\ref{eq:def_iota_gamma}), we transport a hermitian form $h$ on $V$ to $(-,-)_\gamma$ on $\widetilde{F}_{\chi_\gamma}$ defined by
\[
(\alpha,\beta)_\gamma := h(\iota_{\gamma}(\alpha), \iota_{\gamma}(\beta)) \left(= h(\alpha(\gamma) e_0, \beta(\gamma) e_0)\right)
\]
for $\alpha, \beta \in \widetilde{F}_{\chi_\gamma}$.
Then we have the isomorphism between hermitian spaces
\begin{align}\label{equation:iota}
    \iota_{\gamma}: (\widetilde{F}_{\chi_{\gamma}}, (-,-)_{\gamma}) \cong (V, h(-,-)), \quad x^i \mapsto \gamma^i e_0.
\end{align}

We extend the involution $\sigma$ to that on $\widetilde{F}_{\chi_{\gamma}}$ such that $\sigma(x)=-x$. 
Note that the action of $\sigma$ on $\widetilde{F}_{\chi_{\gamma}}$ is well-defined by Equation (\ref{equation:chir}).
Then 
we can easily see that the equation  $h(\gamma -,-)+h(-,\gamma -)=0$
 yields
\begin{equation}\label{equation:alphabetasigma}
    (\alpha,\beta)_\gamma = (1,\sigma(\alpha)\beta)_\gamma   ~~~~~   \textit{    for $\alpha, \beta \in \widetilde{F}_{\chi_{\gamma}}$.}   
\end{equation}

\begin{remark}\label{remark:unitarygroupdesc}
Recall that $\mathrm{G}_F$ is the subgroup of $\mathrm{Res}_{\widetilde{F}/F}(\mathrm{Aut}_{\widetilde{F}}(V))$ stabilizing the hermitian form $h$. 
Using the transport     $\iota_{\gamma}: (\widetilde{F}_{\chi_{\gamma}}, (-,-)_{\gamma}) \cong (V, h(-,-))$, 
the group  $\mathrm{G}_F$
can also be viewed as the subgroup of $\mathrm{Res}_{\widetilde{F}/F}(\mathrm{Aut}_{\widetilde{F}}(\widetilde{F}_{\chi_{\gamma}}))$ stabilizing the hermitian form $(-,-)_{\gamma}$.
\end{remark}

\subsection{Factorization of $\widetilde{F}_{\chi_{\gamma}}$}\label{desc_field}
We define an étale $F$-algebra 
\[F^\sigma_{\chi_{\gamma}}:=\left(\widetilde{F}_{\chi_{\gamma}}\right)^\sigma \left(=  \left(\widetilde{F}[x]/(\chi_{\gamma}(x))\right)^\sigma \right).
\]
Here $\sigma$ in the superscript stands for the set of invariant elements under $\sigma$.
\begin{itemize}
    \item  If $(\widetilde{F}, \epsilon)=(E, 1)$, then there is $\alpha \in \mathfrak{o}_E^\times$ such that $\alpha + \sigma(\alpha) = 0$ by the proof of Lemma \ref{lem:c+sigma(c)_bij}.
Then it is easy to see that $\chi_{\alpha \gamma}(x)\in \mfo[x]$ and that $\mfo_E[x]/(\chi_{\gamma}(x)) \cong  \mfo_E[x]/(\chi_{\alpha\gamma}(x))$ for $x \mapsto x/\alpha$, where $\chi_{\alpha \gamma}(x) (= \alpha^n\cdot \chi_{\gamma}(x/\alpha))$ is the characteristic polynomial of $\alpha\gamma$. 
Let \[
\textit{$\psi(x):=\chi_{\alpha\gamma}(x) \in \mfo[x]$ so that $\sigma(x)=x$ in $\mfo_E[x]/(\psi(x))$.} 
\]

\item If $(\widetilde{F}, \epsilon)=(F, -1)$, then there exists $\psi(x)\in \mfo[x]$ such that $\chi_{\gamma}(x)=\psi(x^2)$.
Note that $\sigma(x)=-x$ in $\mfo[x]/(\chi_{\gamma}(x))$.
\end{itemize}

Let
\begin{equation}\label{equation:yxx2}
    y=\left\{
\begin{array}{l l}
x \textit{ so that } \psi(y)=\chi_{\alpha \gamma}(x)  & \textit{if $(\widetilde{F}, \epsilon)=(E, 1)$};\\
x^2 \textit{ so that } \psi(y)=\chi_{\gamma}(x) & \textit{if $(\widetilde{F}, \epsilon)=(F, -1)$}.
 \end{array}\right.
\end{equation}
Note that the degree of $\psi(y) (\in \mfo[y])$ is $n$  and that $\sigma(y)=y$ in both cases. Then we have 
\begin{equation}\label{equation:F_gamma_iso} 
F_{\chi_{\gamma}}^\sigma = F[y]/(\psi(y)).  
\end{equation}

Write $\psi(y)=\prod\limits_{i \in B(\gamma)} \psi_i(y)$, where $\psi_i(y)\in \mfo[y]$ is irreducible over $F$.
Let  $B(\gamma)^{irred}$ be the subset of $B(\gamma)$ such that 
$\psi_i(y)$ is irreducible in $\mfo_{\widetilde{F}}[x]$, equivalently
\[
\left\{
\begin{array}{l l}
\psi_i(x) \textit{ is irreducible in $\mfo_E[x]$}  & \textit{if $(\widetilde{F}, \epsilon)=(E, 1)$};\\
\psi_i(x^2) \textit{ is irreducible in $\mfo[x]$}  & \textit{if $(\widetilde{F}, \epsilon)=(F, -1)$}.
 \end{array}\right.
\]
Let  $B(\gamma)^{split}$ be the complement  of $B(\gamma)^{irred}$.
Then we have
\begin{equation}\label{equation:irrfactor}
F_{\chi_{\gamma}}^\sigma \cong \prod_{i \in B(\gamma)} F[y]/(\psi_i(y))
= \prod_{i \in B(\gamma)^{irred}} F[y]/(\psi_i(y)) \times \prod_{i \in B(\gamma)^{split}} F[y]/(\psi_i(y)).
\end{equation}

\begin{definition}\label{def:echigamma}
    For each $i\in B(\gamma)$, we define
\[
F_i:=F[y]/(\psi_i(y)) ~~~~~~~~~ \textit{    and    } ~~~~~~~~~~~~~~~~~
  \widetilde{F}_i:= \widetilde{F}[x]/(\psi_i(y)) ~~~~~~~~~~~~~~~~~~~~~ \textit{so that}\]
\[F^\sigma_{\chi_{\gamma}} \cong \prod\limits_{i \in B(\gamma)} F_i ~~~~~~~~~~~ 
\textit{    and      } ~~~~~~~~~~~~~~~~  \widetilde{F}_{\chi_{\gamma}} \cong \prod\limits_{i \in B(\gamma)} \widetilde{F}_i=
\prod\limits_{i \in B(\gamma)^{irred}} \widetilde{F}_i \times \prod\limits_{i \in B(\gamma)^{split}} \widetilde{F}_i.
\]
Then $\sigma$ acts on each $\widetilde{F}_i$ such that $\left(\widetilde{F}_i\right)^{\sigma}=F_i$. 
\end{definition}

\begin{remark}\label{remark:fitilde}
Note that $\widetilde{F}_i/F_i$ is an étale extension of degree $2$.
This is  a  field extension if and only if $i\in  B(\gamma)^{irred}$.
\begin{enumerate}
    \item Suppose that  $i\in  B(\gamma)^{irred}$. If $(\widetilde{F},\epsilon) = (E,1)$, then $\widetilde{F}_i/F_i$ is always an unramified field extension. If $(\widetilde{F},\epsilon) = (F,-1)$, then $\widetilde{F}_i/F_i$ is either ramified or unramified.

\item Suppose that  $i\in  B(\gamma)^{split}$. Then
$\widetilde{F}_i\cong \widetilde{F}_i^1\oplus \widetilde{F}_i^2$ such that $\widetilde{F}_i^1\cong \widetilde{F}_i^2 \cong F_i$.
Since $\sigma$ acts on $\widetilde{F}_i$, the involution $\sigma$ switches $\widetilde{F}_i^1$ and  $\widetilde{F}_i^2$. 
In particular,  $\psi_i(y)=\psi_i^1(x)\psi_i^2(x)$ such that $(\sigma(\psi_i^1(x)))=(\psi_i^2(x))$ as ideals in $\mfo_{\widetilde{F}}[x]$.
Thus the degree of $\psi_i(y)$ in $\mfo_{\widetilde{F}}[x]$ is even.
\end{enumerate}
\end{remark}

\subsection{Explicit description of $\mathrm{T}_{\gamma, F}$ and $\mathrm{T}_c$}\label{sec:cent}
Let     $\mathrm{T}_{\gamma, F}$  be the centralizer of $\gamma$  via the adjoint action of $\mathrm{G}_{F}$ and let 
    $\mathrm{T}_c$  be the maximal compact open subgroup of $\mathrm{T}_{\gamma, F}(F)$.
In this subsection, we will describe $\mathrm{T}_{\gamma, F}$ and $\mathrm{T}_c$ explicitly.

Since $\sigma$ acts on each $\widetilde{F}_i$ fixing $F_i$ (cf. Definition \ref{def:echigamma}), we can define the norm map
\[
N_{\widetilde{F}_i/F_i}: \mathrm{Res}_{\widetilde{F}_i/F_i}\mathbb{G}_m \longrightarrow \mathbb{G}_{m, F_i}, ~~~~~~ g \mapsto g\cdot \sigma(g). 
\]
Here  $\widetilde{F}_i/F_i$ is an étale extension of degree $2$ (cf. Remark \ref{remark:fitilde}). 
We denote by  $N^1_{\widetilde{F}_i/F_i}$ the kernel of  $N_{\widetilde{F}_i/F_i}$.
Then $N^1_{\widetilde{F}_i/F_i}$ is a split torus of dimension $1$  if $i\in  B(\gamma)^{split}$  and a non-split torus of dimension $1$ if $i\in  B(\gamma)^{irred}$.

\begin{lemma}\label{lem: Cent_unitary}
For $F^\sigma_{\chi_{\gamma}} \cong \prod\limits_{i \in B(\gamma)} F_i$ (cf. Definition \ref{def:echigamma}), we have
\[
\mathrm{T}_{\gamma, F} \cong\prod_{i \in B(\gamma)} \mathrm{T}_{\gamma,i, F} ~~~~~~~~~ \textit{    where   } ~~~~~~~~ \mathrm{T}_{\gamma,i, F} \cong \operatorname{Res}_{F_i/F}N^1_{\widetilde{F}_i/F_i}.
\]
\end{lemma}
\begin{proof}
By Remark  \ref{remark:unitarygroupdesc}, the group $\mathrm{G}_F$
is  the subgroup of $\mathrm{Res}_{\widetilde{F}/F}(\mathrm{Aut}_{\widetilde{F}}(\widetilde{F}_{\chi_{\gamma}}))$ stabilizing the hermitian form $(-,-)_{\gamma}$.
The centralizer of $\gamma$ in $\mathrm{Res}_{\widetilde{F}/F}(\mathrm{Aut}_{\widetilde{F}}(\widetilde{F}_{\chi_{\gamma}}))$ is $\prod\limits_{i \in B(\gamma)} \mathrm{Res}_{\widetilde{F}_i/F}\mathbb{G}_m$ by Section \ref{sec322} or \cite[Section 3.2.4]{Yun16}. Thus we have 
\[
\mathrm{T}_{\gamma, F} \cong\prod_{i \in B(\gamma)} \mathrm{Res}_{\widetilde{F}_i/F}\mathbb{G}_m \cap \mathrm{G}_{F}. 
\]
Remark \ref{remark:unitarygroupdesc} yields the following characterization: 
\[
\mathrm{Res}_{\widetilde{F}_i/F}\mathbb{G}_m \cap \mathrm{G}_{F}=\{g\in \mathrm{Res}_{\widetilde{F}_i/F}\mathbb{G}_m | (g-,g-)_{\gamma}=(-,-)_{\gamma}\}. 
\]
Since $(g-,g-)_{\gamma}=(-,\sigma(g)g-)_{\gamma}$ by Equation (\ref{equation:alphabetasigma}), we have
$$\mathrm{Res}_{\widetilde{F}_i/F}\mathbb{G}_m \cap \mathrm{G}_{F}=\{g\in \mathrm{Res}_{\widetilde{F}_i/F}\mathbb{G}_m | g\cdot \sigma(g)=1\}.$$
This completes the proof.
\end{proof}

The lemma shows that the centralizer $\mathrm{T}_{\gamma, F}$ is connected. 
To normalize the measure on $\mathrm{T}_{\gamma, F}(F)$, we need to describe 
the \textit{ft-Néron model}  $\underline{\mathrm{T}}_{\gamma}$ of $\mathrm{T}_{\gamma, F}$, which is defined to be the smooth integral model defined over $\mfo$  such that $\underline{\mathrm{T}}_\gamma(\mathfrak{o})$ is the maximal compact subgroup of $\mathrm{T}_{\gamma, F}(F)$
in \cite[Corllary 2.10.11 and Proposition B.7.2]{KP23}.
Consider the norm map
\[
N_{\mathfrak{o}_{\widetilde{F}_i}/\mathfrak{o}_{F_i}}: \mathrm{Res}_{\mathfrak{o}_{\widetilde{F}_i}/\mathfrak{o}_{F_i}}\mathbb{G}_m \longrightarrow \mathbb{G}_{m, \mathfrak{o}_{F_i}}, ~~~~~~ g \mapsto g\cdot \sigma(g)
\]
defined over $\mathfrak{o}_{F_i}$ with the kernel $N^1_{\mathfrak{o}_{\widetilde{F}_i}/\mathfrak{o}_{F_i}}$.
Here if $\widetilde{F}_i/F_i$ is split (cf. Remark \ref{remark:fitilde}), then $\mathfrak{o}_{\widetilde{F}_i}=\mathfrak{o}_{F_i}\oplus \mathfrak{o}_{F_i}$.
Then $\prod\limits_{i \in B(\gamma)}N^1_{\mathfrak{o}_{\widetilde{F}_i}/\mathfrak{o}_{F_i}}(\mathfrak{o}_{F_i})$ is the maximal compact subgroup of $\mathrm{T}_{\gamma, F}(F)$ by Lemma \ref{lem: Cent_unitary}.

The scheme $N^1_{\mathfrak{o}_{\widetilde{F}_i}/\mathfrak{o}_{F_i}}$ is smooth whenever $\widetilde{F}_i/F_i$ is split or unramified, or whenever $\widetilde{F}_i/F_i$ is ramified with $char(\kappa) > 2$ by \cite[Example B.4.11]{KP23}.
When $\widetilde{F}_i/F_i$ is ramified with $char(\kappa) = 2$, there does exist the desired smooth integral model, denoted by $\left(N^1_{\mathfrak{o}_{\widetilde{F}_i}/\mathfrak{o}_{F_i}}\right)^{sm}$, which is constructed by the smoothening process of \cite[Definition A.6.3]{KP23} applied to $N^1_{\mathfrak{o}_{\widetilde{F}_i}/\mathfrak{o}_{F_i}}$ (cf. \cite[Definition B.7.1]{KP23}). We summarize this discussion as the following corollary:

\begin{corollary}\label{corollary:maxopencpt}
 The ft-Néron model $\underline{\mathrm{T}}_{\gamma}$ of $\mathrm{T}_{\gamma, F}$ is  
     $\underline{\mathrm{T}}_{\gamma} \cong \prod\limits_{i \in B(\gamma)} \underline{\mathrm{T}}_{\gamma,i}$ such that 
     \[
         \underline{\mathrm{T}}_{\gamma,i}  \cong
         \left\{
\begin{array}{l l}
\operatorname{Res}_{\mathfrak{o}_{F_i}/\mathfrak{o}}\left(N^1_{\mathfrak{o}_{\widetilde{F}_i}/\mathfrak{o}_{F_i}}\right)^{sm}  & \textit{if $\widetilde{F}_i/F_i$ is ramified with $char(\kappa) = 2$};\\
\operatorname{Res}_{\mathfrak{o}_{F_i}/\mathfrak{o}}N^1_{\mathfrak{o}_{\widetilde{F}_i}/\mathfrak{o}_{F_i}}  & \textit{otherwise}.
 \end{array}\right.
     \]
Here $\prod\limits_{i \in B(\gamma)}N^1_{\mathfrak{o}_{\widetilde{F}_i}/\mathfrak{o}_{F_i}}(\mathfrak{o}_{F_i})$ is the maximal compact subgroup of $\mathrm{T}_{\gamma, F}(F)$.
\end{corollary}

We now compute  $\#\underline{\mathrm{T}}_{\gamma}(\kappa)$ explicitly.
By Corollary \ref{corollary:maxopencpt}, we may and do work with the case that  $F_{\chi_\gamma}^\sigma/F$ is a finite field extension so that 
$\underline{\mathrm{T}}_{\gamma} = \operatorname{Res}_{\mathfrak{o}_{F_{\chi_\gamma}^\sigma}/\mathfrak{o}} N^1_{\mathfrak{o}_{\widetilde{F}_{\chi_{\gamma}}}/\mathfrak{o}_{F_{\chi_\gamma}^\sigma}}.$

\begin{corollary}\label{lem:pts_spfib_tori}
Suppose that $F_{\chi_\gamma}^\sigma/F$ is a finite field extension of degree $n$.
Let  $K$ be the maximal unramified extension of $F$ contained in $F_{\chi_\gamma}^\sigma$ and let $d=[K:F]$.
Then we have
    \[
    \#\underline{\mathrm{T}}_{\gamma}(\kappa)=
    \begin{cases}
        q^n-q^{n-d} &\textit{if } \widetilde{F}_{\chi_\gamma}/F_{\chi_\gamma}^\sigma \textit{ is split}; \\
        q^n + q^{n-d}&\textit{if } \widetilde{F}_{\chi_\gamma}/F_{\chi_\gamma}^\sigma \text{ is unramified}; \\
       2q^{n} &\textit{if } \widetilde{F}_{\chi_\gamma}/F_{\chi_\gamma}^\sigma \textit{ is ramified and } char(F) > 2.
    \end{cases}
    \]
\end{corollary}
\begin{proof}
    Since $F_{\chi_\gamma}^\sigma/K$ is totally ramified of degree $n/d$, we may write $\mathfrak{o}_{F_{\chi_\gamma}^\sigma} \cong \mathfrak{o}_K[x]/(f(x))$ where $f(x)$ is an Eisenstein polynomial of degree $n/d$.
    It follows that 
    $$\mathfrak{o}_{F_{\chi_\gamma}^\sigma} \otimes_{\mfo} \kappa \cong \mathfrak{o}_{K}[x]/(f(x)) \otimes_{\mfo_K} \kappa_{F_{\chi_\gamma}^\sigma} \cong \kappa_{F_{\chi_\gamma}^\sigma}[x]/(x^{n/d})$$.
\begin{enumerate}
    \item     Firstly suppose that $\widetilde{F}_{\chi_\gamma}/F_{\chi_\gamma}^\sigma$ is split.
    Then $N^1_{\mathfrak{o}_{\widetilde{F}_{\chi_\gamma}}/\mathfrak{o}_{F_{\chi_\gamma}^\sigma}} \cong \mathbb{G}_m$ so that $    \underline{\mathrm{T}}_{\gamma} \cong \operatorname{Res}_{\mathfrak{o}_{F_{\chi_\gamma}^\sigma}/\mathfrak{o}_F} \mathbb{G}_m$.
Thus we have 
    \[
    \#\underline{\mathrm{T}}_{\gamma}(\kappa)
    = \#\left(\frac{\kappa_{F_{\chi_\gamma}^\sigma}[x]}{(x^{n/d})}\right)^\times
    =(q^d-1)q^{d(n/d-1)}= q^{n-d}(q^d-1)=q^n-q^{n-d}.
    \]
    
\item    Suppose that $\widetilde{F}_{\chi_\gamma}/F_{\chi_\gamma}^\sigma$ is unramified.
We use the following notation in this proof exclusively:
\[
\widetilde{\kappa}_{\widetilde{F}}:= \textit{the residue field of }\mathfrak{o}_{\widetilde{F}_{\chi_\gamma}}.
\]
Note that the residue field of $\mathfrak{o}_{F_{\chi_\gamma}^\sigma}$ is $\kappa_{F_{\chi_\gamma}^\sigma}$ and that $\widetilde{\kappa}_{\widetilde{F}}/\kappa_{F_{\chi_\gamma}^\sigma}$ is a quadratic field extension.
Then $\underline{\mathrm{T}}_{\gamma}(\kappa)$ is the kernel of the following norm map:
\[
N_{\widetilde{\kappa}_{\widetilde{F}}/\kappa_{F_{\chi_\gamma}^\sigma}}: \left(\widetilde{\kappa}_{\widetilde{F}}[x]/(x^{n/d})\right)^{\times} \longrightarrow \left(\kappa_{F_{\chi_\gamma}^\sigma}[x]/(x^{n/d})\right)^{\times}, ~~~~~~~~~ ax^i\longmapsto  N_{\widetilde{\kappa}_{\widetilde{F}}/\kappa_{F_{\chi_\gamma}^\sigma}}(a)x^i.
\]
Since the norm map on a finite field extension is surjective, the above map is surjective as well.
Therefore, 
\[
 \#\underline{\mathrm{T}}_{\gamma}(\kappa)=\frac{(q^{2d}-1)q^{2d(n/d-1)}}{(q^d-1)q^{d(n/d-1)}}=(q^d+1)q^{d(n/d-1)}=q^n+q^{n-d}.
\]


\item Suppose that $\widetilde{F}_{\chi_\gamma}/F_{\chi_\gamma}^\sigma$ is ramified and that $char(F)\neq 2$. 
We may choose uniformizers $\widetilde{\varpi} \in \widetilde{F}_{\chi_\gamma}$ and  $\varpi \in F_{\chi_\gamma}^\sigma$ such that  the minimal polynomial of $\widetilde{\varpi}$ over $\mathfrak{o}_{\chi_\gamma}$ is $y^2 - a\varpi y + \varpi$ for some $a \in \mathfrak{o}_{F_{\chi_\gamma}^\sigma}$. Note that  the  different ideal  is  $(\widetilde{\varpi}^r) = (2\widetilde{\varpi} - a \varpi)$ by \cite[Proposition III.2.4]{Neu}.
By \cite[Lemma B.1]{GHY}\footnote{\cite{GHY} assumes the restriction $char(F)=0$.
But \cite[Lemma B.1]{GHY} holds whenever $char(F)\neq 2$.},  $\left(N^1_{\mathfrak{o}_{\widetilde{F}_i}/\mathfrak{o}_{F_i}}\right)^{sm}= \mathrm{Spec} ~\mathfrak{o}_{F_{\chi_\gamma}^\sigma}[y,z]/(g(y,z))$
where
\[
g(y,z)
=
\begin{cases}    
    y^{2} + 2\varpi^{-(r-1)/2}y + \varpi(a\varpi^{-(r-1)/2} z + a yz + z^2) &\textit{if } r \textit{ is odd}; \\
    z^2 + a \varpi^{-(r-2)/2} z + \varpi(y^{2} + 2 \varpi^{-r/2} y +  a yz) &\textit{if } r $\textit{ is even}$.
\end{cases}
\]
We claim that  the coefficients $2\varpi^{-(r-1)/2}$ and $a \varpi^{-(r-2)/2}$ of $y$ and $z$ in each case are units.
We observe that  $\mathrm{ord}_{\widetilde{\varpi}} (2 \widetilde{\varpi})$ is odd and  $\mathrm{ord}_{\widetilde{\varpi}}(a \varpi)$ is even since $2, a, \varpi \in \mathfrak{o}_{\chi_{\gamma}}$.
Thus $r$ is odd if and only if $(\widetilde{\varpi}^r) = (2 \widetilde{\varpi})$, and $r$ is even if and only if $(\widetilde{\varpi}^r) = (a \varpi) = (a \widetilde{\varpi}^2)$.
In particular, $(2) = (\varpi^{(r-1)/2})$ if $r$ is odd, and $(a) = (\varpi^{(r-2)/2})$ if $r$ is even.
Note that $\mathrm{ord}_{\widetilde{\varpi}}(2) = 0$ and $r = 1$ if $char(\kappa) >  2$.


The proof for the case when $r$ is even is identical to that for the case when $r$ is odd (by exchanging the roles of $y$ and $z$) and thus we will treat the case of $r$  odd.
To ease  notations, we rewrite 
\[
g(y,z) = y^2 + u y + \varpi b(y,z) \in \mathfrak{o}_{F_{\chi_\gamma}^\sigma}[y,z]
\]
with $u \in \mathfrak{o}_{F_{\chi_\gamma}^\sigma}^\times$ and $b(y,z) \in \mathfrak{o}_{F_{\chi_\gamma}^\sigma}[y,z]$.
If we identify $\mathfrak{o}_{F_{\chi_\gamma}^\sigma} \cong \mathfrak{o}_K[x]/(f(x))$ then we may write $\varpi = \varpi(x)$, $u = u(x)$, and $b(y,z) = b(x,y,z)$ in $\mathfrak{o}_K[x]/(f(x))$.
\cite[Lemma A.3.12]{KP23} yields that
\[
\left(\mathrm{Res}_{\mathfrak{o}_{F_{\chi_\gamma}^\sigma/\mathfrak{o}}}\left(N^{1}_{\mathfrak{o}_{\widetilde{F}_{\chi_\gamma}}/\mathfrak{o}_{F_{\chi_\gamma}^\sigma}}\right)^{sm}\right)
\otimes_\mathfrak{o} \kappa
\cong
\operatorname{Res}_{\frac{\kappa_{F_{\chi_\gamma}^\sigma}[x]}{(x^{n/d})}/\kappa}\left(
    \mathrm{Spec} \frac{\kappa_{F_{\chi_\gamma}^\sigma}[x,y,z]}{(x^{n/d},\overline{g}(x,y,z))} \right)
\]
where
\[
\overline{g}(x,y,z) = y^{2} + \overline{u}(x)y +
\overline{\varpi}(x)(\overline{b}(x,y,z)).
\]
Here we emphasize that $\overline{\varpi}(x), \overline{u}(x), \overline{b}(x,y,z) \in \kappa_{F_{\chi_\gamma}^\sigma}[x,y,z]$ are reductions of $\varpi(x), u(x), b(x,y,z)$ modulo $\pi$, not modulo $\varpi$.
Then $\underline{\mathrm{T}}_\gamma(\kappa) = \mathrm{Res}_{\mathfrak{o}_{F_{\chi_\gamma}^\sigma/\mathfrak{o}}}\left(N^{1}_{\mathfrak{o}_{\widetilde{F}_{\chi_\gamma}}/\mathfrak{o}_{F_{\chi_\gamma}^\sigma}}\right)^{sm}
(\kappa)$ consists of the set of pairs $(y(x), z(x))$ inside $\kappa_{F_{\chi_\gamma}^\sigma}[x]/(x^{n/d})$ satisfying the equation
\begin{equation}\label{eq:count_tori_ramified}
y(x)^{2} + \overline{u}(x)y(x) +
\overline{\varpi}(x)(\overline{b}(x,y(x),z(x)))= 0  \textit{ in } \kappa_{F_{\chi_\gamma}^\sigma}[x]/(x^{n/d}).
\end{equation}
We fix $z(x)$ arbitrarily in $\kappa_{F_{\chi_\gamma}^\sigma}[x]/(x^{n/d})$.
Then we claim that there are exactly two choices of $y(x)$ satisfying the above equation.
The claim directly yields
\[
\#\underline{\mathrm{T}}_\gamma(\kappa)
= 2 (q^d)^{n/d} = 2q^n.
\]

Since $u(x)$ is a unit and $\varpi(x)$ is  a unifomizer in $\mathfrak{o}_K[x]/(f(x))$, we have that $\overline{u}(0)$ is a unit in $\kappa_{F_{\chi_\gamma}^\sigma}$ and that $\overline{\varpi}(x)$ is divisible  by $x$.
If we write $y(x) = \sum\limits_{i=0}^{n/d-1} y_i x^i$ with $y_i \in \kappa_{F_{\chi_\gamma}^\sigma}$, then comparing coefficients of $x^i$ in Equation (\ref{eq:count_tori_ramified}) yields the following $n/d$-equations
\[
\begin{cases}
y_0^2 + \overline{u}(0) y_0 =0 &\textit{when }i=0;\\
(2y_0 + \overline{u}(0))y_i + p(y_0, \cdots, y_{i-1})= c_i &\textit{when } 0< i < n/d
\end{cases}
\]
where $p(y_0, \cdots, y_{i-1})\in \kappa_{F_{\chi_\gamma}^\sigma}[y_0, \cdots, y_{i-1}]$ and $c_i \in \kappa_{F_{\chi_\gamma}^\sigma}$.
The equation for $i = 0$ yields that there are exactly two choices for $y_0$, either zero or $-\overline{u}(0)$.
It follows that $2y_0 + \overline{u}(0)$ is always a unit. 
Thus the rest equations determine $y_i$'s with $i>0$ uniquely.
\end{enumerate}
\end{proof}

\subsection{Comparison of two normalizations}\label{section:comparison}
As in Section \ref{sscotn}, it is necessary to compare the stable orbital integral $\sog$ in Definition \ref{def:stableorbital} and the stable orbital integral which can be found in the literature. 
In this subsection, we will precisely compare  them, following Section \ref{sscotn} and argument of \cite[Section 3.4.3]{Gor22}. 

\subsubsection{Orbital integral with respect to the quotient measure}\label{subsubsection331} 
Our normalization follows \cite[Section 3.2]{Yun16}, which will be used in 
Proposition \ref{orb:lattice}. 
Let
\[
\left\{
\begin{array}{l}
    dt \textit{ be the Haar measure on } \mathrm{T}_{\gamma,F}(F) \textit{ such that } \operatorname{vol}(dt, \mathrm{T}_c) = 1; \\
    dg \textit{ be the Haar measure on } \mathrm{G}(F) \textit{ such that } \operatorname{vol}(dg, \mathrm{G}(\mathfrak{o})) = 1; \\
    d\mu = \frac{dg}{dt} \textit{ be the quotient measure on } \mathrm{T}_{\gamma,F}(F)\backslash \mathrm{G
    }(F).
\end{array}
\right.
\]
The orbital integral with respect to the quotient measure $d\mu$ on $\mathrm{T}_{\gamma,F}(F) \backslash \mathrm{G}(F)$ is defined by
\begin{equation}\label{equation:rationaloi}
        \mathcal{O}^\mathrm{G}_{\gamma,d\mu} = \int_{\mathrm{T}_{\gamma,F}(F) \backslash \mathrm{G}(F)} \mathbbm{1}_{\mathfrak{g}(\mathfrak{o})}(g^{-1} \gamma g) d\mu,
\end{equation}

where $\mathbbm{1}_{\mathfrak{g}(\mathfrak{o})}$ is the characteristic function of $\mathfrak{g}(\mathfrak{o})\subset \mathfrak{g}(F)$.
If there is no confusion then we sometimes omit $\mathrm{G}$ in the superscript to express the orbital integral so that $\mathcal{O}_{\gamma,d\mu}$ stands for $\mathcal{O}^{\mathrm{G}}_{\gamma,d\mu}$.

\subsubsection{Stable orbital integral with respect to the quotient measure}\label{subsubsection331111} 

In this subsection, we define the stable orbital integral for an unramified reductive group, denoted by $\mathcal{SO}^\mathrm{G}_{\gamma,d\mu}$, with respect to the quotient measure $d\mu$ following \cite[Section 3.4.5]{Yun16}.

\begin{enumerate}
   \item 
For $\gamma'$ which is stably conjugate to $\gamma$, we have a canonical isomorphism  $\iota: \mathrm{T}_{\gamma,F}\cong \mathrm{T}_{\gamma',F}$  by conjugation as a morphism of algebraic groups over $F$ (cf. \cite[Section 3.4.5]{Yun16}).
Let $dt'$ be the Haar measure on $\mathrm{T}_{\gamma',F}(F)$ transferred from $dt$ along the isomorphism $\iota$ and let $\mathrm{T}_c'$ be the maximal compact open subgroup of $\mathrm{T}_{\gamma',F}(F)$.
By uniqueness of the maximal compact open subgroup of a torus, the image of $\mathrm{T}_c$ under $\iota$ is $\mathrm{T}_c'$. Therefore 
$\operatorname{vol}(dt', \mathrm{T}'_c) = 1$.

\item 
We denote by $d\mu'$ the quotient measure $\frac{dg}{dt'}$ on $ \mathrm{T}_{\gamma',F}(F)\backslash \mathrm{G}(F)$ and by $\mathcal{O}^\mathrm{G}_{\gamma',d\mu'}$  the associated orbital integral as in Equation (\ref{equation:rationaloi}).
The stable orbital integral for $\gamma$ with respect to the quotient measure $d\mu$ is defined to be
\[
    \mathcal{SO}^\mathrm{G}_{\gamma, d\mu}
    := \sum_{\gamma' \sim \gamma} \mathcal{O}^\mathrm{G}_{\gamma',d\mu'},
\]
where the sum runs over the set of representatives $\gamma'$ of conjugacy classes within the stable conjugacy class of $\gamma$.
This index set is bijective with
$
\operatorname{ker}(H^1(F,\mathrm{T}_{\gamma,F}) \to H^1(F,\mathrm{G}_F))
$ by \cite[Section 3.4.1]{Yun16}. We sometimes omit $\mathrm{G}$ in the superscript so that $\mathcal{SO}_{\gamma, d\mu}$ stands for $\mathcal{SO}^{\mathrm{G}}_{\gamma, d\mu}$, if this does not cause confusion.
\end{enumerate}

\begin{remark}\label{remark:sordmu}
 Sometimes it is convenient to use the \textit{normalized orbital integral} $|D^\mathrm{G}(\gamma)|^{1/2}\mathcal{O}^\mathrm{G}_{\gamma,d\mu}$, where $D^\mathrm{G}(\gamma)$ is the \textit{Weyl discriminant} defined by $\det(\operatorname{ad}(\gamma) : \mathfrak{g}_F/\mathfrak{t}_{\gamma,F} \to \mathfrak{g}_F/\mathfrak{t}_{\gamma,F})$.
Here $\mathfrak{t}_\gamma$ is the Lie algebra of $\mathrm{T}_{\gamma,F}$.
The description of $D^{\mathrm{G}}(\gamma)$ in \cite[Section 7.4]{Kot05} yields that it is preserved under base change and stable conjugation.

\cite[Equation (20)]{Gor22}\footnote{It seems to us that  the formula of the Weyl discriminant for $\mathfrak{sp}_{2n}$ in \cite[Example 3.6]{Gor22} contains a minor error.} gives a description for $D^\mathrm{G}(\gamma)$ in terms of roots of $(\mathrm{G}_F,\mathrm{T}_{\gamma,F})$.
Using \cite[Section 12.1]{Hum72} for a precise description of roots, we have
\[
|D^{\mathrm{G}}(\gamma)|
=
\begin{cases}
    |\Delta_{\gamma}| &\textit{if } (\widetilde{F},\epsilon) = (E,1); \\
    |\Delta_{\gamma}|/|\Delta_{\psi}|  &\textit{if } (\widetilde{F},\epsilon) = (F,-1)
\end{cases}
\]
where $\Delta_\psi$ is the discriminant of the polynomial $\psi(x)$  defined in Section \ref{desc_field}. Here, we recall that $\Delta_\gamma \in \mfo$ (cf. Notations of Part 2).
Thus $|D^{\mathrm{G}}(\gamma)|$ is determined by coefficients of $\chi_{\gamma}(x)$.
Note that $|\Delta_{\gamma}| = |\Delta_{\psi}|^2 \cdot |\chi_{\gamma}(0)| \cdot |2^{2n}|$ if $(\widetilde{F},\epsilon) = (F,-1)$.
\end{remark}

\subsubsection{Comparison}
A comparison between two measures defining  $\mathcal{SO}_{\gamma}$ and $\mathcal{SO}_{\gamma,d\mu}$ is basically explained in \cite[Proposition 3.29]{FLN}.
In loc. cit., they treat quasi-split reductive groups and \cite[Proposition 3.9]{Gor22} treats split reductive Lie algebras.
Since we need to treat  quasi-split reductive Lie algebras and our normalizations are different from these two, we will explicitly prove our comparison result.
We emphasize that  our argument is  parallel to those of \cite[Section 3.5]{FLN} and \cite[Section 3.4.3]{Gor22}.

\begin{proposition}\label{prop:comparison_measures} 
Suppose\footnote{The assumption is necessary in Step (5) `\textit{Comparison between $\omega_{\mathfrak{t}_\gamma}$ and $\omega_{\mathcal{A}^n_\mathfrak{o}}$}' of the proof.} that $char(F) = 0$ or $char(F)>n$.
The difference between $\mathcal{SO}_{\gamma}$ and $\mathcal{SO}_{\gamma,d\mu}$ is described as follows:
\[
    \mathcal{SO}_{\gamma}
    = |D(\gamma)|^{1/2} \prod_{i \in B(\gamma)} |N_{F_{i}/F}(\Delta_{\widetilde{F}_i/F_i})|^{-1/2}|\Delta_{F_i/F}|^{-1/2}\frac{\#\mathrm{G}
    (\kappa) q^{-\dim \mathrm{G}}}{\#\underline{\mathrm{T}}_{\gamma}(\kappa) q^{-\dim \mathrm{T}_{\gamma,F}}} \mathcal{SO}_{\gamma,d\mu}.
\]
Here $\Delta_{\widetilde{F}_i/F_i}$ and $\Delta_{F_i/F}$ are the discriminants of étale extensions $\widetilde{F}_i/F_i$ and $F_i/F$ respectively.
The Weyl discriminant $D(\gamma)$ is described in Remark \ref{remark:sordmu}, and $|N_{F_i/F}(\Delta_{\widetilde{F}_i/F_i})|$ is described in Lemma \ref{lem:Conductor_unitary}.
In particular, if $(\widetilde{F},\epsilon) = (E,1)$, then $|N_{F_{i}/F}(\Delta_{\widetilde{F}_i/F_i})| = 1$ for any $i \in B(\gamma)$ by Remark \ref{remark:fitilde}.
\end{proposition}
\begin{proof}

  \begin{enumerate} 
  \item {Comparison between two schemes $\mathrm{T}_{\gamma,F} \backslash \mathrm{G}_F$ and $G_{\gamma}$}

Let $\mathcal{O}(\gamma)$ be the orbit of $\gamma$ under the adjoint action of $\mathrm{G}_{F}$ on $\mfu_{F}$, which is defined in \cite[Section 1.7]{Bor91}. Then $\mathcal{O}(\gamma)$ is a smooth variety by \cite[Proposition in Section 1.8]{Bor91} and is closed in  $\mfu_{F}$ by \cite[Theorem 9.2]{Bor91}.
Furthermore the adjoint morphism yields an isomorphism $\mathrm{Ad}: \mathrm{T}_{\gamma,F}\backslash \mathrm{G}_{F} \rightarrow \mathcal{O}(\gamma)$ by \cite[Proposition (2) in Section 9.1]{Bor91}. 

We claim that 
$$\left(\mathrm{Ad}:\mathrm{T}_{\gamma,F}\backslash \mathrm{G}_{F} \cong \right)\mathcal{O}(\gamma) = G_{\gamma}(:=\rho_n^{-1}(\chi_{\gamma})) \textit{ as subvarieties of } \mathfrak{g}_{F}.
$$ 
Note that both varieties $\mathcal{O}(\gamma)$ and $G_{\gamma}$ are closed in $\mfu_{F}$ and smooth. 
Since they have the same sets of   $\bar{F}$-points, it suffices to show that $\mathcal{O}(\gamma)$ is a subvariety of  $G_{\gamma}$.
This follows from \textit{the universal mapping property} in \cite[Section 6.3]{Bor91}.

Then the set of $F$-points of $G_{\gamma}$ is characterized as follows:

    \[
\bigsqcup_{\gamma' \sim \gamma} \mathrm{T}_{\gamma'}(F) \backslash \mathrm{G}(F) \cong    G_{\gamma}(F), \quad g \mapsto g^{-1} \gamma' g ~~ for ~~~ g\in  \mathrm{T}_{\gamma'}(F) \backslash \mathrm{G}(F), 
    \]
    where $\gamma'$ runs over a set of representatives of the rational conjugacy classes in the stable conjugacy class of $\gamma$.
    Here  (the image of) each $\mathrm{T}_{\gamma'}(F)\backslash \mathrm{G}(F)$ is an open subset of $G_\gamma(F)$ (cf. \cite[Section 20.2]{Kot05}).

    Recall that we defined the measure $d \mu'$ on $\mathrm{T}_{\gamma'}(F)\backslash \mathrm{G}(F)$ in Section \ref{subsubsection331111}.
    On the other hand, it makes sense to restrict the measure $|\omega_{\chi_\gamma}^{\operatorname{ld}}|$ to $\mathrm{T}_{\gamma',F}(F)\backslash \mathrm{G}(F)$. 
Therefore, it suffices to show that for each $\gamma'$, we have
    \[
|\omega_{\chi_{\gamma}}^{\mathrm{ld}}|
= |D(\gamma)|^{1/2} \prod_{i \in B(\gamma)} |N_{F_{i}/F}(\Delta_{\widetilde{F}_i/F_i})|^{-1/2}
|\Delta_{F_i/F}|^{-1/2}\frac{\#\mathrm{G}(\kappa) \cdot q^{-\dim \mathrm{G}}}{\#\underline{\mathrm{T}}_{\gamma}(\kappa) \cdot q^{-\dim   \mathrm{T}_{\gamma,F}}} d\mu'
    \]
    as measures on $\mathrm{T}_{\gamma'}(F)\backslash \mathrm{G}(F)$.
    Here $\#\underline{\mathrm{T}}_{\gamma}(\kappa)=\#\underline{\mathrm{T}}_{\gamma'}(\kappa)$ by Corollary \ref{lem:pts_spfib_tori}.
    Since the coefficient of the right hand side is the same as that for $\gamma'$ (cf. Corollary \ref{corollary:maxopencpt} and Remark \ref{remark:sordmu}), 
    it suffices to show that 
    \begin{equation}
    |\omega_{\chi_{\gamma}}^{\mathrm{ld}}|
= |D(\gamma)|^{1/2} \prod_{i \in B(\gamma)} 
|N_{F_{i}/F}(\Delta_{\widetilde{F}_i/F_i})|^{-1/2} |\Delta_{F_i/F}|^{-1/2}\frac{\#\mathrm{G}(\kappa) \cdot q^{-\dim \mathrm{G}}}{\#\underline{\mathrm{T}}_{\gamma}(\kappa) \cdot q^{-\dim   \mathrm{T}_{\gamma,F}}} d\mu.    
    \end{equation}

        \item {Volume forms}
        
        We denote by $\mathfrak{t}_{\gamma, F}$  the Lie algebra of $\mathrm{T}_{\gamma, F}$.
        Then $\underline{\mathfrak{t}}_\gamma := \operatorname{Lie}(\underline{\mathrm{T}}_\gamma)$ is an affine space over $\mathfrak{o}$, which is a smooth integral model of $\mathfrak{t}_{\gamma, F}$.
        Here $\underline{\mathrm{T}}_\gamma$ is the   ft-Néron model of $\mathrm{T}_{\gamma, F}$ (cf. Corollary \ref{corollary:maxopencpt}).
        We introduce a series of volume forms as follows:
    \[
    \begin{cases}
        \omega_{\mathrm{G}}: \textit{the form on } \mathrm{G}(F) \textit{ such that } \operatorname{vol}(|\omega_{\mathrm{G}}|, \mathrm{G}(\mfo)) =  \frac{\# \mathrm{G}(\kappa)}{q^{\dim \mathrm{G}}}; \\
        \omega_{\mathrm{T}_\gamma}: \textit{the form on }\mathrm{T}_{\gamma,F}(F) \textit{ such that } \operatorname{vol}(|\omega_{\mathrm{T}_\gamma}|, \underline{\mathrm{T}}_\gamma(\mathfrak{o})) =  \frac{\# \underline{\mathrm{T}}_\gamma(\kappa)}{q^{\dim \mathrm{T}_{\gamma,F}}};\\
        \omega_{\mathfrak{g}} : \textit{the form on } \mathfrak{g}_{F} \textit{ such that } \operatorname{vol}(|\omega_{\mathfrak{g}}|,\mathfrak{g}(\mathfrak{o})) = 1; \\
        \omega_{\mathfrak{t}_{\gamma}} : \textit{the form on } \mathfrak{t}_{\gamma, F} \textit{ such that } \operatorname{vol}(|\omega_{\mathfrak{t}_\gamma}|,\underline{\mathfrak{t}}_\gamma(\mfo)) = 1; \\
        \omega_{\mathcal{A}^n_\mathfrak{o}} : \textit{the form on } \mathcal{A}^n_F(F) \textit{ such that } \operatorname{vol}(|\omega_{\mathcal{A}^n_\mathfrak{o}}|,\mathcal{A}^n_\mathfrak{o}(\mathfrak{o})) = 1.
    \end{cases}
    \]
Note that each of these forms is determined by a top degree differential form on each smooth integral model, which is invariant under translation and which is nonzero everywhere on the special fiber over $\kappa$ (cf. \cite[Theorem 2.2.5]{Weil}).

In order to  define a volume form on $\mathrm{T}_{\gamma,F}(F)\backslash \mathrm{G}(F)$ which is invariant under translation by $\mathrm{G}(F)$, it is enough to choose  top degree differential forms on $\mathfrak{g}(F)$ and $\mathfrak{t}_{\gamma,F}(F)$, as mentioned in \cite[\S 7.2]{Kot05}. Therefore, we define the following volume form:
    \[
       \omega_{\mathrm{T}_{\gamma}\backslash\mathrm{G}} : \textit{the form on } \mathrm{T}_{\gamma,F}(F)\backslash \mathrm{G}(F)  \textit{ determined by }\omega_{\mathfrak{g}} \textit{ and } \omega_{\mathfrak{t}_{\gamma,F}}.
    \]
Here we refer to \cite[\S 7.2]{Kot05} for a detailed explanation of the relation between $\omega_{\mathrm{T}_{\gamma}\backslash\mathrm{G}}$ and $\omega_{\mathfrak{g}}, \omega_{\mathfrak{t}_{\gamma,F}}$.
\\

\item {Comparison between $\omega_{\mathrm{T}_{\gamma}\backslash\mathrm{G}}$ and $d\mu(:=\frac{dg}{dt})$}
\cite[Proposition 3.15]{EGM} yields  that assigning a top degree differential form on   $\mathrm{G}_{F}$ which is invariant under tranalation is equivalent to assigning a nonzero element in $\mathfrak{g}(F)$. 
Thus $ \omega_{\mathrm{G}}$ and $\omega_{\mathrm{T}_\gamma}$ correspond to $\omega_{\mathfrak{g}}$ and  $\omega_{\mathfrak{t}_\gamma}$ respectively.
Since $|\omega_{\mathrm{G}}|= \frac{\# \mathrm{G}(\kappa)}{q^{\dim \mathrm{G}}}\cdot dg$ and  
             $|\omega_{\mathrm{T}_\gamma}| =  \frac{\# \underline{\mathrm{T}}_\gamma(\kappa)}{q^{\dim \mathrm{T}_{\gamma,F}}}\cdot dt$, we have 
 \begin{equation}\label{equation:comparisontrundmu}
|\omega_{\mathrm{T}_{\gamma}\backslash\mathrm{G}}|=\frac{\# \mathrm{G}(\kappa)\cdot q^{-\dim \mathrm{G}}}{\# \underline{\mathrm{T}}_\gamma(\kappa) \cdot q^{-\dim \mathrm{T}_{\gamma,F}}}d\mu
 \end{equation}
 as forms on $\mathrm{T}_\gamma(F)\backslash \mathrm{G}(F)$.
Here, $dg$ and $dt$ are defined in Section \ref{subsubsection331}.
\\

    \item {Comparison between $\omega_{\mathrm{T}_{\gamma}\backslash\mathrm{G}} \wedge \omega_{\mathfrak{t}_\gamma}$ and $\omega_{\mathfrak{g}}$}

    Consider the map
    \[
    \beta : \mathrm{T}_{\gamma,F}(F) \backslash \mathrm{G}(F) \times \mathfrak{t}_{\gamma,F}(F) \to \mathfrak{g}_{F}(F), \quad (g,x) \mapsto g^{-1} x g.
    \]
    Then \cite[Section 7.2]{Kot05} gives the equation 
    \begin{equation}\label{equation:comparetrun}
            |W_T|^{-1}\cdot |D(\gamma)| \cdot |\omega_{\mathrm{T}_\gamma\backslash \mathrm{G}} \wedge \omega_{\mathfrak{t}_\gamma}| = |\omega_{\mathfrak{g}}|,
        \end{equation}

where $W_T=N_{\mathrm{G}_F}(\mathrm{T}_{\gamma,F})(F)/\mathrm{T}_{\gamma,F}(F)$.
Note that this relation is also explained in \cite[Equation (28)]{Gor22} in the case of split reductive groups.
\\

    \item {Comparison between $\omega_{\mathfrak{t}_\gamma}$ and $\omega_{\mathcal{A}^n_\mathfrak{o}}$}
    
Since $\mathfrak{t}_{\gamma, F}$ is embedded into $\mfu_{F}$, the morphism $\rho_n$ defined in Equation (\ref{eq:chevalleymapp}) yields 
\[
\rho_{n,F}: \mathfrak{t}_{\gamma,F}(F) \longrightarrow \mathfrak{g}(F) \longrightarrow \mathcal{A}^n_F(F).
\]
We claim that 
\begin{equation}\label{equation:comparetranf}
|W_T|^{-1} \cdot |D(\gamma)|^{1/2} \cdot |D_{M_\mathrm{T}}|^{1/2} \cdot |\omega_{\mathfrak{t}_\gamma}|= |\omega_{\mathcal{A}^n_\mathfrak{o}}|.
\end{equation}
where $D_{M_\mathrm{T}}$ is the \textit{refined Artin conductor} of the motive $M_\mathrm{T}$ associated with $\mathrm{T}_{\gamma,F}$, defined in \cite[(4.5)]{GG99}.
In our case, $M_\mathrm{T}$ is the representation of the Galois group on the character lattice of $\mathrm{T}_\gamma$ (cf. \cite[Section 2.3]{Gor22}).

Firstly, we claim that $\omega_{\mathrm{T}_\gamma}$ coincides with  $\omega^{\mathrm{can}}$, up to multiplication by a unit in $\mathfrak{o}$, where  $\omega^{\mathrm{can}}$ is the volume form on $\mathrm{T}_{\gamma, F}$ defined in \cite[Section 2.2.2]{Gor22} which is associated to the relative identity component $\underline{\mathrm{T}}_{\gamma}^0$ of $\underline{\mathrm{T}}_\gamma$.
Here we refer to  \cite[Definition 4.1.18]{KP23} for the notion of the relative identity component.
Note that the subgroup $\underline{\mathrm{T}}_{\gamma}^0(\mathfrak{o}) \leq \mathrm{T}_c$ is of finite index by \cite[Section 2.2.2]{Gor22}.
It suffices to show that $\mathrm{vol}(|\omega^{\mathrm{can}}|,\mathrm{T}_c)=\mathrm{vol}(|\omega_{\mathrm{T}_\gamma}|,\mathrm{T}_c)$.
\cite[Corollary 11.2.1]{KP23} and \cite[Section 2.2.3.(1)]{Gor22} yield that $[\mathrm{T}_c:\underline{\mathrm{T}}_\gamma^0(\mathfrak{o})] = [\underline{\mathrm{T}}_\gamma(\kappa):\underline{\mathrm{T}}_\gamma^0(\kappa)]$.
Then by \cite[Section 2.2.3.(2)]{Gor22} we have
\[
\mathrm{vol}(|\omega^{\mathrm{can}}|,\mathrm{T}_c)
= [\underline{\mathrm{T}}_\gamma(\kappa):\underline{\mathrm{T}}_\gamma^0(\kappa)] \cdot \frac{\# \underline{\mathrm{T}}_\gamma^0(\kappa)}{q^{\dim \mathrm{T}_{\gamma,F}}}
=\frac{\# \underline{\mathrm{T}}_\gamma(\kappa)}{q^{\dim \mathrm{T}_{\gamma,F}}}
=\mathrm{vol}(|\omega_{\mathrm{T}_\gamma}|,\mathrm{T}_c).
\]

Choose coordinates for $\underline{\mathfrak{t}}_\gamma(\mfo)$ and $\mathcal{A}^n_\mathfrak{o}(\mathfrak{o})$ whose top degree exterior powers are the forms $\omega_{\mathfrak{t}_\gamma}$ and $\omega_{\mathcal{A}^n_\mathfrak{o}}$ respectively.
The Jacobian of the above map $\rho_{n,F}: \mathfrak{t}_\gamma(F) \longrightarrow  \mathcal{A}^n_F(F)$ with respect to these coordinates is a polynomial having coefficients in $\mfo$.

On the other hand, the base change to $\mfo_{\widetilde{F}}$ of these coordinates also form coordinates for $\underline{\mathfrak{t}}_\gamma(\mfo_{\widetilde{F}})$ and $\mathcal{A}^n_\mathfrak{o}(\mathfrak{o}_{\widetilde{F}})$. 
With respect to these, the Jacobian of the map $\rho_{n,\widetilde{F}}: \mathfrak{t}_{\gamma,F}(\widetilde{F}) \longrightarrow  \mathcal{A}^n_F(\widetilde{F})$  is then the same as that for $\rho_{n,F}$. 

Since $\mathrm{G}_{\widetilde{F}}$ is split, 
\cite[Equation (30)]{Gor22} yields that  the Jacobian, with respect to the coordinate for $\mathfrak{t}_{\gamma, \widetilde{F}}$ induced by a basis of the character lattice (cf. \cite[Section 2.2]{Gor22}), is $D(\gamma)^{1/2}$ up to multiplication by a unit in $\mfo_{\widetilde{F}}$.  
In \cite[Section 2.2.3.(3)]{Gor22} the difference between $\omega_{\mathrm{T}_\gamma}$ and the form with respect to a basis of the character lattice is explained; 
 it is $D_{M_\mathrm{T}}^{1/2}$ up to multiplication by a unit in $\mfo$, under the assumption that $char(F) = 0$ or that $\mathrm{T}_{\gamma,F}$ splits over a Galois extension of $F$ of degree relatively prime to $char(F)$.
 Note that $\mathrm{T}_{\gamma,F}$ splits over the Galois closure of $\widetilde{F}_{\chi_\gamma}/F$, where $[\widetilde{F}_{\chi_\gamma}:F]=2n$.
 Thus  the second assumption is satisfied since $char(F)$ is either $0$ or $>n$.

Since an invariant top degree differential form on an algebraic group is completely determined by its value at the identity (cf. \cite[Proposition 3.15]{EGM}), this difference is exactly the same as that for  Lie algebras.

In conclusion,  the Jacobian of $\rho_{n,F}: \mathfrak{t}_{\gamma,F}(F) \longrightarrow  \mathcal{A}^n_F(F)$ is also a constant in $F$ whose absolute value is $|D(\gamma)|^{1/2} \cdot |D_{M_T}|^{1/2}$. 
Since the map $\rho_{n,F}: \mathfrak{t}_{\gamma, F}(F) \longrightarrow  \mathcal{A}^n_F(F)$ is $|W_T|:1$ (cf. \cite[Equation (28)]{Gor22}), we have the desired formula.
\\

\item {Computation of $|D_{M_\mathrm{T}}|$}

By \cite[Corollary 4.6]{GG99}, we have
$|D_{M_\mathrm{T}}| = |\pi^{a(\mathrm{T}_{\gamma,F})}|$, 
where $a(\mathrm{T}_{\gamma,F})$ is the \textit{Artin conductor} of a torus $\mathrm{T}_{\gamma,F}$ defined in \cite[\S 4]{GG99}.
It is well-known (cf. \cite[Proposition VII.11.7]{Neu}) that the Artin conductor satisfies the relation 
$a(\mathrm{T}_{\gamma,F}) = a(\prod\limits_{i\in B(\gamma)}\mathrm{T}_{\gamma,i,F}) = \sum\limits_{i \in B(\gamma)} a(\mathrm{T}_{\gamma,i,F}).$
Here, $\mathrm{T}_{\gamma,i,F}$ is defined in Lemma \ref{lem: Cent_unitary}.
Lemma \ref{lem:Conductor_unitary}, which will be provided below, yields that 
\[
|\pi^{a(\mathrm{T}_{\gamma,i,F})}| = |N_{F_i/F}(\Delta_{\widetilde{F}_i/F_i})|\cdot|\Delta_{F_{i}/F}|.
\]

  \item {Comparison between  $\omega_{\chi_{\gamma}}^{\mathrm{ld}}$ and $d\mu$}

Combining Equations (\ref{equation:comparetrun}) and (\ref{equation:comparetranf}), 
    we have
$|D(\gamma)|^{1/2} \cdot |D_{M_\mathrm{T}}|^{-1/2} \cdot|\omega_{\mathrm{T}_\gamma\backslash \mathrm{G}} \wedge \omega_{\mathcal{A}^n_\mathfrak{o}}| = |\omega_{\mathfrak{g}}|$, 
    which yields that 
     $|D(\gamma)|^{1/2} \cdot |D_{M_\mathrm{T}}|^{-1/2} \cdot|\omega_{\mathrm{T}_\gamma\backslash \mathrm{G}}|=|\omega_{\chi_{\gamma}}^{\mathrm{ld}}|.$
Combining Equation (\ref{equation:comparisontrundmu}), we finally obtain the following formula:
\[
    |\omega_{\chi_{\gamma}}^{\mathrm{ld}}|
= |D(\gamma)|^{1/2} \prod_{i \in B(\gamma)} 
|N_{F_i/F}(\Delta_{\widetilde{F}_i/F_i})|^{-1/2}|\Delta_{F_i/F}|^{-1/2}\frac{\#\mathrm{G}(\kappa) \cdot q^{-\dim \mathrm{G}}}{\#\underline{\mathrm{T}}_{\gamma}(\kappa) \cdot q^{-\dim   \mathrm{T}_{\gamma,F}}} d\mu.
\]
    \end{enumerate}
\end{proof}

\begin{lemma}\label{lem:Conductor_unitary}
    Suppose that $F^\sigma_{\chi_\gamma}/F$ is a field extension of degree $n$ and that $char(F)\neq 2$.
    The Artin conductor $a(\mathrm{T}_{\gamma,F})$ of $\mathrm{T}_{\gamma,F}$ satisfies 
    \[
    |\pi^{a(\mathrm{T}_{\gamma,F})}|
    =|N_{F^\sigma_{\chi_\gamma}/F}(\Delta_{\widetilde{F}_{\chi_\gamma}/F^\sigma_{\chi_\gamma}})| \cdot |\Delta_{F^\sigma_{\chi_\gamma}/F}|,     \textit{    where    }
    \]
 \[
    |N_{F^\sigma_{\chi_\gamma}/F}(\Delta_{\widetilde{F}_{\chi_\gamma}/F^\sigma_{\chi_\gamma}})|
    = q^{-rd} \textit{   with     }   \Delta_{{\widetilde{F}_{\chi_\gamma}}/F^\sigma_{\chi_\gamma}} = (\varpi^r).\]
    Here $\Delta_{\widetilde{F}_{\chi_\gamma}/F^\sigma_{\chi_\gamma}}$ and $\Delta_{F^\sigma_{\chi_\gamma}/F}$ are the discriminants of $\widetilde{F}_{\chi_\gamma}/F^\sigma_{\chi_\gamma}$ and $F^\sigma_{\chi_\gamma}/F$ respectively, 
$d=[K:F]$ (cf. Corollary \ref{lem:pts_spfib_tori}),  and $\varpi$ is a uniformizer of $\mathfrak{o}_{F^\sigma_{\chi_\gamma}}$.
\end{lemma}

Note that  $r=0$ if $\widetilde{F}_{\chi_\gamma}/F^\sigma_{\chi_\gamma}$ is unramified and that $r=1$  if $\widetilde{F}_{\chi_\gamma}/F^\sigma_{\chi_\gamma}$ is ramified with $char(\kappa)\neq 2$.
If $\widetilde{F}_{\chi_\gamma}/F^\sigma_{\chi_\gamma}$ is split, then we understand $r=0$.

\begin{proof}
Note that $\widetilde{F}_{\chi_\gamma}/F^\sigma_{\chi_\gamma}$ is either split or a field extension (cf. Remark \ref{remark:fitilde}). 
We treat each case independently.

\begin{enumerate}
    \item Suppose that $\widetilde{F}_{\chi_\gamma}/F^\sigma_{\chi_\gamma}$ is split so that $\mathrm{T}_{\gamma, F} \cong \operatorname{Res}_{F^\sigma_{\chi_\gamma}/F} \mathbb{G}_m$ by Lemma \ref{lem: Cent_unitary}. 
Then   \cite[Proposition 2.2]{Lee21} (see also \cite[Corollary VII.11.8]{Neu})  yields $|\pi^{a(\mathrm{T}_\gamma)}| = |\Delta_{F^\sigma_{\chi_\gamma}/F}|$.

    \item Suppose that $\widetilde{F}_{\chi_\gamma}/F^\sigma_{\chi_\gamma}$ is a field extension.
    We follow the idea of \cite[Proposition 2.2]{Lee21}.
   The exact sequence
    \[
    1 \to \mathrm{T}_{\gamma, F} \to \operatorname{Res}_{F^\sigma_{\chi_\gamma}/F} \left(\operatorname{Res}_{\widetilde{F}_{\chi_\gamma}/F^\sigma_{\chi_\gamma}} 
    \mathbb{G}_m\right) \xrightarrow{\operatorname{Res}_{F^\sigma_{\chi_\gamma}/F}\left(N_{\widetilde{F}_{\chi_\gamma}/F^\sigma_{\chi_\gamma}}\right)} \operatorname{Res}_{F^\sigma_{\chi_\gamma}/F} \mathbb{G}_m
    \to 1
    \]
    yields an existence of an isogeny $\mathrm{T}_{\gamma,F} \times \operatorname{Res}_{F^\sigma_{\chi_\gamma}/F} \mathbb{G}_m \sim \operatorname{Res}_{\widetilde{F}_{\chi_\gamma}/F} \mathbb{G}_m$ by \cite[Equation (4.3)]{Ono}.
    Since the Artin conductor is invariant under an isogeny, we have that
    \[
    a(\mathrm{T}_{\gamma, F}) = a(\operatorname{Res}_{\widetilde{F}_{\chi_\gamma}/F}\mathbb{G}_m) - a(\operatorname{Res}_{F^\sigma_{\chi_\gamma}/F} \mathbb{G}_m).
    \]
    Then \cite[Proposition III. 8]{Ser} and \cite[Proposition 2.2]{Lee21} yield that
    \[
    |\pi^{a(\mathrm{T}_{\gamma, F})}|
    =\frac{|\Delta_{{\widetilde{F}}_{\chi_\gamma}/F}|}{|\Delta_{F^\sigma_{\chi_\gamma}/F}|}
    =\frac{|N_{F^\sigma_{\chi_\gamma}/F}(\Delta_{\widetilde{F}_{\chi_\gamma}/F^\sigma_{\chi_\gamma}}) \cdot \Delta_{F^\sigma_{\chi_\gamma}/F}^{[\widetilde{F}_{\chi_\gamma}:F^\sigma_{\chi_\gamma}]}|}{|\Delta_{F^\sigma_{\chi_\gamma}/F}|}
    = |N_{F^\sigma_{\chi_\gamma}/F}(\Delta_{\widetilde{F}_{\chi_\gamma}/F^\sigma_{\chi_\gamma}}) \cdot \Delta_{F^\sigma_{\chi_\gamma}/F}|.
    \]
\end{enumerate}
\end{proof}

In the next section, we will explain the parabolic descent so that the stable orbital integral is reduced to the case that the set $B(\gamma)^{split}$ is empty, equivalently $\widetilde{F}_i/F_i$ for all $i\in B(\gamma)$ is a quadratic field extension (cf. Remark \ref{remark:fitilde}). 
In Sections \ref{sec:reduction}-\ref{sectypem}, we will treat the case that $B(\gamma)=B(\gamma)^{irred}$ is a singleton and thus we will provide the comparison in this case.

\begin{definition}\label{def:Serre_inv}
Define $S(\psi)$ and $S(\gamma)$ to be the \textit{Serre invariants} (cf. \cite[Section 2.1]{Yun13}) which are described by the relative module lengths
    \[
    \begin{cases}
        S(\psi) = [\mathfrak{o}_{F_{\chi_\gamma}^\sigma}:\mathfrak{o}[y]/(\psi(y))]_{\mfo}; \\
        S(\gamma) = [\mathfrak{o}_{\widetilde{F}_{\chi_\gamma}}:\mathfrak{o}_{\widetilde{F}}[x]/(\chi_\gamma(x))]_{\mfo_{\widetilde{F}}}.
    \end{cases}
    \]
Here $F_{\chi_{\gamma}}^\sigma = F[y]/(\psi(y))$ in Equation (\ref{equation:F_gamma_iso}) and $\widetilde{F}_{\chi_{\gamma}} :=  \widetilde{F}[x]/(\chi_{\gamma}(x))$ at the beginning of Section \ref{sec:trans_herm}.
Note that  $S(\gamma)=S(\psi)$ when $(\widetilde{F}, \epsilon)=(E, 1)$ since
$E/F$ is a unramified field extension and $\psi(y)=\chi_{\alpha\gamma}(x)$ with $\alpha\in \mfo_E^{\times}$. 
\end{definition}

\begin{corollary}\label{cor:comp_irr}
    Suppose that $char(F) = 0$ or $char(F)>n$, and that $\widetilde{F}_{\chi_\gamma}/\widetilde{F}$ is a field extension. 
    The difference between $\mathcal{SO}_{\gamma}$ and $\mathcal{SO}_{\gamma,d\mu}$ is described as follows:
    \begin{enumerate}
        \item If $(\widetilde{F},\epsilon) = (E,1)$, then we have
        \[
        \mathcal{SO}_\gamma^{\mathrm{U}_n}
        =
        q^{-S(\gamma)} \cdot \frac{\#\mathrm{U}_n(\kappa) q^{-n^2}}{1+q^{-d}} \mathcal{SO}_{\gamma,d\mu}^{\mathrm{U}_n}
        \]
    where $d=[K:F]$ (cf. Corollary \ref{lem:pts_spfib_tori}).
        \item If $(\widetilde{F},\epsilon) = (F,-1)$, then we have
        \[
     \mathcal{SO}_\gamma^{\mathrm{Sp}_{2n}}=
    \begin{cases}
        q^{-S(\gamma)+S(\psi)} \cdot \frac{\#\mathrm{Sp}_{2n}(\kappa) q^{-n(2n+1)}}{1+q^{-d}} \cdot \mathcal{SO}_{\gamma,d\mu}^{\mathrm{Sp}_{2n}}  &\textit{if } \widetilde{F}_{\chi_\gamma}/F_{\chi_\gamma}^\sigma \textit{ is unramified}; \\
        q^{-S(\gamma)+S(\psi)} \cdot \frac{\#\mathrm{Sp}_{2n}(\kappa) q^{-n(2n+1)}}{2} \cdot \mathcal{SO}_{\gamma,d\mu}^{\mathrm{Sp}_{2n}}  &\textit{if } \widetilde{F}_{\chi_\gamma}/F_{\chi_\gamma}^\sigma \textit{ is ramified}.
    \end{cases}
    \]
    \end{enumerate}
    Here $d$ is the inertial degree of $F_{\chi_\gamma}^\sigma/F$.
\end{corollary}
\begin{proof}
    Proposition \ref{prop:comparison_measures} yields that 
    \[
    \mathcal{SO}_{\gamma}^{\mathrm{G}}
    = |D(\gamma)|^{1/2} |N_{F_{\chi_\gamma}^\sigma/F}(\Delta_{\widetilde{F}_{\chi_\gamma}/F_{\chi_\gamma}^\sigma})|^{-1/2}|\Delta_{F_{\chi_\gamma}^\sigma/F}|^{-1/2}
    \frac{\#\mathrm{G}
    (\kappa) q^{-\dim \mathrm{G}}}{\#\underline{\mathrm{T}}_{\gamma}(\kappa) q^{-\dim \mathrm{T}_{\gamma,F}}} \mathcal{SO}_{\gamma,d\mu}^{\mathrm{G}}.
    \]
By Lemma \ref{lem:Conductor_unitary},      $|N_{F^\sigma_{\chi_\gamma}/F}(\Delta_{\widetilde{F}_{\chi_\gamma}/F^\sigma_{\chi_\gamma}})| = 1$  if $\widetilde{F}_{\chi_\gamma}/{F}^\sigma_{\chi_\gamma}$ is unramified.
    Proposition \ref{propserre} yields
    \[
    (\#\kappa_{\widetilde{F}})^{-S(\gamma)} = |\Delta_\gamma|_{\widetilde{F}}^{1/2} \cdot |\Delta_{\widetilde{F}_{\chi_\gamma}/\widetilde{F}}|_{\widetilde{F}}^{-1/2} \textit{ and }
    q^{-S(\psi)} = |\Delta_\psi|^{1/2} \cdot |\Delta_{{F}_{\chi_\gamma}^\sigma/F}|^{-1/2}.
    \]
    
    \begin{enumerate}
        \item 
    Suppose that $(\widetilde{F},\epsilon) = (E,1)$.
    In this case, we have $|N_{F^\sigma_{\chi_\gamma}/F}(\Delta_{\widetilde{F}_{\chi_\gamma}/F^\sigma_{\chi_\gamma}})| = 1$.
    Remark \ref{remark:sordmu} for $\mathrm{U}_n$ yields that 
    \[
    |D(\gamma)|^{1/2} \cdot |\Delta_{F_{\chi_\gamma}^\sigma/F}|^{-1/2} 
    = |\Delta_\gamma|^{1/2} \cdot |\Delta_{{F}_{\chi_\gamma}^\sigma/F}|^{-1/2} 
    = |\Delta_\psi|^{1/2} \cdot |\Delta_{{F}_{\chi_\gamma}^\sigma/F}|^{-1/2} 
    = q^{-S(\psi)}=q^{-S(\gamma)}.
    \]
    \item
     Suppose that $(\widetilde{F},\epsilon) = (F,-1)$.
     By the proof of Lemma \ref{lem:Conductor_unitary}, we have $|N_{F_{\chi_\gamma}^\sigma/F}(\Delta_{\widetilde{F}_{\chi_\gamma}/F_{\chi_\gamma}^\sigma})||\Delta_{F_{\chi_\gamma}^\sigma/F}| = |\Delta_{\widetilde{F}_{\chi_\gamma}/F}||\Delta_{F_{\chi_\gamma}^\sigma/F}|^{-1}$.
     Then Remark \ref{remark:sordmu} for $(\widetilde{F},\epsilon) = (F,-1)$ yields that 
     \begin{align*}
     &|D(\gamma)|^{1/2} |N_{F^\sigma_{\chi_\gamma}/F}(\Delta_{\widetilde{F}_{\chi_\gamma}/F^\sigma_{\chi_\gamma}})|^{-1/2}|\Delta_{F^\sigma_{\chi_\gamma}/F}|^{-1/2}
     =\frac{|\Delta_{\gamma}|^{1/2}|\Delta_{\widetilde{F}_{\chi_\gamma}/F}|^{-1/2}}{|\Delta_{\psi}|^{1/2}|\Delta_{F_{\chi_\gamma}^\sigma/F}|^{-1/2}}
     =q^{-S(\gamma)} \cdot q^{S(\psi)}.
     \end{align*}
     \end{enumerate}
\end{proof}

\section{Parabolic descent}\label{sec:parabolicdescent}
Recall that $B(\gamma)=B(\gamma)^{irred} \bigsqcup B(\gamma)^{split}$ is the index set of irreducible factors in Equation (\ref{equation:irrfactor}).
A goal of this subsection is to explain a reduction of $\sog$ to the case that  $B(\gamma)^{irred}=B(\gamma)$. 
We will use a well-known method, parabolic descent, by closely following  \cite[Section 8.6]{GH19}.
In order to choose a precise measure appeared in loc. cit., it is necessary to describe a maximal compact subgroup of a Levi subgroup and its Lie algebra. 
Thus we will start with a precise matrix description of $\gamma$ with $\psi=\prod\limits_{i \in B(\gamma)} \psi_i=\prod\limits_{i\in B(\gamma)^{irred}}\psi_{i} \times \prod\limits_{i\in B(\gamma)^{split}}\psi_{i}$ (cf. Equations (\ref{equation:F_gamma_iso})-(\ref{equation:irrfactor})).
We introduce the following notations:
\[
\left\{
\begin{array}{l}
B(\gamma)^{irred}=\{1, \cdots, k\} \textit{   and   } B(\gamma)^{split}=\{k+1, \cdots, k+k^{\dagger}\};\\
  m_i= \textit{ the degree of }\psi_i(y) \textit{ in $\mfo[x]$ with }i\in B(\gamma)^{irred};\\
  2l_i= \textit{ the degree of }\psi_i(y) \textit{ in $\mfo[x]$ with }i\in B(\gamma)^{split};\\
 m=m_1+\cdots +m_k  \textit{   and   } l=l_{k+1}+\cdots +l_{k+k^{\dagger}}.
 \end{array}
\right.
\]

Note that $m+2l=n$ if $(\widetilde{F}, \epsilon)=(E, 1)$ and $m+2l=2n$ if $(\widetilde{F}, \epsilon)=(F, -1)$. In the latter case, $m_i$ is always even since $y=x^2$.

Since the parabolic descent is unnecesary when $B(\gamma)$ is a singleton, we suppose that $\#B(\gamma)\geq 2$  and $\#B(\gamma)^{irred}\geq 1$   in this section.

\subsection{A matrix description of $\gamma$}\label{section:matdescofr}
In this subsection, we suppose that $char(\kappa)>2$  
if $(\widetilde{F}, \epsilon)=(E, 1)$ and $n$ is even.
We do not impose any restriction on $char(\kappa)$ otherwise.

Let $(L_i, h_i)$ be an hermitian lattice represented by the following matrix for a suitable basis:
\[
\left\{
\begin{array}{l l}
Id_{m_i}  &  \textit{ if 
 $(\widetilde{F}, \epsilon)=(E, 1)$ and $1\leq i \leq k$};\\
\begin{pmatrix} 0 & Id_{m_i/2} \\ - Id_{m_i/2} & 0   \end{pmatrix}  & \textit{ if  $(\widetilde{F}, \epsilon)=(F, -1)$ and $1\leq i \leq k$};\\
\begin{pmatrix} 0 & Id_{l_i} \\ \epsilon Id_{l_i} & 0   \end{pmatrix}  & \textit{ if $k+1\leq i \leq k+k^{\dagger}$ in both cases}.
 \end{array}
\right.
\]
Then $(L,h)$ is the orthogonal sum of $(L_i, h_i)$'s with $1\leq i \leq k+k^{\dagger}$ since a unimodular lattice of rank $n$ is unique up to isometry\footnote{This holds without any restriction on $char(\kappa)$.}.
We denote by $\mathrm{G}_{h_i}$ the (smooth) reductive group scheme defined over $\mfo$ stabilizing $(L_i, h_i)$ and by 
$\mfu_{h_i}$ its Lie algebra over $\mfo$ for $1\leq i \leq k+k^{\dagger}$.

\begin{itemize}
    \item 
If $(\widetilde{F}, \epsilon)=(E, 1)$ and $1\leq i \leq k$, then  choose $\gamma_i\in \mfu_{h_i}(\mfo)$ such that $\chi_{\alpha\gamma_i}(x)=\psi_i(x)$. 
Here existence of $\gamma_i$ is explained in \cite[Section 2.3]{LN} and \cite[Proposition 4.3.2]{BC} with the restriction that  $char(\kappa)>2$  
if $(\widetilde{F}, \epsilon)=(E, 1)$ and $n$ is even (cf. \cite{Lee23} for the correction).

\item 
If $(\widetilde{F}, \epsilon)=(F, -1)$ and $1\leq i \leq k$,   then  choose $\gamma_i\in \mfu_{h_i}(\mfo)$ such that $\chi_{\gamma_i}(x)=\psi_i(x^2)$. 
Here existence of $\gamma_i$ is explained in \cite[Proposition 10]{AFV}.

\item 
If $k+1\leq i \leq k+k^{\dagger}$, then choose $g_i\in \mathfrak{gl}_{l_i}(\mfo_{\widetilde{F}})$ such that $\chi_{\alpha g_i}\cdot \sigma(\chi_{\alpha g_i})=\psi_i(x)$ if  $(\widetilde{F}, \epsilon)=(E, 1)$ and such that $\chi_{g_i}\cdot (-1)^{l_i}\chi_{-g_i}=\psi_i(x)$ if  $(\widetilde{F}, \epsilon)=(F, -1)$.
Here existence of $g_i$ follows from the theory of companion matrices.

Then let $\gamma_i=\begin{pmatrix} g_i & 0 \\ 0 & -{}^t\sigma(g_i)   \end{pmatrix}$.
It is easy to see that $\gamma_i\in \mfu_{h_i}(\mfo)$ such that $\chi_{\alpha\gamma_i}(x)=\psi_i(x)$ if  $(\widetilde{F}, \epsilon)=(E, 1)$ and such that $\chi_{\gamma_i}(x)=\psi_i(x^2)$ if  $(\widetilde{F}, \epsilon)=(F, -1)$.
\end{itemize}

Therefore we may and do consider  $\gamma$ as a diagonal block matrix whose each block is $\gamma_i$ for $1\leq i \leq k+k^{\dagger}$. 



\subsection{A description of a Levi subgroup $\mathrm{L}$ of $\mathrm{G}_{F}$}\label{sec:desc_Levi}
We keep assuming  that $char(\kappa)>2$ if $(\widetilde{F}, \epsilon)=(E, 1)$ and $n$ is even, and no restriction otherwise.
Let $\mathrm{G}_{m}$ be the reductive group scheme defined over $\mfo$ stabilizing the hermitian lattice $\bigoplus\limits_{i=1}^{k}(L_i, h_i)$. 
When $k+1\leq i \leq k+k^{\dagger}$, let $\mathrm{G}_{2l_i}$ be the reductive group scheme stabilizing the hermitian lattice $(L_i, h_i)$. 

When $k+1\leq i \leq k+k^{\dagger}$, we define the group scheme $\mathrm{L}_i$ over $\mfo$ as follows:
\[
\mathrm{L}_i(A)=\left\{  
g=\begin{pmatrix} g_1 & 0 \\ 0 & g_2   \end{pmatrix}\in  \mathrm{G}_{2l_i}(A) 
\right\}
\]
for a commutative $\mfo$-algebra $A$. 
Here $g_1, g_2$ are of size $l_i\times l_i$. 
A simple matrix computation shows that $g_2=\sigma(g_1)^{-1}$ and thus 
the group scheme $\mathrm{L}_i$ is isomorphic to $\mathrm{Res}_{\mfo_{\widetilde{F}}/\mfo}(\mathfrak{gl}_{l_i, \mfo_{\widetilde{F}}})$.

Let $\mathrm{L}=\mathrm{G}_m\times \prod\limits_{i=k+1}^{k+k^{\dagger}}\mathrm{Res}_{\mfo_{\widetilde{F}}/\mfo}(\mathfrak{gl}_{l_i, \mfo_{\widetilde{F}}})$, which is smooth over $\mfo$. 
As in the description of $\gamma$, the group scheme $\mathrm{L}$ is block diagonally embedded into the group scheme $\mathrm{G}$.
Then $\mathfrak{l} = \mfu_{m}\oplus \bigoplus\limits_{i=k+1}^{k+k^{\dagger}}\mathrm{Res}_{\mfo_{\widetilde{F}}/\mfo}(\mathfrak{gl}_{l_i, \mfo_{\widetilde{F}}})$, the Lie algebra of $\mathrm{L}$, is also block diagonally embedded into $\mfu$ and  contains $\gamma$.
Here $\mathfrak{g}_m$ is the Lie algebra of $\mathrm{G}_m$.

\begin{lemma}\label{lemma:levidesc}
The algebraic group $\mathrm{L}_F$ is a Levi subgroup of  $\mathrm{G}_{F}$.
\end{lemma}
\begin{proof}
By Lemma \ref{lem: Cent_unitary}, $\mathrm{T}_{\gamma,i,F}$ is the centralizer of $\gamma_i$ via the adjoint action of $\mathrm{G}_{h_i, F}$ and thus $\mathrm{T}_{\gamma,F}$ is a maximal torus of $\mathrm{L}_F$. 
By \cite[Theorem 4.15.b)]{BA}, the centralizer of the maximal split subtorus of $\mathrm{T}_{\gamma,F}$ in $\mathrm{G}_{F}$ is a Levi subgroup. 

When $1\leq i \leq k$, the torus $\mathrm{T}_{\gamma,i,F}$ is anisotropic and thus its split subtorus is trivial.
When $k+1\leq i \leq k+k^{\dagger}$, the maximal split subtorus of $\mathrm{T}_{\gamma, i}$ is of dimension $1$, which is $\mathbb{G}_m$.
More precisely, this is visualized as follows in terms of matrices:
\[
\left\{\begin{pmatrix} a\cdot Id_{l_i} & 0 \\ 0 & a^{-1}\cdot Id_{l_i}   \end{pmatrix}| a\in F^{\times}\right\}. 
\]

Thus the maximal split subtorus of $\mathrm{T}_{\gamma,F}$ is isomorphic to 
$\prod\limits_{i=1}^{k}1 \times \prod\limits_{i=k+1}^{k+k^{\dagger}}\mathbb{G}_m$, which is block diagonally embedded into $\mathrm{G}_{F}$. 
Now one can check that its centralizer in $\mathrm{G}_{F}$ is  $\mathrm{L}_F$ by easy matrix computation.
\end{proof}

\subsection{Parabolic descent of stable orbital integrals}\label{subsection12.3}
We keep assuming  that $char(\kappa)>2$ if $(\widetilde{F}, \epsilon)=(E, 1)$ and $n$ is even, and no restriction otherwise.
Let
\[
\left\{
\begin{array}{l}
    \mathrm{P}_F \textit{ be a parabolic subgroup of } \mathrm{G}_{F}   \textit{ with a Levi factor } \mathrm{L}_F;\\
    \mathrm{N}_F \textit{ be the unipotent radical of  } \mathrm{P}_F ; \\
    \mathfrak{p}_F, \mathfrak{n}_F \textit{ be the Lie algebras of } \mathrm{P}_F, \mathrm{N}_F \textit{ respectively};\\
    dn \textit{ be the Haar measure on } \mathrm{N}_F(F) \textit{ such that } \operatorname{vol}(dn, \mathrm{N}_F(F)\cap \mathrm{G}(\mathfrak{o})) = 1; \\
    d\delta \textit{ be the Haar measure on } \mathfrak{n}_F(F) \textit{ such that } \operatorname{vol}(d\delta, \mathfrak{n}_F(F) \cap \mathfrak{g}(\mathfrak{o})) = 1.
\end{array}
\right.
\]

\begin{lemma} \label{lem: good_position}
We have 
\begin{enumerate}
    \item     $\mathrm{P}_F(F) \cap \mathrm{G}(\mathfrak{o}) = \mathrm{L}(\mathfrak{o})(\mathrm{N}_F(F)\cap \mathrm{G}(\mathfrak{o}))$;  
    \item $\mathfrak{p}_F(F) \cap \mathfrak{g}(\mathfrak{o}) = \mathfrak{l}(\mathfrak{o}) \oplus (\mathfrak{n}_F(F) \cap \mathfrak{g}(\mathfrak{o}))$,
    \end{enumerate}
 where $\mathfrak{l}$ is the Lie algebra of $\mathrm{L}$ defined in Section \ref{sec:desc_Levi}.
   
\end{lemma}
\begin{proof}
For the first case,     since the parabolic subgroup $\mathrm{P}_F$ admits a split exact sequence
    \[
    1 \to \mathrm{N}_F(F) \to \mathrm{P}_F(F) \to \mathrm{L}_F(F) \to 1,
    \]
    it suffices to show that the image of $\mathrm{P}_F(F) \cap \mathrm{G}(\mathfrak{o}) (\subset \mathrm{P}_F(F))$ under the quotient map $\mathrm{P}_F(F) \to \mathrm{L}(F)$ coincides with $\mathrm{L}(\mathfrak{o})$.
    By the description of $\mathrm{L}$ in Section \ref{sec:desc_Levi}, $\mathrm{L}(\mathfrak{o})$ is a maximal compact subgroup of $\mathrm{L}(F)$ and is contained in $\mathrm{P}_F(F) \cap \mathrm{G}(\mathfrak{o})$.
The claim follows from the maximality of $\mathrm{L}(\mathfrak{o})$.

 The proof of the second case is identical to the above and thus we omit it.
\end{proof}

\begin{remark}
    Lemma \ref{lem: good_position} implies that $\mathrm{G}(\mathfrak{o})$ is in good position with respect to $(\mathrm{P}_F,\mathrm{L}_F)$ in the language of \cite[Section 8.6]{GH19} 
since the construction of the group scheme $\mathrm{L}$   directly yields that  $\mathrm{L}(\mathfrak{o})=\mathrm{L}_F(F)\cap \mathrm{G}(\mfo)$.
\end{remark}

\begin{proposition}\label{lem:par_des_orb}{\cite[Lemma 13.3]{Kot05}} 
For $\gamma \in \mathfrak{l}(\mathfrak{o})$,  we have the following formula:
\[
|D^{\mathrm{G}}(\gamma)|^{\frac{1}{2}} \mathcal{O}_{\gamma,d\mu}^{\mathrm{G}} = |D^\mathrm{L}(\gamma)|^{\frac{1}{2}}\mathcal{O}_{\gamma,d\mu}^\mathrm{L}
\]
where 
    $\mathcal{O}_{\gamma,d\mu}^{\mathrm{G}}$ and $\mathcal{O}_{\gamma,d\mu}^\mathrm{L}$
are orbital integrals of $\gamma$ on $\mathfrak{g}(F)$ and $\mathfrak{l}(F)$ respectively, which are defined in Section \ref{subsubsection331}.
Here $D^{\mathrm{G}}(\gamma)$ and $D^\mathrm{L}(\gamma)$ are Weyl discriminants of $\gamma$ with respect to $\mathrm{G}$ and $\mathrm{L}$, which are defined in Remark \ref{remark:sordmu}.
\end{proposition}
\begin{proof}
    This is a reformulation of parabolic descent described in \cite[Lemma 13.3]{Kot05}.
    To use it explicitly in our situation, we provide a proof. 
It suffices to show two formulas: 
\[
\left\{
\begin{array}{l}
\int_{\mathrm{G}(F)} \mathbbm{1}_{\mathrm{G}(\mathfrak{o})}(g) dg
    =\int_{\mathrm{L}(F)}\int_{\mathrm{N}(F)}\int_{\mathrm{G}(\mathfrak{o})} \mathbbm{1}_{\mathrm{G}(\mathfrak{o})}(lnk) dk dn dl;\\
    \mathbbm{1}_{\mathfrak{g}(\mathfrak{o})}^\mathrm{P} = \mathbbm{1}_{\mathfrak{l}(\mathfrak{o})}.
\end{array}
\right.
\]

    Here the integral is with respect to the Iwasawa decomposition 
    $\mathrm{G}(F) = \mathrm{P}(F)\mathrm{G}(\mathfrak{o}) = \mathrm{L}(F)\mathrm{N}(F)\mathrm{G}(\mathfrak{o})$, which yields the normalization in \cite[Section 13.10]{Kot05},
    where $dg$ and $dl$  are Haar measures on $\mathrm{G}(F)$ and $L(F)$ defined in Section \ref{subsubsection331} respectively, $dk$ is the measure on $\mathrm{G}(\mfo)$ with $\operatorname{vol}(dk, \mathrm{G}(\mathfrak{o})) = 1$, and $\mathbbm{1}$ stands for the characteristic function.     
The function $\mathbbm{1}^P_{\mathfrak{g}(\mathfrak{o})}$ is defined to be 
   $ \mathbbm{1}_{\mathfrak{g}(\mathfrak{o})}^\mathrm{P}(x):=\int_{\mathfrak{n}(F)} \mathbbm{1}_{\mathfrak{g}(\mathfrak{o})}(x + \delta) d\delta ~~~~~~~ \textit{     for $x\in \mathfrak{l}(F)$}$ in \cite[Section 13.1]{Kot05}.

For the first claim,  the left hand side is $1$. 
Using Lemma \ref{lem: good_position}.(1), the right hand side is 
\begin{align*}
\int_{\mathrm{L}(F)}\int_{\mathrm{N}(F)}\int_{\mathrm{G}(\mfo)} \mathbbm{1}_{\mathrm{G}(\mathfrak{o})}(lnk) dkdndl
= \int_{\mathrm{L}(F)} \mathbbm{1}_{\mathrm{L}(\mathfrak{o})}(l) dl \int_{\mathrm{N}(F)}\mathbbm{1}_{\mathrm{N}(F)\cap \mathrm{G}(\mathfrak{o})}(n) dn  = 1.
\end{align*}

For the second claim, choose  $x \in \mathfrak{l}(F)$ and $\delta \in \mathfrak{n}(F)$.
Then  Lemma \ref{lem: good_position}.(2) yields that $x+\delta \in \mathfrak{g}(\mathfrak{o})$ if and only if $x \in \mathfrak{l}(\mathfrak{o})$ and $\delta \in \mathfrak{n}(F)\cap \mathfrak{g}(\mathfrak{o})$.
This directly implies that $\mathbbm{1}_{\mathfrak{g}(\mathfrak{o})}^P = \mathbbm{1}_{\mathfrak{l}(\mathfrak{o})}$.
\end{proof}

We now extend  Proposition \ref{lem:par_des_orb} to the situation of stable orbital integrals.
Recall from Section \ref{sec:desc_Levi} that  $\mathrm{L} = \mathrm{G}_{m} \times \prod\limits_{i = k+1}^{k+k^\dagger} \operatorname{Res}_{\mathfrak{o}_{\widetilde{F}}/\mathfrak{o}}(\operatorname{GL}_{l_i,\mathfrak{o}_{\widetilde{F}_i}})$ and that $\mathfrak{l} = \mathfrak{g}_m \oplus \bigoplus\limits_{i = k+1}^{k+k^\dagger} \operatorname{Res}_{\mathfrak{o}_{\widetilde{F}}/\mathfrak{o}}(\mathfrak{gl}_{l_i,\mathfrak{o}_{\widetilde{F}}})$.
Thus as mentioned in Section \ref{section:matdescofr}, we write  $\gamma = (\gamma_m,\gamma_{k+1},\cdots,\gamma_{k+k^\dagger})\in \mathfrak{l}(\mathfrak{o})$ with $\gamma_m \in \mathfrak{g}_m(\mathfrak{o})$ and with $\gamma_i \in \mathrm{Res}_{\mathfrak{o}_{\widetilde{F}}/\mathfrak{o}}(\mathfrak{gl}_{l_i,\mathfrak{o}_{\widetilde{F}}})(\mathfrak{o})$ for $i = k+1, \cdots, k+k^\dagger$.

\begin{proposition}\label{lem:par_des_sorb}
For $\gamma \in \mathfrak{l}(\mathfrak{o})$, we have the following identity of stable orbital integrals:
\[
|D^{\mathrm{G}}(\gamma)|^{\frac{1}{2}} \mathcal{SO}_{\gamma,d\mu}^{\mathrm{G}} = 
|D^{\mathrm{G}_m}(\gamma_m)|^{\frac{1}{2}} \mathcal{SO}_{\gamma_m,d\mu}^{\mathrm{G}_m} \times \prod\limits_{i = k+1}^{k+k^\dagger} |D_i(\gamma_i)|^{\frac{1}{2}}\mathcal{SO}_{\gamma_i,d\mu}^{\operatorname{Res}_{\mathfrak{o}_{\widetilde{F}}/\mathfrak{o}}(\operatorname{GL}_{l_i,\mathfrak{o}_{\widetilde{F}}})}.
\]
Here, we shorten the notation $D^{\operatorname{Res}_{\mathfrak{o}_{\widetilde{F}}/\mathfrak{o}}(\operatorname{GL}_{l_i,\mathfrak{o}_{\widetilde{F}}})}(\gamma_i)$ to $D_i(\gamma_i)$.
\end{proposition}
\begin{proof}
We first claim that $$|D^{\mathrm{G}}(\gamma)|^{\frac{1}{2}} \mathcal{SO}_{\gamma,d\mu}^{\mathrm{G}} = |D^\mathrm{L}(\gamma)|^{\frac{1}{2}}\mathcal{SO}_{\gamma,d\mu}^\mathrm{L}.$$
By Remark \ref{remark:sordmu}, the Weyl discriminant is invariant under stable conjugacy.
Using Proposition \ref{lem:par_des_orb}, it suffices to show that two sets of conjugacy classes within the stable conjugacy class of $\gamma$ in $\mathfrak{g}(F)$ and in $\mathfrak{l}(F)$ have the same cardinality.
These two sets are parametrized by $\operatorname{ker}(H^1(F,\mathrm{T}_\gamma) \to H^1(F,\mathrm{L}_F))$ and $\operatorname{ker}(H^1(F,\mathrm{T}_\gamma) \to H^1(F,\mathrm{G}_{F}))$ respectively, which is explained in Section \ref{subsubsection331111}.

\cite[Fact 2.6.1]{KP23} yields that  the natural map $ H^1(F,\mathrm{L}_F) \to H^1(F,\mathrm{G}_{F})$ is injective, which directly shows  bijectivity between the above two kernels. 

Next, we claim that 
\[
|D^L(\gamma)|^{\frac{1}{2}}\mathcal{SO}_{\gamma,d\mu}^\mathrm{L}=
|D^{\mathrm{G}_m}(\gamma_m)|^{\frac{1}{2}} \mathcal{SO}_{\gamma_m,d\mu}^{\mathrm{G}_m} \times \prod\limits_{i = k+1}^{k+k^\dagger} |D_i(\gamma_i)|^{\frac{1}{2}}\mathcal{SO}_{\gamma_i,d\mu}^{\operatorname{Res}_{\mathfrak{o}_{\widetilde{F}}/\mathfrak{o}}(\operatorname{GL}_{l_i,\mathfrak{o}_{\widetilde{F}}})}.
\]
The description of Weyl discriminant $D^{\mathrm{L}}(\gamma)$ in Remark \ref{remark:sordmu} yields that $D^\mathrm{L}(\gamma) = D^{\mathrm{G}_m}(\gamma_m) \times \prod\limits_{i = k+1}^{k+k'} D_{i}(\gamma_i)$.
It is  obvious that 
$\mathcal{O}_{\gamma,d\mu}^\mathrm{L}=
\mathcal{O}_{\gamma_m,d\mu}^{\mathrm{G}_m} \times \prod\limits_{i = k+1}^{k+k^\dagger} \mathcal{O}_{\gamma_i,d\mu}^{\operatorname{Res}_{\mathfrak{o}_{\widetilde{F}}/\mathfrak{o}}(\operatorname{GL}_{l_i,\mathfrak{o}_{\widetilde{F}}})}$. Thus as in the first claim, it suffices to compare sets of conjugacy classes within the stable conjugacy of $r_m, r_{k+1}, \cdots, r_{k+k^{\dagger}}$. 
This is also obvious since both $\mathrm{T}_{\gamma,F}$ and $\mathrm{L}_F$ split along the index set $\{m\}\cup \{k+1, \cdots, k+k^{\dagger}\}$ which appear in $\operatorname{ker}(H^1(F,\mathrm{T}_\gamma) \to H^1(F,\mathrm{L}_{F}))$.
\end{proof}

We translate the above result in terms of geometric measures. 
Note that 
$|D^{\operatorname{Res}_{\mathfrak{o}_{\widetilde{F}}/\mathfrak{o}}(\operatorname{GL}_{l_i,\mathfrak{o}_{\widetilde{F}}})}(\gamma_i)|_F=|D^{\operatorname{GL}_{l_i,\mathfrak{o}_{\widetilde{F}}}}(\gamma_i)|_{\widetilde{F}}$
for $k+1\leq i\leq k+k^{\dagger}$. 
This, combined with 
Propositions \ref{proptrans}, \ref{prop:comparison_measures}, and \ref{lem:par_des_sorb},  and Lemma \ref{lemma:comparisonofsogweilrest} (to be provided in the next subsection), yields  the following formula.

\begin{proposition}\label{pro:par_des_sorb}
Suppose that $char(F) = 0$ or $char(F)>n$.
With respect to geometric measures, we have 
\[
    \mathcal{SO}_{\gamma}^{\mathrm{G}}
    = \frac{\# \mathrm{G}(\kappa) \cdot q^{-\dim \mathrm{G}}}{\# \mathrm{G}_m(\kappa) \cdot q^{-\dim\mathrm{G}_m} \cdot \prod\limits_{i = k+1}^{k+k^\dagger}\# \operatorname{GL}_{l_i}(\kappa_{\widetilde{F}}) \cdot q^{-[\widetilde{F}:F]\cdot  l_i^2}}   \times
    \mathcal{SO}_{\gamma_m}^{\mathrm{G}_m} \times \prod\limits_{i = k+1}^{k+k^\dagger} \mathcal{SO}_{\gamma_i}^{\operatorname{GL}_{l_i,\mathfrak{o}_{\widetilde{F}}}}.
\]  
\end{proposition}
Since $\mathcal{SO}_{\gamma_i}^{\operatorname{GL}_{l_i,\mathfrak{o}_{\widetilde{F}}}}$ is well-studied in Part 1, we only need to work with $\mathcal{SO}_{\gamma_m}^{\mathrm{G}_m}$.
Thus, our study on $\mathcal{SO}_{\gamma}^{\mathrm{G}}$ is reduced to the case that $B(\gamma)=B(\gamma)^{irred}$. 

\subsection{Comparison between stable orbital integrals for $\mathfrak{gl}_{n, \mfo_E}$ and $\mathrm{Res}_{\mfo_E/\mfo}(\mathfrak{gl}_{n, \mfo_E})$}\label{subsection:weilrestglnoe}
In this subsection, we will not make any restriction on $char(\kappa)$. 
First we show that how Weil restrictions act with Chevalley morphism and stable orbital integrals, which will be used parabolic descent especially if $(\widetilde{F},\epsilon) = (E,1)$.
Since we treat  $\mathfrak{gl}_{n, \mfo_E}$ and $\mathrm{Res}_{\mfo_E/\mfo}(\mathfrak{gl}_{n, \mfo_E})$, we consider $\gamma$ to be a regular semisimple element in $\mathfrak{gl}_{n, \mfo_{E}}(\mfo_{E})$ only in this subsection.
Let  $\sog^{\mathfrak{gl}_{n, \mfo_{E}}}$ to be the stable orbital integral with respect to the geometric measure arising from  the Chevalley morphism $\varphi_{n, \mathfrak{o}_{E}} : \mathfrak{gl}_{n,\mfo_{E}} \longrightarrow  \mathbb{A}^n_{\mathfrak{o}_{E}}, \gamma \mapsto \chi_{\gamma}$ (cf. Section \ref{measure}).

The morphism $\varphi_{n, \mathfrak{o}_{E}}$ 
yields the  Chevalley morphism on $\mathrm{Res}_{\mfo_{E}/\mfo}(\mathfrak{gl}_{n, \mfo_{E}})$ described as follows:
\[
\mathrm{Res}_{\mfo_{E}/\mfo}(\varphi_{n, \mathfrak{o}_{E}}) : \mathrm{Res}_{\mfo_{E}/\mfo}(\mathfrak{gl}_{n,\mfo_{E}}) \longrightarrow  \mathrm{Res}_{\mfo_{E}/\mfo}(\mathbb{A}^n_{\mfo_{E}}). 
\]
Then we choose two volume forms 
 $\omega_{\mfu_{R_{E/F}(\mathfrak{gl}_n)}}$ and $\omega_{R_{E/F}(\mathbb{A}^n)}$ be nonzero, translation-invariant forms on $\mathrm{Res}_{E/F}(\mathfrak{gl}_{n,E})$ and $\mathrm{Res}_{E/F}(\mathbb{A}^n_{E})$, respectively, with normalizations
\[
\int_{\mathrm{Res}_{\mfo_{E}/\mfo}(\mathfrak{gl}_{n,\mfo_{E}})(\mfo)} |\omega_{\mfu_{R_{E/F}(\mathfrak{gl}_n)}}| = 1
\text{ and }
\int_{\mathrm{Res}_{\mfo_{E}/\mfo}(\mathbb{A}^n_{\mathfrak{o}_{E}})(\mfo)} |\omega_{R_{E/F}(\mathbb{A}^n)}| = 1.
\]
Note that $\mathrm{Res}_{\mfo_{E}/\mfo}(\varphi_{n, \mathfrak{o}_{E}})^{-1}(\chi_{\gamma})(\mfo)=\varphi_{n, \mathfrak{o}_{E}}^{-1}(\chi_{\gamma})(\mfo_F)$. 
Let $\sog^{\mathrm{Res}_{\mfo_E/\mfo}(\mathfrak{gl}_{n, \mfo_E})}$ be the volume of $\varphi_{n, \mathfrak{o}_{E}}^{-1}(\chi_{\gamma})(\mfo_F)$ with respect to the quotient of $\omega_{\mfu_{R_{E/F}(\mathfrak{gl}_n)}}$ by  $\omega_{R_{E/F}(\mathbb{A}^n)}$ as in Definition \ref{def:stableorbital}.

\begin{lemma}\label{lemma:comparisonofsogweilrest}
    We have the following identities of stable orbital integrals:
    \[
    \sog^{\mathfrak{gl}_{n, \mfo_{E}}}=\sog^{\mathrm{Res}_{\mfo_{E}/\mfo}(\mathfrak{gl}_{n, \mfo_{E}})}  ~~~~~~   \textit{    and   }   ~~~~~~~~~~~~~~
      \mathcal{SO}_{\gamma,d\mu}^{\mathfrak{gl}_{n, \mfo_{E}}}= \mathcal{SO}_{\gamma,d\mu}^{\mathrm{Res}_{\mfo_{E}/\mfo}(\mathfrak{gl}_{n, \mfo_{E}})}.
      \]
\end{lemma}

\begin{proof}
The second identity is clear by the definition given in Section \ref{subsubsection331}.
For the first identity, put
\[
 \underline{G}_{\gamma, \mfo_E}:=\varphi_{\gamma, \mfo_E}^{-1}(\chi_{\gamma}) ~~~~~~   \textit{    and   }   ~~~~~~~~~~~~~~
 R(\underline{G}_{\gamma}):=\left(\mathrm{Res}_{\mfo_{E}/\mfo}(\varphi_{n, \mathfrak{o}_{E}})\right)^{-1}(\chi_{\gamma}).
\]

We claim that 
\[
\left\{
\begin{array}{l}
\sog^{\mathfrak{gl}_{n, \mfo_{E}}}=\lim\limits_{N\rightarrow \infty} q^{-2N\cdot (n^2-n)}
\#\underline{G}_{\gamma, \mfo_E}(\mfo_E/\pi^N \mfo_E); \\
\sog^{\mathrm{Res}_{\mfo_{E}/\mfo}(\mathfrak{gl}_{n, \mfo_{E}})}=\lim\limits_{N\rightarrow \infty} q^{-N\cdot (2n^2-2n))}
\#R(\underline{G}_{\gamma, \mfo_E})(\mfo/\pi^N \mfo),
\end{array}
\right.
\]
where the limit stabilizes for $N$ sufficiently large.
Here, $n^2-n$ (resp. $2n^2-2n$) is the dimension of the generic fiber of $\underline{G}_{\gamma, \mfo_E}$ (resp. $R(\underline{G}_{\gamma, \mfo_E})$).
Then both right hand sides are equal, which completes the proof. 
We believe that this is a well-known fact but  could not find a proof in the literature. 
Thus we provide a sketch of the proof. 

We will treat  a more general situation. Consider a morphism $\varphi$ from an affine space $\mathbb{A}^m$ to another affine space $\mathbb{A}^n$ defined over an Henselian ring  $R$ having a finite residue field $\kappa_R$ such that the generic fiber of $\varphi$ at $\gamma(\in \mathbb{A}^n(R))$ is smooth.
Choose two volume forms $\omega_{\mathbb{A}^m}$ and $\omega_{\mathbb{A}^n}$ such that the volumes of $\mathbb{A}^m(R)$ and $\mathbb{A}^m(R)$ are $1$. 
Then we claim that  the volume of $\varphi^{-1}(\varphi(\gamma))(R)$, with respect to the quotient of $\omega_{\mathbb{A}^m}$ by $\omega_{\mathbb{A}^n}$, is described as the limit $\lim\limits_{N\rightarrow \infty} q^{-N\cdot (m-n)}
\#\varphi^{-1}(\varphi(\gamma))(R/\pi_R R)$, where $\pi_R$ is a uniformizer in $R$ and $q=\#\kappa_R$ . 

The proof is identical to that of \cite[Lemma 3.2]{CY}, if we prove that  the function defined on $U$ evaluating the volume of $\varphi^{-1}(a)(R)$ for $a\in U$ is locally constant, where $U$ is a small open neighborhood in the target space $\mathbb{A}^n$ containing $\varphi(\gamma)$.
This is the statement of   \cite[Theorem 7.6.1]{Igu}. 
\end{proof}
\begin{remark}
The above lemma holds 
whenever $E$ is replaced by any finite and unramified extension $F'$ of $F$ and whenever $\mathfrak{gl}_{n, \mfo_E}$ is replaced by any quasi-split Lie algebra over $\mfo_{F'}$. 
\end{remark}

\section{Reductions and stratification}\label{sec:reduction}
By Proposition \ref{pro:par_des_sorb} (or Proposition \ref{lem:par_des_sorb}) and Lemma \ref{lemma:comparisonofsogweilrest}, we may and do assume that $B(\gamma)=B(\gamma)^{irred}$.
See Equation (\ref{equation:irrfactor}) for the notion of $B(\gamma)$ and $B(\gamma)^{irred}$.
In this section we suppose that  $B(\gamma)=B(\gamma)^{irred}$ is a singleton, equivalently $\chi_{\gamma}$ is irreducible over $\widetilde{F}$ (or $F$),  equivalently $\widetilde{F}_{\chi_{\gamma}}$ is a field extension of $\widetilde{F}$ (or $F$).

\subsection{Reductions}\label{subsec: Reductions}

The goal of this section is to simply the form of the characteristic polynomial in order to apply smoothening process in Section \ref{sectypem}.
Recall from Sections \ref{sec:trans_herm}-\ref{desc_field} that  $\widetilde{F}_{\chi_{\gamma}}=  \widetilde{F}[x]/(\chi_{\gamma}(x))$  and $F^\sigma_{\chi_{\gamma}}=F[y]/(\psi(y))$. 
We refer to  Equation (\ref{equation:yxx2}) for the relation between $x$ and $y$. Note that
\[
\left\{
\begin{array}{l l}
\widetilde{F}_{\chi_{\gamma}}\cong E[x]/(\psi(x)), ~~~  x\mapsto x/\alpha  & \textit{if $(\widetilde{F}, \epsilon)=(E, 1)$};\\
\widetilde{F}_{\chi_{\gamma}}= F[x]/(\psi(x^2)) & \textit{if $(\widetilde{F}, \epsilon)=(F, -1)$}
 \end{array}\right.
\]
for $\alpha \in \mathfrak{o}_E^\times$ such that $\alpha + \sigma(\alpha) = 0$.
We introduce the following notations:
\begin{equation}
   R(\widetilde{F}):= \mfo_{\widetilde{F}}[x]/(\chi_{\gamma}(x)) ~~~~~~~~  \textit{    and    } ~~~~~~~~~~    R:=(R(\widetilde{F}))^{\sigma}=\mfo[y]/(\psi(y))
   \end{equation}
so that 
\[
\left\{
\begin{array}{l l}
R(\widetilde{F})\cong \mfo_E[x]/(\psi(x)), ~~~  x\mapsto x/\alpha  & \textit{if $(\widetilde{F}, \epsilon)=(E, 1)$};\\
R(\widetilde{F})= \mfo[x]/(\psi(x^2)) & \textit{if $(\widetilde{F}, \epsilon)=(F, -1)$}.
 \end{array}\right.
\]


Let $\widetilde{K}$ be the unique unramified extension of $\widetilde{F}$ in $\widetilde{F}_{\chi_{\gamma}}$ corresponding to the residue field $\kappa_{R(\widetilde{F})}$.
Then $\mfo_{\widetilde{K}}$ is contained in $R(\widetilde{F})$ by Lemma \ref{unramadd}.
Since $\widetilde{K}$ is Galois over $F$, the involution $\sigma$ preserves $\widetilde{K}$ and $\mathfrak{o}_{\widetilde{K}}$.
Let $\widetilde{l}= [\widetilde{K}:F]$.
Our setting is  visualized as follows:
\begin{equation}\label{diag2_unsp}
\begin{tikzcd}
F \arrow[r,"\widetilde{l}"] \arrow[d, no head]               & \widetilde{K} \arrow[rr,"2n/\widetilde{l}"] \arrow[d, no head]     &                        & \widetilde{F}_{\chi_{\gamma}} \arrow[ld, no head] \arrow[d, no head]    \\
\mfo \arrow[r, "\widetilde{l}"] \arrow[d, no head] & \mfo_{\widetilde{K}} \arrow[r, "2n/\widetilde{l}"] \arrow[d, no head]      & R(\widetilde{F}) \arrow[r] \arrow[ld, no head] & \mfo_{\widetilde{F}_{\chi_{\gamma}}} \arrow[d, no head] \\
\kappa \arrow[r,"\widetilde{l}"]                     & \kappa_{R(\widetilde{F})} \arrow[rr] &                      & \kappa_{\widetilde{F}_{\chi_{\gamma}}}.
\end{tikzcd}\end{equation}
Here the letter over an arrow stands for the rank of an extension as a free module.

Let $K^\sigma = (\widetilde{K})^\sigma$ so that $\mathfrak{o}_{K^\sigma} = (\mathfrak{o}_{\widetilde{K}})^\sigma$.
Then $K^\sigma$ is  the unramified extension of $F$ in $F^\sigma_{\chi_\gamma}$ corresponding to the residue field $\kappa_R$.
Let $l = [K^\sigma:F]$.
Then we have \[
\begin{cases}
[\widetilde{K}:K^\sigma]=2 \textit{ and } \widetilde{l} = 2l
 &\textit{if } \widetilde{K}\not\subset F_{\chi_\gamma}^\sigma; \\
  K^{\sigma}=\widetilde{K} \textit{ and } \widetilde{l} = l &\textit{if } \widetilde{K}\subset F^\sigma_{\chi_\gamma}.
\end{cases}
\]

Under the action of $\sigma$, Diagram (\ref{diag2_unsp}) descends to the diagram
\begin{equation}\label{diag_reduction}
    \begin{tikzcd}
F \arrow[r,"l"] \arrow[d, no head] &K^\sigma \arrow[rr,"n/l"]\arrow[d, no head]& & F^\sigma_{\chi_{\gamma}} \arrow[ld, no head] \arrow[d, no head]    \\
\mfo \arrow[r, "l"] \arrow[d, no head] & \mfo_{K^\sigma} \arrow[r, "n/l"] \arrow[d, no head]      & R \arrow[r] \arrow[ld, no head] & \mfo_{F^\sigma_{\chi_{\gamma}}} \arrow[d, no head] \\
\kappa \arrow[r,"l"]                     & \kappa_{R} \arrow[rr] &                      & \kappa_{F^\sigma_{\chi_{\gamma}}}.            
\end{tikzcd}
\end{equation}

Let
\[
K_\gamma (\supset \widetilde{F}) =
\begin{cases}
    \widetilde{K} &\textit{if } (\widetilde{F},\epsilon) = (E,1); \\
    K^\sigma &\textit{if } (\widetilde{F},\epsilon) = (F,-1).
\end{cases}
\]

\begin{remark}\label{rmk:reduction_cond}
\begin{enumerate}
    \item Suppose that $(\widetilde{F},\epsilon) = (E,1)$.
    Then  $[\widetilde{K}:K^\sigma] = 2$ and the integer $l$ is odd  since $E/F$ is unramified and 
    $\chi_{\gamma}(x)$ is irreducible over $E$.
    
    \item Suppose  that $(\widetilde{F},\epsilon) = (F,-1)$.
    If $K^\sigma = \widetilde{K}$ and $char(\kappa) > 2$, then $\overline{\psi}(x)$ and $\overline{\psi}(x^2)$ share the same irreducible factor so that  $\overline{\chi}_\gamma(x) = \overline{\psi}(x^2) = x^{2n}$ in $\kappa[x]$.
    For an instance, if $\widetilde{F}_{\chi_\gamma}/F^\sigma_{\chi_\gamma}$ is a ramified (quadratic) extension, then $\widetilde{K}\subset F^\sigma_{\chi_\gamma}$ and thus $K^{\sigma}=\widetilde{K}$.
\end{enumerate}
\end{remark}

\subsubsection{Comparison between unimodular hermitian forms on $\mfo_{\widetilde{F}}$ and on $\mfo_{K_\gamma}$}

In order to emphasize that $(L,h)$ is defined over $\mathfrak{o}_{\widetilde{F}}$, we temporarily use the notion $(L(\widetilde{F}), h_{\widetilde{F}})$ for $(L, h)$ in this subsection.

The $\mathfrak{o}_{\widetilde{F}}$-lattice $L$ is an $R(\widetilde{F})$-module so as to be  a free $\mfo_{K_\gamma}$-module since $\mfo_{\widetilde{F}} \subset \mfo_{K_\gamma} \subset R(\widetilde{F})$ (cf. Diagram (\ref{diag2_unsp})).
From now on, when we view $L$ as an $\mfo_{K_\gamma}$-lattice, we denote it by $L(K_\gamma)$.
The rank of  $L(K_\gamma)$ as an $\mfo_{K_\gamma}$-lattice is 
\[
\begin{cases}
  n/l  &\textit{if } (\widetilde{F},\epsilon) = (E,1); \\
   2n/l  &\textit{if } (\widetilde{F},\epsilon) = (F,-1).
\end{cases}
\]
Since $\gamma$ commutes with any element of $K_\gamma$,  $\gamma$ can also be considered as an element of  $\mathrm{End}_{\mfo_{K_\gamma}}(L(K_\gamma))$.
In this case,  we use the notion  $\gamma_K$ for $\gamma$.

Let $\chi_{\gamma_K}(x)\in \mfo_{K_\gamma}[x]$ be the minimal polynomial of $\gamma_K$ over $\mfo_{K_\gamma}$.
Then $\chi_{\gamma_K}(x)$ is the characteristic polynomial of $\gamma_K$ and is irreducible over $K_\gamma$.
This is due to the fact that $\widetilde{F}_{\chi_\gamma} \cong K_\gamma[x]/(\chi_{\gamma_K}(x))$ is a field extension of $K_{\gamma}$ whose degree is same with the rank of $L(K_\gamma)$ as a free  $\mfo_{K_\gamma}$-module. 

We consider a unimodular hermitian lattice  $h_{K_\gamma} : L(K_\gamma) \times L(K_\gamma) \to \mathfrak{o}_{K_\gamma}$.

\begin{lemma}\label{extherm}
The map 
\[
h'_{\widetilde{F}} : L(\widetilde{F}) \times L(\widetilde{F}) \to \mathfrak{o}_{\widetilde{F}}, \quad  (x,y) \mapsto \mathrm{Tr}_{K_\gamma/\widetilde{F}}(h_{K_\gamma}(x,y))
\]
is a unimodular hermitian form on $L(\widetilde{F})$.
\end{lemma}
\begin{proof}
For any $x,y \in L(\widetilde{F})$, the following computation shows that $h'_{\widetilde{F}}$ is a hermitian form: 
\[
h'_{\widetilde{F}}(y,x) =\mathrm{Tr}_{K_\gamma/\widetilde{F}}(h_{K_\gamma}(y,x)) = \mathrm{Tr}_{K_\gamma/
\widetilde{F}}(\epsilon \sigma(h_{K_\gamma}(x,y)))
=\epsilon \sigma( \mathrm{Tr}_{K_\gamma/\widetilde{F}}(h_{K_\gamma}(x,y)))
=\epsilon \sigma( h'_{\widetilde{F}}(x,y)).
\] 
It suffices to show that the lattice $L(\widetilde{F})$ is unimodular with respect to $h'_{\widetilde{F}}$. 
Since $K_\gamma/\widetilde{F}$ is unramified, \cite[Proposition III.7]{Ser} yields that $h'_{\widetilde{F}}(x,L(\widetilde{F}))\subset \mfo_{\widetilde{F}}$ if and only if $h_{K_\gamma}(x,L(K_\gamma))\subset \mfo_{K_\gamma}$.
The claim follows since $L(K_\gamma)$ is unimodular with respect to $h_{K_\gamma}$.
\end{proof}

It is well known that  two unimodular hermitian lattices are isometric (cf. \cite[Proposition 4.2]{GY})
 so that there exists an isometry 
 \[
 g \left(\in \mathrm{GL}_{n, \mfo_{\widetilde{F}}}(\mfo_{\widetilde{F}})\right) : (L(\widetilde{F}),h_{\widetilde{F}}) \longrightarrow (L(\widetilde{F}),h_{\widetilde{F}}'), \quad
 ~~~~ x \mapsto g\cdot x ~~~~ \textit{ for $x\in L(\widetilde{F})$}.
 \]

\begin{lemma}\label{red:equiv_herm}
  We have
    \[
 \mathcal{SO}_{\gamma} =   \mathcal{SO}_{g \gamma g^{-1}}.
    \]
Here, $\mathcal{SO}_{g \gamma g^{-1}}$ is the stable orbital integral associated to the hermitian lattice $(L(\widetilde{F}),h_{\widetilde{F}}')$. 
\end{lemma}
\begin{proof}
The proof is identical to that of Lemma \ref{constantlemma}. 
The isometry $g$ yields an isomorphism 
    \[
g: \mathfrak{g}(\mfo) \longrightarrow \mathfrak{g}'(\mfo), ~~~~~~~~~ \textit{    }  ~~~~~~~~~ \gamma \mapsto g \gamma g^{-1}.
    \]
    Here $\mathfrak{g}'$ is the Lie algebra associated to $(L(\widetilde{F}),h_{\widetilde{F}}')$.
Thus we have the following commutative diagram as schemes defined over $\mfo$:
    \[
    \begin{tikzcd}
    \mathfrak{g} \arrow[d, "\gamma \mapsto g \gamma g^{-1}"'] \arrow[rr, "\varphi_{n}"] &  & \mathcal{A}^{n}_{\mfo} \arrow[d, "identity"] \\
    \mathfrak{g}' \arrow[rr, "\varphi_{n}"]                                            &  & \mathcal{A}^{n}_{\mfo}
    \end{tikzcd}.
    \]
This completes the proof.
\end{proof}

By Lemma \ref{red:equiv_herm}, we may and do assume that 
$h_{\widetilde{F}}=\mathrm{Tr}_{K_\gamma/\widetilde{F}}(h_{K_\gamma})$.

\subsubsection{Interpretation of $\gamma$ in $\mathfrak{g}(\mfo)$ as an element in $\mathfrak{g}_{h_{K_{\gamma}}, \mfo_{K^\sigma}}(\mfo_{K^\sigma})$}

In this subsection, we temporarily use the following notations:
\[
\left\{
\begin{array}{l}
\textit{$\mathrm{G}_{h_{K_{\gamma}}, \mfo_{K^\sigma}}$ be the (smooth) reductive group scheme  defined over $\mfo_{K^\sigma}$ stabilizing $(L(K_\gamma),h_{K_{\gamma}})$};\\
\textit{$\mathfrak{g}_{h_{K_{\gamma}}, \mfo_{K^\sigma}}$ be the Lie algebra of $\mathrm{G}_{h_{K_{\gamma}}, \mfo_{K^\sigma}}$}.
 \end{array}\right.
\]

\begin{lemma}\label{lem:gammainnd} 
  The element $\gamma_K \in \mathrm{End}_{\mfo_{K_\gamma}}(L(K_\gamma))$ is a regular semisimple element of  $\mathfrak{g}_{h_{K_{\gamma}}, \mfo_{K^\sigma}}(\mfo_{K^\sigma})$.
\end{lemma}
\begin{proof}
Since $\chi_{\gamma_K}(x)$ divides $\chi_\gamma(x)$, which is separable, it suffices to show that $\gamma_K$ is an element of $\mathfrak{g}_{K^\sigma}(\mfo_{K^\sigma})$, equivalently  that $h_{K_\gamma}(\gamma_K x,y)+h_{K_\gamma}(x,\gamma_K y)=0$ for all $x,y \in L(K_\gamma)$.
Suppose there exist nonzero $x,y \in L(K_\gamma)$ such that $h_{K_\gamma}(\gamma_K x,y)+h_{K_\gamma}(x,\gamma_K y) = b$ for some $b\neq 0$ in $\mfo_{K_\gamma}$.
For any $a \in K_{\gamma}$, we have
\begin{align*}
    \mathrm{Tr}_{K_\gamma/\widetilde{F}}(a)
  &= \mathrm{Tr}_{K_\gamma/\widetilde{F}}(ab^{-1}b)
  =\mathrm{Tr}_{K_\gamma/\widetilde{F}}(ab^{-1}(h_{K_{\gamma}}(\gamma_K x,y)+h_{K_{\gamma}}(x,\gamma_K y)))\\
  &=\mathrm{Tr}_{K_\gamma/\widetilde{F}}(h_{K_{\gamma}}(\gamma_K x,ab^{-1}y)+h_{K_\gamma}(x,ab^{-1}\gamma_K y))
=h_{\widetilde{F}}'(\gamma_K x,ab^{-1}y)+h_{\widetilde{F}}'(x,\gamma_K ab^{-1}y)=0.
\end{align*}
Here $ab^{-1}\gamma_K y=\gamma_K ab^{-1}y$ since $\gamma_K$ is $K_\gamma$-linear.
This yields that $\mathrm{Tr}_{K_\gamma/\widetilde{F}}(K_{\gamma}) = 0$, which contradicts that $K_{\gamma}/\widetilde{F}$ is unramified.
\end{proof}

\subsubsection{Comparison between $\mathcal{SO}_{\gamma,d\mu}$ and $\mathcal{SO}_{\gamma_{K},d\mu_K}$}

We define the stable orbital integral $\mathcal{SO}_{\gamma_K,d\mu_K}$ of $\gamma_K$ on $\mathfrak{g}_{h_{K_{\gamma}}, \mfo_{K^\sigma}}$, where $d \mu_K$ is the quotient measure on $\mathrm{T}_{\gamma_K, K^\sigma} (K^\sigma)\backslash \mathrm{G}_{h_{K_\gamma},K^\sigma}(K^\sigma)$ as defined  in Section \ref{subsubsection331}.
Here $\mathrm{T}_{\gamma_K, K^\sigma}$ is the centralizer of $\gamma_K$ via the adjoint action of $\mathrm{G}_{h_{K_\gamma}, {K^\sigma}}$ (cf. the beginning of Section \ref{sec:cent}). 
We will prove that $\mathcal{SO}_{\gamma,d\mu} = \mathcal{SO}_{\gamma_{K},d\mu_K}$, by using lattice counting arguments.

Recall that $V = L(\widetilde{F})\otimes_{\mathfrak{o}_{\widetilde{F}}} \widetilde{F}=L(K_\gamma)\otimes_{\mfo{K_\gamma}}K_\gamma$.
When we view $V$ as an $\widetilde{F}$-vector space or as a $K_\gamma$-vector space, we use the notion $V(\widetilde{F})$ or $V(K_\gamma)$ respectively. 
We define
\[
\left\{
\begin{array}{l}
     X_{\gamma}:=\{M \subset V(\widetilde{F})| M \textit{ is a unimodular $\mfo_{\widetilde{F}}$-lattice under } h_{\widetilde{F}} \textit{ such that } \gamma M \subset M\};  \\
      X_{\gamma_K}:=\{N \subset V(K_\gamma)| N \textit{ is a unimodular $\mfo_{K_\gamma}$-lattice under } h_{K_{\gamma}} \textit{ such that } \gamma_K N \subset N \}.  \\
\end{array}\right.
\]

\begin{lemma}\label{lem:lattice_reduction}
We have the identification
 $X_{\gamma} = X_{\gamma_K}$.
\end{lemma}
\begin{proof}
    Firstly we claim that $X_{\gamma} \subset X_{\gamma_K}$.
    For $M\in X_\gamma$, since $\gamma M \subset M$, the lattice $M$ has a $R(\widetilde{F})$-module structure, where $x \in R(\widetilde{F}) = \mathfrak{o}_{\widetilde{F}}[x]/(\chi_\gamma(x))$ acts on $M$ by $\gamma$.
    Since $\mathfrak{o}_{K_\gamma} \subset R(\widetilde{F})$ (cf. diagram (\ref{diag2_unsp})), the lattice  $M$ has  an $\mathfrak{o}_{K_\gamma}$-module structure. 
Thus it suffices to show that $M$ is unidmodular with respect to $h_{K_\gamma}$. 
This is because \cite[Proposition III.7]{Ser} yields that $h_{K_{\gamma}}(x,M)\subset \mfo_{K_\gamma}$ if and only if $h_{\widetilde{F}}(x,M)\subset \mfo_{\widetilde
    {F}}$ for any $x \in V$.

    The opposite direction $X_{\gamma_K} \subset X_{\gamma}$ directly follows from Lemma \ref{extherm}.
\end{proof}

\begin{proposition}\label{orb:lattice}
 We have
$\mathcal{O}_{\gamma, d\mu}=\# X_{\gamma}$ and 
    $\mathcal{O}_{\gamma_K,d\mu_K} = \# X_{\gamma_K}$ and thus Lemma \ref{lem:lattice_reduction} yields
    \[
    \mathcal{O}_{\gamma,d\mu}=\mathcal{O}_{\gamma_{K},d\mu_K}.
    \]

\end{proposition}
\begin{proof}
The proof of the second claim is identical to that of the first one.
Thus we will only prove the first equality $\mathcal{O}_{\gamma, d\mu}=\# X_{R(F)}$.
We claim that the following map $f$ is bijective:
    \[
f:     X'_\gamma \left(:= \{[g] \in \mathrm{G}(F)/\mathrm{G}(\mathfrak{o}) \mid g^{-1} \gamma g \in \mathfrak{g}(\mathfrak{o})\}\right)  \longrightarrow  X_{\gamma}, \quad ~~~~~~ [g] \mapsto g L.
\]
    This is well-defined since  $gL$ is a unimodular $\mathfrak{o}_{\widetilde{F}}$-lattice for any $g \in \mathrm{G}(F)$ and $kL = L$ for any $k \in \mathrm{G}(\mathfrak{o})$.
    In order to show that $f$ is injective, suppose that $gL = g'L$ for some $g, g' \in G(F)$.
    Then we have $g'^{-1}g L = L$, which is equivalent to say that $g'^{-1}g \in \mathrm{G}(\mathfrak{o})$.

    For  the surjectivity, 
choose $M \in X_\gamma$.
Since $L$ and $M$  are unimodular lattices contained in the hermitian space $(V(\widetilde{F}), h_{\widetilde{F}})$, there exists an isometry  $g \in \mathrm{G}(F)$ such that $M = gL$.
    Note that  $(g^{-1}\gamma g)L \subset L$ since $M = gL$ is preserved by $\gamma$.
    This is equivalent that $g^{-1} \gamma g \in \mathfrak{g}(\mathfrak{o})$.

    {\cite[Lemma 3.2.8]{Yun16}\footnote{\cite{Yun16} assumes that $char(\kappa)$ is ``large enough'' and that $\mathrm{G}$ is split. However, \cite[Lemma 3.2.8]{Yun16} holds for an unramified group $\mathrm{G}$ without restriction on $char(F)$ and $char(\kappa)$.}}
    yields that $\mathcal{O}_{\gamma,d\mu} = \#(\Lambda_{\gamma} \backslash X'_{\gamma})$, where $\Lambda_{\gamma}$ is the complementary to the maximal compact subgroup $\mathrm{T}_{c}$ in $\mathrm{T}_{\gamma,F}(F)$.
    Since $\chi_{\gamma}(x)$ is irreducible, Lemma \ref{lem: Cent_unitary} and Corollary \ref{corollary:maxopencpt} yield that $\Lambda_\gamma = 0$.
    Thus we have $\mathcal{O}_{\gamma,d\mu} = \# X'_{\gamma} = \# X_{\gamma}$.
\end{proof}

We now extend this to the equality between the stable orbital integrals.

\begin{proposition}\label{red:result1}
We have
    $
\mathcal{SO}_{\gamma,d\mu}=\mathcal{SO}_{\gamma_{K},d\mu_K}.
    $
\end{proposition}
\begin{proof}
Recall that in Section \ref{subsubsection331111} the stable orbital integrals are defined by
\[
\mathcal{SO}_{\gamma,d\mu} = \sum_{\gamma'\sim \gamma} \mathcal{O}_{\gamma',d\mu}, \textit{ and } 
\mathcal{SO}_{\gamma_K,d\mu_K} = \sum_{\gamma'_K\sim \gamma_K} \mathcal{O}_{\gamma'_K,d\mu_K}.
\]
Here the second sum runs over the set of representatives $\gamma'_K$ of conjugacy classes within the stable conjugacy class of $\gamma_K$.
By Proposition \ref{orb:lattice}, it suffices to show that there is a bijection between the index sets $\{\gamma' \in \mathfrak{g}(F) \mid \gamma' \sim \gamma\}$ and $\{\gamma'_K \in \mathfrak{g}_{h_{K_\gamma},{K^\sigma}}(K^\sigma) \mid \gamma'_K \sim \gamma_K\}$.
By Section \ref{subsubsection331111} these index sets are bijectively identified with $\mathrm{ker}(H^1(F,\mathrm{T}_{\gamma,F}) \to H^1(F,\mathrm{G}_F))$.

\begin{enumerate}
    \item Suppose that $(\widetilde{F},\epsilon) = (E,1)$.
    Since $\chi_{\gamma}(x)$ is irreducible, 
 \cite[Proposition 3.6]{Xiao18} yields that the index set is trivial. The same for $\chi_{\gamma_K}(x)$.

    \item Suppose that $(\widetilde{F},\epsilon) = (F,-1)$.
Then $H^1(F,\operatorname{G}_{F}) = 1$ by \cite[Theorem 5.12.24]{Poonen}.
Since $\chi_{\gamma}(x)$ is irreducible, Lemma \ref{lem: Cent_unitary} and the local class field theory yield that
\[
H^1(F,\mathrm{T}_\gamma) \cong H^1(F_{\chi_\gamma}^\sigma,N^1_{\widetilde{F}_{\chi_\gamma}/F^\sigma_{\chi_\gamma}}\mathbb{G}_m) \cong 
(F^\sigma_{\chi_\gamma})^{\times} /N_{\widetilde{F}_{\chi_\gamma}/F^{\sigma}_{\chi_\gamma}}({\widetilde{F}_{\chi_\gamma}}^\times)
\cong \mathrm{Gal}(\widetilde{F}_{\chi_\gamma}/F^\sigma_{\chi_\gamma})
\cong \mathbb{Z}/2\mathbb{Z}.
\]
It follows that $\mathrm{ker}(H^1(F,\mathrm{T}_{\gamma,F}) \to H^1(F,\mathrm{G}_F)) \cong \mathbb{Z}/2\mathbb{Z}$.
Since $\chi_{\gamma_K}(x)$ is irreducible as well, the second index set is also bijective to  $\mathbb{Z}/2\mathbb{Z}$.
Therefore it suffices to show that if $\gamma$ and $\gamma'$ represent different conjugacy classes, so do $\gamma_K$ and $\gamma'_K$.
Suppose that there exists $g_K \in \mathrm{G}_{h_{K_\gamma},K^\sigma}(K^\sigma)$ such that $g_K^{-1}\cdot \gamma_K\cdot g_K = \gamma'_K$.
From the following computation $$h_{\widetilde{F}}(g_K\cdot x,g_K\cdot y) = \mathrm{Tr}_{K_\gamma/\widetilde{F}} (h_{K_{\gamma}}(g_K \cdot x, g_K \cdot y)) = \mathrm{Tr}_{K_\gamma/\widetilde{F}} (h_{K_{\gamma}}(x,y)) = h_{\widetilde{F}}(x,y),$$
we have  that $g_K \in \mathrm{G}(F)$. Thus $\gamma$ and $\gamma'$ represent the same conjugacy class in $\mathfrak{g}(F)$.
This completes the proof. 
\end{enumerate}
\end{proof}

We translate the above identity in terms of the geometric measure based on  Proposition \ref{prop:comparison_measures}.

\begin{corollary}\label{red:result2}
Suppose that $char(F) = 0$ or $char(F)>n$.
Then we have
\[
\mathcal{SO}_{\gamma}=
\begin{cases}
\frac{\#\mathrm{U}_n(\kappa)}{\#\mathrm{U}_{n/l}(\kappa_{R})\cdot q^{n^2-n^2/l}}\cdot \mathcal{SO}_{\gamma_K} &\textit{if } (\widetilde{F},\epsilon)=(E,1);\\
  \frac{\#\mathrm{Sp}_{2n}(\kappa)}{\#\mathrm{Sp}_{2n/l}(\kappa_{R})\cdot q^{2n^2-2n^2/l}}\cdot \mathcal{SO}_{\gamma_K} &\textit{if } (\widetilde{F},\epsilon) = (F,-1).
\end{cases}
\]
See Diagram (\ref{diag_reduction}) for the notion of $l$ and $\kappa_R$.
\end{corollary}

\subsubsection{Invariance under the translation}\label{sec:inv}  

By Corollary \ref{red:result2}, our study on $\mathcal{SO}_{\gamma}$ is reduced to the case that $[\kappa_{R}:\kappa]=1$.
This is equivalent that  the characteristic polynomial $\chi_\gamma(x)$ is of the following form:
\[
\overline{\chi}_\gamma(x) =
\begin{cases}
    (x-a)^n \textit{ for some } a \in \kappa_E &\textit{if } (\widetilde{F},\epsilon) = (E,1);\\
    (x^2-a)^n \textit{ for some } a \in \kappa &\textit{if } (\widetilde{F},\epsilon) = (F,-1).
\end{cases}
\]
In this subsection, we will prove  invariance of $\mathcal{SO}_{\gamma}$ under the translation  when $(\widetilde{F},\epsilon) = (E,1)$, as Lemma \ref{constantlemma} for $\mathfrak{gl}_n$ case. 
This invariance does not hold when   $(\widetilde{F},\epsilon) = (F,-1)$.
See Remark \ref{noinvariancerforsp} for further discussion.
\begin{lemma}
      Suppose that $(\widetilde{F}, \epsilon)=(E, 1)$.
      If $\overline{\chi}_{\gamma}(x)=(x-a)^{n}$ for some $a\in \kappa_{E}$, then $\sigma(a)=-a$.
\end{lemma}
\begin{proof}
    Equation (\ref{equation:chir}) yields that $(x-a)^{n}=(x+\sigma(a))^{n}$.
    Let $p = char(\kappa_E)$.
    If $(n,p)=1$, then we have $-a=\sigma(a)$ by comparing the coefficients of $x^{n-1}$ since $n$ is a unit in $\kappa_E$.

    If $(n,p)\neq 1$, then we can write $n=p^m\cdot l$ such that $(p,l)=1$.
    Then $(x-a)^n=(x^{p^m}+(-a)^{p^m})^l$ and $(x+\sigma(a))^n=(x^{p^m}+(\sigma(a))^{p^m})^l$.
    Then $(-a)^{p^m}=(\sigma(a))^{p^m}$ by comparing the coefficients of $x^{p^m(l-1)}$ since $l$ is a unit in $\kappa_E$. Since $0=(\sigma(a))^{p^m}-(-a)^{p^m}=(\sigma(a)-(-a))^{p^m}$, we conclude that $\sigma(a)=-a$.
\end{proof}

\begin{lemma}\label{lem:invarianct_translation}
    Suppose that $(\widetilde{F}, \epsilon)=(E, 1)$.
    We have that $\mathcal{SO}_{\gamma}=\mathcal{SO}_{\gamma+c}$ for a constant matrix $c\in\mathfrak{u}_n(\mfo)$.
\end{lemma}
\begin{proof}
The proof is similar to that of Lemma \ref{constantlemma}.
    We note that a constant matrix $c\in\mathfrak{u}_n(\mfo)$ satisfies $\sigma(c)=-c$ and that the morphism $\mathfrak{u}_n \rightarrow \mathfrak{u}_n,\, m \mapsto m+c$ is an isomorphism. 
    Then it suffices to show that there exists an automorphism $\iota_{c}$ of $\mathcal{A}_{\mfo}$, depending on the choice of $c$, which makes the following diagram commutes;
    \[
    \begin{tikzcd}
    \mathfrak{u}_n \arrow[d, "\gamma \mapsto \gamma +c"'] \arrow[rr, "\varphi_{n}"] &  & \mathcal{A}^{n}_{\mfo} \arrow[d, "\iota_{c}"] \\
    \mathfrak{u}_n \arrow[rr, "\varphi_{n}"]                                            &  & \mathcal{A}^{n}_{\mfo}.
    \end{tikzcd}
    \]
The construction of $\iota_c$ is similar to that in the proof of Lemma \ref{constantlemma}.
Only difference is to check that the coefficients of $f(x-c)$ are an element of $\mathcal{A}^{n}_{\mfo}$.
This is because $f(x-c)$ is also the characteristic polynomial of $m+c \in \mathfrak{u})_n$ if $f(x)$ is the characteristic polynomial of $m\in \mathfrak{u}_n$. 
\end{proof}

\begin{remark}\label{noinvariancerforsp}
    Suppose that $(\widetilde{F},\epsilon) = (F,-1)$.
    Then the only constant matrix $c\in \mathfrak{g}(\mfo)$ is the zero matrix.
In addition to this, two characteristic polynomials $\chi^1(x)$ and $\chi^2(x)$ with $\overline{\chi}^1(x)=x^{2n}$ and $\overline{\chi}^2(x)=(x^2-a)^n$ have different discriminants whenever $a\neq 0$ in $\kappa$. 

Therefore it seems to us that invariance under the translation (with respect to any element in $\mathfrak{g}$) does not work. 
On the other hand, if $K^\sigma = \widetilde{K}$ and $char(\kappa) > 2$, then  $\overline{\chi}_\gamma(x) = x^{2n}$ by Remark \ref{rmk:reduction_cond}.
\end{remark}

\subsection{Stratification}\label{sec:st_unsp}
The goal of this section is to construct a stratification of $\underline{G}_{\gamma}(\mfo)$.
Here $\underline{G}_{\gamma}(\mfo)=\{g\gamma g^{-1}\in \mfu(\mfo)|g\in \mathrm{G}_{F}(\bar{F})\}$ in Proposition \ref{prop:git}.
As in Section \ref{sec42st}, let us introduce the following notations and related facts:
\begin{itemize}
\item  For a sublattice  $M$ of $L$ of the same rank $\wn$, 
define the subset $\widetilde{\cL}(L,M)(\mfo)$ of $\mfu(\mfo)$ as follows:
\[
\widetilde{\cL}(L,M)(\mfo)=\{f:L \rightarrow M\in \mfu(\mfo)\}.
\]
Note that  the set $\widetilde{\cL}(L,M)(\mfo)$ is the intersection of  $\mfu(\mfo)$ and the set $\{f:L \rightarrow M\in \mathrm{End}_{\mfo_{\widetilde{F}}}(L)\}$.
Since the latter set is an open subset of $\mathrm{End}_{\mfo_{\widetilde{F}}}(L)$ in terms of $\pi$-adic topology, 
the set $\widetilde{\cL}(L,M)(\mfo)$ is  open in $\mfu(\mfo)$ as well.

\item  For a sublattice  $M$ of $L$ of the same rank $\wn$, 
define another subset $\cL(L,M)(\mfo)$ of $\mathfrak{g}(\mfo)$ as follows:
\[
\cL(L,M)(\mfo)=\{f:L \rightarrow M\in \mathfrak{g}(\mfo)\mid\textit{$f$ is surjective}\}.
\]
The set  $\cL(L,M)(\mfo)$ is also open in $\mfu(\mfo)$ since  the surjectivity is an open condition.

\item Define the orbit of $\gamma$ inside $\cL(L,M)(\mfo)$ as follows:
\[O_{\gamma, \cL(L,M)}=\uGr(\mfo)\cap \cL(L,M)(\mfo)=\{f\in \cL(L,M)(\mfo)\mid \vpi_n(f)=\vpi_n(\gamma)\}\] 
where the intersection is taken inside $\mfu(\mfo)$.
Then $O_{\gamma, \cL(L,M)}$ is an open subset of $\uGr(\mfo)$.

\end{itemize}

We now have the following stratification;
\begin{equation}\label{st1_unsp}
O_{\gamma, \mfu(\mfo)} \left(=\uGr(\mfo)\right)  =\bigsqcup_M O_{\gamma, \cL(L,M)}
\end{equation}
where $M$ runs over all $\mfo_{\widetilde{F}}$-sublattices of $L$ such that $[L:M]=d_n$.
Here, $[L:M]$ is the index as $\mfo_{\widetilde{F}}$-modules.
Recall from Notations given at the beginning of Part 2  that $d_n$ is defined to be $\mathrm{ord}(c_n)$, where $c_{n}$ is the constant term of  $\chi_{\gamma}(x)$.
Thus there are finitely many such $M$'s and so the above disjoint union is finite.

In this stratification, we emphasize that the set $O_{\gamma, \cL(L,M)}$ could be empty. 
This will be explained more carefully in Section \ref{sectypem}.

\begin{definition}\label{def:tjt}
Let $M$ be a sublattice of $L$ with rank $\wn$. In the following,  we define two types to $M$ as a sublattice of $L$ and as an hermitian lattice.
\begin{enumerate}
    \item{\textbf{Type of $M$, denoted by $\mathcal{T}(M)$}}
    
The type of $M$, denoted by $\T(M)$, is defined to be $(k_1, \cdots, k_{\wn-m})$, where $k_i\in \Z_{\geq 1}$ and $k_i\leq k_j$ for $i\leq j$, such that 
\[
L/M \cong \mfo_{\widetilde{F}}/\pi^{k_1}\mfo_{\widetilde{F}} \oplus \cdots \oplus \mfo_{\widetilde{F}}/\pi^{k_{\wn-m}}\mfo_{\widetilde{F}}.
\]
Here  $k_1+\cdots +k_{\wn-m}=[L:M]$.

    \item{\textbf{Jordan Type of $M$, denoted by $\JT(M)$}}
    
    For a hermitian lattice $(M, h)$, there always exists a basis $(e_1, \cdots, e_{\wn})$ of $M$ which satisfies the following condition, so that we define the Jordan type of $M$ as follows in each case.
    \begin{itemize}
        \item In the case that $(\widetilde{F},\epsilon)=(E,1)$, 
        with respect to $(e_1, \cdots, e_{n})$ the hermitian form $h$ is represented by a diagonal matrix with diagonal entries $u_i\pi^{t_i}$'s such that $t_i\leq t_{j}$ for $i<j$ (cf. Section 7 of \cite{Jac}).
       
        The Jordan type of $(M,h)$ is defined to be $(t_{l}, \cdots, t_n)$ such that $t_l>0$ and $t_{l-1}=0$.
        
        \item In the case that $(\widetilde{F},\epsilon)=(F,-1)$,
        with respect to $(e_1, \cdots, e_{2n})$ the hermitian form $h$ is represented by an orthogonal direct sum of lattices of $\mathbb{H}(u_i\pi^{t_i})$ such that $t_i\leq t_j$ for $i<j$ by \cite{GY}[Proposition 4.2 and Section 9.1], where $u_i\in \mfo^{\times}$.

        Here we denote by $\mathbb{H}(a)$ the rank-2 lattice $(\mfo\cdot f_1+\mfo \cdot f_2,\ \langle\cdot,\cdot\rangle)$ with the hermitian form such that $\langle f_1,f_1\rangle=\langle f_2,f_2\rangle=0$ and $\langle f_1,f_2\rangle=a$ for $a\in \mfo$.

        The Jordan type of $(M,h)$ is defined to be $(t_l,t_l, \cdots,t_n, t_n)$ such that $t_l>0$ and $t_{l-1}=0$.
    \end{itemize}
    Then  both $\T(M)$ and $\JT(M)$ are invariants of $M$. 
\end{enumerate}
\end{definition}
 Using $\mathcal{T}(M)$,  the above stratification in Equation (\ref{st1_unsp}) is refined as follows:
\begin{equation}\label{st2_unsp}
O_{\gamma, \mfu(\mfo)} \left(=\uGr(\mfo)\right) =\bigsqcup_{\substack{(k_1, \cdots, k_{\wn-m}),  \\ \sum k_i=d_n,\\ m\geq 0 }}
\left(\bigsqcup_{\substack{M:\T(M)= \\ (k_1, \cdots, k_{\wn-m})}} O_{\gamma, \cL(L,M)}\right).
\end{equation}
Here, the set $O_{\gamma, \cL(L,M)}$ could be empty for a large set of  $M$'s.

Let
\begin{equation}\label{equation:som}
\mathcal{SO}_{\gamma, M}=\int_{O_{\gamma, \cL(L,M)}}|\omega_{\chi_{\gamma}}^{\mathrm{ld}}|.
\end{equation}
\textit{ }

\subsection{Reduction when $\T(M)=(k_1, \cdots, k_{\wn})$}

Recall that we keep assuming $\chi_{\gamma}(x)$ of being irreducible over $\mfo_{\widetilde{F}}$.
Suppose that    $\T(M)=(k_1, \cdots, k_{\wn})$ so that $M\subset \pi^{k_1} L$. 
We define the following two items:
  \[
\left\{
  \begin{array}{l }
\chi_{\gamma^{(k_1)}}(x):=\chi_{\gamma}(\pi^{k_1} x)/\pi^{k_1\wn} \in \mfo_{\widetilde{F}}[x];\\
\textit{$\gamma^{(k_1)}\in\mfu(\mfo)$ whose characteristic polynomial is $\chi_{\gamma^{(k_1)}}(x)$.}
    \end{array} \right.
\]
More precisely, 
\[    \chi_{\gamma^{(k_1)}}(x)=
\left\{
\begin{array}{l l}
x^n+c_{1}^{(k_1)}x^{n-1}+\cdots +c_{n-1}^{(k_1)}x+c_n^{(k_1)} \textit{ with }c_i^{(k_1)}:=c_i/\pi^{ik_1}\in \mathcal{A}_i(\mfo)  & \textit{if $(\widetilde{F}, \epsilon)=(E, 1)$};\\
x^{2n}+c_1^{(k_1)}x^{2n-2}+\cdots + c_{n-1}^{(k_1)}x^2+c_n^{(k_1)} \textit{ with }c_i^{(k_1)}:=c_i/\pi^{2ik_1}\in \mfo   & \textit{if $(\widetilde{F}, \epsilon)=(F, -1)$}.
 \end{array}\right.
\]
Here the Newton polygon of an irreducible polynomial $\chi_{\gamma}(x)$  confirms $\chi_{\gamma^{(k_1)}}(x)\in \mfo_{\widetilde{F}}[x]$.

\begin{proposition}\label{propendred_unsp}
If $\T(M)=(k_1, \cdots, k_{\wn})$, then we have
\[    \mathcal{SO}_{\gamma, M}=
\left\{
\begin{array}{l l}
q^{-k_1\cdot \frac{n(n-1)}{2}}\mathcal{SO}_{\gamma^{(k_1)}, \pi^{-k_1}M}  & \textit{if $(\widetilde{F}, \epsilon)=(E, 1)$};\\
q^{-k_1\cdot n^2}\mathcal{SO}_{\gamma^{(k_1)}, \pi^{-k_1}M}  & \textit{if $(\widetilde{F}, \epsilon)=(F, -1)$}.
 \end{array}\right.
\]
Here $\T(\pi^{-k_1}M)=(0, k_2-k_1, \cdots, k_{\wn}-k_1)$. 
\end{proposition}
The proof is identical to that of Proposition \ref{propendred}, which is based on the proof of  Propositions \ref{propredend},  so that we skip it.

\section{Geometric formulation of $\mathcal{SO}_{\gamma}$}\label{sectypem}
In this section, we will suppose that $\chi_\gamma(x)$ is irreducible over $\mfo_{\widetilde{F}}$ and that $\overline{\chi}_{\gamma}(x)=x^{\wn}$ with $\wn \geq 2$. 
Here, when $(\widetilde{F},\epsilon)=(E,1)$, the case that  $\chi_\gamma(x)$ is irreducible over $\mfo_{\widetilde{F}}$ is reduced to this assumption by Proposition \ref{red:result1} (or Corollary \ref{red:result2}) and Lemma \ref{lem:invarianct_translation}. 
The goal of this section is to provide a precise formula for the volume of \begin{equation}\label{mtypestrata}
\bigsqcup_{M:\T(M)=(d_n)} O_{\gamma, \cL(L,M)}\end{equation} appeared in Equation (\ref{st2_unsp}) under the assumption that $\chi_\gamma$ is irreducible over $\widetilde{F}$.
Thus we will concentrate on the stratum whose type is $(d_n)$.


\subsection{Jordan type of a lattice $M$}\label{section:jordan}
Let $M$ be a sublattice of $L$ such that $\mathcal{T}(M)=(d_n)$ with $d_n>0$. 
We choose a basis $(e_1, \cdots, e_{\wn})$  of $L$ such that $(e_1, \cdots, e_{\wn-1}, \pi^{d_n}e_{\wn})$ is a basis of $M$.
With respect to this basis, the Gram matrix of $h$ and the matrix representation of $X\in\cL(L,M)(\mfo_{\tilde{F}})$    are given as follows:
\begin{equation}\label{matrixforx}
h=\left(
\begin{array}{c c c c}
     a_{1,1}&  a_{1,2}&\cdots & a_{1,\wn}\\
     \epsilon\sigma(a_{1,2})&  a_{2,2}&\cdots & a_{2,\wn}\\
     \vdots&  \vdots&\ddots & \vdots\\
     \epsilon\sigma(a_{1,\wn})&  \epsilon\sigma(a_{2,\wn})&\cdots & a_{\wn,\wn}
\end{array}
\right)\text{ and }
X=
\left(
\begin{array}{c c c c}
     x_{1,1}&  x_{1,2}&\cdots & x_{1,\wn}\\
     x_{2,1}&  x_{2,2}&\cdots & x_{2,\wn}\\
     \vdots&  \vdots&\ddots & \vdots\\
     \pi^{d_n}x_{\wn, 1}&  \pi^{d_n}x_{\wn, 2}&\cdots & \pi^{d_n}x_{\wn,\wn}
\end{array}
\right)
\in \mathrm{M}_{\wn}(\mfo_{\tilde{F}}) 
\end{equation}
such that $\epsilon\sigma(a_{i,i})=a_{i,i}$ for $1\leq i\leq \wn$ and such that $hX+\sigma({}^tX)h=0$.

\begin{lemma}\label{lemma61}
Suppose that  $O_{\gamma, \cL(L,M)}$ is non-empty. Then the rank of the Gram matrix $h|_M$ modulo $\pi$ is $\wn-2$.
\end{lemma}

\begin{proof}
For $X\in O_{\gamma,\cL(L,M)}$, we denote by $X_{\wn-1,\wn-1}$ the  left and upper submatrix of $X$ of size $(\wn-1)\times (\wn-1)$ in Equation (\ref{matrixforx}).
Since  $\overline{\chi}_\gamma(x)=x^{\wn}$,  the characteristic polynomial of $\overline{X}_{\wn-1,\wn-1}$ is $x^{\wn-1}$.
We note that the rank of $\overline{X}_{\wn-1,\wn-1}$ is at least $\wn-2$ by Lemma \ref{lem411}.
Therefore $\overline{X}_{\wn-1,\wn-1}$ is conjugate to the following Jordan canonical form
\[J_{\wn}:=
\begin{array}{c c c c c}
     \begin{pmatrix}
        0 &1 &\cdots &0 \\
        \vdots &\vdots &\ddots &\vdots \\
        0 &0 &\cdots &1 \\
        0 &0 &\cdots &0 \\
    \end{pmatrix}
\end{array}
\]by an element in $\mathrm{GL}_{\wn-1}(\kappa_{\widetilde{F}})$.
By choosing another basis for the $\mfo_{\widetilde{F}}$-span of $(e_1, \cdots, e_{\wn-1})$ (if necessary),  we may and do assume that $\overline{X}_{\wn-1,\wn-1} = J_{\wn}$.
Then  $\bar{x}_{i,j}=\left\{
    \begin{array}{l l}
         1&\textit{for $i=j-1$}  \\
         0&\textit{otherwise} 
    \end{array}
    \right.$ 
    for $1\leq i,j\leq \wn-1$ and  $x_{\wn-1,\wn}\in \mfo_{\widetilde{F}}^{\times}$ since $\mathrm{ord}(\mathrm{det}(X))=d_n$.

Using the matrix representations of $h$ and $X$ given in Equation (\ref{matrixforx}),  the $(i,j)$-entry of the equation $\overline{hX}+ \sigma({}^t\overline{X})\overline{h}=0$ over $\kappa_{\widetilde{F}}$, for  $1\leq i< j \leq \wn$, is given as follows:
\begin{equation}\label{ijcomponent}
(i,j): ~~~ \sum_{k=1}^{i-1}\epsilon\sigma(\bar{a}_{k,i})\bar{x}_{k,j}+\sum_{k=i}^{\wn-1}\bar{a}_{i,k}\bar{x}_{k,j} + \sum_{k=1}^{j-1}\bar{a}_{k,j}\sigma(\bar{x}_{k,i})+\sum_{k=j}^{\wn-1}\epsilon\sigma(\bar{a}_{j,k}\bar{x}_{k,i})=0. \end{equation}
 We claim that 
\[
\bar{h}=\begin{pmatrix}
    0&\cdots&0&\ast\\
    \vdots&&\iddots&\vdots\\
    0&\iddots&&\vdots\\
    \ast&\cdots&\cdots&\ast
\end{pmatrix}.
\]
Since $h$ defines a unimodular hermitian form on $L$,  the entries on anti-diagonal of $h$ are units  in $\mfo_{\widetilde{F}}^{\times}$. 
We note that $(e_1,\cdots,e_{\wn-1},\pi^{d_n}e_{\wn})$ is a basis of $M$. 
Then the Gram matrix of $\bar{h}|_M$ has a rank of $\wn-2$, which completes the proof.

To prove the claim, we use an  induction on $i$. 
\begin{enumerate}
    \item When $i=1$,  Equation (\ref{ijcomponent}) for $(1,j+1)$ with $1\leq j \leq \wn-2$ yields that      $\bar{a}_{1j}=0$, and that for $(1,\wn)$  yields that 
    $\bar{a}_{1,\wn-1} \bar{x}_{\wn-1,\wn}=0$.
Since  $x_{\wn-1,\wn}\in\mfo_{\widetilde{F}}^{\times}$, we have $\bar{a}_{1,\wn-1}=0$.

\item When $1\leq i\leq \lfloor\frac{\wn}{2}\rfloor-1$, suppose that $\bar{a}_{ij}=0$ for $i\leq j \leq \wn-i$ (so that $\bar{a}_{ji}=0$ by the symmetry of $h$). 
Then we claim that $\bar{a}_{i+1, i+l}=0$ for $1\leq l\leq \wn-2i-1$. 
In Equation (\ref{ijcomponent}) for $(i+1,i+1+l)$,
we have
\begin{align*}
\textit{the first two sums in (\ref{ijcomponent})}= &\sum_{k=1}^{i}\sigma(\bar{a}_{k,i+1})\bar{x}_{k,i+1+l}+\sum_{k=i+1}^{\wn-1}\bar{a}_{i+1,k}\bar{x}_{k,i+1+l}=\bar{a}_{i+1,i+l};\\
\textit{the last two sums in (\ref{ijcomponent})}=&\sum_{k=1}^{i+l}\bar{a}_{k,i+1+l}\sigma(\bar{x}_{k,i+1})+ \sum_{k=i+1+l}^{\wn-1}\sigma(\bar{a}_{i+1+l,k}\bar{x}_{k,i+1})=\bar{a}_{i,i+1+l}=0.
\end{align*}
We then have $\bar{a}_{i+1,i+l}=0$.
\end{enumerate}
By the symmetry of $h$, we obtain the desired claim. 
\end{proof}

In order to ease the computation, we rewrite the matrix $X$ as the following block matrix
\begin{equation}\label{eq:form_X}
    X =
    \begin{pmatrix}
        X_{1,1}  &X_{1,2} &X_{1,3} \\
        X_{2,1} &X_{2,2} &X_{2,3} \\
        \pi^{d_{n}} X_{3,1} &\pi^{d_{n}} X_{3,2} & \pi^{d_{n}} X_{3,3}
    \end{pmatrix},    
\end{equation}
where the size of $X_{2,2}$ is $(\wn-2)\times (\wn-2)$ and so on.
We will simplify the matrix form of $h$ by choosing a suitable $\mathfrak{o}_{\widetilde{F}}$-basis of $L$.

\begin{lemma}\label{modif_matr}
Suppose that  $O_{\gamma, \cL(L,M)}$ is non-empty. 
Then there exists an $\mfo_{\widetilde{F}}$-basis of $L$ such that 
$h$ and $X \in O_{\gamma, \cL(L,M)}$ are represented by the following matrices respectively:
\begin{equation}\label{eq:form_X_1}
h=
    \begin{pmatrix}
            a&0&u\\
            0&h_{2,2}&0\\
            \epsilon\sigma(u)&0&b
    \end{pmatrix}
        \textit{ and }
        X =
    \begin{pmatrix}
        X_{1,1}  &X_{1,2} &X_{1,3} \\
        X_{2,1} &X_{2,2} &X_{2,3} \\
        \pi^{d_{n}} X_{3,1} &\pi^{d_{n}} X_{3,2} & \pi^{d_{n}} X_{3,3}
    \end{pmatrix}, \textit{ where}
\end{equation}
 $h_{2,2}$ is a unimodular hermitian matrix of size $(\wn-2)\times (\wn-2)$,  $u\in\mfo^{\times}_{\widetilde{F}}$, and $a,b \in \mfo_{\widetilde{F}}$ such that
\[\left\{\begin{array}{l l}
    a=\pi^{\xi} \textit{ with } \xi>0 & \textit{if $(\widetilde{F},\epsilon)=(E,1)$}; \\
    a=b=0 & \textit{if $(\widetilde{F},\epsilon)=(F,-1)$}.
\end{array}\right.
\]
\end{lemma}
\begin{proof}
We start with a basis $(e_1, \cdots, e_{\wn})$ of $L$ which is chosen in Equation (\ref{matrixforx}).
Let $L'$ be the sublattice  of $L$ spanned by $(e_1, \cdots, e_{\widetilde{n}-1})$ so that the rank of $L'$ is $\wn-1$.
Our idea is to choose a suitable basis for $L'$ so that the matrix form for $X$ in Equation (\ref{eq:form_X}) is unchanged.

The restriction of $h$ to $L'$, which we denote by $h'$, is also a hermitian lattice (possibly degenerate).
Lemma \ref{lemma61} yields that $\overline{h'}$ is of rank $\wn-2$.
Thus we can choose a  basis $(e_1', \cdots, e'_{\widetilde{n}-1})$ for $L'$ with respect to which  $h$ is represented by 
   $h =
    \begin{pmatrix}
        a & c' & u \\
        \epsilon\sigma(c') & h_{2,2}& c \\
        \epsilon\sigma(u) &\epsilon\sigma(^t c)  &b' 
    \end{pmatrix}$, 
where $h_{2,2}$ is unimodular of size $(\wn-2)\times (\wn-2)$ and $\overline{a}=0, \overline{c'}=0$.
Adding to  $e_1'$ a suitable linear combination of $(e_2', \cdots, e'_{\widetilde{n}-1})$, we may and do assume that $c'=0$ so that  $h =
    \begin{pmatrix}
        a & 0 & u \\
        0 & h_{2,2}& c \\
        \epsilon\sigma(u) &\epsilon\sigma(^t c)  &b' 
    \end{pmatrix}$. 
Since $h$ is unimodular, $u$ should be a unit. 

To eliminate $c$, we choose another basis for $L$ consisting of  columns vectors  of the matrix $g = \begin{pmatrix}
        1 &0 &0 \\
        0 &Id_{\wn-2} &-h_{2,2}^{-1}c \\
        0 &0 &1 \\
    \end{pmatrix}
    \in \operatorname{GL}_{\wn}(\mathfrak{o}_{\widetilde{F}})$.
Then $$\sigma({}^t g) \cdot h \cdot g
    = \begin{pmatrix}
    a & 0 & u \\
    0 & h_{2,2} & 0 \\
    \epsilon\sigma(u) &0 & b
    \end{pmatrix} \textit{  and  } g^{-1} \cdot X \cdot g = \begin{pmatrix}
        X'_{1,1}  &X'_{1,2} &X'_{1,3} \\
        X'_{2,1} &X'_{2,2} &X'_{2,3} \\
        \pi^{d_{n}} X'_{3,1} &\pi^{d_{n}} X'_{3,2} & \pi^{d_{n}} X'_{3,3}
    \end{pmatrix}.$$
If $(\widetilde{F},\epsilon) = (F,-1)$, then both $a$ and $b$ are $0$.
If $(\widetilde{F},\epsilon) = (E,1)$, then multiplying a suitable unit in $\mfo_E$ to $e_1'$ guarantees that $a = \pi^\xi$ for $\xi>0$.
    This completes the proof.
\end{proof}

\begin{lemma}\label{lem:compare_xi}
    Suppose that $(\widetilde{F},\epsilon)=(E,1)$ and that  $O_{\gamma, \cL(L,M)}$ is non-empty. 
    Then the exponential order of the $(1,1)$-entry of $h$ in Equation (\ref{eq:form_X_1}), denoted by $\xi$, is described as follows:
    \[
    \begin{cases}
    \xi = d_{n-1} &\textit{if } d_{n-1}<d_n; \\
    \xi \geq d_{n} &\textit{if }d_{n-1} \geq d_n.
    \end{cases}
    \]
\end{lemma}
\begin{proof}
For $X\in O_{\gamma,\cL(L,M)}$, the  equation
    $h X + \sigma({}^tX)h = 0$ in the setting of Equation (\ref{eq:form_X_1}) yields the following matrix equations:
\begin{equation}\label{eq:eqs_about_X_1}
\left\{
\begin{array}{l l}
        \pi^{\xi}(X_{1,1} + \sigma(X_{1,1})) = -\pi^{d_{n}}(uX_{3,1} + \sigma(uX_{3,1}))  & \textit{by the $1\times 1$-block};\\
        \pi^{\xi} X_{1,2} + \pi^{d_{n}}u X_{3,2} = -\sigma({}^t X_{2,1}) h_{2,2}  & \textit{by the $1\times 2$-block};\\
        \pi^{\xi} X_{1,3} + \pi^{d_{n}}(u X_{3,3} +  b \sigma(X_{3,1})) = - u \sigma(X_{1,1}) & \textit{by the $1\times 3$-block};\\
        h_{2,2} X_{2,2} + \sigma({}^{t}X_{2,2}) h_{2,2} = 0& \textit{by the $2\times 2$-block};\\
        h_{2,2} X_{2,3} + \pi^{d_{n}} b \sigma({}^{t}X_{3,2}) = -u \sigma({}^{t}X_{1,2})  & \textit{by the $2\times 3$-block};\\
        (\sigma(u)X_{1,3} + u \sigma(X_{1,3})) = -\pi^{d_{n}}b( X_{3,3} +  \sigma(X_{3,3}))& \textit{by the $3\times 3$-block}.
 \end{array}\right.
\end{equation}
We claim that if $\xi < d_n$, then $\xi = d_{n-1}$.
Suppose that the claim is true. Then the desired results are proved as follows:
\begin{enumerate}
    \item {The case $d_{n-1}<d_n$}
    
If $\xi \geq d_{n}$, then the second and the third of Equation (\ref{eq:eqs_about_X_1}) yield that each entry of $X_{1,1}, X_{2,1}$ is contained in the ideal $(\pi^{d_{n}})$.
By computing the sum of all principal minors of $X$ of size $n-1$, we have $d_{n-1} \geq d_{n}$, which is a contradiction to the assumption that $d_{n-1} < d_{n}$.

Thus $\xi < d_{n}$, which directly yields that $\xi = d_{n-1}$ by the claim.

    \item{The case $d_{n-1} \geq d_n$}
    
    If $\xi<d_n$, then  the claim yields that $\xi = d_{n-1}$, which is a contradiction to the assumption $d_{n-1} \geq d_n$.
    Thus  $\xi \geq d_{n}$.
\end{enumerate}

It remains to prove the claim and thus we suppose that $\xi<d_n$.
Using the second and the third of Equation  (\ref{eq:eqs_about_X_1}), we have that each entry of $X_{1,1}, X_{2,1}$ is contained in the ideal $(\pi^\xi)$.
Thus we replace $X_{1,1},X_{2,1}$ with $\pi^\xi X_{1,1}, \pi^\xi X_{2,1}$ respectively so that one may and do rewrite Equations (\ref{eq:form_X_1}) and (\ref{eq:eqs_about_X_1}) by
\begin{equation} \label{eq:form_X_dn-1<dn}
      h=  \begin{pmatrix}
        \pi^\xi &0 &u \\
        0 &h_{2,2} &0 \\
        \sigma(u) &0 & b
    \end{pmatrix},
    \quad
    X =
    \begin{pmatrix}
        \pi^\xi X_{1,1}  &X_{1,2} &X_{1,3} \\
        \pi^\xi X_{2,1} &X_{2,2} &X_{2,3} \\
        \pi^{d_n} X_{3,1} &\pi^{d_n} X_{3,2} & \pi^{d_n} X_{3,3}
    \end{pmatrix} 
\end{equation}
with a unimodular hermitian matrix $h_{2,2}$, $u \in \mathfrak{o}_E^\times$, and
\begin{equation}\label{eq:eqs_about_X_dn-1<dn_1}
\left\{
\begin{array}{l l}
        \pi^{2\xi}(X_{1,1} + \sigma(X_{1,1})) = -\pi^{d_n}(uX_{3,1} + \sigma(uX_{3,1}))  & \textit{by the $1\times 1$-block};\\
        X_{1,2} + \pi^{d_n-\xi}u X_{3,2} = -\sigma({}^t X_{2,1}) h_{2,2}   & \textit{by the $1\times 2$-block};\\
        X_{1,3} + \pi^{d_n - \xi}(u X_{3,3} +  b \sigma(X_{3,1})) = - u \sigma(X_{1,1})   & \textit{by the $1\times 3$-block};\\ 
        h_{2,2} X_{2,2} + \sigma({}^{t}X_{2,2}) h_{2,2} = 0 & \textit{by the $2\times 2$-block};\\
        h_{2,2} X_{2,3} + \pi^{d_{n}} b \sigma({}^{t}X_{3,2}) = -u \sigma({}^{t}X_{1,2})& \textit{by the $2\times 3$-block};\\
        (\sigma(u)X_{1,3} + u \sigma(X_{1,3})) = -\pi^{d_{n}}(b X_{3,3} + b \sigma(X_{3,3}))& \textit{by the $3\times 3$-block}.
  \end{array}\right.
\end{equation}
Recall that  the matrix $X_{ij}$ modulo $\pi$ is denoted 
by $\overline{X}_{ij}$ (cf. Notations in Part \ref{part1}).
If the matrix $\begin{pmatrix}
        \overline{X}_{1,1} &\overline{X}_{1,2} \\
        \overline{X}_{2,1} &\overline{X}_{2,2}
    \end{pmatrix}$ is invertible over $\kappa_E$, then by summing all the principal minors of $X$ of size $n-1$ we have that $\xi = d_{n-1}$.

Therefore, it suffices to show that the matrix $\begin{pmatrix}
        \overline{X}_{1,1} &\overline{X}_{1,2} \\
        \overline{X}_{2,1} &\overline{X}_{2,2}
    \end{pmatrix}$ is invertible.    
Recall from Lemma \ref{modif_matr} that $\xi>0$. Since $\overline{\chi}_{\gamma}(x)=x^n\in \kappa_E[x]$, the characteristic polynomial of $\overline{X}_{2,2}$ is $x^{n-2}$, by the description of $X$ in Equation (\ref{eq:form_X_dn-1<dn}).
Thus the rank of $\overline{X}_{2,2}$ is at most $n-3$.
On the other hand, since $X$ is an element of $O_{\gamma, \cL(L,M)}$ which is a subset of $\cL(L,M)(\mfo)$, 
the matrix $    \begin{pmatrix}
        0  &\overline{X}_{1,2} &\overline{X}_{1,3} \\
        0 &\overline{X}_{2,2} &\overline{X}_{2,3} \\
        \overline{X}_{3,1} &\overline{X}_{3,2} & \overline{X}_{3,3}
    \end{pmatrix}$ is invertible over $\kappa_E$. 
This shows that the rank of  $\overline{X}_{2,2}$  is at least $n-3$ and thus should be $n-3$.
The invertibility of this matrix  also yields that the  matrix $    \begin{pmatrix}
        \overline{X}_{1,2} \\
        \overline{X}_{2,2}
    \end{pmatrix}$  is of rank $n-2$.

Now suppose that  
$\begin{pmatrix}
        \overline{X}_{1,1} &\overline{X}_{1,2} \\
        \overline{X}_{2,1} &\overline{X}_{2,2}
    \end{pmatrix}$  is not invertible.
    Then its rank should be $n-2$ 
since the rank of  $    \begin{pmatrix}
        \overline{X}_{1,2} \\
        \overline{X}_{2,2}
    \end{pmatrix}$ is $n-2$.
This yields that   $\begin{pmatrix}
        \overline{X}_{1,1} \\
        \overline{X}_{2,1}
    \end{pmatrix}$  is a linear combination of the column vectors of  $    \begin{pmatrix}
        \overline{X}_{1,2} \\
        \overline{X}_{2,2}
    \end{pmatrix}$. 
In particular,  $\overline{X}_{2,1}$  is a linear combination  of column vectors of 
 $\overline{X}_{2,2}$.

On the other hand, the second and the fourth of Equation (\ref{eq:eqs_about_X_dn-1<dn_1}) modulo $\pi$ yield that $$\overline{X}_{1,2} = - \sigma({}^t \overline{X}_{2,1}) \overline{h}_{2,2} ~~~~~~~~~~ \textit{ 
     and         } ~~~~~~~~~~~~ \overline{h}_{2,2}\overline{X}_{2,2} + \sigma({}^t \overline{X}_{2,2}) \overline{h}_{2,2} = 0$$ since $\xi<d_n$. 
These two equations imply that  $\overline{X}_{1,2}$ is a linear combination of row vectors of 
 $\overline{X}_{2,2}$.
 This is a contradiction since  the rank of  $    \begin{pmatrix}
        \overline{X}_{1,2} \\
        \overline{X}_{2,2}
    \end{pmatrix}$   is $n-2$, whereas  $\overline{X}_{2,2}$ is  of rank $n-3$.
This completes the proof of the claim.
\end{proof}

We finally have the following formula for the Jordan type of $M$.

\begin{proposition} \label{lem:Jordan_of_type_m}
For a lattice $M$ of type $(d_n)$ such that $O_{\gamma, \cL(L,M)}$ is non-empty, we have
\[
\mathcal{JT}(M)=(\widetilde{d}_{n-1},2d_n-\widetilde{d}_{n-1}),
\]
where $\widetilde{d}_{n-1}$  is an integer defined by
\begin{equation}\label{eq:def_widetilde_d_n-1}
    \widetilde{d}_{n-1} =
\begin{cases}
    d_{n-1} &\textit{if } (\widetilde{F},\epsilon)=(E,1) \textit{ and } d_{n-1}<d_n; \\
    d_n &\textit{if } (\widetilde{F},\epsilon)=(E,1) \textit{ and }d_{n-1} \geq d_{n}, \textit{ or } (\widetilde{F},\epsilon)=(F,-1).
\end{cases}    
\end{equation}
\end{proposition}

\begin{proof}
If we choose an $\mathfrak{o}_{\widetilde{F}}$-basis $(e_1,\cdots,e_{\wn-1},e_{\wn})$ of $L$ such that $h$ is represented as in Equation (\ref{eq:form_X_1}), then we have
\[
h|_M = \begin{pmatrix}
 a &0 &\pi^{d_n}u \\
 0 &h_{2,2} &0 \\
 \pi^{d_n}\sigma(u) &0 &\pi^{2d_n}b
\end{pmatrix}
\]
with respect to a basis $(e_1,\cdots,e_{\wn-1},\pi^{d_n} e_{\wn})$ of $M$.
In the case that $(\widetilde{F},\epsilon)=(F,-1)$, we have $a=b=0$ and $u\in \mfo_{\widetilde{F}}^{\times}$ by Lemma \ref{modif_matr}. Therefore the Jordan type of $M$ is $(d_n,d_n)$.

Now suppose that $(\widetilde{F},\epsilon)=(E,1)$ so that $a=\pi^\xi$ with $\xi>0$. 
In order to compute the Jordan type of $h|_M$, we will choose an element $g \in \operatorname{GL}_{\wn}(\mathfrak{o}_{\widetilde{F}})$ such that $\sigma({}^t g) \cdot h|_{M} \cdot g$ is diagonal.

\begin{enumerate}
    \item The case  $d_{n-1}<d_n$

By Lemma \ref{lem:compare_xi} we have $\xi = d_{n-1}$ .
For $g=\begin{pmatrix}
1&0&-\pi^{d_n-d_{n-1}}u\\
    0&Id_{n-2}&0\\
    0&0&1
\end{pmatrix}\in \mathfrak{gl}_n(\mfo_E)$, we have
\[
\sigma({}^tg)\cdot h|_M \cdot g=\begin{pmatrix}
        \pi^{d_{n-1}} &0 &0 \\
        0 &h_{2,2} &0 \\
        0 &0 & \pi^{2 d_n - d_{n-1}}  b'
    \end{pmatrix}
\]
for some $b'\in \mfo^\times$.
Since $h_{2,2}$ is unimodular, the Jordan type of $M$ is $(d_{n-1},2d_n-d_{n-1})$.

\item The case  $d_{n-1}\geq d_n$

By Lemma \ref{lem:compare_xi} we have $\xi \geq d_n$.
When $\xi = d_n$, choose
$
g = \begin{pmatrix}
  1 &0 &-u \\
  0 &Id_{n-2} &0 \\
  0 & 0 & 1
\end{pmatrix}
$
so that 
$\sigma({}^t g) \cdot h|_M \cdot g
= \begin{pmatrix}
    \pi^{d_n}  &0 &0 \\
    0 &h_{2,2} &0 \\
    0 &0 &\pi^{d_n} b'
\end{pmatrix}$
for some $b' \in \mathfrak{o}^\times$.
Since $h_{2,2}$ is unimodular, the Jordan type of $M$ is $(d_n,d_n)$.

When $\xi>d_n$, choose
$
g = \begin{pmatrix}
    1 &0 &0 \\
    0 &Id_{n-2} &0 \\
    \beta u^{-1} &0 &1
\end{pmatrix}
$ where $\beta \in \mathfrak{o}_E$ such that $\beta + \sigma(\beta) = 1$.
Then we have
$\sigma({}^t g) \cdot h|_M \cdot g
= \begin{pmatrix}
    \pi^{d_n} v &0 & \pi^{d_n}u' \\
    0 &h_{2,2} &0 \\
    \pi^{d_n}\sigma(u') &0 &\pi^{2d_n} b
\end{pmatrix}$
for some $v \in \mathfrak{o}^\times$, $u' \in \mathfrak{o}_E^\times$, and $b \in \mathfrak{o}$.
Note that such $\beta$ exists since the trace map $\mathfrak{o}_E \to \mathfrak{o}$ is surjective.
Then for
$
g' = \begin{pmatrix}
    1 &0 &-u'v^{-1} \\
    0 &Id_{n-2} &0 \\
    0 &0 &1
\end{pmatrix}
$
we have
$\sigma({}^t g')\cdot \sigma({}^t g) \cdot h|_{M} \cdot g \cdot g'
= \begin{pmatrix}
    \pi^{d_n} v &0 &0 \\
    0 &h_{2,2} &0 \\
    0 &0 &\pi^{d_n} b'
\end{pmatrix}$
for some $b' \in \mathfrak{o}^\times$.
Since $h_{2,2}$ is unimodular, the Jordan type of $M$ is $(d_n,d_n)$.
\end{enumerate}
\end{proof}

\subsection{Geometric formulation of \texorpdfstring{$O_{\gamma, \cL(L,M)}$}{Ogamma(L,M)}}


By Proposition \ref{lem:Jordan_of_type_m}, Equation (\ref{mtypestrata}) is refined as follows:
\begin{equation}\label{equation:stradn}
    \bigsqcup_{M:\T(M)=(d_n)} O_{\gamma, \cL(L,M)}=
    \left\{\begin{array}{c l}
    \bigsqcup\limits_{\substack{M:\T(M)=(d_{n});\\ \JT(M)=(\widetilde{d}_{n-1},2d_{n}-\widetilde{d}_{n-1})}}O_{\gamma,\cL(L,M)}&\textit{if $(\widetilde{F},\epsilon)=(E,1)$};\\
    \bigsqcup\limits_{\substack{M:\T(M)=(d_{n});\\ \JT(M)=(d_{n},d_{n})}}O_{\gamma,\cL(L,M)}&\textit{if $(\widetilde{F},\epsilon)=(F,-1)$}.
    \end{array}
    \right.
\end{equation}

We fix a sublattice $M$ of $L$ such that   
$$\T(M)=(d_n)  \textit{   and   }  \mathcal{JT}(M)=(\widetilde{d}_{n-1},2d_n-\widetilde{d}_{n-1}).$$
Here we do not impose the condition that $O_{\gamma, \cL(L,M)}$ is non-empty (cf. Proposition \ref{lem:Jordan_of_type_m}).
The goal of this section is to find a closed formula for $\mathcal{SO}_{\gamma, M}$ in Proposition \ref{cor:610} and a closed formula for $\sum\limits_{M:\mathcal{T}(M)=(d_n)}\mathcal{SO}_{\gamma,M}$ in Theorem \ref{theorem:closedformulafordn},  which only depend on two integers $d_n$ and $\widetilde{d}_{n-1}$.
See Equation (\ref{equation:som}) for the notion of $\mathcal{SO}_{\gamma, M}$.

\begin{lemma}\label{lemma:matrixformoflmh}
    There exists a basis $(e_1, \cdots, e_{\wn})$ for $L$ such that $(e_1, \cdots, e_{\wn-1}, \pi^{d_n}e_{\wn})$ is a basis for $M$ and with respect to which the matrix form of $h$ is 
\[
    h=  \begin{pmatrix}
        a &0 &u \\
        0 &h_{2,2} &0 \\
        \epsilon\sigma(u) &0 & b
    \end{pmatrix} \textit{ such that } \begin{cases}
   a=\pi^{d_{n-1}} &\textit{if $(\widetilde{F}, \epsilon)=(E, 1)$ and $d_{n-1}<d_n$};\\
   a=\pi^{\xi} \textit{ with } \xi\geq d_n  &\textit{if $(\widetilde{F}, \epsilon)=(E, 1)$ and $d_{n-1}\geq d_n$};\\
a=b=0   &\textit{if $(\widetilde{F}, \epsilon)=(F, -1)$}.
\end{cases}
\]
Here $u\in\mfo_{\widetilde{F}}^{\times}$, $b\in\mfo_{\widetilde{F}}$ and $h_{2,2}$ is unimodular. 
If $\xi=\infty$, then we understand $\pi^{\infty}=0$. 
\end{lemma}
A construction of such a basis is exactly the same as that explained in the proof of  Lemma \ref{modif_matr} and so we skip the proof here.

\subsubsection{Construction of $\varphi_{n,M}: \cL(L,M) \longrightarrow \mathcal{A}_{M}$}
Using a basis explained in Lemma \ref{lemma:matrixformoflmh}, we define the functor $\widetilde{\cL}(L,M)$ from the category of flat $\mfo$-algebras to the category of abelian groups by
\begin{equation}\label{defL_1}
    \widetilde{\cL}(L,M)(R)=
    \left\{\begin{pmatrix}
       \pi^{\widetilde{d}_{n-1}}X_{1,1}  &X_{1,2} &X_{1,3} \\
        \pi^{\widetilde{d}_{n-1}}X_{2,1} &X_{2,2} &X_{2,3} \\
        \pi^{d_{n}} X_{3,1} &\pi^{d_{n}} X_{3,2} & \pi^{d_{n}} X_{3,3}
    \end{pmatrix}: L\otimes R \longrightarrow M\otimes R\in \mathfrak{g}(R)
    \right\}
\end{equation}
for a flat $\mfo$-algebra $R$, regardless of non-emptiness of $O_{\gamma, \cL(L,M)}$.

We claim that  $O_{\gamma, \cL(L,M)}\subset \widetilde{\cL}(L,M)(\mfo)$.
If $O_{\gamma, \cL(L,M)}$ is empty, then it is obvious. 
If $O_{\gamma, \cL(L,M)}$ is non-empty, then 
  the second and third  of Equation (\ref{eq:eqs_about_X_1}), combined with Lemma \ref{lem:compare_xi}, yield that each entry of $X_{1,1}, X_{2,1}$ in Equation (\ref{eq:form_X_1}) is contained in the ideal  $(\pi^{\widetilde{d}_{n-1}})$.
Here $\widetilde{d}_{n-1}$ is an integer defined in Equation (\ref{eq:def_widetilde_d_n-1}).
Note that Equation (\ref{eq:eqs_about_X_1}) is for the case $(\widetilde{F},\epsilon)=(E,1)$ but is applicable for the other case $(\widetilde{F},\epsilon)=(F,-1)$ if we replace $\pi^\xi$ and $b$ by $0$.

\begin{lemma}\label{lem:64affine}
    The functor $\widetilde{\cL}(L,M)$ is represented by an affine  space over $\mfo$ of dimension \[\left\{\begin{array}{l l}
    n^2&\textit{if $(\widetilde{F},\epsilon)=(E,1)$};\\
    n(2n+1)&\textit{if $(\widetilde{F},\epsilon)=(F,-1)$}.
    \end{array}\right.\]
\end{lemma}

\begin{proof}
We use a basis for $L$  explained in Lemma \ref{lemma:matrixformoflmh} so that 
 Equation (\ref{eq:eqs_about_X_1}) is applicaable.
 Each block of the  equation
    $h X + \sigma({}^tX)h = 0$  for $X\in \widetilde{\cL}(L,M)(R)$ with a flat $\mfo$-algebra $R$, yields the following matrix equations:

\begin{itemize}
\item for all cases,
    \begin{equation}\label{eq:eqs_about_X_uniform}
    \left\{\begin{array}{l l} (1):  h_{2,2} X_{2,2} + \sigma({}^{t}X_{2,2}) h_{2,2} = 0& \textit{by the $2\times 2$-block};\\
       (2):  h_{2,2} X_{2,3} + \pi^{d_{n}} b \sigma({}^{t}X_{3,2}) +u \sigma({}^{t}X_{1,2})=0  & \textit{by the $2\times 3$-block};\\
      (3):   \epsilon\sigma(u)X_{1,3} + u \sigma(X_{1,3}) +\pi^{d_{n}}b( X_{3,3} +  \sigma(X_{3,3}))=0& \textit{by the $3\times 3$-block}; 
  \end{array}\right.
  \end{equation}

    \item for the case $(\widetilde{F},\epsilon)=(E,1)$ and $d_{n-1}<d_n$,
\begin{equation}\label{eq:eqs_about_X_dn-1<dn}
       \left\{\begin{array}{l l}
       (4): \pi^{d_n}\left(\pi^{2d_{n-1}-d_n}(X_{1,1} + \sigma(X_{1,1})) +(uX_{3,1} + \sigma(uX_{3,1}))\right)=0  & \textit{by the $1\times 1$-block};\\
       (5):    \pi^{d_{n-1}}\left( X_{1,2} + \pi^{{d}_{n}-d_{n-1}}u X_{3,2} +\sigma({}^t X_{2,1}) h_{2,2}\right)=0  & \textit{by the $1\times 2$-block};\\
      (6):    \pi^{d_{n-1}}\left(X_{1,3} + \pi^{{d}_{n}-d_{n-1}}(u X_{3,3} +  b \sigma(X_{3,1})) + u \sigma(X_{1,1})\right)=0  & \textit{by the $1\times 3$-block}.
  \end{array}\right. 
\end{equation}
Here $2 d_{n-1} - d_n \geq 0$ by the Newton polygon of an irreducible polynomial $\chi_{\gamma}(x)$.
  
  \item for the case $(\widetilde{F},\epsilon)=(E,1)$ and $d_{n-1} \geq d_n$,
    \begin{equation}\label{eq:eqs_about_X_dn-1>=dn}
\left\{\begin{array}{l l}
    (4'): \pi^{d_n}\left(\pi^{\xi}(X_{1,1} + \sigma(X_{1,1})) + (uX_{3,1} + \sigma(uX_{3,1}))\right)=0  & \textit{by the $1\times 1$-block};\\
     (5'): \pi^{d_n}\left( \pi^{\xi-d_{n}}X_{1,2} + u X_{3,2} +\sigma({}^t X_{2,1}) h_{2,2}\right)=0 & \textit{by the $1\times 2$-block};\\
      (6'):\pi^{d_n}\left( \pi^{\xi-d_{n}} X_{1,3} + (u X_{3,3} +  b \sigma(X_{3,1})) + u \sigma(X_{1,1})\right)=0 & \textit{by the $1\times 3$-block}.
      \end{array}\right.
    \end{equation}

    \item for the case $(\widetilde{F},\epsilon)=(F,-1)$,
    \begin{equation}\label{eq:eqs_about_X_sp2n}
       \left\{\begin{array}{l l}
       (4''): \pi^{d_n}\left(uX_{3,1}-uX_{3,1}\right)=0  & \textit{by the $1\times 1$-block};\\
      (5''): \pi^{d_n}\left(uX_{3,2}+{}^tX_{2,1}h_{2,2}\right)=0  & \textit{by the $1\times 2$-block};\\
     (6''):  \pi^{d_n}\left(uX_{3,3}+uX_{1,1}\right)=0  & \textit{by the $1\times 3$-block}.
  \end{array}\right. 
\end{equation}
\end{itemize}

Let us interpret the above matrix equations case by case as follows:
\begin{enumerate}
    \item for all cases,
    \begin{itemize}
        \item Since $h_{2,2}$ is unimodular, Equation (\ref{eq:eqs_about_X_uniform}).(1) yields that $X_{2,2}$ forms $\mathfrak{u}_{n-2}(R)$ if $(\widetilde{F},\epsilon) = (E,1)$, and $\mathfrak{sp}_{2n-2}(R)$ if $(\widetilde{F},-1) = (F,-1)$.

    \item Equation (\ref{eq:eqs_about_X_uniform}).(2) yields that $X_{2,3}$ is completely determined by $X_{1,2}$ and $X_{3,2}$  since $h_{2,2}$ is unimodular.

    \item If $(\widetilde{F},\epsilon)=(E,1)$, then Equation (\ref{eq:eqs_about_X_uniform}).(3) yields that the real part of $\sigma(u)X_{1,3}$ is completely determined by $ X_{3,3}$.
    Here, we refer to Remark \ref{remark:reim} for  the notion of the real part.    
    If $(\widetilde{F},\epsilon)=(F,-1)$, Equation (\ref{eq:eqs_about_X_uniform}).(3) is equivalent to the equation $uX_{1,3}-uX_{1,3}=0$ (note $b=0$), and thus is redundant.
    \end{itemize}
    
    \item for the case $(\widetilde{F},\epsilon) = (E,1)$,
    \begin{itemize}
        \item Equation (\ref{eq:eqs_about_X_dn-1<dn}).(4) and Equation (\ref{eq:eqs_about_X_dn-1>=dn}).(4$'$) yield that the real part of $u X_{3,1}$ is determined by the real part of $X_{1,1}$.
        By plugging these into Equation (\ref{eq:eqs_about_X_dn-1<dn}).(6) and Equation \ref{eq:eqs_about_X_dn-1>=dn}.(6$'$) respectively, $X_{1,1}$ is completely determined by $X_{1,3}$, $X_{3,3}$, and the imaginary part of $u X_{3,1}$.
        Here, we refer to Remark \ref{remark:reim} for  the notion of the imaginary part.

        \item Since $h_{2,2}$ is unimodular, Equation (\ref{eq:eqs_about_X_dn-1<dn}).(5) and Equation (\ref{eq:eqs_about_X_dn-1>=dn}).(5$'$) yield that $X_{2,1}$ is completely determined by $X_{1,2}$ and $X_{3,2}$.
    \end{itemize}

    \item for the case $(\widetilde{F},\epsilon) = (F,-1)$,
    \begin{itemize}
        \item Equation (\ref{eq:eqs_about_X_sp2n}).(4$''$) is redundant.
        \item Since $h_{2,2}$ is unimodular, Equation (\ref{eq:eqs_about_X_sp2n}).(5$''$) yields that $X_{2,1}$ is completely determined by $X_{3,2}$.
        \item Equation (\ref{eq:eqs_about_X_sp2n}).(6$''$) yields that $X_{1,1}$ is completely determined by $X_{3,3}$.
    \end{itemize}
\end{enumerate}

In summary,  when $(\widetilde{F},\epsilon)=(E, 1)$, each entry of $X$ is completely determined by $X_{2,2} \in \mathfrak{u}_{n-2}(R)$, $X_{1,2}$, $X_{3,2}$, $X_{3,3}$, the imaginary part of $\sigma(u)X_{1,3}$, and the imaginary part of $uX_{3,1}$.
When $(\widetilde{F},\epsilon)=(F,-1)$, each entry of $X$ is completely determined by $X_{2,2}\in \mathfrak{sp}_{2n-2}(R),X_{1,2}, X_{3,2}, X_{3,3},X_{1,3}$, and $X_{3,1}$.
We note that these variables are all algebraically independent.
The number of these algebraically independent variables   is \[\left\{\begin{array}{l l}
(n-2)^2+2(n-2)+2(n-2)+2+1+1=n^2 &\textit{if $(\widetilde{F},\epsilon)=(E,1)$};\\
(n-1)(2n-1)+(2n-2)+(2n-2)+1+1+1=n(2n+1)&\textit{if $(\widetilde{F},\epsilon)=(F,-1)$}.
\end{array}\right.
\]
This completes the proof.
\end{proof}

\begin{definition}\label{rmk:const_phi_n,M}
We will define a morphism     $\varphi_{n,M}: \cL(L,M) \longrightarrow \mathcal{A}_{M}$ step by step below.
\begin{enumerate}
\item Define the functor $\cL(L,M)$ to be the subfunctor of $\widetilde{\cL}(L,M)$  on the category of flat $\mfo$-algebras such that
\[
\cL(L,M)(R)=\{X:L\otimes_{\mfo} R\longrightarrow M\otimes_{\mfo} R\in \widetilde{\cL}(L,M)(R)\mid\textit{$X$ is surjective}\}
\]
for a flat $\mfo$-algebra $R$. Then $\cL(L,M)$ is an open subscheme of $\widetilde{\cL}(L,M)$ so as to be smooth over $\mfo$.

    \item Define the another functor $\mathcal{A}_{M}$ on the category of flat $\mfo$-algebras such that 
\[
\mathcal{A}_{M}(R) =\left\{ 
\begin{array}{l l}
\prod\limits_{i = 1}^{n-2} \mathcal{A}_i(R) \times  
 \pi^{\widetilde{d}_{n-1}} \mathcal{A}_{n-1}(R) \times \pi^{d_{n}} \mathcal{A}_n(R)&\textit{if $(\widetilde{F},\epsilon)=(E,1)$}\\
R^{n-1}  \times \pi^{d_{n}} R&\textit{if $(\widetilde{F},\epsilon)=(F,-1)$}
 \end{array}\right.
 \]
for a flat $\mfo$-algebra $R$. 
Here $\widetilde{d}_{n-1}$ is described in Equation (\ref{eq:def_widetilde_d_n-1}).
The functor $\mathcal{A}_{M}$ is then represented by an affine space $\mathbb{A}^n_\mathfrak{o}$ as in  Lemma \ref{lem:c+sigma(c)_bij}.

\item Note that  both $\cL(L,M)$ and $\mathcal{A}_{M}$ are subfunctors  of $\mathfrak{g}$ and $\mathcal{A}^n$ (thus $\mathbb{A}^n$) on the category of flat $\mfo$-algebras, respectively. 
Here, $\mathcal{A}^n$ is defined in Section \ref{subsubsection:chevalley}.
If we restrict the Chevalley morphism $\varphi_n$, which is defined in Equation (\ref{eq:chevalleymap}) and Remark \ref{remark:chevnotation}, to $\cL(L,M)$ on the category of flat $\mfo$-algebras, then it induces a well-defined map 
\[
\varphi_{n,M}: \cL(L,M)(R) \longrightarrow \mathcal{A}_{M}(R)
\]
for a flat $\mfo$-algebra $R$.
This map is represented by a morphism of schemes over $\mfo$.
\end{enumerate}
\end{definition}
\begin{remark}\label{remark:rmkchangemea}
In this remark, we will explain how to interpret the morphism $\varphi_{n,M}$ as matrices for a general $\mfo$-algebra $R$ (especially when $R$ is a $\kappa$-algebra).

    \begin{enumerate}
        \item 
We  first explain how to interpret an element of $\cL(L,M)(R)$ as a matrix. 
We formally write  $X=\begin{pmatrix}
       \pi^{\widetilde{d}_{n-1}}X_{1,1}  &X_{1,2} &X_{1,3} \\
        \pi^{\widetilde{d}_{n-1}}X_{2,1} &X_{2,2} &X_{2,3} \\
        \pi^{d_{n}} X_{3,1} &\pi^{d_{n}} X_{3,2} & \pi^{d_{n}} X_{3,3}
    \end{pmatrix}$, where each $X_{i,j}$ has entries in $R$ (or $\mfo_E\otimes_{\mfo}R$ if $(\widetilde{F},\epsilon)=(E,1)$). 
    Then we formally compute $h X + \sigma({}^tX)h$, which is  of the form $\begin{pmatrix}
      \pi^{d_{n}} Y_{1,1}  &\pi^{\widetilde{d}_{n-1}}Y_{1,2} &\pi^{\widetilde{d}_{n-1}}Y_{1,3} \\
       \pi^{\widetilde{d}_{n-1}} \epsilon\sigma({}^t Y_{1,2}) &Y_{2,2} &Y_{2,3} \\
       \pi^{\widetilde{d}_{n-1}} \epsilon\sigma({}^t Y_{1,3}) &\epsilon\sigma({}^t Y_{2,3}) & Y_{3,3}
    \end{pmatrix}$.
    Here each $Y_{i,j}$ has entries in $R$ (or $\mfo_E\otimes_{\mfo}R$ if $(\widetilde{F},\epsilon)=(E,1)$) and $\epsilon\sigma({}^t Y_{i,i})=Y_{i,i}$.
    Note that $Y_{1,1}, Y_{3,3}$, and the diagonal entries of $Y_{2,2}$ are $0$ if $(\widetilde{F},\epsilon)=(F,-1)$ and in the real part of $\mfo_E\otimes_{\mfo}R$ if $(\widetilde{F},\epsilon)=(E,1)$.

    Each matrix $Y_{i,j}$ with $i\leq j$ is described as a linear polynomial of $X_{i', j'}$'s by Equations (\ref{eq:eqs_about_X_uniform})-(\ref{eq:eqs_about_X_sp2n}). 
    Then an element of $\cL(L,M)(R)$ is described as the above formal matrix $X$ such that  $Y_{i,j}=0$ for $i\leq j$.
            
            \item 
        With the above formal matrix $X$ in $\cL(L,M)(R)$, we compute  $\varphi_{n,M}(X)$ formally. It is of the form 
    \[\left\{\begin{array}{l l}
    \begin{pmatrix}r_1, \cdots, r_{n-2}, \pi^{\widetilde{d}_{n-1}}r_{n-1},  \pi^{d_n}r_n\end{pmatrix}&\textit{if $(\widetilde{F},\epsilon)=(E,1)$};\\
    \begin{pmatrix}r_1,\cdots,r_{n-1},\pi^{d_n}r_n\end{pmatrix}&\textit{if $(\widetilde{F},\epsilon)=(F,-1)$},
    \end{array}
    \right.
    \]
    where $r_i\in \mathcal{A}_i(R)$.
    The image of $X$, under the morphism $\varphi_{n,M}$, is then $(r_1, \cdots, r_{n-1}, r_n)$.
    \end{enumerate}
\end{remark}

\subsubsection{Counting  $\#\vpi_{n,M}^{-1}(\chi_{\gamma})(\kappa)$}

\begin{lemma} \label{lem:form_h22}
For $J_{\wn}=    \begin{pmatrix}
        0 &1 &\cdots &0 \\
        \vdots &\vdots &\ddots &\vdots \\
        0 &0 &\cdots &1 \\
        0 &0 &\cdots &0 \\
    \end{pmatrix} \in \mathrm{M}_{\wn}(\kappa_{\widetilde{F}})$ (cf. Lemma \ref{lemma61}),  
if $\overline{h}\in \mathrm{M}_{\wn}(\kappa_{\widetilde{F}})$ is a non-degenerate hermitian matrix satisfying
    \begin{align} \label{eq:hx+x^th}
      \overline{h} J_{\wn} + \sigma({}^t J_{\wn}) \overline{h} = 0,  
    \end{align}
    then $\overline{h}$ is of the form
    \[
    \overline{h} =
    \begin{pmatrix}
        0 &0 &0 &\cdots &0 &0  &c \\
        0 &0 &0 &\cdots &0 &-c  &d \\
        0 &0 &0 &\cdots &c &-d &* \\
        \vdots &\vdots &\vdots &\iddots &\vdots &\vdots &\vdots \\
        0 &0 &\epsilon\sigma(c) &\cdots &* &* &* \\
        0 &-\epsilon\sigma(c) &-\epsilon\sigma(d)  &\cdots &* &* &* \\
        \epsilon\sigma(c) &\epsilon\sigma(d) &* &\cdots &* &* &*
    \end{pmatrix},
    \]
    where  the anti-upper triangular submatrix is zero.
    Here $c\in \kappa_{\widetilde{F}}^{\times}$ and $d \in \kappa_{\widetilde{F}}$ are the $(1,\widetilde{n})$-entry and the $(2,\widetilde{n})$-entry of $\overline{h}$ respectively such that $c + (-1)^{\wn} \epsilon\sigma(c) = 0$ and $d + (-1)^{\widetilde{n}-1}\epsilon \sigma(d)=0$.
\end{lemma}
\begin{proof}
Denote  by $\overline{h}_{i,j}$ the $(i,j)$-entry of an hermitian matrix $\overline{h}$.
 Equation (\ref{eq:hx+x^th}) induces the equation
\[
\overline{h}\cdot J_{\wn}^k = (-1)^k \sigma({}^t J_{\wn}^k)\overline{h}
=(-1)^k\epsilon\sigma({}^t(\overline{h}\cdot J_{\wn}^k))
\]
for any $1\leq k\leq \wn$.
Combining this with the computation
\begin{align*}
\overline{h}\cdot J_{\wn}^k
&=
\begin{pmatrix}
    \overline{h}_{1,1} &\cdots  &\overline{h}_{1,k+1} &\cdots &\overline{h}_{1,\wn} \\
    \vdots &\ddots &\vdots &\ddots &\vdots \\
    \overline{h }_{k+1,1} &\cdots &\overline{h}_{k+1,k+1}  &\cdots &
    \overline{h}_{k+1,\wn} \\
    \vdots  &\ddots & \vdots &\ddots &\vdots \\
    \overline{h}_{\wn,1}  &\cdots &\overline{h}_{\wn,k+1}  &\cdots &\overline{h}_{\wn,\wn}
\end{pmatrix}
\begin{pmatrix}
    0 &\cdots  &1 &\cdots &0 \\
    \vdots &\ddots &\vdots &\ddots &\vdots \\
    0 &\cdots &0  &\cdots &1 \\
    \vdots  &\ddots & \vdots &\ddots &\vdots \\
    0  &\cdots &0  &\cdots &0
\end{pmatrix}     \\
&=\begin{pmatrix}
    0 &\cdots &0  &\overline{h}_{1,1} &\cdots &\overline{h}_{1,\wn-k} \\
    \vdots &\ddots &\vdots &\vdots &\ddots &\vdots \\
    0 &\cdots &0 &\overline{h}_{k+1,1}  &\cdots &\overline
    {h}_{k+1,\wn-k} \\
    \vdots  &\ddots &\vdots &\vdots &\ddots &\vdots \\
    0  &\cdots &0 &\overline{h}_{\wn,1}  &\cdots &\overline{h}_{\wn,\wn-k}
\end{pmatrix}  ~~~~   \textit{for $1\leq k\leq \wn-1$},
\end{align*}
 we have that $\overline{h}_{i,j} = 0$ if $i+j \leq \wn$, and $\overline{h}_{k+l,\wn-k} = (-1)^k \epsilon\sigma(\overline{h}_{\wn,l})$ for $1 \leq l \leq \widetilde{n}-k$.

We write  $c:=\overline{h}_{1,\wn}$, so that  $\overline{h}_{\wn,1} = \epsilon\sigma(c)$ and  $\overline{h}_{k+1,\wn-k} = (-1)^k\epsilon c$ for any $0 \leq k \leq \wn-1$.
Putting $k=\wn-1$, we have $\overline{h}_{\wn,1} = \epsilon\sigma(c) = (-1)^{\wn-1}c$ so that  $c$ satisfies $c + (-1)^{\wn} \epsilon \sigma(c) = 0$.
In the similar manner, we write $d := \overline{h}_{2,\wn}$ so that it satisfies $d + (-1)^{\widetilde{n}-1}\epsilon \sigma(d) = 0$.
Since $\overline{h}$ is non-degenerate, the element $c$ is a unit. 
\end{proof}

\begin{lemma}\label{lem:6.7countingu}
The number of regular and nilpotent elements in $\mathfrak{g}_{\kappa}(\kappa)$ is 
\begin{enumerate}
    \item $\frac{\#\mathrm{U}_n(\kappa)}{(q+1)q^{n-1}}$    \textit{   }    if $(\widetilde{F}, \epsilon)=(E, 1)$;
    \item $\frac{\#\mathrm{Sp}_{2n}(\kappa)}{q^{n}}$     \textit{   }    if $(\widetilde{F}, \epsilon)=(F, -1)$;
\end{enumerate}
\end{lemma}

\begin{proof}
Let $\overline{X}\in \mathfrak{g}_{\kappa}(\kappa)$ be a regular and nilpotent element. Let $\mathrm{Z_{G_{\kappa}}}(\overline{X})$ be its centralizer, defined over $\kappa$, with respect to the Adjoint action.  
By \cite[Theorem C]{Hum17}, we have that 
$$\mathrm{Z_{G_{\kappa}}}(\overline{X})\cong \mathrm{Z(G_{\kappa})}\times \mathrm{Z_{U_{\kappa}}}(\overline{X}),$$
where 
$Z(G_\kappa)$ is the center of $G_{\kappa}$ and $\mathrm{Z_{U_{\kappa}}}(\overline{X})$ is the unipotent radical of a Borel subgroup of $G_{\kappa}$, which is  a connected unipotent group of dimension $n-dim(Z(G_{\kappa}))$. 
Note that \cite[Theorem C]{Hum17} is stated when $\mathrm{\mathrm{G}_{\kappa}}$ is a semisimple algebraic group defined over an algebraically closed field. However, this theorem is applicable in our cases.

On the other hand, \cite[Theorem 8.5]{Hum} (which is a simple application of the \textit{Lang's Theorem}) yields that  the number of conjugacy classes in the stable conjugacy class containing $\overline{X}$ is bijectively identified with $H^1(\kappa, \mathrm{Z(G_{\kappa})}/\mathrm{Z(G_{\kappa})^{\circ}})$, where $\mathrm{Z(G_{\kappa})^{\circ}}$ is the neutral component of $\mathrm{Z(G_{\kappa})}$. 
By \cite[Theorem B]{Hum17}, the stable conjugacy class containing $\overline{X}$ is the same as the set of regular and nilpotent elements in $\mathfrak{g}_{\kappa}(\kappa)$.
Here we have
\[
\mathrm{Z(G_{\kappa})}=\left\{
\begin{array}{l l}
\mathrm{N^1_{\kappa_E/\kappa}} (=\textit{the norm $1$-torus})  & \textit{if $(\widetilde{F}, \epsilon)=(E, 1)$};\\
\mu_2  & \textit{if $(\widetilde{F}, \epsilon)=(F, -1)$};\\
 \end{array}\right.
\]
Thus $\#H^1(\kappa, \mathrm{Z(G_{\kappa})}/\mathrm{Z(G_{\kappa})^{\circ}})=2$ if $(\widetilde{F}, \epsilon)=(F, -1)$ with $char(\kappa)>2$, and $1$ otherwise.
\cite[Lemma 6.3.3]{GY} yields that $\#\mathrm{Z_{U_{\kappa}}}(\overline{X})(\kappa)=q^{\dim(\mathrm{Z_{U_{\kappa}}}(\overline{X}))}$.
Thus we have the desired formula.
\end{proof}

\begin{proposition}\label{corcounting}
The order of the set  $\vpi_{n,M}^{-1}(\chi_{\gamma})(\kappa)$ is 
\[ 
\#\vpi_{n,M}^{-1}(\chi_{\gamma})(\kappa)=
\begin{cases}
    \#\mathrm{U}_{n-2}(\kappa)\cdot q^{3n-4} &\textit{if $(\widetilde{F}, \epsilon)=(E, 1)$ and $d_{n-1}<d_n$};\\
    \#\mathrm{U}_{n-2}(\kappa)\cdot (q-1) q^{3n-5} &\textit{if $(\widetilde{F}, \epsilon)=(E, 1)$ and $d_{n-1}\geq d_n$};\\
\#\mathrm{Sp}_{2n-2}(\kappa)\cdot(q-1)q^{3n-2}   &\textit{if $(\widetilde{F}, \epsilon)=(F, -1)$}.
\end{cases}
\]
Here we regard $\#\mathrm{U}_{2-2}(\kappa)=\#\mathrm{Sp}_{2-2}(\kappa)=1$.
\end{proposition}

\begin{proof}
By Definition \ref{rmk:const_phi_n,M} and Remark \ref{remark:rmkchangemea}, an 
    element $\overline{X} \in \vpi_{n,M}^{-1}(\chi_{\gamma})(\kappa)$ is formally represented by
    $\overline{X} =
    \begin{pmatrix}
        \pi^{\widetilde{d}_{n-1}} \overline{X}_{1,1}  &\overline{X}_{1,2} &\overline{X}_{1,3} \\
        \pi^{\widetilde{d}_{n-1}} \overline{X}_{2,1} &\overline{X}_{2,2} &\overline{X}_{2,3} \\
        \pi^{d_n} \overline{X}_{3,1} &\pi^{d_n} \overline{X}_{3,2} & \pi^{d_n} \overline{X}_{3,3}
    \end{pmatrix}$, where each matrix $\overline{X}_{i,j}$ has entries in $\kappa_{
\widetilde{F}}$. 
    Then Equations (\ref{eq:eqs_about_X_uniform})-(\ref{eq:eqs_about_X_sp2n}) should be intrepreted as in Remark \ref{remark:rmkchangemea}.(1). These, together with  Remark \ref{remark:rmkchangemea}.(2),  yield the set of equations which will be listed  below precisely.
      Let  $\delta(-,-)$ be a function on $\mathbb{Z}\times \mathbb{Z}$  defined by
    $\delta(m,m')
    =\begin{cases}
        1 &\text{if } m = m'; \\
        0 &\text{otherwise}.
    \end{cases}$
 
    \begin{itemize}
        \item for all cases (cf. Equation (\ref{eq:eqs_about_X_uniform})),
        \begin{equation}\label{eq:eqs_about_X_uniform_red}
        \left\{
    \begin{array}{l l}
        (1): \overline{h}_{2,2} \overline{X}_{2,2} + \sigma({}^t\overline{X}_{2,2}) \overline{h}_{2,2} = 0  & \textit{by the $2\times 2$-block};\\
        (2): \overline{u} \sigma({}^t\overline{X}_{1,2}) +  \overline{h}_{2,2} \overline{X}_{2,3} =0  & \textit{by the $2\times 3$-block};\\
        (3): \epsilon\sigma(\overline{u})\overline{X}_{1,3} + \overline{u} \sigma(\overline{X}_{1,3}) =0 &\textit{by the $3\times 3$-block}; \\
        (7): \overline{X}_{3,1}\cdot \det\begin{pmatrix}
        \overline{X}_{1,2} & \overline{X}_{1,3} \\     \overline{X}_{2,2} & \overline{X}_{2,3}
        \end{pmatrix}=u_{1}  & \textit{by the constant term in }
         \chi_{\gamma}(x); \\
        (8): \chi_{\overline{X}_{2,2}}(x)= x^{\wn-2} \in \kappa_E[x] & \textit{by the coefficients of $x^{\wn}, \cdots, x^2$ in $\chi_{\gamma}(x)$}.
    \end{array}\right.
    \end{equation}
    Here $u_{1}=(-1)\cdot\overline{c_n/\pi^{d_n}}\in \mathcal{A}_{n}(\kappa)$ and $u_1\neq 0$.

        \item for the case $(\widetilde{F},\epsilon) = (E,1)$ and $d_{n-1}<d_n$ (cf. Equation (\ref{eq:eqs_about_X_dn-1<dn})),
        \begin{equation}\label{eq:eqs_about_X_dn-1<dn_red}
        \left\{
            \begin{array}{l l}
            (4): \overline{u} \overline{X}_{3,1} + \sigma(\overline{u}\overline{X}_{3,1}) =-\delta(2d_{n-1},d_n)(\overline{X}_{1,1} + \sigma(\overline{X}_{1,1}))  & \textit{by the $1\times 1$-block};\\
            (5): \overline{X}_{1,2} + \sigma({}^t\overline{X}_{2,1}) \overline{h}_{2,2} =0  & \textit{by the $1\times 2$-block};\\
            (6): \overline{u} \sigma(\overline{X}_{1,1}) + \overline{X}_{1,3}  =0  & \textit{by the $1\times 3$-block};\\
            (9): \det\begin{pmatrix}
             \overline{X}_{1,1} & \overline{X}_{1,2} \\     \overline{X}_{2,1} & \overline{X}_{2,2}
             \end{pmatrix}=u_{2}
             & \textit{by the coefficient of $x$ in $\chi_{\gamma}(x)$.}
        \end{array}\right.
        \end{equation}
        Here $u_{2}=(-1)^{n-1}\cdot\overline{c_{n-1}/\pi^{d_{n-1}}}\in \mathcal{A}_{n-1}(\kappa)$ and $u_2\neq 0$.

        \item for the case $(\widetilde{F},\epsilon) = (E,1)$ and $d_{n-1}\geq d_n$ (cf. Equation (\ref{eq:eqs_about_X_dn-1>=dn})),
        \begin{equation}\label{eq:eqs_about_X_dn-1>=dn_red}
        \left\{
        \begin{array}{l l}
            (4'): \overline{u} \overline{X}_{3,1} + \sigma(\overline{u}\overline{X}_{3,1}) = 0  & \textit{by the $1\times 1$-block};\\
            (5'): \delta(\xi,d_n)\overline{X}_{1,2} + \overline{u} \overline{X}_{3,2} + \sigma({}^t\overline{X}_{2,1}) \overline{h}_{2,2} =0 & \textit{by the $1\times 2$-block}; \\
            (6'): \delta(\xi,d_n) \overline{X}_{1,3} + \overline{u} \sigma(\overline{X}_{1,1}) + \overline{u}\overline{X}_{3,3} + \overline{b}\sigma(\overline{X}_{3,1})  =0 & \textit{by the $1\times 3$-block};  \\
            (9'): 
        \det\begin{pmatrix}
            \overline{X}_{1,1} & \overline{X}_{1,2} \\     \overline{X}_{2,1} & \overline{X}_{2,2}
        \end{pmatrix}
        +\det\begin{pmatrix}
            \overline{X}_{2,2} & \overline{X}_{2,3} \\
            \overline{X}_{3,2} &\overline{X}_{3,3}
        \end{pmatrix} & \textit{by the coefficient of $x$ in $\chi_{\gamma}(x)$.}  \\
        \quad \quad +(-1)^{n-2}\overline{X}_{3,1} \cdot \textit{the sum of (i+1,i)-minors of}\begin{pmatrix}
            \overline{X}_{1,2} & \overline{X}_{1,3} \\
            \overline{X}_{2,2} &\overline{X}_{2,3}
        \end{pmatrix}  \\
        \quad\quad = \delta(d_{n-1},d_n) u_2  
        \end{array}\right.
        \end{equation}
        Here $u_{2}=(-1)^{n-1}\cdot\overline{c_{n-1}/\pi^{d_{n-1}}}\in \mathcal{A}_{n-1}(\kappa)$, and the index $i$ in Equation (\ref{eq:eqs_about_X_dn-1>=dn_red}).(9$'$) runs over $i = 1, \cdots, n-2$.
        If $d_{n-1} = \infty$, then we understand $u_2 = 0$.
        
        \item for the case $(\widetilde{F},\epsilon) = (F,-1)$ (cf. Equation (\ref{eq:eqs_about_X_sp2n})),
        \begin{equation}\label{eq:eqs_about_X_sp2n_red}
        \left\{
            \begin{array}{l l}
            (4''): \overline{u} \overline{X}_{3,1} -\overline{u}\overline{X}_{3,1} = 0  & \textit{by the $1\times 1$-block};\\
            (5''):\overline{u}\overline{X}_{3,2}+{}^t\overline{X}_{2,1}\overline{h}_{2,2}=0 & \textit{by the $1\times 2$-block}; \\
            (6''):\overline{u}\overline{X}_{3,3}+\overline{u}\overline{X}_{1,1}=0 & \textit{by the $1\times 3$-block}.
        \end{array}\right.
        \end{equation}
    \end{itemize}
    For all cases, $\overline{h}_{2,2}, \overline{u}, \overline{b}$ are the reductions of $h_{2,2}, u,b$ in Equation (\ref{eq:form_X_1}) modulo $\pi$ respectively (cf. Notations in Part \ref{part1}).
    Equation (\ref{eq:eqs_about_X_uniform_red}).(7)-(8), Equation (\ref{eq:eqs_about_X_dn-1<dn}).(9), and Equation (\ref{eq:eqs_about_X_dn-1>=dn_red}).(9') guarantee that $\overline{X}$ is contained in the fiber of  $\varphi_{n,M}$ at $\chi_\gamma$.

        We will interpret Equation (\ref{eq:eqs_about_X_uniform_red})-Equation (\ref{eq:eqs_about_X_sp2n_red}) below case by case. 
        To do that, let 
        \[
        \left\{
        \begin{array}{l}
        \overline{z}_1 \textit{ and } \overline{z}_1' \textit{ be the first and the second entries of } \overline{X}_{1,2} \textit{ respectively}; \\
        \overline{z}_2 \textit{ be the last entry of } \overline{X}_{2,1}; \\
        \overline{z}_3 \textit{ and } \overline{z}_3' \textit{ be the $(\wn-3)$-th and the $(\wn-2)$-th entries of } \overline{X}_{2,3} \textit{ respectively};\\
        \overline{z}_4 \textit{ be the first entry of } \overline{X}_{3,2}. \\
        \end{array}\right.
        \]
    \begin{enumerate}
        \item for all cases,
        \begin{itemize}
            \item 
            Equation (\ref{eq:eqs_about_X_uniform_red}).(7)-(8) yields that the rank of $\overline{X}_{2,2}$ is 
$\wn-3$. This, combined with  Equation (\ref{eq:eqs_about_X_uniform_red}).(1),   yields that $\overline{X}_{2,2}$ is a regular nilpotent element of the Lie algebra $\mathfrak{u}_{n-2}(\kappa)$ if $(\widetilde{F},\epsilon) = (E,1)$, or $\mathfrak{sp}_{2n-2}(\kappa)$ if $(\widetilde{F},\epsilon) = (F,-1)$.
            By Lemma \ref{lem:6.7countingu}, the number of such $\overline{X}_{2,2}$ is
            \begin{equation}\label{eq:count_X22}
            \left\{\begin{array}{l l}
            \frac{\#\mathrm{U}_{n-2}(\kappa)}{(q+1)q^{n-3}}&\textit{if $(\widetilde{F},\epsilon)=(E,1)$};\\
            \frac{\#\mathrm{Sp}_{2n-2}(\kappa)}{q^{n-1}}&\textit{if $(\widetilde{F},\epsilon)=(F,-1)$}.
            \end{array}\right.
            \end{equation}

As in the case of $\mathfrak{gl}_n$ (cf. (2) just below Equation (\ref{eq83})),              
           we may and do choose a $\kappa_{\widetilde{F}}$-basis for the $\kappa_{\widetilde{F}}$-span of $(\overline{e}_2, \cdots, \overline{e}_{\wn-1})$ such that $\overline{X}_{2,2}=J_{\wn-2}$ so that  $\overline{h}_{2,2}$ is represented by a matrix in Lemma \ref{lem:form_h22}.
           We keep using $c$ and $d$ to stand for entries of $\overline{h}_{2,2}$ as in Lemma \ref{lem:form_h22}.


        \item  Equation (\ref{eq:eqs_about_X_uniform_red}).(2) is equivalent that  $\overline{X}_{2,3}$ is determined by $\overline{X}_{1,2}$.
        Especially, we have the equations
        \begin{equation}\label{eq:z3_as_z_1}
        \overline{u}\sigma(\overline{z}_1) + c \overline{z}'_3 = 0, \quad \textit{and} \quad
        \overline{u}\sigma(\overline{z}'_1) - c \overline{z}_3 + d\overline{z}_3' = 0.
        \end{equation}
        
        \item Using the first in the above two,  Equation (\ref{eq:eqs_about_X_uniform_red}).(7) is equivalent to the equation 
        \begin{equation}\label{eq:eq_comp_sp1}
            \overline{z}_1\sigma(\overline{z}_1)=-(cu_{1})/(\overline{u}\overline{X}_{3,1}),
        \end{equation}
        where $\overline{z}_1, \overline{z}_3', \overline{X}_{3,1}\neq 0$.
    We remark that when $(\widetilde{F},\epsilon) = (E,1)$  this equation has a solution only if its right hand side  is contained in $\kappa^{\times}$.
        Since $c + (-1)^{n} \sigma(c) = 0$ and $u_1 (\neq 0) \in \mathcal{A}_{n}(\kappa)$, it suffices to claim that   the real part of $\overline{u}\overline{X}_{3,1}$ is zero (cf. Remark \ref{remark:reim}).
        This will be proved below when we treat Equations (\ref{eq:eqs_about_X_dn-1<dn_red}).(4) (if $d_{n-1}<d_n$) and (\ref{eq:eqs_about_X_dn-1>=dn_red}).(4$'$) (if $d_{n-1} \geq d_n$).
        \end{itemize}
        
        \item for the case $(\widetilde{F},\epsilon) = (E,1)$ and $d_{n-1} < d_n$,
        \begin{itemize}
            \item Equation (\ref{eq:eqs_about_X_uniform_red}).(3) is equivalent that the real part of $\sigma(\overline{u})\overline{X}_{1,3}$ is zero.

            \item Then Equation(\ref{eq:eqs_about_X_dn-1<dn_red}).(6) is equivalent that the real part of $\overline{X}_{1,1}$ is zero and the imaginary part of $\overline{X}_{1,1}$ is completely determined by the imaginary part of $\sigma(\overline{u})\overline{X}_{1,3}$.

            \item Then Equation (\ref{eq:eqs_about_X_dn-1<dn_red}).(4) is equivalent that the real part of $\overline{u}\overline{X}_{3,1}$ is zero.
        
            \item Equation (\ref{eq:eqs_about_X_dn-1<dn_red}).(5) is equivalent that $\overline{X}_{2,1}$ is completely determined by $\overline{X}_{1,2}$. Especially, we have the equation
            \begin{equation}\label{eq:z2_as_z1_dn-1<d_n}
            \overline{z}_1+\sigma(c\overline{z}_2)=0.
            \end{equation}

            \item By using Equations (\ref{eq:eq_comp_sp1})-(\ref{eq:z2_as_z1_dn-1<d_n}),  Equation (\ref{eq:eqs_about_X_dn-1<dn_red}).(9) is equivalent to the equation 
            \begin{equation}\label{eq:eq_comp_sp2}
               (-1)^n \cdot \overline{u}\overline{X}_{3,1} = u_{1}/u_{2} (\neq 0).
            \end{equation}
            Note that the real part of both sides of Equation (\ref{eq:eq_comp_sp2}) is zero since
            the real part of $\overline{u}\overline{X}_{3,1}$ is zero, $u_{1} (\neq 0) \in \mathcal{A}_{n}(\kappa)$ and $u_{2} (\neq 0) \in \mathcal{A}_{n-1}(\kappa)$.
            It turns out that the imaginary part of $\overline{u}\overline{X}_{3,1}$ is determined.
        \end{itemize}

        In summary, each entry of $\overline{X}$ is completely determined by $\overline{X}_{12}, \overline{X}_{2,2}, \overline{X}_{3,2}, \overline{X}_{3,3}$, and the imaginary part of $\sigma(\overline{u})\overline{X}_{1,3}$.
We analyze this more precisely.

By Equation (\ref{eq:count_X22}), the contribution from the set of $\overline{X}_{2,2}$'s is $\frac{\#\mathrm{U}_{n-2}(\kappa)}{(q+1)q^{n-3}}$.
        The imaginary part of $\sigma(\overline{u})\overline{X}_{1,3}$ contributes to $\kappa$.
        Since Equation (\ref{eq:eq_comp_sp2}) determines $\overline{u}\overline{X}_{3,1}$, 
      Equation (\ref{eq:eq_comp_sp1}) yields that  $\overline{z}_1$ contributes to $N^1_{\kappa_E/\kappa}(\kappa)$.
        The rest  of $\overline{X}_{1,2}$ contributes to $\kappa^{2n-6}$.
        The entries of $\overline{X}_{3,2}$ and $\overline{X}_{3,3}$ contribute to $\kappa^{2n-4}$ and $\kappa^2$ respectively.
        By summing up, we have the desired result:
        \[
\#\vpi_{n,M}^{-1}(\chi_{\gamma})(\kappa)=        \frac{\#\mathrm{U}_{n-2}(\kappa)}{(q+1)q^{n-3}} \cdot q \cdot \#N^1_{\kappa_E/\kappa}(\kappa) \cdot q^{2n-6} \cdot q^{2n-4} \cdot q^2
        = \#\mathrm{U}_{n-2}(\kappa) \cdot q^{3n-4}.
        \]
        
        \item for the case $(\widetilde{F},\epsilon) = (E,1)$ and $d_{n-1} \geq d_n$,
        \begin{itemize}
            \item Equation (\ref{eq:eqs_about_X_uniform_red}).(3) is equivalent that the real part of $\sigma(\overline{u})\overline{X}_{1,3}$ is zero.

            \item Equation (\ref{eq:eqs_about_X_dn-1>=dn_red}).(4$'$) is equivalent that that the real part of $\overline{u} \overline{X}_{3,1}$ is zero.
            
            \item Then Equation (\ref{eq:eqs_about_X_dn-1>=dn_red}).(6$'$) is equivalent that $\overline{X}_{1,1}$ is completely determined by  $\overline{X}_{3,3}$, the imaginary part of $\overline{u}\overline{X}_{3,1}$, and the imaginary part of $\sigma(\overline{u})\overline{X}_{1,3}$.

            \item Equation (\ref{eq:eqs_about_X_dn-1>=dn_red}).(5$'$) is equivalent that $\overline{X}_{2,1}$ is completely determined by $\overline{X}_{1,2}$, $\overline{X}_{3,2}$.
            Especially, we have the equation
            \begin{equation}\label{eq:z2_as_z1_dn-1>=dn}
            \delta(\xi,d_n) \overline{z}_1 + \overline{u}\overline{z}_4 + \sigma(c\overline{z}_2) = 0.
            \end{equation}
            
            \item
            Equation (\ref{eq:eqs_about_X_dn-1>=dn_red}).(9$'$) is equivalent to the equation
            \[
            (-1)^{n-2}(\overline{z}_1 \overline{z}_2 + \overline{z}_3'\overline{z}_4 + \overline{X}_{3,1}(\overline{z}_1\overline{z}_3 + \overline{z}_1' \overline{z}_3'))
            =\delta(d_{n-1},d_n)u_2.
            \]
           By plugging Equations (\ref{eq:z3_as_z_1}) and  (\ref{eq:z2_as_z1_dn-1>=dn}), the above is equivalent to 
            \begin{equation}\label{eq:eq_comp_sp3}
            \begin{aligned}
            &-(\overline{z}_1 \sigma(\overline{u}\overline{z}_4) +  \sigma(\overline{z}_1)\overline{u}\overline{z}_4) -\delta(\xi,d_n)\overline{z}_1\sigma(\overline{z}_1)
            + \overline{u}\overline{X}_{3,1}(\overline{z}_1\sigma(\overline{z}_1') - \overline{z}_1'\sigma(\overline{z}_1) - c^{-1}d \overline{z}_1 \sigma(\overline{z}_1) ) \\
            &= (-1)^{n-2}\delta(d_{n-1},d_n) c u_{2}.
            \end{aligned}
            \end{equation}
            Note that $\overline{z}_1$ and $\overline{u}\overline{X}_{31}$ are nonzeros by Equation (\ref{eq:eq_comp_sp1}).
            Both sides of Equation (\ref{eq:eq_comp_sp3}) are contained in $\kappa$ since $c + (-1)^{n} \sigma(c) = 0$, $d + (-1)^{n-1} \sigma(d) = 0$, $u_2 \in \mathcal{A}_{n-1}(\kappa)$, and the real part of $\overline{u}\overline{X}_{3,1}$ is zero.
            Equation (\ref{eq:eq_comp_sp3}) is equivalent that the real part of $\sigma(\overline{z}_1)\overline{u}\overline{z}_4$ is determined by $\overline{z}_1$, $\overline{z}_1'$, and the imaginary part of $\overline{u}\overline{X}_{3,1}$.
        \end{itemize}
        In summary, each entry of $\overline{X}$ is completely determined by $\overline{X}_{12}, \overline{X}_{2,2}, \overline{X}_{3,2}, \overline{X}_{3,3}$, the imaginary part of $\overline{u}\overline{X}_{3,1}$, and the imaginary part of $\sigma(\overline{u})\overline{X}_{1,3}$.
        We analyze this more precisely. 
        
        By Equation (\ref{eq:count_X22}), the contribution from the set of $\overline{X}_{2,2}$'s is $\frac{\#\mathrm{U}_{n-2}(\kappa)}{(q+1)q^{n-3}}$.
        The imaginary part of $\sigma(\overline{u})\overline{X}_{1,3}$ contributes to $\kappa$.
        Equation (\ref{eq:eq_comp_sp1}) yields that the choice of $\overline{z}_1$ and the imaginary part of $\overline{u}\overline{X}_{3,1}$ contribute to $N^1_{\kappa_E/\kappa}(\kappa) \times \kappa^\times$.
        The rest entries of $\overline{X}_{1,2}$ contribute to $\kappa^{2n-6}$.
        If $\overline{X}_{1,2}$ and the imaginary part of $\overline{u}\overline{X}_{3,1}$ are determined, then the real part of $\sigma(\overline{z}_1)\overline{u}\overline{z}_4$ is determined as well by Equation (\ref{eq:eq_comp_sp3}).
        Since $\overline{z}_1$ is a unit by Equation (\ref{eq:eq_comp_sp1}), the imaginary part of $\sigma(\overline{z}_1)\overline{u}\overline{z}_4$ contributes to $\kappa$.
        The rest entries of $\overline{X}_{3,2}$ contribute to $\kappa^{2n-6}$, and the choice of $\overline{X}_{3,3}$ contribute to $\kappa^2$.
        By summing up, we have the desired result:
        \[
        \frac{\#\mathrm{U}_{n-2}(\kappa)}{(q+1)q^{n-3}} \cdot q \cdot \#N^1_{\kappa_E/\kappa}(\kappa) \cdot (q-1) \cdot q^{2n-6} \cdot q \cdot q^{2n-6} \cdot q^2 = \#\mathrm{U}_{n-2}(\kappa) \cdot (q-1)q^{3n-5}.
        \]
        
        \item for the case $(\widetilde{F},\epsilon) = (F,-1)$,
        \begin{itemize}
            \item Equations (\ref{eq:eqs_about_X_uniform_red}).(3) and  (\ref{eq:eqs_about_X_sp2n_red}).(4$''$) are redundant.
            \item Equation (\ref{eq:eqs_about_X_sp2n_red}).(6$''$) is equivalent that $\overline{X}_{1,1}$ is completely determined by $\overline{X}_{3,3}$.
            \item Equation (\ref{eq:eqs_about_X_sp2n_red}).(5$''$)  is equivalent that $\overline{X}_{2,1}$ is completely determined by $\overline{X}_{3,2}$. 
            Especially, we have the equation 
            \begin{equation}\label{eq:z2_as_z1_sp}
            \overline{u}\overline{z}_4 - c\overline{z}_2 = 0.
            \end{equation}
        \end{itemize}
        In summary, each entry of $\overline{X}$ is completely determined by $\overline{X}_{12}, \overline{X}_{2,2}, \overline{X}_{3,2}, \overline{X}_{3,3}$, and  $\overline{X}_{1,3}$.
        We analyze this more precisely.
        
        By Equation (\ref{eq:count_X22}), the contribution from the set of $\overline{X}_{2,2}$'s is $\frac{\# \mathrm{Sp}_{2n-2}(\kappa)}{q^{n-1}}$.
        The choice of $\overline{X}_{1,3}$ contributes to $\kappa$.
        By Equation (\ref{eq:eq_comp_sp1}), the entry $\overline{z}_1$ contributes to $\kappa^\times$ and $\overline{z}_1$ determines $\overline{X}_{3,1}$.
        The rest entries of $\overline{X}_{1,2}$ contribute to $\kappa^{2n-3}$.
        The entries of $\overline{X}_{3,2}$ and $\overline{X}_{3,3}$ contribute to $\kappa^{2n-2}$ and $\kappa$ respectively.
        By summing up, we have the desired result:
        \[
        \frac{\#\mathrm{Sp}_{2n-2}(\kappa)}{q^{n-1}} \cdot q \cdot (q-1) \cdot q^{2n-3} \cdot q^{2n-2} \cdot q
        = \#\mathrm{Sp}_{2n-2}(\kappa) \cdot (q-1)q^{3n-2}.
        \]
    \end{enumerate}
\end{proof}

\subsubsection{Smootheness of $\varphi_{n,M}^{-1}(\chi_\gamma)$}
\begin{theorem}\label{theorem:smoothness}
The scheme
    $\varphi^{-1}_{n,M}(\chi_{\gamma})$
    is smooth over $\mathfrak{o}$.
\end{theorem}
\begin{proof}
The argument of the proof follows that of Theorem \ref{thm55}.
It suffices to show that for any $\overline{X}\in \varphi_{n,M}^{-1}(\chi(\gamma))(\bar{\kappa})$,
the induced map on the Zariski tangent space 
\[
d(\vpi_{n,M})_{\ast, \overline{X}}: T_{\overline{X}} \longrightarrow T_{\vpi_{n,M}(\overline{X})}
\] 
is surjective, where $T_{\overline{X}}$ is the Zariski tangent space of $ \cL(L,M)\otimes \bar{\kappa}$ at $\overline{X}$ and
 $T_{\vpi_{n,M}(\overline{X})}$ is the Zariski tangent space of $\mathcal{A}_{M}\otimes \bar{\kappa}$ at $\vpi_{n,M}(\overline{X})$. 

Let us identify $T_{\overline{X}}$ with the subset of $\mathcal{L}(L,M)(\overline{\kappa}[\epsilon])$, where $\epsilon^2 = 0$, whose element is of the form $\overline{X}+\epsilon \overline{Y}$.
Since the equation
$    \overline{h}(\overline{X}+ \epsilon \overline{Y}) + \sigma({}^t \overline{X} + \epsilon {}^t \overline{Y}) \overline{h} = 0$ yields the equation $    \overline{h} \overline{Y} + \sigma({}^t \overline{Y}) \overline{h} = 0$, 
 $\overline{Y}$ can be viewed as an element of  $\widetilde{\mathcal{L}}(L,M)(\overline{\kappa})$.
  We then formally write
\begin{align} \label{eq:form_Y} 
    \overline{X}+ \epsilon \overline{Y} = 
    \begin{pmatrix}
        \pi^{\widetilde{d}_{n-1}}\overline{X}_{1,1}  &\overline{X}_{1,2} &\overline{X}_{1,3}  \\
        \pi^{\widetilde{d}_{n-1}}\overline{X}_{2,1} &\overline{X}_{2,2}  &\overline{X}_{2,3}  \\
        \pi^{d_n}\overline{X}_{3,1} &\pi^{d_n}\overline{X}_{3,2} & \pi^{d_n}\overline{X}_{3,3} 
    \end{pmatrix}
    +
    \epsilon\begin{pmatrix}
        \pi^{\widetilde{d}_{n-1}}\overline{Y}_{1,1}  &\overline{Y}_{1,2} &\overline{Y}_{1,3}  \\
        \pi^{\widetilde{d}_{n-1}}\overline{Y}_{2,1} &\overline{Y}_{2,2}  &\overline{Y}_{2,3}  \\
        \pi^{d_n}\overline{Y}_{3,1} &\pi^{d_n}\overline{Y}_{3,2} & \pi^{d_n}\overline{Y}_{3,3}
    \end{pmatrix}
\end{align}
where $\overline{X}_{i,j}$ and $\overline{Y}_{i,j}$ are matrices with entries in $\kappa_{\widetilde{F}}\otimes\overline{\kappa}$.
This representation satisfies the following conditions:
\begin{enumerate}
    \item for the case $(\widetilde{F},\epsilon)=(E,1)$ and $d_{n-1}<d_n$, 
    $\overline{X}$ satisfies Equations (\ref{eq:eqs_about_X_uniform_red}) and  (\ref{eq:eqs_about_X_dn-1<dn_red}), while
    $\overline{Y}$ satisfies 
    Equation (\ref{eq:eqs_about_X_uniform_red}) for all indices except $(7)$ and $(8)$, and Equation (\ref{eq:eqs_about_X_dn-1<dn_red}) for all indices except $(9)$;

    \item for the case $(\widetilde{F},\epsilon)=(E,1)$ and $d_{n-1} \geq d_n$, 
    $\overline{X}$ satisfies Equations (\ref{eq:eqs_about_X_uniform_red}) and 
    (\ref{eq:eqs_about_X_dn-1>=dn_red}), while        
    $\overline{Y}$ satisfies Equations (\ref{eq:eqs_about_X_uniform_red}) for all indices except $(7)$ and $(8)$, and Equation (\ref{eq:eqs_about_X_dn-1>=dn_red}) for all indices except $(9')$;

    \item for the case $(\widetilde{F},\epsilon)=(F,-1)$, 
     $\overline{X}$ satisfies Equations (\ref{eq:eqs_about_X_uniform_red}) and  (\ref{eq:eqs_about_X_sp2n_red}), while
    $\overline{Y}$ satisfies Equation (\ref{eq:eqs_about_X_uniform_red}) for all indices except $(7)$ and $(8)$, and Equation (\ref{eq:eqs_about_X_sp2n_red}).
\end{enumerate}
Here we consider these equations after replacing $\kappa_{
\widetilde{F}}$ with $\kappa_{\widetilde{F}}\otimes_\kappa \overline{\kappa}$.
We formally compute the image $\varphi_{n,M}(\overline{X} + \epsilon \overline{Y}) \in \mathcal{A}_M(\overline{\kappa}[\epsilon])$ based on Remark \ref{remark:rmkchangemea}.
It consists of $n$-entries whose $j$-th entry is the sum of $\frac{\wn}{n}\cdot j\times \frac{\wn}{n}\cdot j$ principal minors of $\overline{X} + \epsilon \overline{Y}$ up to the sign $\pm 1$.
It is then of the form
\[
    \varphi_{n,M}(\overline{X}+\epsilon \overline{Y})
    =\left\{\begin{array}{l l}
    \varphi_{n,M}(\overline{X}) + \epsilon \cdot (w_1, \cdots, w_{n-2}, \pi^{\widetilde{d}_{n-1}} w_{n-1}, \pi^{d_n} w_n)&\textit{if $(\widetilde{F},\epsilon)=(E,1)$};\\
    \varphi_{n,M}(\overline{X})+\epsilon\cdot (w_1,\cdots,w_{n-1},\pi^{d_n}w_n)&\textit{if $(\widetilde{F},\epsilon)=(F,-1)$}
    \end{array}\right.
\]
where $w_j \in \mathcal{A}_j(\overline{\kappa})$. 
The image of $\overline{Y}$ is then $\left( w_1, \cdots, w_{n-1}, w_n\right)$.

Our method to prove the surjectivity of $d(\varphi_{n,M})_{\ast,\overline{X}}$ is to choose a certain subspace of $T_{\overline{X}}$ mapping onto $T_{\varphi_{n,M}(\overline{X})}$. 
We proceed the proof in two steps. 
At first we identify a vector whose image is $(w_1,\cdots,w_{n-2n/\wn},\cdots)$ for any given $w_i \in \mathcal{A}_i(\overline{\kappa})$ with $1\leq i \leq n-2n/\wn$. 
We next identify a vector whose image is $\left\{\begin{array}{l l}(0,\cdots,0,w_{n-1},w_{n})&\textit{if $(\widetilde{F},\epsilon)=(E,1)$};\\
(0,\cdots,0,w_n)&\textit{if $(\widetilde{F},\epsilon)=(F,-1)$}\end{array}\right.$ with any given $w_i\in\mathcal{A}_i(\overline{\kappa})$ for $n-2n/\wn+1\leq i\leq n$.
By combining these two results, we obtain the desired surjectivity of $d(\varphi_{n,M})_{\ast,\overline{X}}$.

First of all, we treat the first $n-2n/\widetilde{n}$ entries $(w_1, w_2, \cdots, w_{n-2n/\wn})$. 
Since all entries in the last row and the first column of $\overline{X} + \epsilon \overline{Y}$ involve  $\pi$,  $w_j$'s with $1\leq j\leq n-2n/\wn$ are completely determined by principal minors of $\overline{X}_{2,2}+\epsilon \overline{Y}_{2,2}$.
Note that $\overline{X}_{2,2}$ is a regular element in $\mathfrak{u}_{n-2}(\bar{\kappa})$ or $\mathfrak{sp}_{2n-2}(\bar{\kappa})$ as mentioned in the paragraph just before Equation (\ref{eq:count_X22}) . By \cite[Theorem 4.20]{Hum}, the Chevalley map on $\overline{X}_{2,2}$ is smooth so that the induced map on the tangent space is surjective. 
Thus there exists $\overline{Y}_{2,2}$ which maps to  $(w_1, w_2, \cdots, w_{n-2n/\wn})$. 
By putting $\overline{Y}_{i,j}=0$ for all $(i,j)\neq (2,2)$, we have the desired result.

We now treat the last $2n/\wn$ entries.
As in the proof of Proposition \ref{corcounting}, we may  and do assume that $\overline{X}_{2,2}=J_{\wn-2}$ so that $\overline{h}_{22}$ is represented by a matrix in Lemma \ref{lem:form_h22}.
Subsequently, we can formally write 
\begin{equation}\label{eq:forms_h_and_X_sp}
\overline{h}
    =     \begin{pmatrix}
        \pi^\xi &0 &0&\cdots &0&0 &\overline{u} \\
        0 &0 &0&\cdots &0&c &0 \\
        0 &0 &0&\cdots &-c&\ast &0 \\
        \vdots &\vdots &\vdots&\iddots &\vdots&\vdots &\vdots \\
        0 &0&-\epsilon\sigma(c) &\cdots &\ast&\ast &0 \\
        0 &\epsilon\sigma(c)&\ast &\cdots &\ast&\ast &0 \\
        \epsilon\sigma(\overline{u})&0 &0 &\cdots &0&0 &\overline{b}
    \end{pmatrix},
    \overline{X} = 
    \begin{pmatrix}
        \pi^{\widetilde{d}_{n-1}} \overline{X}_{1,1} &\overline{z}_1 &\overline{z}_1' &\cdots &* &\overline{X}_{1,3} \\
        \pi^{\widetilde{d}_{n-1}} * &0 &1 &\cdots &0 &* \\
        \vdots &\vdots &\ddots &\ddots &\vdots &\vdots \\
        \pi^{\widetilde{d}_{n-1}} * &0 &\cdots &0 &1 & \overline{z}_3 \\
        \pi^{\widetilde{d}_{n-1}} \overline{z}_2 &0 &\cdots &0 &0 & \overline{z}_3' \\
        \pi^{d_{n}} \overline{X}_{3,1} &\pi^{d_n} \overline{z}_4 &\cdots &\pi^{d_n} * &\pi^{d_n} * &\pi^{d_n}\overline{X}_{3,3}
    \end{pmatrix}
    \end{equation}
with $\overline{X}_{3,1}, \overline{z}_1, \overline{z}'_3 , \overline{u}, c \in (\kappa_{\widetilde{F}}\otimes_{\kappa}\overline{\kappa})^\times$ and $\overline{z}_1', \overline{z}_2,  \overline{z}_3, \overline{z}_4,\bar{b} \in \kappa_{\widetilde{F}}\otimes_\kappa \overline{\kappa}$.
if $(\widetilde{F},\epsilon)=(F,-1)$, then $\pi^{\xi}=\bar{b}=0$.
Here notations coincide with those used in  the proof of Proposition \ref{corcounting}.
We reproduce Equations (\ref{eq:z3_as_z_1}), (\ref{eq:z2_as_z1_dn-1<d_n}),  (\ref{eq:z2_as_z1_dn-1>=dn}),  and  (\ref{eq:z2_as_z1_sp}):
  \begin{equation}\label{equation:13.31}
        \overline{z}_3' = -c^{-1}\overline{u}\sigma(\overline{z}_1)
  ~~~~~~   \textit{ and }  ~~~~~~~
    \overline{z}_3 = c^{-1}\overline{u}(\sigma(\overline{z}_1')-c^{-1}d\sigma(\overline{z}_1)); 
      \end{equation}
      \begin{equation}\label{equation:13.32}
          \overline{z}_2 =
\begin{cases}
    -c^{-1}\sigma(\overline{z}_1) &\textit{if $(\widetilde{F},\epsilon)=(E,1)$ and $d_{n-1}<d_n$}; \\
    -c^{-1}(\delta(\xi,d_n)\sigma(\overline{z}_1) + \sigma(\overline{u}\overline{z}_4)) &\textit{if $(\widetilde{F},\epsilon)=(E,1)$ and $d_{n-1}\geq d_n$}.
\end{cases}
      \end{equation}

We claim that there exists $\overline{Y}\in\widetilde{\cL}(L,M)(\overline{\kappa})$ whose image under $d(\varphi_{n,M})_{\ast,\overline{X}}$ is 
\[\left\{\begin{array}{l l}
(0,\cdots,0,w_{n-1},w_{n})&\textit{if $(\widetilde{F},\epsilon)=(E,1)$};\\
(0,\cdots,0,w_n)&\textit{if $(\widetilde{F},\epsilon)=(F,-1)$}
\end{array}\right.\]
for any given $w_i \in \mathcal{A}_i(\overline{\kappa})$ where $n- 2n/\wn\leq i\leq n$.

\begin{enumerate}
    \item For the case $(\widetilde{F},\epsilon)=(E,1)$ and $d_{n-1}<d_n$, we set
\[\overline{Y}=
\begin{pmatrix}
    0&\overline{Y}_{1,2}&0\\
    \overline{Y}_{2,1}&0&\overline{Y}_{2,3}\\
    v_2&0&0
\end{pmatrix}\in \widetilde{\cL}(L,M)(\overline{\kappa}),
\]
where $\overline{Y}_{1,2} = (v_1,0,\cdots,0)$. Here, $v_1,v_2 \in \kappa_E\otimes_{\kappa}\overline{\kappa}$ such that $\overline{u}v_2 + \sigma(\overline{u}v_2) = 0$.
Applying Equations (\ref{eq:z3_as_z_1}) and (\ref{eq:z2_as_z1_dn-1<d_n}) to $\overline{Y}$, the last entry of $\overline{Y}_{2,3}$ is  $-c^{-1}\overline{u}\sigma(v_1)$ and    the last entry of $\overline{Y}_{2,1}$ is  $-c^{-1}\sigma(v_1)$, respectively. To prove the claim, we find suitable $v_1$ and $v_2$ as the following process:
\begin{itemize}
    \item Since $\overline{Y}_{2,2}=0$, the first $(n-2)$-entries of  $d(\varphi_{n,M})_{\ast,\overline{X}}(\overline{Y})$ are vanished.
    \item
  The computation of the ($n-1$)-th entry of $\varphi_{n,M}(\overline{X} + \epsilon \overline{Y})$ is as follows:
\[
    \det
    \begin{pmatrix}
        \overline{X}_{1,1} + \epsilon \overline{Y}_{1,1} &\overline{X}_{1,2} + \epsilon \overline{Y}_{1,2} \\
        \overline{X}_{2,1} + \epsilon \overline{Y}_{2,1} &\overline{X}_{2,2} + \epsilon \overline{Y}_{2,2}
    \end{pmatrix}
    = (-1)^{n-1}\cdot \overline{c_{n-1}/\pi^{d_{n-1}}}
    + (-1)^{n-2}\epsilon\{-\overline{z}_1c^{-1}\sigma(v_1)  + v_1 \overline{z}_2\}.
\]
By Equation (\ref{equation:13.32}), we consider the following equation for $v_1$:
\[
(-1)^{n-1}w_{n-1} = (-1)^{n-1}c^{-1}(\overline{z}_1 \sigma(v_1)+\sigma(\overline{z}_1)v_1).
\]
Since $c$ and $\overline{z}_1$ are units, the existence of a solution $v_1\in \kappa_E\otimes_{\kappa}\overline{\kappa}$ of this equation is guaranteed 
by the surjectivity of the trace mapping $\mathrm{Tr}_{\kappa_E\otimes _\kappa \overline{\kappa}/\overline{\kappa}}:\kappa_E\otimes_{\kappa}\overline{\kappa} \rightarrow \overline{\kappa}$.
We fix such a $v_1$ to be a solution for this equation.
\item
The computation of the $n$-th entry of $\varphi_{n,M}(\overline{X}+ \epsilon \overline{Y})$ is as follows:
\[
    \det(\overline{X} + \epsilon \overline{Y})
    = (-1)^n\cdot \overline{c_n/\pi^{d_n}}
    +(-1)^{n-1}\epsilon\{v_1\overline{z}_3'\overline{X}_{3,1} - c^{-1}\overline{u}\sigma(v_1)\overline{z}_1\overline{X}_{3,1} + v_2 \overline{z}_1\overline{z}_3'\}.
\]
By Equation (\ref{equation:13.31}), we have the following equation for $v_2$:
\begin{align*}
    (-1)^{n}w_n
= (-1)^{n} c^{-1}\overline{u}(v_1\sigma(\overline{z}_1)\overline{X}_{3,1}+\sigma(v_1)\overline{z}_1\overline{X}_{3,1}+v_2\overline{z}_1\sigma(\overline{z}_1)).
\end{align*}
This equation has a solution $v_2$ since $\overline{u}, \overline{z}_1, c$ are units.
It remains to show that the solution $v_2$ satisfies the equation $\overline{u}v_2 + \sigma(\overline{u}v_2) = 0$.
From the above equation we obtain that
\[
\overline{u}v_2
= \frac{cw_n- \overline{u}\overline{X}_{3,1}( v_1\sigma(\overline{z}_1) +\sigma(v_1)\overline{z}_1 )}{\overline{z}_1\sigma(\overline{z}_1)}.
\]
Since $\overline{u}\overline{X}_{3,1} + \sigma(\overline{u}\overline{X}_{3,1}) = 0$, $w_n \in \mathcal{A}_n(\overline{\kappa})$, and $c + (-1)^{\wn}\sigma(c) = 0$, we can deduce that $\overline{u}v_2+\sigma(\overline{u}v_2)=0$.
\end{itemize}

\item For the case $(\widetilde{F},\epsilon)=(E,1)$ and $d_{n-1} \geq d_n$,
we set
\[\overline{Y}=
\begin{pmatrix}
    0&\overline{Y}_{1,2}&0\\
    \overline{Y}_{2,1}&0&\overline{Y}_{2,3}\\
    0&\overline{Y}_{3,2}&0
\end{pmatrix}\in \widetilde{\cL}(L,M)(\overline{\kappa}),
\]
where $\overline{Y}_{1,2} = (v_1,0,\cdots,0)$ and $\overline{Y}_{3,2} = (v_2,0,\cdots,0)$.
Here,  $v_1,v_2 \in \kappa_E\otimes_{\kappa}\overline{\kappa}$.
Applying Equations (\ref{eq:z3_as_z_1}) to $\overline{Y}$, the last entry and the $(\wn-3)$-th entry of $\overline{Y}_{2,3}$ are  $-c^{-1}\overline{u}\sigma(v_1)$ and $-dc^{-2}\overline{u}\sigma(v_1)$, respectively.
Applying Equation (\ref{eq:z2_as_z1_dn-1>=dn}) to $\overline{Y}$, the last entry of $\overline{Y}_{2,1}$ is $-c^{-1}(\delta(\xi,d_n)\sigma(v_1) + \sigma(\overline{u}v_2))$. To prove the claim, we find suitable $v_1$ and $v_2$ as the following process:

\begin{itemize}
    \item Since $\overline{Y}_{2,2}=0$, the first $(n-2)$-entries of  $d(\varphi_{n,M})_{\ast,\overline{X}}(\overline{Y})$ are vanished.
    \item 
The computation of $n$-th entry of $\varphi_{n,M}(\overline{X}+ \epsilon \overline{Y})$ is as follows:
\[    \det(\overline{X} + \epsilon \overline{Y})
    = (-1)^n\cdot \overline{c_n/\pi^{d_n}}+ (-1)^{n-1}\epsilon (
    v_1\overline{z}_3'\overline{X}_{3,1}-c^{-1}\overline{u}\sigma(v_1)\overline{z}_1\overline{X}_{3,1}).
\]
By Equation (\ref{equation:13.31}), we have the following equation for $v_1$:
\[
(-1)^{n}w_n=
(-1)^{n} c^{-1}\overline{u} (v_1\sigma(\overline{z}_1)\overline{X}_{3,1}+\sigma(v_1)\overline{z}_1\overline{X}_{3,1}).
\]
Since $\overline{X}_{3,1}, \overline{u},\overline{z}_1,c$ are units, the existence of a solution $v_1 \in \kappa_E\otimes \overline{\kappa}$ is guaranteed by the surjectivity of the trace mapping $\mathrm{Tr}_{\kappa_E\otimes_\kappa \overline{\kappa}/\overline{\kappa}}:\kappa_E\otimes_{\kappa}\overline{\kappa} \rightarrow \overline{\kappa}$. We fix such a $v_1$ to be a solution for this equation.
    \item
The computation of $(n-1)$-th entry of $\varphi_{n,M}(\overline{X}+\epsilon \overline{Y})$ is as follows:
\begin{align*}
    &\det
    \begin{pmatrix}
        \overline{X}_{1,1} + \epsilon \overline{Y}_{1,1} &\overline{X}_{1,2} + \epsilon \overline{Y}_{1,2} \\
        \overline{X}_{2,1} + \epsilon \overline{Y}_{2,1} &\overline{X}_{2,2} + \epsilon \overline{Y}_{2,2}
    \end{pmatrix}
     +
     \det
    \begin{pmatrix}
        \overline{X}_{2,2} + \epsilon \overline{Y}_{2,3} &\overline{X}_{3,2} + \epsilon \overline{Y}_{3,3} \\
        \overline{X}_{2,1} + \epsilon \overline{Y}_{2,1} &\overline{X}_{2,2} + \epsilon \overline{Y}_{2,2}
    \end{pmatrix} \\
    &+(-1)^{n-2} (\overline{X}_{3,1} + \epsilon \overline{Y}_{3,1}) \cdot \textit{the sum of $(i+1,i)$-minors of }
    \begin{pmatrix}
        \overline{X}_{1,2} + \epsilon \overline{Y}_{1,2} &\overline{X}_{1,3} + \epsilon \overline{Y}_{1,3} \\
        \overline{X}_{2,2} + \epsilon \overline{Y}_{2,2} &\overline{X}_{2,3} + \epsilon \overline{Y}_{2,3}
    \end{pmatrix}
    \\
    &=(-1)^{n-1}\cdot\overline{c_{n-1}/\pi^{d_n}}
    +(-1)^{n-2}\epsilon
    \Bigl(v_1\overline{z}_2-c^{-1}\overline{z}_1(\delta(\xi,d_n)\sigma(v_1)+\sigma(\overline{u}v_2))
    +\overline{z}_{3}'v_2-c^{-1}\overline{z}_{4}\overline{u}\sigma(v_1)\\
    &-\overline{z}_1'c^{-1}\overline{u}\sigma(v_1)\overline{X}_{3,1}+
    \overline{X}_{3,1}(\overline{z}_3v_1-dc^{-2}\overline{z}_1\overline{u}\sigma(v_1))\Bigr)
\end{align*}
with $1 \leq i \leq  \wn-2$.
By Equations (\ref{equation:13.31}) and (\ref{equation:13.32}), we have 
\begin{align*}
    w_{n-1}
    =&c^{-1}\Bigl(
    \delta(\xi,d_n)(\sigma(\overline{z}_1)v_1+\overline{z}_1\sigma(v_1))
    +(v_1\sigma(\overline{u}\overline{z_4})+\sigma(v_1)\overline{u}\overline{z}_4)
    +(\overline{z}_1\sigma(\overline{u}v_2)+\sigma(\overline{z}_1)\overline{u}v_2)
    \\&+\overline{u}\overline{X}_{3,1}(\overline{z}_1'\sigma(v_1)-\sigma(\overline{z}_1') v_1)+\overline{u}\overline{X}_{3,1}c^{-1}d(\sigma(\overline{z}_1)v_1+\overline{z}_1\sigma(v_1))
    \Bigr).
\end{align*}
Since $\overline{u},\overline{z}_1,c$ are units, the existence of a solution $v_2 \in \kappa_E \otimes _\kappa \overline{\kappa}$ of this equation is guaranteed by the  surjectivity of the trace map $\mathrm{Tr}_{\kappa_E\otimes _\kappa \overline{\kappa}/\overline{\kappa}}:\kappa_E\otimes_{\kappa}\overline{\kappa} \rightarrow \overline{\kappa}$.

\end{itemize}
\item For the case $(\widetilde{F},\epsilon)=(F,-1)$, we set 
\[
\overline{Y}=\begin{pmatrix}
    0&0&0\\
    0&0&0\\
    v&0&0
\end{pmatrix}\in \widetilde{\cL}(L,M)(\overline{\kappa}),
\]
where $v\in\overline{\kappa}$. 
The $n$-th entry of $\varphi_{n,M}(\overline{X}+\epsilon \overline{Y})$ is described as follows:
\[\mathrm{det}(\overline{X}+\epsilon \overline{Y})=
(-1)^{2n}\cdot \overline{c_n/\pi^{d_n}}
-\epsilon\overline{z}_1\overline{z}_3'\cdot v.
\]
Then the equation 
$w_n=-\overline{z}_1\overline{z}_3'\cdot v$ has a solution $v$ since 
$\overline{z}_1$ and $\overline{z}_3'$ are units.
\end{enumerate}
\end{proof}

\subsubsection{Formula for $\mathcal{SO}_{\gamma, M}$}

\begin{proposition}\label{cor:610}
We have the following formula:
\[ 
\mathcal{SO}_{\gamma, M}
=\begin{cases}
    \frac{\#\mathrm{U}_{n-2}(\kappa)}{q^{(n-2)^2+(2n-2)d_n -d_{n-1}}} &\textit{if $(\widetilde{F}, \epsilon)=(E, 1)$ and $d_{n-1}<d_n$};\\
    \frac{\#\mathrm{U}_{n-2}(\kappa)\cdot (q-1)}{q^{(n-2)^2+(2n-3)d_n+1}} &\textit{if $(\widetilde{F}, \epsilon)=(E, 1)$ and $d_{n-1}\geq d_n$};\\
\frac{\#\mathrm{Sp}_{2n-2}(\kappa)\cdot(q-1)}{q^{2n^2-3n+2+(2n-1)d_n}}   &\textit{if $(\widetilde{F}, \epsilon)=(F, -1)$}.
\end{cases}
\]
\end{proposition}

\begin{proof}
    Note that generic fibers of $\widetilde{\mathcal{L}}(L,M)$ and $\mathcal{A}_M$ are isomorhpic to $\mathfrak{g}_{F}$ and $\mathcal{A}^n_F$ respectively.
    Let $\omega_{\widetilde{\cL}(L,M)}$ and $\omega_{\mathcal{A}_{M}}$ be nonzero,  translation-invariant forms on   $\mathfrak{g}_{F}$ and $\mathcal{A}^n_F$
  with normalizations
\[
\int_{\widetilde{\cL}(L,M)(\mathfrak{o})}|\omega_{\widetilde{\cL}(L,M)}|=1 \textit{ and }  \int_{\mathcal{A}_{M}(\mathfrak{o})}|\omega_{\mathcal{A}_{M}}|=1.
\]
Let $\omega^{\mathrm{ld}}_{(\chi_{\gamma},\widetilde{\cL}(L,M))}=\omega_{\widetilde{\cL}(L,M)}/\rho_{n}^{\ast}\omega_{\mathcal{A}_{M}}$ be a differential on $G_{\gamma}$. 

Recall from Section \ref{subsubsection:measure} that 
 $\omega_{\mathfrak{g}_{\mathfrak{o}}}$ and $\omega_{\mathcal{A}^n_{\mathfrak{o}}}$ are invariant forms on $\mathfrak{g}_{F}$ and $\mathcal{A}_{F}^n$ with normalizations
\[
\int_{\mathfrak{g}(\mathfrak{o})} |\omega_{\mathfrak{g}_{\mathfrak{o}}}| = 1
\textit{ and }
\int_{\mathcal{A}^n(\mathfrak{o})} |\omega_{\mathcal{A}^n_{\mathfrak{o}}}| = 1.
\]

In order to compare two volume forms $\omega_{\mathfrak{g}_{\mathfrak{o}}}$ and $\omega_{\widetilde{\cL}(L,M)}$, it suffices to compute $[\mathfrak{g}(\mathfrak{o}):\widetilde{\cL}(L,M)(\mfo)]$ since both are affine spaces of the same dimension over $\mfo$ by Lemma \ref{lem:64affine}.
Using Equations (\ref{defL_1})-(\ref{eq:eqs_about_X_sp2n}),  we have the following  comparison among differentials;
\[
  \begin{array}{l l}
\left\{
  \begin{array}{l }
|\omega_{\mathfrak{u}_{n,\mathfrak{o}}}|=|\pi|^{(2n-1)d_n}|\omega_{\widetilde{\cL}(L,M)}|;\\
|\omega_{\mathcal{A}^n_{\mathfrak{o}}}|=
|\pi|^{\widetilde{d}_{n-1}+d_n} |\omega_{\mathcal{A}_{M}}|;\\
|\omega_{\chi_{\gamma}}^{\mathrm{ld}}|=|\pi|^{(2n-2)d_n-\widetilde{d}_{n-1}}|\omega^{\mathrm{ld}}_{(\chi_{\gamma},\widetilde{\cL}(L,M))}|;
    \end{array} \right. & \textit{if $(\widetilde{F},\epsilon)=(E,1)$};\\
\left\{\begin{array}{l}
|\omega_{\mathfrak{sp}_{2n,\mfo}}|=|\pi|^{2n d_n}  |\omega_{\widetilde{\cL}(L,M)}|;\\
     |\omega_{\mathcal{A}_{\mfo}^n}|= |\pi|^{d_n}|\omega_{\mathcal{A}_M}|;\\
     |\omega_{\chi_\gamma}^{\mathrm{ld}}|=|\pi|^{(2n-1)d_n}|\omega^{\mathrm{ld}}_{(\chi_\gamma,\widetilde{\cL}(L,M))}|.
\end{array}
\right.
& \textit{if $(\widetilde{F},\epsilon)=(F,-1)$}.
    \end{array} 
\]

On the other hand, Proposition \ref{corcounting} yields the following formula:
\[
\int_{O_{\gamma, \cL(L,M)}}|\omega^{\mathrm{ld}}_{(\chi_{\gamma},\widetilde{\cL}(L,M))}|=
\begin{cases}
    \frac{\#\vpi_{n,M}^{-1}(\chi_{\gamma})(\kappa)}{q^{n^2-n}}=\frac{\#\mathrm{U}_{n-2}(\kappa)}{q^{(n-2)^2}} 
    &\textit{if $(\widetilde{F}, \epsilon)=(E, 1)$ and $d_{n-1}<d_n$};\\   
    \frac{\#\vpi_{n,M}^{-1}(\chi_{\gamma})(\kappa)}{q^{n^2-n}}=\frac{\#\mathrm{U}_{n-2}(\kappa)\cdot (q-1)}{q^{(n-2)^2+1}} &\textit{if $(\widetilde{F}, \epsilon)=(E, 1)$ and $d_{n-1}\geq d_n$};\\
    \frac{\#\vpi_{n,M}^{-1}(\chi_{\gamma})(\kappa)}{q^{2n^2}}=\frac{\#\mathrm{Sp}_{2n-2}(\kappa)\cdot(q-1)}{q^{2n^2-3n+2}}
    &\textit{if $(\widetilde{F}, \epsilon)=(F, -1)$}
\end{cases}
\]
by \cite[Theorem 2.2.5]{Weil} since $\vpi_{n,M}^{-1}(\chi_{\gamma})$ is smooth over $\mfo$ with 
$\vpi_{n,M}^{-1}(\chi_{\gamma})(\mfo)=O_{\gamma, \cL(L,M)}$ and thus 
$\omega^{\mathrm{ld}}_{(\chi_{\gamma},\widetilde{\cL}(L,M))}$ is a nowhere vanishing differential of top degree on $\vpi_{n,M}^{-1}(\chi_{\gamma})$ over $\mfo$.
Here, the dimension of the special fiber of $\vpi_{n,M}^{-1}(\chi_{\gamma})$ is  $n^2-n$ if $(\widetilde{F}, \epsilon)=(E, 1)$ and $2n^2$ if $(\widetilde{F}, \epsilon)=(F, -1)$.
Therefore,
\[\mathcal{SO}_{\gamma, M}
=\int_{O_{\gamma, \cL(L,M)}}|\omega_{\chi_{\gamma}}^{\mathrm{ld}}|
=
\begin{cases}
  q^{-(2n-2)d_n + \widetilde{d}_{n-1}}\cdot\int_{O_{\gamma, \cL(L,M)}}|\omega^{\mathrm{ld}}_{(\chi_{\gamma},\widetilde{\cL}(L,M))}| &\textit{if } (\widetilde{F}, \epsilon)=(E, 1) \\
    q^{-(2n-1)d_n}\cdot \int_{O_{\gamma, \cL(L,M)}}|\omega^{\mathrm{ld}}_{(\chi_{\gamma},\widetilde{\cL}(L,M))}| &\textit{if } (\widetilde{F}, \epsilon)=(F, -1)
\end{cases}.
\]
This directly yields the desired formula. 
\end{proof}

Thus, Equation (\ref{equation:stradn}) yields the following formula.
\begin{theorem}\label{theorem:closedformulafordn}
We have the following formula:
\[
  \sum_{M:\T(M)=(d_n)} \mathcal{SO}_{\gamma, M}
    =\begin{cases}
    \#\mathcal{S}^{(E,1)}_{(d_{n-1},2d_{n}-d_{n-1})}\cdot\frac{\#\mathrm{U}_{n-2}(\kappa)}{q^{(n-2)^2+(2n-2)d_n -d_{n-1}}}
    &\textit{if $(\widetilde{F}, \epsilon)=(E, 1)$ and $d_{n-1}<d_n$};\\   
    \#\mathcal{S}^{(E,1)}_{(d_n, d_n)}\cdot\frac{\#\mathrm{U}_{n-2}(\kappa)\cdot (q-1)}{q^{(n-2)^2+(2n-3)d_n+1}}
    &\textit{if $(\widetilde{F}, \epsilon)=(E, 1)$ and $d_{n-1}\geq d_n$};\\
\#\mathcal{S}^{(F,-1)}_{(d_n,d_n)}\cdot \frac{\#\mathrm{Sp}_{2n-2}(\kappa)\cdot(q-1)}{q^{2n^2-3n+2+(2n-1)d_n}}
    &\textit{if $(\widetilde{F}, \epsilon)=(F, -1)$}.
\end{cases}
\]
Here $\mathcal{S}^{(\widetilde{F},\epsilon)}_{(l,2d_n-l)}$ defined to be 
\[
\mathcal{S}^{(\widetilde{F},\epsilon)}_{(l,2d_n-l)} := \{M \subset L\mid \T(M)=(d_{n}), ~~~~~~~~ \JT(M)=(l,2d_n-l)\}.
\]

\end{theorem}

The formula for $\# \mathcal{S}^{(\widetilde{F},\epsilon)}_{(l,2d_n-l)}$ will be computed in Proposition \ref{cor:Sab} below.
Using this,  Theorem \ref{theorem:closedformulafordn} is rewritten more concretely  as follows:

\begin{theorem}\label{thm:SO_dn_type}
    For a regular semisimple $\gamma\in \mathfrak{g}(\mfo)$ such that $\chi_\gamma(x)$ is irreducible over $\widetilde{F}$ and $\overline{\chi}_\gamma(x)=x^{\wn}$,
    we have
    \[
    \sum_{M:\T(M)=(d_n)} \mathcal{SO}_{\gamma, M}
    =
    \left\{
    \begin{array}{l l}
    \frac{\#\mathrm{U}_n(\kappa)}{(1+q^{-1})q^{n^2}}&\textit{if } (\widetilde{F},\epsilon) = (E,1);\\
    \frac{\#\mathrm{Sp}_{2n}(\kappa)}{q^{n(2n+1)}}
    &\textit{if } (\widetilde{F},\epsilon) = (F,-1).
    \end{array}
    \right.
    \]
\end{theorem}
Note that this formula holds for an arbitrary local field $F$ of any characteristic.

\subsubsection{Formula for $\#\mathcal{S}^{(\widetilde{F},\epsilon)}_{(a,b)}$}\label{subsubsec:Sab}
\begin{proposition}\label{cor:Sab}
 We have the following formula for  $\#\mathcal{S}^{(\widetilde{F},\epsilon)}_{(l,2d-l)}$:
 \[\#\mathcal{S}^{(\widetilde{F},\epsilon)}_{(l,2d-l)}=\left\{\begin{array}{l l}
    \frac{(q^n-(-1)^{n})(q^{n-1}-(-1)^{n-1})}{q+1}\cdot q^{2(d-1)(n-1)-l}&\textit{if $(\widetilde{F},\epsilon)=(E,1)$ and 
    $1\leq l\leq d-1$};\\
    \frac{(q^n-(-1)^n)(q^{n-1}-(-1)^{n-1})}{q^2-1}\cdot q^{(d-1)(2n-3)}&\textit{if $(\widetilde{F},\epsilon)=(E,1)$ and $l=d$};\\
    \frac{q^{2n}-1}{q-1}\cdot q^{(2n-1)(d-1)}&\textit{if $(\widetilde{F},\epsilon)=(F,-1)$ and $l=d$}.
    \end{array}
    \right.
    \]
\end{proposition}
The proof is postponed to the end of this subsection.
We will first investigate the description of a basis for a sublattice $M$ in $L$ of type $(d)$ and then will study its Jordan type.

\begin{lemma}\label{lema1cl}
Let $M \subset L$ be a sublattice of  type $(d)$. Choose a basis  $(e_1,\cdots,e_{\wn})$ for $L$.
    \begin{enumerate}
        \item 
        There exists an integer $1\leq k\leq \widetilde{n}$ such that  the following $\wn$-tuple forms a basis for $M$:
    \begin{equation}\label{basisformallcase} 
(e_1+a_1e_k,\cdots, e_{k-1}+a_{k-1}e_k,\pi^d e_k,e_{k+1}+\pi a_{k+1}e_k,\cdots, e_{\wn}+\pi a_{\wn}e_{k})
    \end{equation}
    where $a_i \in \mfo_{\widetilde{F}}$ for $1 \leq i(\neq k) \leq \wn$. 
    

    \item For two sublattices $M, M' \subset L$ with $\mathcal{T}(M) = \mathcal{T}(M') = (d)$, consider 
    \[
    \begin{cases}
        \textit{a basis for } M : (e_1+a_1e_k,\cdots, e_{k-1}+a_{k-1}e_k,\pi^d e_k,e_{k+1}+\pi a_{k+1}e_k,\cdots, e_{\wn}+\pi a_{\wn}e_{k});  \\
    \textit{a basis for } M' : (e_1+a'_1e_{k'},\cdots, e_{k'-1}+a'_{k'-1}e_{k'},\pi^d e_{k'},e_{k'+1}+\pi a'_{k'+1}e_{k'},\cdots, e_{\wn}+\pi a'_{\wn}e_{k'}).     
    \end{cases}
    \]
    Then $M = M'$ if and only if 
    $\left\{
    \begin{array}{l l}
        k = k';& \\
        a_i \equiv a'_i \textit{ modulo } \pi^d &\textit{ for } 1\leq i \leq k-1; \\
        \pi a_i \equiv \pi a'_i \textit{ modulo } \pi^d &\textit{ for } k+1\leq i \leq \widetilde{n}.
    \end{array}\right.$
    \end{enumerate}
\end{lemma}

\begin{proof}
\begin{enumerate}
    \item 
We define a map
\[
   \phi : \{M \subset L  \mid \mathcal{T}(M) = (d)\} \longrightarrow \mathrm{Gr}_{\kappa_{\widetilde{F}}}(\wn-1,\wn),\
    M \mapsto M/(M\cap \pi L)
\] where $\mathrm{Gr}_{\kappa_{\widetilde{F}}}(\wn-1,\wn)$ is the Grassmannian variety classifying $(\wn-1)$-dimensional subspaces of a $\wn$-dimensional vector space $L/\pi L$ over $\kappa_{\widetilde{F}}$.
Our strategy is to find a basis for $\phi(M)$ and then lift to $M$. 

We denote by $\overline{e}_i$ the image of $e_i$ under the quotient map $L\rightarrow L/\pi L$.
We will first analyze a basis for  $W\in \mathrm{Gr}_{\kappa_{\widetilde{F}}}(\wn-1,\wn)$ with respect to $(\overline{e}_1, \cdots, \overline{e}_{\wn})$. 
If we consider the matrix $X$ of size $\wn\times (\wn-1)$ whose columns consist of a basis of $W$, then the set of column vectors of any matrix resulting from column operations on $X$ forms a basis for $W$ as well.
Then it is easy to see that 
$W$ is uniquely determined by the basis of the  following form:
\begin{equation}\label{basisforgrassmanian}
   (\overline{e}_1+\alpha_1 \overline{e}_k ,\cdots, \overline{e}_{k-1}+\alpha_{k-1}\overline{e}_k,\overline{e}_{k+1} ,\cdots, \overline{e}_{\wn}) ~~~~ \textit{ for } 1\leq k \leq \wn.
\end{equation}
Here  $\overline{e}_k$ is missing. 
In other words, $\mathrm{Gr}_{\kappa_{\widetilde{F}}}(\wn-1,\wn)$ is completely characterized by  $1\leq k\leq \wn$ and $(\alpha_1, \cdots, \alpha_{k-1})$ with $\alpha_i\in \kappa_{\widetilde{F}}$.
Note that this observation directly yields the well-known formula that
$\#\mathrm{Gr}_{\kappa_{\widetilde{F}}}(\wn-1,\wn)=\frac{\widetilde{q}^{\wn}-1}{\widetilde{q}-1}$, where $\widetilde{q}=\#\kappa_{\widetilde{F}}$.

Choose $M (\subset L)$ of type $(d)$ such that $\phi(M)=W$, where $W$ is spanned by a basis of the form in Equation (\ref{basisforgrassmanian}). 
Then $M$ contains $\wn-1$ elements in $L$ of the following forms;
\[
(e_1+\widetilde{\alpha}_1e_k+v_1,\cdots, e_{k-1}+\widetilde{\alpha}_{k-1}e_k+v_{k-1}, e_{k+  1}+v_{k+1}, \cdots, e_{\wn}+v_{\wn})
\]
where $v_j \in \pi L$ for $1\leq j \leq \wn$ with $j\neq k$ and
where $\widetilde{\alpha}_i \in \mfo_{\widetilde{F}}$ is a lift of $\alpha_i$ for $1\leq i \leq k-1$.
Note that the above $\wn-1$ vectors are linearly independent over $\mfo_{\widetilde{F}}$ since their reductions modulo $\pi$ are.

On the other hand, $M$ contains $\pi^d e_k$ since $\pi^d L \subset M$.
Then as in Equation (\ref{basisforgrassmanian}),  we consider the matrix $\widetilde{X}$ consisting of the columns of the above vectors together with $\pi^d e_k$. 
By applying a finite sequence of column operations on $\widetilde{X}$ and by using the fact that $v_j \in \pi L$, 
the following $\wn$ vectors are linearly independent and are contained in $L$:
\[
(e_1+a_1e_k,\cdots, e_{k-1}+a_{k-1}e_k,\pi^d e_k,e_{k+1}+\pi a_{k+1}e_k,\cdots, e_{\wn}+\pi a_{\wn}e_{k}),
\]
where $a_i\in \mfo_{\widetilde{F}}$ with $1\leq i (\neq k) \leq \wn$ such that $\bar{a}_i=\alpha_i$ for $1\leq i \leq k-1$.
The exponential order of the determinant of the matrix whose columns consist of these vectors is $d$. 
Therefore, these form a basis for $M$.
\\

\item
Suppose that  $k = k'$, 
        $a_i \equiv a'_i$ modulo $\pi^d$ for  $1\leq i \leq k-1$, and 
        $\pi a_i \equiv \pi a'_i$ modulo  $\pi^d$ for  $k+1\leq i \leq \widetilde{n}$.
Since $M$ contains $\pi^d e_k$, all vectors in a basis for  $M'$ are contained in $M$ and vice versa.
This yields that $M=M'$.

Conversely  suppose that $M=M'$.
Then $\phi(M)=\phi(M')$ so that  $k=k'$ and $\bar{a}_i=\bar{a}_i'$ for $1\leq i \leq k-1$, since 
Equation (\ref{basisforgrassmanian}) characterizes $\phi(M)$ uniquely (cf. the proof of the above (1)).
On the oher hand, $M(=M')$ contains the following vectors:
\[\left\{
\begin{array}{l l}
(e_i+a_ie_k)-(e_i+a_i'e_k)=
(a_i-a_i')e_k
\in  \mfo_{\widetilde{F}} \cdot e_k &\textit{for $1\leq i\leq k-1$ and};\\
(e_i+\pi a_i e_k)-(e_i+\pi a_i' e_k)=
(\pi a_i-\pi a_i')e_k
\in \mfo_{\widetilde{F}} \cdot e_k &\textit{for $k+1\leq i\leq \wn$.}
\end{array}\right.
\]
If $a_i \not\equiv a'_i$ modulo $\pi^d$ for some $1\leq i \leq k-1$, then $M$ contains $\pi^{\beta}e_k$ with $\beta<d$. This is a contradiction since $[L;M]\leq \beta$.
The same argument also works with $\pi a_i-\pi a_i'$ for for $k+1\leq i\leq \wn$. 
This completes the proof. 
\end{enumerate}
\end{proof}

We choose a basis for $L$ and the associated Gram matrix for $h$ as follows: 
\[\left\{\begin{array}{l l}
   (e_1,\cdots,e_{\wn}) \textit{ and } h=Id_{\wn} & \textit{if $(\widetilde{F}, \epsilon)=(E, 1)$};
    \\
(\underbrace{e_1,e_1'}_{1},\cdots, \underbrace{e_n,e_n'}_n)  \textit{ and }  h=\begin{pmatrix}
    0&1\\
    -1&0
\end{pmatrix}^n
& \textit{if $(\widetilde{F}, \epsilon)=(F, -1)$}.
\end{array}
\right.
\]

In the case that $(\widetilde{F}, \epsilon)=(F, -1)$,  Lemma \ref{lema1cl}
is interpreted as one of the following forms:
\begin{equation}\label{basisforsympl}
\left\{\begin{array}{l}
( \underbrace{e_1+a_{1}e_k , e_{1}'+a_{1}'e_k}_1,\cdots,\underbrace{\pi^{d}e_k,e_{k}'+\pi a_{k}' e_k}_k,\cdots, \underbrace{e_n+\pi a_n e_k,e_n'+\pi a_{n}'e_k}_n);
    \\
(\underbrace{e_1+a_{1}e_k' , e_{1}'+a_{1}'e_k'}_1,\cdots, \underbrace{e_k+a_ke_k',\pi^d e_{k}'}_k,\cdots, \underbrace{e_n+\pi a_{n}e_k',e_n'+\pi a_{n}'e_k'}_n).
\end{array}
\right.
\end{equation}

\begin{proposition}\label{jtofalllattice}
Let $M \subset L$ be a sublattice of type $(d)$ corresponding to a basis described in Equation (\ref{basisformallcase}). 
The Jordan type of $M$ is described as follows:
\[
\JT(M)=\left\{\begin{array}{l l}
  (l,2d-l) & \textit{if $(\widetilde{F}, \epsilon)=(E, 1)$};  \\
(d,d) & \textit{if $(\widetilde{F}, \epsilon)=(F, -1)$}.
\end{array}
\right.
\]
Here $l:=
min\{\mathrm{ord}(1+\sum\limits_{i=1}^{k-1}a_i\sigma(a_i)+\pi^2\sum\limits_{i=k+1}^{n} a_i\sigma(a_i))
,d\}$.
\end{proposition}

\begin{proof}
    \begin{enumerate}
        \item For the case $(\widetilde{F},\epsilon)=(E,1)$, the Gram matrix of $h|_M$ with respect to Equation (\ref{basisformallcase}) is 
        \begin{equation}\label{matrixformsyme}
\begin{pmatrix}       
    1+a_{1}\sigma(a_{1}) &a_{1}\sigma(a_{2}) &\cdots & a_{1}\pi^d &  a_{1}\sigma(a_{k+1})\pi&\cdots\\
    a_{2}\sigma(a_{1}) & 1+a_{2}\sigma(a_{2})& \cdots & a_{2}\pi^d&a_{2}\sigma(a_{k+1})\pi&\cdots\\
    \vdots&\vdots&\ddots&\vdots&\vdots&\ddots\\
    \sigma(a_{1})\pi^d&\sigma(a_{2})\pi^d&\cdots&\pi^{2d}&\sigma(a_{k+1})\pi^{d+1}&\cdots\\
    a_{k+1}\sigma(a_{1})\pi&a_{k+1}\sigma(a_{2})\pi&\cdots &a_{k+1}\pi^{d+1} &1+ a_{k+1}\sigma(a_{k+1})\pi^2&\cdots\\
    \vdots &\vdots&\ddots &\vdots&\vdots&\ddots
\end{pmatrix}. \end{equation}
        Let $M'$ be the sublattice of $M$ spanned by \[(e_1+a_1e_k,\cdots, e_{k-1}+a_{k-1}e_k,e_{k+1}+\pi a_{k+1}e_k,\cdots, e_{\wn}+\pi a_{\wn}e_{k}). \] 
        Here $\pi^d e_k$ is missing so that the rank of $M'$ is $\wn-1$. 
        Then the Gram matrix of $h|_{M'}$ is the $(k,k)$-minor of Equation (\ref{matrixformsyme}). 
        We denote it by $h_{k,k}$.
        The determinant of $h_{k,k}$ is equal to the value of the characteristic polynomial of 
        $I-h_{k,k}$ at $1$.
        Since the rank of $I-h_{k,k}$ is $1$, its characteristic polynomial is given by $x^{n}+\mathrm{Tr}(h_{k,k}-I)x^{n-1}=x^{n}+(\sum_{i=1}^{k-1}a_i\sigma(a_i)+\pi^2\sum_{i=k+1}^{\wn} a_i\sigma(a_i))x^{n-1}$. Therefore we have
       \[
        \mathrm{det}(h_{k,k})=
        1+\sum_{i=1}^{k-1}a_i\sigma(a_i)+\pi^2\sum_{i=k+1}^{\wn} a_i\sigma(a_i).
        \] 
        We denote  by $m$ the exponential order of $\mathrm{det}(h_{k,k})$. 
       On the other hand, the reduction of $h_{k,k}$ modulo $\pi$ is of   rank $\geq \wn-2$ since   $(L,h)$ is unimodular. 
       Then by  \cite[Section 7]{Jac}, there exists a basis for $M'$ where 
        $h_{k,k}$ is represented by the matrix $h_{k,k}=\begin{pmatrix}
    \pi^m &0\\
    0&Id_{n-2}
\end{pmatrix}$. 
By adding the element $\pi^{d} e_k$ to this basis, we have a basis for $M$ whose associated Gram matrix  is 
 $h|_{M}=
\begin{pmatrix}
    \pi^m &0&\pi^d b_1\\
    0&Id_{\wn-2}& \pi^d b_2\\
    \pi^d\sigma(^t b_1)&\pi^d \sigma(^t b_2)&\pi^{2d}b'
\end{pmatrix}.$ Here, $b_2,b'$ have entries in $\mfo_E$ and $b_1\in \mfo_E^{\times}$.

To simplify this Gram matrix,
we let  $g=\begin{pmatrix}
    1&0&0\\
    0&Id_{\wn-2}&-\pi^{d}b_2\\
    0&0&1
\end{pmatrix}$ (cf. the proof of Lemma \ref{modif_matr}).
Then  
$\sigma({}^t g )\cdot h|_M\cdot g=\begin{pmatrix}
    \pi^m &0&\pi^d b_1\\
    0&Id_{\wn-2}& 0\\
    \pi^d\sigma(^t b_1)&0&\pi^{2d}b
\end{pmatrix}$.
This yields that 
        \[\mathcal{JT}(M)=\left\{\begin{array}{l l}
        (m,2d-m) &\textit{if $m<d$ and};\\
        (d,d) &\textit{if $m\geq d$}
        \end{array}\right.
        \]
by using the argument of the proof of Lemma $\ref{lem:Jordan_of_type_m}$.
\\

        \item For the case $(\widetilde{F},\epsilon)=(F,-1)$,
        we will treat   the first case of Equation (\ref{basisforsympl}), as the same argument works for the second case.
     The Gram matrix of $h|_M$ is 
$$    \begin{pmatrix}
        0&1&\cdots &0&a_1&\cdots&0&0\\
        -1&0&\cdots&0&a_1'&\cdots&0&0\\
        \vdots&\vdots&\ddots&\vdots&\vdots&\ddots&\vdots&\vdots\\
        0&0&\cdots &0&\pi^d&\cdots&0&0\\
        -a_1&-a_1'&\cdots&-\pi^d&0&\cdots&-\pi a_n&-\pi a_n'\\
        \vdots&\vdots&\ddots&\vdots&\vdots&\ddots&\vdots&\vdots\\
        0&0&\cdots &0&\pi a_n&\cdots&0&1\\
        0&0&\cdots&0&\pi a_n'&\cdots&-1&0
    \end{pmatrix}.$$
If we consider the rank of the reduction of this matrix modulo $\pi$, then its rank is at least $2n-2$ and at most $2n-1$ and thus it should be $2n-2$. Since the exponential order of the determinant of $h|_M$ is $2d$, we conclude that $\mathcal{JT}(M)=(d,d)$ (cf. Definition \ref{def:tjt}. (2)). 
     \end{enumerate}
\end{proof}

The following lemma is needed  in the case that $(\widetilde{F},\epsilon)=(E,1)$.
\begin{lemma}\label{refined_norm}
    Let $a$ be an element in $ \mfo_E^{\times}$. Then the following map
    \[
    \mathrm{Nm}_{d,\overline{a}}:
    \{x\in \mfo_E/\pi^d\mfo_E \mid x\equiv a\  \textit{modulo }\pi\} 
    \rightarrow 
    \{x'\in \mfo/\pi^{d}\mfo\mid x'\equiv a \sigma(a)\ \textit{modulo }\pi\},\ x\mapsto x\sigma(x)
    \]
    is surjective and the cardinality of each fiber is  $q^{d-1}$.
\end{lemma}
\begin{proof}
We consider the following commutative diagram
\[
\begin{tikzcd}
\{x\in \mfo_E/\pi^d\mfo_E \mid x\equiv 1\  \textit{modulo }\pi\} \arrow[r,"\mathrm{Nm}"]\arrow[d, "\textit{multiplication by } a"]&
\{x'\in \mfo/\pi^{d}\mfo\mid x'\equiv 1\ \textit{modulo }\pi\}
\arrow[d,"\textit{multiplication by } a\sigma(a)"]\\
    \{y\in \mfo_E/\pi^d\mfo_E \mid y\equiv a\  \textit{modulo }\pi\} 
    \arrow[r,"{\mathrm{Nm}_{d,\overline{a}}}"]&
    \{y'\in \mfo/\pi^{d}\mfo\mid y'\equiv a \sigma(a)\ \textit{modulo }\pi\}.
\end{tikzcd}
\]
Two vertical maps are bijective since both allow the inverse maps, as $a\in \mfo_E^\times$ and $a\sigma(a)\in \mfo^\times$.
In addition, the first horizontal map $\mathrm{Nm}$ is a surjective group homomorphism by \cite{Ser}[Proposition V.3].
Thus $\mathrm{Nm}_{d,\overline{a}}$ is surjective and  the cardinality of each fiber is  $q^{2(d-1)}/q^{d-1}=q^{d-1}$.
\end{proof}

\begin{proof}[Proof of Proposition \ref{cor:Sab}]
In the case that $(\widetilde{F},\epsilon)=(F,-1)$,  Proposition \ref{jtofalllattice} yields that  any sublattice $M$ of type $(d)$ in $L$ satisfies $\mathcal{JT}(M)=(d,d)$. By Lemma \ref{counttype}, we then have
\[
\#\mathcal{S}_{(d,d)}^{(F,-1)}=\#\mathrm{Gr}_{\kappa}(2n-1,2n)\cdot q^{(2n-1)(d-1)}=\frac{q^{2n}-1}{q-1}\cdot q^{(2n-1)(d-1)}.
\]

Suppose that $(\widetilde{F},\epsilon)=(E,1)$. 
By Lemma \ref{lema1cl} and Proposition \ref{jtofalllattice}, $\#\mathcal{S}^{(E,1)}_{(l,2d-l)}$ is equal to the number of  $(n-1)$-tuples $(a_1,\cdots,a_{k-1},\pi a_{k+1},\cdots,\pi a_n)$ satisfying 
\begin{equation}\label{countingcond}
\left\{\begin{array}{l}
1\leq k\leq n;\\  
a_i \in \mfo_E/\pi^d\mfo_E    \textit{     for $1\leq i\leq k-1$};\\ 
\pi a_i \in \pi\mfo_E/\pi^d\mfo_E    \textit{     for $k+1\leq i\leq n$};\\
l=min\{\mathrm{ord}(1+\sum\limits_{i=1}^{k-1}a_i\sigma(a_i)+\pi^2\sum\limits_{i=k+1}^{n} a_i\sigma(a_i))
,d\}.
\end{array}\right.
\end{equation}
Note that since $l>0$, we have $1+\sum\limits_{i=1}^{k-1}\overline{a}_i\sigma(\overline{a}_i)= 0$ and that $k\geq 2$.

    Fix $k$ with $2\leq k\leq n$ and $(\alpha_1,\cdots,\alpha_{k-1})$ with $\alpha_i \in\kappa_E$ such that $1+\sum\limits_{i=1}^{k-1}\alpha_i\sigma(\alpha_i)=0$.   
    We will firstly compute the number of the $(n-1)$-tuples $(a_1,\cdots,a_{k-1},\pi a_{k+1},\cdots, \pi a_n)$ satisfying Equation (\ref{countingcond}) such that $\overline{a}_i=\alpha_i$ for $1\leq i\leq k-1$.
    Since $1+\sum\limits_{i=1}^{k-1}\alpha_i\sigma(\alpha_i)= 0$,  at least one $\alpha_i$ is nonzero.
    Thus, we may assume that $\overline{a}_1=\alpha_1\neq 0$.
    We consider the following map from Lemma \ref{refined_norm}:
    \[
        \mathrm{Nm}_{d,\alpha_1}:
    \{x\in \mfo_E/\pi^d\mfo_E \mid x\equiv \alpha_1\  \textit{modulo }\pi\} 
    \rightarrow 
    \{x'\in \mfo/\pi^{d}\mfo\mid x\equiv \alpha_1 \sigma(\alpha_1)\ \textit{modulo }\pi\},\ x\mapsto x\sigma(x)
    \]
    \begin{itemize}
        \item For the case that $l<d$, the condition that 
        $l=\mathrm{ord}(1+\sum\limits_{i=1}^{k-1}a_i\sigma(a_i)+\pi^2\sum\limits_{i=k+1}^{n} a_i\sigma(a_i))$
 is equivalent to the condition that
        \[
        a_1 \in \mathrm{Nm}_{d,\alpha_1}^{-1}\left(-(1+\sum\limits_{i=2}^{k-1}a_i\sigma(a_i)+\sum\limits_{i=k+1}^{n} \pi a_i\sigma(\pi a_i))+\pi^l\mfo^{\times} \right).
        \] 
        By Lemma \ref{refined_norm}, the number of such $a_1$'s, for fixed $(a_2,\cdots,a_{k-1},\pi a_{k+1},\cdots, \pi a_{n})$,  is $(q-1)\cdot q^{(d-1)-l}\cdot q^{d-1}$.
The number of $(a_2,\cdots,a_{k-1},\pi a_{k+1},\cdots, \pi a_{n})$'s with $\overline{a}_i =\alpha_i$ for $1\leq i\leq k-1$ is $(q^{2(d-1)})^{n-2}$. 
Thus the number of $(n-1)$-tuples $(a_1,\cdots,a_{k-1},\pi a_{k+1},\cdots,\pi a_n)$ satisfying 
Equation (\ref{countingcond}) with fixed $k$ and $(\alpha_1,\cdots,\alpha_{k-1})$ is 
        \[
        (q-1)\cdot q^{2(d-1)-l}\cdot (q^{2(d-1)})^{n-2}=(q-1)\cdot q^{2(d-1)(n-1)-l}.
        \]

        \item
        For the case that $l=d$, the condition that 
        $\mathrm{ord}(1+\sum\limits_{i=1}^{k-1}a_i\sigma(a_i)+\pi^2\sum\limits_{i=k+1}^{n} a_i\sigma(a_i))\geq d$
         is equivalent to the condition that
        \[
        a_1 \in \mathrm{Nm}_{d,\alpha_1}^{-1}\left(-(1+\sum\limits_{i=2}^{k-1}a_i\sigma(a_i)+\sum\limits_{i=k+1}^n \pi a_i\sigma(\pi a_i))\right).
        \] 
As in the above case,   the number of $(n -1)$-tuples $(a_1,\cdots,a_{k-1},\pi a_{k+1},\cdots,\pi a_n)$ satisfying Equation (\ref{countingcond}) with fixed $k$ and $(\alpha_1,\cdots, \alpha_{k-1
        })$ is 
        \[
        q^{d-1}\cdot (q^{2(d-1)})^{n-2}=q^{d-1}\cdot q^{2(d-1)(n-2)}.
        \]
    \end{itemize}

Let $\mathrm{NmGr}(k):=\{(\alpha_1,\cdots,\alpha_k)\in \kappa_E^k \mid 1+\alpha_1\sigma(\alpha_1)+\cdots +\alpha_{k}\sigma(\alpha_k)=0\}$ for $1\leq k \leq n-1$.
Then the desired counting number is 
    \[\#\mathcal{S}_{(l,2d-l)}^{(E,1)}=
    \left\{\begin{array}{l l}
    (\sum\limits_{k=1}^{n-1}\# \mathrm{NmGr}(k))\cdot (q-1)\cdot q^{2(d-1)(n-1)-l} &\textit{if $0<l<d$};\\
    (\sum\limits_{k=1}^{n-1}\# \mathrm{NmGr}(k))\cdot q^{d-1}\cdot q^{2(d-1)(n-2)}&\textit{if $l=d$}.
    \end{array}
    \right.
    \]
    
We claim that the number $\# \mathrm{NmGr}(k)$ satisfies  the following recurrence relation:
    \[
\#\mathrm{NmGr}(k)=\#\mathrm{NmGr}(k-1)+(q^{2(k-1)}-\#\mathrm{NmGr}(k-1))\cdot \# N^1_{\kappa_E/\kappa}(\kappa),
    \]
    where $\#\mathrm{NmGr}(1)=\#N^1_{\kappa_E/\kappa}(\kappa)=q+1$.
For an element $(\alpha_1,\cdots, \alpha_{k})\in \mathrm{NmGr}(k)$, if $(\alpha_1,\cdots, \alpha_{k-1})\in \mathrm{NmGr}(k-1)$, then $\alpha_k=0$.
 Otherwise, we have that $\alpha_k\sigma(\alpha_k)=-(1+\alpha_1\sigma(\alpha_1)+\cdots+\alpha_{k-1})\neq 0$.
The number of such $(\alpha_1,\cdots,\alpha_{k-1})$'s is $q^{2(k-1)}-\#\mathrm{NmGr}(k-1)$ and  the number of $a_k$'s, for such  a fixed $(\alpha_1,\cdots,\alpha_{k-1})$,  is $\#N^{1}_{\kappa_E/\kappa}(\kappa)$. This directly yields the above recurrence relation.

We then inductively obtain that
    \begin{align*}
\sum\limits_{k=1}^{n-1}\# \mathrm{NmGr}(k)&=(q+1)\left(
\sum_{k=1}^{n-1}(-q)^{k-1}\cdot \frac{q^{2(n-k)}-1}{q^{2}-1}
\right)
=\frac{(q^n-(-1)^{n})(q^{n-1}-(-1)^{n-1})}{q^2-1}.
\end{align*}
\end{proof}


\subsection{A lower bound  for $\mathcal{SO}_{\gamma}$}

    Recall that Theorem \ref{thm:SO_dn_type} gives the volume of the strata $\bigsqcup\limits_{M : \mathcal{T}(M) = (d_n)} \mathcal{O}_{\gamma,\mathcal{L}(L,M)}$, which serves as  a lower bound for $\mathcal{SO}_\gamma$.
    

    We will weaken the assumption in Theorem \ref{thm:SO_dn_type} by using the reduction steps (cf. Section \ref{sec:reduction}) and the parabolic descent (cf. Section \ref{sec:parabolicdescent}).
    Here we apply the reduction step only for the case $(\widetilde{F},\epsilon) = (E,1)$ since the translation in Section \ref{sec:inv} does not work for the case $(\widetilde{F},\epsilon) = (F,-1)$ (cf. Remark \ref{noinvariancerforsp}).
    Using Corollary \ref{red:result2} and Lemma \ref{lem:invarianct_translation} in the case $(\widetilde{F},\epsilon)=(E,1)$, Theorem \ref{thm:SO_dn_type} yields the following result:

    \begin{theorem}\label{thm:lowerbound_geom}
    Suppose that $char(F) = 0$ or $char(F)>n$.
For a regular semisimple element  $\gamma \in \mathfrak{g}(\mfo)$  such that $\chi_\gamma(x)$ is irreducible over $\widetilde{F}$, we have the following bounds:
\[ 
\mathcal{SO}_{\gamma}>
\begin{cases}
    \frac{\#\mathrm{U}_n(\kappa)}{(1 + q^{-l})q^{n^2}} &\textit{if $(\widetilde{F}, \epsilon)=(E, 1)$};\\
\frac{\#\mathrm{Sp}_{2n}(\kappa)}{q^{n(2n+1)}}   &\textit{if $(\widetilde{F}, \epsilon)=(F, -1)$ and $\overline{\chi}_\gamma(x) = x^{2n}$},
\end{cases}
\]
\[ 
\mathcal{SO}_{\gamma,d\mu}>
\begin{cases}
   q^{S(\gamma)} \cdot \frac{1+q^{-d}}{1+q^{-l}} &\textit{if $(\widetilde{F}, \epsilon)=(E, 1)$};\\
q^{S(\gamma)-S(\psi)} \cdot (1+q^{-d})   &\textit{if $(\widetilde{F},\epsilon) = (F,-1)$,  $\widetilde{F}_{\chi_\gamma}/F_{\chi_\gamma}^\sigma$ is unramified, and $\overline{\chi}_\gamma(x) = x^{2n}$};\\
q^{S(\gamma)-S(\psi)} \cdot 2   &\textit{if $(\widetilde{F},\epsilon) = (F,-1)$,  $\widetilde{F}_{\chi_\gamma}/F_{\chi_\gamma}^\sigma$ is ramified, and $\overline{\chi}_\gamma(x) = x^{2n}$}.
\end{cases}
\]
Here $d$ is the inertial degree of $F_{\chi_\gamma}^\sigma/F$ and $l = [\kappa_R:\kappa]$ (cf. Diagram (\ref{diag_reduction})). 
We refer to Definition \ref{def:Serre_inv} for the notion of $S(\gamma)$ and $S(\psi)$.
    \end{theorem}
    Note that if $(\widetilde{F},\epsilon) = (F,-1)$,  $\widetilde{F}_{\chi_\gamma}/F_{\chi_\gamma}^\sigma$ is ramified, and $char(\kappa) > 2$, then $\overline{\chi}_\gamma(x) = x^{2n}$ always so that the above theorem is applicable (cf. Remark \ref{rmk:reduction_cond}).

    We extend Theorem \ref{thm:lowerbound_geom} by using the parabolic descent.
    Recall the following factorization 
    \[
    \psi(y) = \prod_{i\in B(\gamma)}\psi_i(y)=\left(\prod_{i\in B(\gamma)^{irred}}\psi_i(y)\right) \times \left(\prod_{i\in B(\gamma)^{split}}\psi_i(y)\right),
    \]
where $\psi(y)$ is a polynomial in $\mfo[y]$ related to $\chi_{\gamma}(x)$ (cf. Section \ref{desc_field}) and
$y$ is explained in Equation (\ref{equation:yxx2}).
With respect to this factorization, 
 we may and do write 
$\gamma = \gamma_0 \oplus \bigoplus\limits_{i \in B(\gamma)^{split}} \gamma_i$
where $\gamma_0 \in \mathfrak{g}_{m}(\mathfrak{o})$ and $\gamma_i = \begin{pmatrix} g_i & 0 \\ 0 & -{}^t\sigma(g_i)   \end{pmatrix}$ with $g_i \in \mathfrak{gl}_{l_i,\mathfrak{o}_{\widetilde{F}}}(\mathfrak{o}_{\widetilde{F}})$
(cf. Section \ref{section:matdescofr}), where $\mathfrak{g}_{m}(\mathfrak{o})$ is explained in Section \ref{sec:desc_Levi}. 
Here we need to suppose that $char(\kappa)>2$ if $(\widetilde{F},\epsilon) = (E,1)$  and $n$ is even,  and no restriction otherwise due to the existence of a Kostant section (cf. Section \ref{section:matdescofr}). 
The elements $\gamma_0$ and $\gamma_i$ have the following interpretation:
\[
\begin{cases}
    \chi_{\alpha \gamma_0}(x) = \prod\limits_{i\in B(\gamma)^{irred}}\psi_i(x) \textit{ and } \chi_{\alpha \gamma_i}(x) = \psi_i(x) \textit{ for each } i \in B(\gamma)^{split} &\textit{if } (\widetilde{F},\epsilon) = (E,1); \\
    \chi_{\gamma_0}(x) = \prod\limits_{i\in B(\gamma)^{irred}}\psi_i(x^2) \textit{ and } \chi_{\gamma_i}(x) = \psi_i(x^2) \textit{ for each } i \in B(\gamma)^{split} &\textit{if } (\widetilde{F},\epsilon) = (F,-1).
\end{cases}
\]
Here $2l_i = \deg \psi_i(x)$ for each $i \in B(\gamma)^{split}$.

    By applying Theorem \ref{thm:lowerbound_geom} to $\gamma_0 \in \mathfrak{g}_{m}(\mathfrak{o})$ and Theorem \ref{genlb1} to $g_i \in \mathfrak{gl}_{m_i,\mathfrak{o}_{\widetilde{F}}}(\mathfrak{o}_{\widetilde{F}})$ with $i \in B(\gamma)^{split}$,  the parabolic descent in Proposition  \ref{pro:par_des_sorb} and the comparison of two measures in Proposition \ref{prop:comparison_measures} yield the following lower bounds.
    
    \begin{theorem}\label{thm:lowerbound_SO}
        Suppose that $char(F)=0$ or $char(F)>n$ and that  $B(\gamma)^{irred}$ is a singleton. 
        We further assume that $char(\kappa)>2$ if $(\widetilde{F},\epsilon) = (E,1)$, $n$ is even, and $\#B(\gamma) \geq 2$ (cf. Section \ref{section:matdescofr}).
        Then we have the following lower bounds:
        \[ 
\mathcal{SO}_{\gamma}>
\begin{cases}
    \frac{\#\mathrm{U}_n(\kappa)}{(1+q^{-l})q^{n^2}} \prod\limits_{i \in B(\gamma)^{split}} N'_{g_{i}}(q^{2d_{i}}) &\textit{if }(\widetilde{F}, \epsilon)=(E, 1);\\
\frac{\#\mathrm{Sp}_{2n}(\kappa)}{q^{n(2n+1)}} \prod\limits_{i \in B(\gamma)^{split}}N'_{g_{i}}(q^{d_{i}})   &\textit{if $(\widetilde{F}, \epsilon)=(F, -1)$ and $\overline{\chi}_\gamma(x) = x^{2n}$},
\end{cases}
\]
\[ 
\mathcal{SO}_{\gamma,d\mu}>
\begin{cases}
  q^{\rho(\gamma)} \cdot \frac{1 + q^{-d}}{1+q^{-l}} \cdot\prod\limits_{i \in B(\gamma)^{split}} N'_{g_{i},d\mu}(q^{2d_{i}}) &\textit{if $(\widetilde{F}, \epsilon)=(E, 1)$};\\
q^{\rho(\gamma) - S(\psi)} \cdot (1 + q^{-d}) \cdot \prod\limits_{i \in B(\gamma)^{split}} N'_{g_{i},d\mu}(q^{d_{i}})   &\textit{if $(\widetilde{F},\epsilon) = (F,-1)$,  $\widetilde{F}_{\chi_\gamma}/F_{\chi_\gamma}^\sigma$unramified, $\overline{\chi}_\gamma(x) = x^{2n}$};\\
q^{\rho(\gamma) - S(\psi)} \cdot 2 \cdot\prod\limits_{i \in B(\gamma)^{split}} N'_{g_{i},d\mu}(q^{d_{i}})   &\textit{if $(\widetilde{F},\epsilon) = (F,-1)$,  $\widetilde{F}_{\chi_\gamma}/F_{\chi_\gamma}^\sigma$ramified, $\overline{\chi}_\gamma(x) = x^{2n}$}.
\end{cases}
\]
Here the notations are as follows:
\[
\begin{cases}
    d : \textit{ inertial degree of }F^\sigma_{\chi_{\gamma_0}}/F; \\
    l = [\kappa_{(\mathfrak{o}_E(x)/\chi_{\gamma_0}(x))^\sigma}:\kappa] ~ (= \textit{the degree of the irreducible factor of } \chi_{\alpha\gamma_0}(x)) ~ (\textit{cf. Diagram (\ref{diag_reduction})});  \\
    S(\gamma), S(\psi) : \textit{ Serre invariants (cf. Definition \ref{def:Serre_inv})} ;\\
    \rho(\gamma) = S(\gamma) - \sum\limits_{i \in B(\gamma)^{split}} S(g_i).
\end{cases}
\]
    The other notations $d_i$, $N'_{g_i}(-)$, $N'_{g_i,d\mu}(-)$, and $S(g_i)$  are referred to Equation (\ref{eq:Ngamma'})-(\ref{eq:notations_rho_gamma}) for each $i \in B(\gamma)^{split}$ (by replacing $\gamma \in \mathfrak{gl}_n(\mathfrak{o})$ with $g_i \in \mathfrak{gl}_{l_i}(\mathfrak{o}_{\widetilde{F}})$ so that $B(g_i)$ is a singleton).
    
\end{theorem}

\section{Closed formula for $\mathcal{SO}_{\gamma}$ in $\mathfrak{u}_{2}$ when $\#B(\gamma)$=1}\label{sectionformulaforu2}

In this section, we will work with the case that $(\widetilde{F}, \epsilon)=(E, 1)$ and that $n=2$. 
The condition that $\# B(\gamma)=1$ is equivalent that $\chi_{\alpha\gamma}(x)$ is irreducible over $F$, equivalently $F_{\chi_\gamma}^\sigma/F$ is a quadratic field extension (cf. Section \ref{desc_field}).
More precisely we have that 
\[\#B(\gamma)=\left\{\begin{array}{l l}
\#B(\gamma)^{split}=1 & \textit{ if $F_{\chi_\gamma}^\sigma/F$ is unramified};\\
\#B(\gamma)^{irred}=1 & \textit{ if $F_{\chi_\gamma}^\sigma/F$ is ramified}.
\end{array}\right.
\]

Applying  Proposition \ref{charred3} to $\chi_{\alpha\gamma}(x) (= \alpha^2\cdot \chi_{\gamma}(x/\alpha))\in \mfo[x]$, it is easy to see that there exists an element $a\in \mfo\cdot \alpha^{-1}$ 
such that
\[\chi_\gamma(x+a)=x^{2}+c_{1,a}x+c_{2,a} \textit{ where }\ord(c_{2,a})\equiv
\left\{\begin{array}{l l}
     0 \textit{ modulo 2}& \textit{ if $F_{\chi_\gamma}^\sigma/F$ is unramified};  \\
     1\textit{ modulo 2}& \textit{ if $F_{\chi_\gamma}^\sigma/F$ is ramified}.
\end{array}
\right.
\]
Here $\alpha \in \mfo_E^{\times}$ such that $\alpha+\sigma(\alpha)=0$ in the proof of Lemma \ref{lem:c+sigma(c)_bij}. 
Note that when 
$F_{\chi_\gamma}^\sigma/F$ is unramified, $\chi_{\gamma}(x)$ is not irreducible over $E$ but $\chi_{\alpha\gamma}(x)$ is irreducible over $F$. 

When $F_{\chi_\gamma}^\sigma/F$ is ramified, we denote $\mathrm{ord}(c_{2,a})$ by $d_\gamma$.
Then we claim that $S(\gamma)=\frac{d_{\gamma}-1}{2}$, which implies the independency of $d_\gamma$ with respect to the choice of $a$.
 Proposition \ref{da2} yields that $\frac{d_{\gamma}-1}{2}=S(\psi)$, which is equal to $S(\gamma)$ as explained in Definition \ref{def:Serre_inv}. 


\begin{theorem}\label{theorem12.1sorn=2}
Suppose that $char(F)=0$ or $char(F)>2$. 
For $\gamma \in \mathfrak{u}_2(\mfo)$, we have the following formula:
\[
\left(\mathcal{SO}_{\gamma},  \mathcal{SO}_{\gamma,d\mu}\right)=
\left\{\begin{array}{l l}
\left(\frac{q+1}{q},  q^{S(\gamma)}\right) & \textit{ if $F_{\chi_\gamma}^\sigma/F$ is unramified};\\
\left(\frac{(q+1)(q^{S(\gamma)+1}-1)}{q^{S(\gamma)+2}}, \frac{q^{S(\gamma)+1}-1}{q-1}\right) & \textit{ if $F_{\chi_\gamma}^\sigma/F$ is ramified}.
\end{array}\right.
\]
\end{theorem}
     \begin{proof}
         \begin{enumerate}
             \item 
If $F_{\chi_\gamma}^\sigma/F$ is unramified, then $ F_{\chi_\gamma}^{\sigma}\cong E$ and thus $\widetilde{F}_{\chi_\gamma}\cong F_{\chi_\gamma}^{\sigma}\times F_{\chi_\gamma}^{\sigma}$.
Proposition \ref{pro:par_des_sorb} then yields that 
\[
\mathcal{SO}_{\gamma}=\frac{\#\mathrm{U}_{2}(\kappa)\cdot q^{-4}}{\#\mathrm{GL}_{1}(\kappa_E)\cdot q^{-2}}\cdot 1=\frac{q+1}{q}.
\]
Then by Proposition \ref{prop:comparison_measures}, 
\[
\mathcal{SO}_{\gamma,d\mu}=  |N_{F_{\chi_\gamma}^{\sigma}/F}(\Delta_{\widetilde{F}_{\chi_\gamma}/F_{\chi_\gamma}^{\sigma}})|^{1/2}|\cdot|\Delta_{F_{\chi_\gamma}^\sigma/F}|^{1/2}\cdot |D(\gamma)|^{-1/2}\cdot\frac{\#\underline{\mathrm{T}}_{\gamma}(\kappa)\cdot q^{-2}}{\#\mathrm{U}_2(\kappa)\cdot q^{-4}}\cdot \frac{q+1}{q}.
\]
Here $\#\underline{\mathrm{T}}_{\gamma}(\kappa)=q^2-1$   by Corollary \ref{lem:pts_spfib_tori},
$|N_{F_{\chi_\gamma}^{\sigma}/F}(\Delta_{\widetilde{F}_{\chi_\gamma}/F_{\chi_\gamma}^{\sigma}})|=1$ by Lemma \ref{lem:Conductor_unitary}, and $|\Delta_{F_{\chi_\gamma}^\sigma/F}|=1$ since $F_{\chi_\gamma}^\sigma/F$ is unramified.
Remark \ref{remark:sordmu} yields that 
$|D(\gamma)|^{-1/2}=|\Delta_{\gamma}|^{-1/2}=|\Delta_{\alpha\gamma}|^{-1/2}$, which is equal to $q^{S(\alpha\gamma)}$ by Proposition \ref{propserre}. 
Here $\Delta_{\gamma} (\in \mfo)$ is the discriminant of $\chi_{\gamma}(x)$ (cf. Notations of Part 2). 
The integer $S(\alpha\gamma)$ was defined to be  $S(\psi)$ in Definition \ref{def:Serre_inv}, which is the same as $S(\gamma)$. 

Combining these, we have the desired formula.

\item
If $F_{\chi_\gamma}^\sigma/F$ is ramified, then  
  Equation (\ref{st2_unsp}) and Proposition \ref{propendred_unsp} yield that
  \begin{align*}
    \mathcal{SO}_\gamma 
    =\sum_{k_1=0}^{(d_\gamma-1)/2}\sum_{\mathcal{T}(M) = (k_1,d_\gamma-k_1)} SO_{\gamma,M}
    =\sum_{k_1=0}^{(d_\gamma-1)/2}q^{-k_1}\sum_{\mathcal{T}(\pi^{-k_1}M) = (d_\gamma - 2k_1)} SO_{\gamma^{(k_1)},\pi^{-k_1}M},
  \end{align*}
  where $\chi_{\gamma^{(k_1)}}(x) = x^2 + c_1/\pi^{k_1} x + c_2/\pi^{2k_1}$.
Note that $S(\gamma)=\frac{d_{\gamma}-1}{2}$.  
Theorem \ref{thm:SO_dn_type} then yields that
  \[
  \mathcal{SO}_\gamma= \sum_{k_1=0}^{S(\gamma)}q^{-k_1}\cdot \frac{\#\mathrm{U}_2(\kappa)}{(q+1)q^{3}}
  =q^{-S(\gamma)} \cdot \frac{q^2-1}{q^2} \cdot \frac{q^{S(\gamma)+1}-1}{q-1}.
  \]
  By applying Corollary \ref{cor:comp_irr}, we have
  \[
  \mathcal{SO}_{\gamma,d\mu}=q^{S(\gamma)}\cdot\frac{1+q^{-1}}{\#\mathrm{U}_2(\kappa)q^{-4}}\cdot \frac{(q+1)(q^{S(\gamma)+1}-1)}{q^{S(\gamma)+2}}=\frac{q^{S(\gamma)+1}-1}{q-1}.
  \]
         \end{enumerate}
\end{proof}

\bibliographystyle{alpha}
\bibliography{References}

\begin{thebibliography}{EvdGM}

\bibitem[AFV18]{AFV}
Jeffrey~D. Adler, Jessica Fintzen, and Sandeep Varma.
\newblock {On Kostant sections and topological nilpotence}.
\newblock {\em J. Lond. Math. Soc.}, 97(2):325--351, 2018.

\bibitem[BC22]{BC}
Alexis Bouthier and Kęstutis Cesnavičius.
\newblock {Torsors on loop groups and the Hitchin fibration}.
\newblock {\em Annales scientifiques de l'ENS}, pages 791--864, 2022.

\bibitem[BLR90]{BRL}
Siegfried Bosch, Werner L{\"u}tkebohmert, and Michel Raynaud.
\newblock {\em N{\'e}ron models}, volume~21.
\newblock Springer Science \& Business Media, 1990.

\bibitem[Bor91]{Bor91}
Armand Borel.
\newblock {\em Linear algebraic groups}, volume 126.
\newblock Springer New York, NY, 1991.

\bibitem[BT65]{BA}
Armand Borel and Jacques Tits.
\newblock Groupes r\'{e}ductifs.
\newblock {\em Inst. Hautes \'{E}tudes Sci. Publ. Math.}, (27):55--150, 1965.

\bibitem[Cho15]{C1}
Sungmun Cho.
\newblock Group schemes and local densities of quadratic lattices in residue
  characteristic 2.
\newblock {\em Compositio Mathematica}, 151(5):793--827, 2015.

\bibitem[Cho16]{C2}
Sungmun Cho.
\newblock {Group schemes and local densities of ramified hermitian lattices in
  residue characteristic 2 Part I}.
\newblock {\em Algebra \& Number Theory}, 10(3):451--532, 2016.

\bibitem[Cho18a]{C3}
Sungmun Cho.
\newblock {A Uniform Construction of Smooth Integral Models and a Conjectural
  Recipe for Computing Local Densities}.
\newblock {\em International Mathematics Research Notices},
  2018(12):3870--3907, 2018.

\bibitem[Cho18b]{C2'}
Sungmun Cho.
\newblock {Group schemes and local densities of ramified hermitian lattices in
  residue characteristic 2. Part II}.
\newblock In {\em Forum Mathematicum}, volume~30, 2018.

\bibitem[CL]{CL}
Sungmun Cho and Yuchan Lee.
\newblock {An explicit formula for the orbital integrals on the spherical Hecke
  algebra of $\mathrm{GL}_3$, preprint available at
  \url{https://arxiv.org/abs/2404.04666}}.

\bibitem[CR10]{CR}
Pierre-Emmanuel Chaput and Matthieu Romagny.
\newblock {On the adjoint quotient of Chevalley groups over arbitrary base
  schemes}.
\newblock {\em J. Inst. Math. Jussieu}, (4):673--704, 2010.

\bibitem[CY20]{CY}
Sungmun Cho and Takuya Yamauchi.
\newblock {A reformulation of the Siegel series and intersection numbers}.
\newblock {\em Mathematische Annalen}, 377(3):1757--1826, 2020.

\bibitem[EvdGM]{EGM}
Bas Edixhoven, Gerard van~der Geer, and Ben Moonen.
\newblock {Abelian varieties, preprint available at
  \url{http://van-der-geer.nl/~gerard/AV.pdf}}.

\bibitem[FLN10]{FLN}
Edward Frenkel, Robert Langlands, and B\'{a}o~Ch\^{a}u Ng\^{o}.
\newblock {Formule des traces et fonctorialit\'{e}: le d\'{e}but d'un
  programme}.
\newblock {\em Ann. Sci. Math. Qu\'{e}bec}, 34(2):199--243, 2010.

\bibitem[Gek03]{Gek}
Ernst-Ulrich Gekeler.
\newblock Frobenius distributions of elliptic curves over finite prime fields.
\newblock {\em International Mathematics Research Notices},
  2003(37):1999--2018, 2003.

\bibitem[GG99]{GG99}
Benedict~H. Gross and Wee~Teck Gan.
\newblock {Haar measure and the Artin conductor}.
\newblock {\em Trans. Amer. Math. Soc.}, 351(4):1691--1704, 1999.

\bibitem[GH19]{GH19}
Jayce~R Getz and Heekyoung Hahn.
\newblock An introduction to automorphic representations with a view towards
  trace formulae.
\newblock {\em Graduate Studies in Mathematics}, 6, 2019.

\bibitem[GHY01]{GHY}
Wee~Teck Gan, Jonathan~P. Hanke, and Jiu-Kang Yu.
\newblock {On an exact mass formula of Shimura}.
\newblock {\em Duke Math. J.}, 107(1):103--133, 2001.

\bibitem[Gor22]{Gor22}
Julia Gordon.
\newblock Orbital integrals and normalizations of measures.
\newblock {\em arXiv preprint arXiv:2205.02391}, 2022.

\bibitem[Gro05]{Gro05}
Benedict~H. Gross.
\newblock On the centralizer of regular, semi-simple, stable conjugacy class.
\newblock {\em Represent. Theory}, 9:287--296, 2005.

\bibitem[GY00]{GY}
Wee~Teck Gan and Jiu-Kang Yu.
\newblock Group schemes and local densities.
\newblock {\em Duke mathematical journal}, 105(3):497--524, 2000.

\bibitem[HM07]{HM}
Geir~T Helleloid and Ursula Martin.
\newblock The automorphism group of a finite p-group is almost always a
  p-group.
\newblock {\em Journal of Algebra}, 312(1):294--329, 2007.

\bibitem[Hum72]{Hum72}
James~E. Humphreys.
\newblock {\em {Introduction to Lie algebras and representation theory}},
  volume Vol. 9 of {\em Graduate Texts in Mathematics}.
\newblock Springer-Verlag, New York-Berlin, 1972.

\bibitem[Hum95]{Hum}
James~E Humphreys.
\newblock {\em Conjugacy classes in semisimple algebraic groups}.
\newblock Number~43. American Mathematical Soc., 1995.

\bibitem[Hum17]{Hum17}
James~E Humphreys.
\newblock Notes on regular unipotent and nilpotent elements, available at
  \url{https://people.math.umass.edu/~jeh/pub/regular.pdf}.
\newblock 2017.

\bibitem[Igu00]{Igu}
Jun-ichi Igusa.
\newblock {\em An introduction to the theory of local zeta functions}.
\newblock Number~14. American Mathematical Soc., 2000.

\bibitem[Jac62]{Jac}
Ronald Jacobowitz.
\newblock Hermitian forms over local fields.
\newblock {\em American Journal of Mathematics}, 84(3):441--465, 1962.

\bibitem[Kac94]{Kac}
Victor Kac.
\newblock {Invariant Theory, available at
  \url{https://people.kth.se/~laksov/notes/invariant.pdf}}.
\newblock 1994.

\bibitem[Kot05]{Kot05}
Robert~E Kottwitz.
\newblock {Harmonic analysis on reductive p-adic groups and Lie algebras}.
\newblock In {\em Harmonic analysis, the trace formula, and Shimura varieties},
  volume~4, pages 393--522. Citeseer, 2005.

\bibitem[KP23]{KP23}
Tasho Kaletha and Gopal Prasad.
\newblock {\em {Bruhat-{T}its theory---a new approach}}, volume~44 of {\em New
  Mathematical Monographs}.
\newblock Cambridge University Press, Cambridge, 2023.

\bibitem[Lee21]{Lee21}
Jungin Lee.
\newblock {Counting algebraic tori over $\mathbb{Q}$ by Artin conductor}.
\newblock {\em arXiv preprint arXiv:2104.02855}, 2021.

\bibitem[Lee23]{Lee23}
Yuchan Lee.
\newblock On a kostant section for the unitary group.
\newblock {\em preprint}, 2023.

\bibitem[LN08]{LN}
G\'{e}rard Laumon and Bao~Ch\^{a}u Ng\^{o}.
\newblock {Le lemme fondamental pour les groupes unitaires}.
\newblock {\em Ann. of Math. (2)}, 168(2):477--573, 2008.

\bibitem[LZ22a]{LZ1}
Chao Li and Wei Zhang.
\newblock {Kudla--Rapoport cycles and derivatives of local densities}.
\newblock {\em Journal of the American Mathematical Society}, 35(3):705--797,
  2022.

\bibitem[LZ22b]{LZ2}
Chao Li and Wei Zhang.
\newblock {On the arithmetic Siegel--Weil formula for GSpin Shimura varieties}.
\newblock {\em Inventiones mathematicae}, 228(3):1353--1460, 2022.

\bibitem[Mac98]{Mac}
Ian~Grant Macdonald.
\newblock {\em {Symmetric functions and Hall polynomials}}.
\newblock Oxford university press, 1998.

\bibitem[Mil17]{Mil}
James~S Milne.
\newblock {Algebraic number theory (v3. 07)}, 2017.

\bibitem[Neu99]{Neu}
J\"{u}rgen Neukirch.
\newblock {\em Algebraic number theory}, volume 322.
\newblock Springer-Verlag, Berlin, 1999.
\newblock Translated from the 1992 German original and with a note by Norbert
  Schappacher, With a foreword by G. Harder.

\bibitem[Ono87]{Ono}
Takashi Ono.
\newblock {On some class number relations for {G}alois extensions}.
\newblock {\em Nagoya Math. J.}, 107:121--133, 1987.

\bibitem[Poo17]{Poonen}
Bjorn Poonen.
\newblock {\em Rational points on varieties}, volume 186 of {\em Graduate
  Studies in Mathematics}.
\newblock American Mathematical Society, Providence, RI, 2017.

\bibitem[Ser79]{Ser}
Jean-Pierre Serre.
\newblock {\em Local fields}, volume~67 of {\em Graduate Texts in Mathematics}.
\newblock Springer-Verlag, New York-Berlin, 1979.
\newblock Translated from the French by Marvin Jay Greenberg.

\bibitem[Wei12]{Weil}
Andr{\'e} Weil.
\newblock {\em Adeles and algebraic groups}, volume~23.
\newblock Springer Science \& Business Media, 2012.

\bibitem[Xia18]{Xiao18}
Jingwei Xiao.
\newblock {Endoscopic transfer for unitary Lie algebras}.
\newblock {\em arXiv preprint arXiv:1802.07624}, 2018.

\bibitem[Yun11]{Yun11}
Zhiwei Yun.
\newblock {The fundamental lemma of {J}acquet and {R}allis}.
\newblock {\em Duke Math. J.}, 156(2):167--227, 2011.
\newblock With an appendix by Julia Gordon.

\bibitem[Yun13]{Yun13}
Zhiwei Yun.
\newblock Orbital integrals and {D}edekind zeta functions.
\newblock In {\em The legacy of {S}rinivasa {R}amanujan}, volume~20 of {\em
  Ramanujan Math. Soc. Lect. Notes Ser.}, pages 399--420. Ramanujan Math. Soc.,
  Mysore, 2013.

\bibitem[Yun16]{Yun16}
Zhiwei Yun.
\newblock {Lectures on Springer theories and orbital integrals}.
\newblock {\em In Geometry of Moduli Spaces and Representation Theory. IAS/Park
  City Mathematics Series}, 2016.

\end{thebibliography}

\end{document}